\documentclass[12pt]{article}
\usepackage{graphicx}
\usepackage{amsmath}
\usepackage{amsthm}
\usepackage{graphics}
\usepackage{epsfig}
\usepackage{amssymb}
\usepackage{comment} 

\usepackage[toc,page]{appendix}
\usepackage{xcolor}

 \usepackage{setspace}
\usepackage{extarrows}

\usepackage[authoryear]{natbib}
\usepackage{subfig}

\usepackage{pdflscape}
\usepackage{subfig}
\usepackage{enumitem}
\allowdisplaybreaks

\newcommand{\A}{\mathbf{A}}

\newcommand{\R}{\mathbb{R}}
\newcommand{\lb}{\left(}
\newcommand{\rb}{\right)}

\newcommand{\cov}{\operatorname{cov}}
\newcommand{\E}{\mathbb{E}}
\newcommand{\N}{\mathbb{N}}
\newcommand{\ntn}{\lfloor nt \rfloor}
\newcommand{\nt}{\lfloor nt_1 \rfloor}
\newcommand{\ntt}{\lfloor nt_2 \rfloor}

\newcommand{\cond}{\stackrel{\mathcal{D}}{\to}}

\newcommand{\conp}{\stackrel{\mathbb{P}}{\to}}
\newcommand{\tr}{\operatorname{tr}}
\newcommand{\D}{\mathbf{D}}
\newcommand{\rd}{\mathbf{r}}
\newcommand{\T}{\mathbf{T}}
\newcommand{\inv}{^{-1}}
\newcommand{\su}{\underline{s}}
\newcommand{\sq}{^{\frac{1}{2}}}
\newcommand{\im}{\operatorname{Im}}
\newcommand{\re}{\operatorname{Re}}
\newcommand{\PR}{\mathbb{P}}

\newcommand{\sut}{\tilde{\underline{s}}}

\newcommand{\bfSigma}{\mathbf{\Sigma}}
\newcommand{\bfx}{\mathbf{x}}

\newtheorem{theorem}{Theorem}[section]
\newtheorem{lemma}{Lemma}[section]
\newtheorem{proposition}{Proposition}[section]
\newtheorem{corollary}{Corollary}[section]

\newtheorem{example}{Example}[section]

\newtheorem{remark}{Remark}[section]

\numberwithin{equation}{section}

\allowdisplaybreaks

\begin{document}
\title{Linear spectral statistics of sequential 
sample covariance matrices 
}
% in high dimension with an application to monitoring sphericity}
% \title{Monitoring sphericity in large dimension via sequential processes of linear spectral statistics}
\author{
{\small Nina D\"ornemann, Holger Dette} \\
{\small Fakult\"at f\"ur Mathematik} \\
{\small Ruhr-Universit\"at Bochum} \\
{\small 44801 Bochum, Germany} \\
}
\maketitle

\begin{abstract}
Let $\mathbf{x}_1, \ldots , \mathbf{x}_n$ denote independent
$p$-dimensional vectors with independent complex or real valued entries
such that $\mathbb{E} [\mathbf{x}_i] = \mathbf{0}$, ${\rm Var } (\mathbf{x}_i) = \mathbf{I}_p$, $i=1, \ldots,n$, let $\mathbf{T }_n$ be a $p \times p$ Hermitian nonnegative definite matrix and $f  $ be a given function.  We prove that an approriately standardized version of the stochastic process   $  \big ( {\tr} ( f(\mathbf{B}_{n,t}) ) \big   )_{t \in [t_0, 1]}  $
corresponding to a linear spectral statistic
of the sequential empirical covariance estimator 
$$
\big ( \mathbf{B}_{n,t}  )_{t\in [ t_0 ,   1]}  =  \Big ( \frac{1}{n}    \sum_{i=1}^{\lfloor n t \rfloor}  \mathbf{T }^{1/2}_n   \mathbf{x}_i   \mathbf{x}_i ^\star \mathbf{T }^{1/2}_n \Big)_{t\in [ t_0 ,   1]}  
$$
converges weakly to a non-standard Gaussian process for $n,p\to\infty$. As an application we use these results 
to develop a novel approach for monitoring the sphericity assumption  in a high-dimensional framework, even if the dimension of the underlying data is larger than the sample size.

\end{abstract}
Keywords: linear spectral statistic, sequential sample covariance matrix, sequential process, 
sphericity test, Stieltjes transform, monitoring spherictiy.

AMS subject classification:  Primary 15A18, 60F17; Secondary 62H15

\section{Introduction} \label{sec1}

Estimation and testing  of a high-dimensional covariance matrix is a fundamental problem of statistical inference with numerous  applications in biostatistics, wireless communications and finance
 (see, e.g., \cite{fanli2006}, \cite{Johnstone} and the references therein). Linear spectral statistics are frequently used to construct 
 tests   for various hypotheses. For example,  \cite{mauchy1940}  proposes a likelihood ratio test for the hypothesis of sphericity (of a normal distribution), which 
  has been extended by \cite{guptaxu2006} to the non-normal case  and by  \cite{baietal2009}  and \cite{wang2013} to the high-dimensional case, where the dimension $p$ is of the same order as the sample size $n$, that is $p/n \to y \in (0,1)$ as $p,n \to \infty$ (see also  Theorem 9.12 in the monograph of \cite{yao2015} for a   further extension).
  Alternative tests based on  distances between the sample covariance matrix and a multiple of the identity matrix
  have been considered in  \cite{ledoit_wolf_2002} and  \cite{chenzhangzhong2010} among others. 
  \cite{fisher2010} suggest a generalization of 
 John's test  for sphericity, which is based on   a ratio of arithmetic means of   the eigenvalues of different powers of  the sample covariance matrix. 
  Among other testing problems such as sphericity, \cite{jiangyang2013} consider some classical $q$-sample testing problems under normality in a high-dimensional setting, which are further generalized in \cite{dettedoernemann2020} for an increasing number $q$ of groups.  
  Other authors concentrate on linear spectral statistics of $F$-matrices 
\citep[see, for example,][]{zheng2012,zhengbaiyao2017,botdetpar2019}, auto-cross covariance \citep{jinwangbaietal2014}, large-dimensional matrices with bivariate dependence measures as entries 
\citep{baolinpanzhou2015,li2019central} or information-plus-noise matrices \citep{banna2020}.
  
Because of its importance in statistics numerous authors have investigated the asymptotic properties of linear spectral statistics from a more general perspective.
An early reference is  \cite{johnson1982} and  in their pioneering paper, \cite{baisilverstein2004} proved a central limit theorem for linear spectral statistics of the form
$$
\sum_{i=1}^p  f(\lambda_i(\mathbf{B}_n)) 
$$
of the sample covariance matrices  $\mathbf{B}_n = \frac{1}{n}    \sum_{i=1}^n \mathbf{T }^{1/2}_n  \mathbf{x}_i   \mathbf{x}_i ^\star \mathbf{T }^{1/2}_n  $ under rather general conditions,  where   $ \mathbf{x}_1 , \ldots,  \mathbf{x}_n$  are independent   $p $-dimensional  random vectors 
with independent 
real or complex valued (centered) entries  $x_{ij}$,  $ \mathbf{T }_n$ is a $p \times p$ (non-random) Hermitian nonnegative definite matrix
and $\lambda_1(\mathbf{B}_n)  \leq  \ldots \leq   \lambda_p(\mathbf{B}_n)$ are the ordered eigenvalues of the matrix  $\mathbf{B}_n$.
Several authors have followed this line of research and tried to relax the assumptions for such statements
\citep[see][among others]{panzhou2008,lytovapastur2009,pan2014,zheng_et_al_2015,najimyao2016}.  

In this paper we will take a  different point of view on linear spectral statistics and study these objects from a sequential perspetive.  More precisely,
we consider a sequential version of the empirical   covariance estimator
\begin{equation} \label{1.1}
\mathbf{B}_{n,t}  
%=   {\tr} \big (  f(\mathbf{B}_{n,t})  \big ) 
= \frac{1}{n}    \sum_{i=1}^{\lfloor n t \rfloor}  \mathbf{T }^{1/2}_n  \mathbf{x}_i   \mathbf{x}_i ^\star \mathbf{T }^{1/2}_n   ~, ~~0 \leq t \leq 1,
\end{equation} 
and  investigate the probabilistic properties of the stochastic process corresponding to linear spectral statistics of $\mathbf{B}_{n,t} $, that is 
\begin{equation} \label{1.2}
S_{t}  =  \frac{1}{p} \sum_{i=1}^p  f(\lambda_i (\mathbf{B}_{n,t} ) )    ~, ~~0 \leq t \leq 1,
\end{equation} 
where $\lambda_1 (\mathbf{B}_{n,t} )  \leq  \ldots \leq   \lambda_p (\mathbf{B}_{n,t} )$ are the ordered eigenvalues of  the matrix $\mathbf{B}_{n,t} $. In particular, we prove that 
for any $0 < t_0 < 1$, an appropriately normalized  and centered version of the process $(S_{t})_{t \in [ t_0 , 1]}$ converges weakly to non-standard Gaussian process.

Our interest in these processes is  partially motivated by the fact that the sequential covariance estimator plays a central role in the construction of methodology for the detection of structural breaks in 
the covariance structure  \citep[see][among others]{Aue2009b,dettegoesmann2020}. In this field various functionals of the process $( \mathbf{B}_{n,t} )_{~0 \leq t \leq 1}$
%  such as $ \big ( {\rm tr }  ( \hat \Sigma_{n,t} - t \hat \Sigma_{n,1} )^2 \big )_{0 \leq t \leq 1}$
have been   studied  in the case of fixed dimension, and we expect that results on the weak convergence of the process $(S_{t})_{t \in [ t_0 , 1]}$ 
will be useful in the context of change-point analysis 
for high-dimensional covariance covariance matrices. In fact, we use the probabilistic results  presented in this paper  to develop a procedure for monitoring deviations from sphericity,
see Section \ref{sec2} for more details.  

Surprisingly, sequential processes of the form \eqref{1.2} have not found much attention in the literature. To our best knowledge we
are only aware of the work of  \cite{aristotile2000} and \cite{nagel2020}, who  considered 
sequential aspects of large dimensional random matrices  
from a different point of view. More precisely, \cite{aristotile2000} studied a sequential process generated from
the first $\lfloor nt \rfloor $ diagonal elements of a random  matrix chosen according to the Haar measure on the unitary group of $n \times n$ matrices and showed that this process converges weakly to  a standard complex-valued Brownian motion \citep[see also][for similar results]{aristotileetal2003}. Recently, \cite{nagel2020} proved a functional central limit theorem for   
the sum of the first $\lfloor nt \rfloor $ diagonal elements of  an $n \times n$  matrix $f(Z)$, where
$Z$ has an orthogonal or unitarily invariant distribution such that ${\rm tr} \big ( f(Z) \big )$ satisfies a CLT.
Compared to these contributions  the results of  the present paper are conceptually different, because - in contrast to the cited references - the parameter $t$ 
used in the definition of the process \eqref{1.1} also appears in the eigenvalues $\lambda_i (\mathbf{B}_{n,t})$. This ``non-linearity''  results in a substantially more complicated structure of the problem. In particular, the limiting processes 
of $(S_t)_{t \in [ t_0, 1]}$ are non-standard Gaussian processes (except in the case $f(x)=x$), and the proofs of our results 
(in particular the proof of tightness) require an extended machinery, which has so far not been considered in the literature on linear spectral statistics. As a consequence  we provide a substantial generalization of the classical CLT 
for linear spectral statistics \citep[see, for example,][]{bai2004}, which is obtained  from the process convergence of  $(S_{t})_{t \in [ t_0 , 1]}$  (appropriately standardized) via continuous mapping.

\newpage

\section{A sequential look at linear spectral statistics} \label{sec3}	
Let $\bfx_1, \ldots, \bfx_n $ be independent $p$-dimensional random vectors with  real or complex entries and covariance matrix given by the identity matrix $\mathbf{I} = \mathbf{I}_p \in \R^{p \times p}$ . We use the notation  $\mathbf{x}_j = (x_{1j}, \ldots, x_{pj})^\top$ for the components of $\mathbf{x}_j$ and assume that $\E[x_{ij}]=0$ and $\E[x_{ij}^2] =1$. 
When considering asymptotics, the dimension $p = p_n$ of the data is allowed to increase with the sample size $n \to \infty $ at same order, that is, $p/n \to y \in (0,\infty)$ as $n\to\infty$. 
	Recall the notation of the sequential covariance estimator 	$\mathbf{B}_{n,t}$ in \eqref{1.1} and consider the  
 corresponding   linear spectral statistic (as a function of $t$)
	\begin{align*}
S_t = \frac{1}{p} \tr  \big ( f(\mathbf{B}_{n,t})
\big ) 
= 		\frac{1}{p} \sum\limits_{j=1}^p
		f \lb \lambda_i (\mathbf{B}_{n,t}) \rb, ~ t\in[0,1],
	\end{align*}
	where $f$  is an appropriate function defined on a subset of the complex plane.
	For a given $t_0 \in (0,1] $, we  are interested in the asymptotic properties  of the  process 
	$(S_t)_{t \in [t_0 ,1]}$ and will   prove a weak convergence result for an appropriately standardized version of  
	this process in  the space  $\ell^\infty ([t_0,1])$ of bounded functions defined on the interval $[t_0,1]$. Note that the random variable $S_1$ has been studied intensively in the literature (see the discussion in Section \ref{sec1}).

For the statement of our main result we require some 
notation.  
	Let  
	\begin{align*}
		F^{\mathbf{A}} = \frac{1}{p} \sum\limits_{j=1}^p \delta_{\lambda_j (\mathbf{A})},
	\end{align*}
	be the empirical spectral distribution of a $p\times p$ Hermitian  matrix $\mathbf{A}$,
where  $\lambda_1  (\mathbf{A}) , \ldots , \lambda_p  (\mathbf{A})$ are the eigenvalues of $\mathbf{A}$ 
(often the dependence on  $ \mathbf{A}$ is omitted in the notation, because it is clear from the context) and $\delta_a $ denotes the Dirac measure at a point $a \in \mathbb{R}$. 
A useful tool in random matrix theory is the Stieltjes transform 
	\begin{align*}
		s_{F}(z) = \int \frac{1}{\lambda - z} dF (\lambda)	
	\end{align*}
of a distribution function $F$ on the real line, which is here  defined for 
 $z\in\mathbb{C}^+ = \{ z \in \mathbb{C} : \im(z) > 0 \}$. If $F= F^{\mathbf{A}}$ is an empirical spectral distribution, then its Stieltjes transform has the form
	\begin{align*}
		s_{F^\mathbf{A}} (z) = \frac{1}{p} \tr \left\{ \lb \mathbf{A} - z \mathbf{I} \rb\inv \right\}, z \in \mathbb{C}^+. 
	\end{align*}
	Standard results on linear spectral statistics \citep[see, for example the monograph of][]{bai2004} show that
	(for fixed $t>0$)
	 under certain conditions, with probability $1$, 
	the empirical spectral distribution $F^{(n/\ntn)\mathbf{B}_{n,t}}$ of the (scaled) matrix $(n/\ntn)\mathbf{B}_{n,t}$
	converges weakly. The    limit, say  
	$F^{y_t,H}$,    is the so-called generalized Mar\v{c}enko-Pastur distribution  defined by its  Stieltjes transform $s_t = s_{F^{y_t,H}}$, which  is the unique solution of the equation
	\begin{align} \label{fund_eq}
		s_t (z) = \int \frac{1}{\lambda( 1 - y_t - y_t z s_t (z) ) - z } dH(\lambda) 
	\end{align}	 
	on the set $\{ s_t \in \mathbb{C}^+ : \frac{1-y_t}{z} + y_t s_t \in \mathbb{C}^+ \}$.
	 Here, $H$ denotes the limiting spectral distribution
	 $H_n = F^{\T_n}$  of the Hermitian matrix $\mathbf{T}_n$ which will be assumed to exist 
	 throughout this paper and $y_t=y/t$.
Hence, we have
\begin{align} \label{def_gen_MP_lim}
\tilde{F}^{y_t,H} (x) :=
\lim_{n \to \infty}
F^{\mathbf{B}_{n,t}} (x) =  
%\tilde{F}^{y_t,H}(x)=
F^{y_t,H}(x/t).
\end{align}
at all points, where $\tilde{F}^{y_t,H}$  is continuous.

For  the following discussion,
	%we denote by  
 % $H_n = F^{\T_n}$ the spectral distribution of the matrix $\T_n$,
 define for $\mathbf{B}_{n,t}$ the $(\ntn \times \ntn)$-dimensional companion matrix
	\begin{align} \label{comp}
		\mathbf{\underline{B}}_{n,t} = \frac{1}{n} \mathbf{X}_{n,t}^\star \mathbf{T}_n \mathbf{X}_{n,t} 
	\end{align}	
	and denote the limit (if it  exists) of its spectral distribution 
	$F^{\mathbf{\underline{B}}_{n,t}}$ 
and its corresponding  Stieltjes transform  by
	\begin{align} \underline{\tilde{F}}^{y_t,H}~~~\text{ and } ~~ ~~ 
	%\underline{\tilde{s}}_t=
	\underline{\tilde{s}}_t(z)=s_{\underline{\tilde{F}}^{y_t,H}} (z),
	\label{def_sut}
	\end{align}
	respectively.  A straightforward  calculation (using \eqref{fund_eq}) shows that this
	 Stieltjes transform satisfies the equation
		\begin{align} \label{repl_a47}
		z & = - \frac{1}{\tilde{\su}_t(z)} + y \int \frac{\lambda}{1 + \lambda t \tilde{\su}_t(z) } dH(\lambda).
	\end{align}

Our main result provides the asymptotic properties of   the process $(	X_{n}(f,t))_{t \in [t_0,1]}$, where $t_0 \in (0,1]$, $f$ is a given function,  
	\begin{align} \label{def_X}
		X_{n}(f,t) = \int f(x) d G_{n,t} (x),
	\end{align}
 the process $G_{n,t} $ is defined by 
$$
G_{n,t}(x) = p \big ( F^{\mathbf{B}_{n,t}}(x) - \tilde{F}^{y_{\ntn}, H_{n} } (x) \big ), ~ t\in [t_0,1]
$$
  and 
	\begin{align} \label{def_gen_MP}
	\tilde{F}^{y_{\ntn},H_n} ( x) = F ^{y_{\ntn}, H_n} \lb \frac{n}{\ntn} x  \rb
	\end{align} 
	is a rescaled version of the generalized Mar\v{c}enko-Pastur distribution defined by \eqref{fund_eq}. 
The proof is challenging and therefore deferred to Section \ref{sec40} and the Appendix.

%	 \subsection{Weak convergence for linear spectral statistics of sequential sample covariance matrices} \label{sec32}
	 
	 \begin{theorem} \label{thm}
Assume that $ 
 p/ \ntn \to y_{t} = y/t \in (0,\infty)$ and that 
the following additional conditions are satisfied: 
	\begin{enumerate}
	\item[(a)]  
	For each $n$, the random variables $x_{ij}=x_{ij}^{(n)}$ are independent with  $\E x_{ij} = 0,$ $\E |x_{ij}|^2=1$, $\max\limits_{i,j,n} \E |x_{ij}|^{12} < \infty$.
% 	Moreover,   the condition
% 	\begin{align} \label{assumption_lindeberg}
% 		\frac{1}{n p} \sum\limits_{i=1}^p \sum\limits_{j=1}^n \E \left[ |x_{ij}|^4 I(|x_{ij}| \geq \sqrt{n} \eta ) \right] \to 0
% 	\end{align}
% 	holds for any $\eta >0$.
	\item[(b)] $(\mathbf{T}_n)_{n\in \mathbb{N}}$ is  a sequence of $p\times p$ Hermitian non-negative definite matrices with bounded spectral norm
	and  the sequence of spectral distributions $(F^{\mathbf{T}_n})_{n\in \mathbb{N}}$ converges to a  proper c.d.f. $H$. 
	\item[(c)]
	Let $t_0 \in (0,1]$ and $f_1, f_2$ be functions, which are analytic on an open region containing the interval
	\begin{align} \label{interval}
		\Big [ \liminf\limits_{n\to \infty} \lambda_{\min}(\T_n) I_{(0,1)} (y_{t_0}) t_0 (1-\sqrt{y_{t_0}})^2 ,  
		\limsup\limits_{n\to\infty} \lambda_{\max}(\T_n) ( 1+ \sqrt{y_{t_0}} )^2 \Big ].
	\end{align}
	\end{enumerate}
	\begin{enumerate}
	\item[(1)] %\label{real} 
	If the random variables $x_{ij}$ are real and $\E x_{ij}^4 = 3$, then the process  
	\begin{align*}
	 (X_{n}(f_1,t), X_{n}(f_2,t))_{t\in[t_0,1]}   
	\end{align*}
	converges weakly to a Gaussian process $\lb X(f_1,t), X(f_2,t) \rb _{t\in[t_0,1]}$ in the space $\lb \ell^\infty([t_0,1]) \rb^2$  with means 
	\begin{align*}
		\E[ X(f_i,t)] =  - \frac{1}{2\pi i} \int\limits_{\mathcal{C}} f_i(z) 
		\frac{ t y  \int \frac{\sut_t^3(z)\lambda^2}{(t \sut_t(z) \lambda + 1)^3 } dH(\lambda) }
			{\lb 1 - t y  \int \frac{\sut_t^2(z)\lambda^2}{( t \sut_t(z) \lambda + 1 )^2}  dH(\lambda) \rb^2} 
			dz~,~~ i=1,2,
	\end{align*}
	and covariance kernel 
	%for $t_1 \geq t_2$
	\begin{align*}
		& \cov (X(f_1,t_1), X(f_2,t_2)) 
		=   \frac{1}{2 \pi^2 } \int_{\mathcal{C}_1} \int_{\mathcal{C}_2} f_1(z_1) \overline{f_2(z_2)} 
	 \sigma_{ t_1, t_2} ^2 (z_1,\overline{z_2})
	 \overline{dz_2} dz_1 ,
	\end{align*}
	where  $\mathcal{C}, \mathcal{C}_1, \mathcal{C}_2$ are  arbitrary closed,  
	positively orientated contours in the complex plane
	enclosing the interval in \eqref{interval}, $\mathcal{C}_1, \mathcal{C}_2$
	are non overlapping and the  function  $\sigma_{t_1,t_2}^2(z_1,z_2)$ is defined  in \eqref{def_sigma}.

	\item[(2)] %\label{complex} 
	If the random variables $x_{ij}$ are complex with $\E x_{ij}^2 = 0$ and $\E |x_{ij}|^4 = 2,$ then (1) also holds with means $\E[ X(f_i,t)] = 0, ~i=1,2,$ and covariance structure 
	\begin{align*}
		& \cov (X(f_1,t_1), X(f_2,t_2)) 
		=  
		\frac{1}{4 \pi^2 } \int_{\mathcal{C}_1} \int_{\mathcal{C}_2} f_1(z_1) \overline{f_2(z_2)} 
	 \sigma_{ t_1, t_2}^2 (z_1,\overline{z_2})
	 \overline{dz_2} dz_1.
	\end{align*}
	% where $t_1 \geq t_2$.
	\end{enumerate}
	\end{theorem}

	\begin{remark}{\rm
	 While linear spectral statistics have been studied intensively  for sample covariance matrices \citep[see, for example,][]{baisilverstein2004,bai2004}, very little effort has been done in a sequential framework so far.
	 In contrast to these ``classical''  CLTs the sequential version in 
Theorem \ref{thm} reveals the asymptotic behaviour of the whole process of linear spectral statistics 
corresponding to the sequential empirical covariance process \eqref{1.1} and thus provides a substantial generalization of its one-dimensional versions. In particular, the limiting process is not a standard Gaussian process and the proofs require   an extended machinery and some additional assumptions. 
	\begin{itemize}
	    \item[(1)] 
	     While assumptions such as 
	    %\eqref{assumption_lindeberg} and 
	    \eqref{interval} and on the spectrum of the population covariance matrix $\T_n$ are common even for a standard CLT of non-sequential linear spectral statistics, we should have a closer look at the moment assumptions. 
	Among many other technical challenges, the most delicate part of the proof of Theorem \ref{thm} lies in controlling the process $(X_n(f,t))_t$ of linear spectral statistics in terms of (asymptotic) tightness,
	which enforces higher-order moment conditions in order to find sharper bounds for the concentration of random quadratic forms of the type
	\begin{align} \label{quad_form}
	    \mathbf{x}_j^\star \mathbf{A} \mathbf{x}_j - \tr (\mathbf{A} ) ,
	\end{align}
	where $\mathbf{A}$ denotes a random $p\times p$ matrix independent of $\mathbf{x}_j$, $j\in\{1, \ldots, n\}$. 	In particular, the existence of  the $12$th-moment in Theorem \ref{thm} is exclusively needed for the  proof of asymptotic tightness and is not used for the  proof of convergence of the finite-dimensional distributions (for details, see Section \ref{sec_proof_tight}). 
	Strengthening the moment conditions on the underlying random variables appears to be a convienient tool for investigating linear spectral statistics of non-standard random matrices. 
	For example, in the work of \cite{banna2020}, the authors consider linear spectral statistics of random information-plus-noise matrices and assume the existence of the $16$th-moment for deriving a non-sequential CLT for linear spectral statistics corresponding to this type of random matrices. Consequently, the higher-order moment condition implies stronger bounds for the moments of
random quadratic forms of the type  \eqref{quad_form} (see their Lemma A.2 for more details). \\
Moreover, note that our condition on the $12$th-moment implies the Lindeberg-type condition (9.7.2) in the work of \cite{bai2004}. 

	%or the uniform convergence of the non-random part. \\
	    \item[(2)]	In order to allow for non-centralized data $(\E[x_{ij}] \neq 0 )$, \cite{zheng_et_al_2015} prove a substitution principle for linear spectral statistics of recentered sample covariance matrices and thus, weakening the conditions of Bai and Silverstein's CLT. We expect that it is possible to pursue such a generalization of  Theorem \ref{thm}  combining the tools developed in this paper with methodology used in the proof of Theorem \ref{thm}. 
    \item[(3)]	
    Furthermore, it might be of interest to  relax  the Gaussian-type 4$th$ moment condition. When allowing for a general finite $4$th moment, additional terms for the covariance structure and the bias arise whose convergence is not guaranteed under the assumptions of Theorem \ref{thm}. In fact, in this case those terms depend also on the eigenvectors of the population covariance matrix $\T_n$, which are not controlled under the conditions of Theorem \ref{thm}.  For instance, in the non-sequential case, \cite{najimyao2016} show that the Lévy–Prohorov distance between the linear statistics’ distribution and a normal distribution, whose mean and variance may diverge, vanishes asymptotically, while \cite{pan2014}  imposes additional conditions on $\T_n$ in order to ensure convergence of the additional terms for mean and covariance. 
	For the sequential version considered in this paper, it seems to be promising to derive the convergence of such additional terms under similar conditions on $\T_n$  as used by  \cite{pan2014} for a proof of a ``classical'' CLT.  
%	The two latter generalizations are beyond the scope of this paper and could be further explored in detail in following papers.
	\end{itemize}
	}
	\end{remark}

%	\subsection{Log-determinant of the sequential sample covariance matrix } \label{sec33}
	In general, the calculation of the limiting parameters appearing in Theorem \ref{thm} might be involved, since mean and covariance are given by contour integrals and rely on the Stieltjes transform $\sut_t(z)$, which is defined implicitly by a equation involving the limiting spectral distribution $H$ 
	and  has in general no closed form. In the
	case $\T_n = \mathbf{I}$
	these integrals can be interpreted as integrals over the unit circle (see Proposition \ref{prop_formula} in the Appendix), and for
	specific  functions $f_1$ and $f_2$ an explicit calculation of the asymptotic expectation and variance in Theorem \ref{thm} is possible.
	In the following corollary we illustrate this
	for the sequential process corresponding to 
	 the log-determinant of   $\mathbf{B}_{n,t}$. Note that the log-determinant  $ \log | \mathbf{B}_{n,1}|$
	 of the sample covariance matrix is a well-studied object in random matrix theory (see, e.g., \cite{bao2015}, \cite{cai2015}, \cite{nguyen_vu}, \cite{wang2018}) and has many applications in statistics. 
	%For example, the following corollary enables one to construct a test for $H_0$ in \eqref{hypothesis} based on the log-likelihood ratio test statistic. However, for reasons discussed in Section \ref{sec2}, we do not pursue this approach further. 
	% To the best of our knowledge, the following result seems to be the first one to deal with the asymptotics of the log determinant of a process sample covariance matrix. 
%	Let $h(x) = \log (x)$. 
A proof can be found in Section \ref{sec_51}. 
\bigskip

			\begin{corollary} \label{thm:logdet}
		Let $t_0 \in (0,1]$, and assume that  condition (a)  of  Theorem \ref{thm} is satisfied and that $p/n \to y \in (0,t_0)$ as $n\to \infty$. 
	\begin{enumerate}
	\item 
If the variables $x_{ij}$ are real and $\E x_{ij}^4 = 3$, then the process
\begin{align*}
\big ( \mathbb{D}_{n}(t) \big ) _{t \in [t_0,1]}= 
\Big (   \log | \mathbf{B}_{n,t} |    + p  + \ntn \log ( 1 - y_{\ntn} ) - p \log \Big (   \frac{\ntn}{n} - y_n \Big ) \Big )_{t \in [t_0,1]},
	\end{align*}
	converges weakly to a Gaussian process
	$(\mathbb{D}(t))_{t \in [t_0,1]}$  in the space $\ell^\infty([t_0,1])$  with mean
	\begin{align*}
		\E[ \mathbb{D}(t)] = \frac{1}{2} \log ( 1 - y_t)
	\end{align*}
	and covariance kernel  
	\begin{align*}
		& \cov (\mathbb{D}(t_1), \mathbb{D}(t_2)) 
		=  - 2 \log (1 - y_{t_1} \wedge y_{t_2}  \ ).
	\end{align*}
	\item If $x_{ij}$ are complex with $\E x_{ij}^2 = 0$ and $\E |x_{ij}|^4 = 2,$ then (1) also holds with mean	$\E[ \mathbb{D}(t)]  = 0$ and 
$
		 \cov (\mathbb{D}(t_1), \mathbb{D}(t_2)) 
		=  -  \log (1 -y_{t_1} \wedge y_{t_2}  ).
$
	\end{enumerate}
	    % hab ich durch Simulationen geprüft 
	\end{corollary}

\section{Monitoring sphericity in large dimension} \label{sec2}

	\begin{comment}
	In this section, we only consider real random variables $x_{ij}$ and in this case, $\mathbf{B}_{n,t}$ takes the form
	\begin{align*}
		\mathbf{B}_{n,t} = \sum\limits_{j=1}^{\ntn}  \rd_j \rd_j^\top.
	\end{align*}
	\end{comment}

% \subsection{The sphericity test}
		In many statistical problems  an important assumption is sphericity, which means, that the components of the random vectors are independent and have common variance. In the present context the corresponding test problem can be formulated as
		\begin{align} \label{3.1}
			\textnormal{H}_0: \T_n = \sigma^2 \mathbf{I}_p \textnormal{ for some } \sigma^2 >0,
~~~	\textnormal{ vs. } ~~
	\textnormal{H}_1: \T_n \neq \sigma^2 \mathbf{I}_p \textnormal{ for all } \sigma^2 >0.
%	\textnormal{H}_0: \bfSigma \in \mathcal{S}_p
%	\textnormal{ vs. } 
%	\textnormal{H}_1: \bfSigma \notin \mathcal{S}_p,
	\end{align}
	
%	\textcolor{red}{Achtung: $\mathbf{\Sigma}$ ist nicht defniniert. }
%	where 
	%$\bfSigma = \bfSigma_1 = \ldots = \bfSigma_n$ and 
% \begin{align*}
% 	\mathcal{S}_p = \{ 
% 	\sigma^2 \mathbf{I}_p : \sigma^2 > 0
% 	\} \subset \R^{p\times p}. 
% 	\end{align*}
% 	denotes the class of covariance matrices of size $p\times p$ that satisfy the sphericity assumption.  \cite{yao2015}  consider the log-likelihood ratio test and provide a central limit theorem  
%	In the latter work, it is discussed that the limiting distribution of the log-likelihood ratio test crucially depends on $-\log(1-y)$ through its asymptotic variance. If $p$ is close to $n$, the asymptotic variance quickly blows up and dramatic effects on the power of the test are expected. 
In general, it is well-known that the likelihood ratio test statistic for the hypotheses in \eqref{3.1} 
is degenerated if $p>n$ \citep[see][]{anderson2003, muirhead2009}). 	   A test statistic which is also applicable  in the case $p\geq n$ 
has been proposed by  \cite{john1971}
 and is based on the statistic
	\begin{align*}
		% U^{(n)} = 
		\frac{1}{p} \tr \Big \{ \Big (  \frac{\mathbf{B}_{n,1}}{\frac{1}{p} \tr \mathbf{B}_{n,1} } - \mathbf{I} \Big  )^2 \Big \} + 1
		=  \frac{ \frac{1}{p} \tr ( \mathbf{B}_{n,1}^2)}{\big ( \frac{1}{p} \tr \mathbf{B}_{n,1} \big )^2 }.
	\end{align*}
	The asymptotic properties of this statistic 
	in the high-dimensional regime  are investigated by \cite{ledoit_wolf_2002} and  \cite{yao2015}
	in the case $y\in(0,\infty)$ 
	and by  \cite{birke_dette} 
	%and  \cite{chen_pan_2015})  
	in the ultra high dimensional case $y=\infty$.
		In the following discussion we will use the results of Section \ref{sec3} to develop a sequential 
	monitoring procedure for the assumption of sphericity.
	
To be precise,  we consider random variables $\mathbf{y}_1, \ldots, \mathbf{y}_n \in\R^p$, where
    \begin{align*}
        \mathbf{y}_i = \bfSigma\sq_i \mathbf{x}_i, ~ 1 \leq i \leq n, 
    \end{align*}
    for symmetric non-negative definite matrices $\bfSigma_1, \ldots, \bfSigma_n \in \R^{p \times p}$ and random variables $\mathbf{x}_1, \ldots, \mathbf{x}_n \in\R^p$ satisfying the asssumptions stated  in Section \ref{sec3}. 
	%independent and centered random variables  $\mathbf{x}_1, \ldots, \mathbf{x}_n \in\R^p$ with covariance matrices $\bfSigma_1, \ldots, \bfSigma_n \in \R^{p\times p}$. 
	We are interested in monitoring the sphericity assumption 
	\begin{align}
	&\textnormal{H}_0: \bfSigma_1 =  \ldots =  \bfSigma_n = \sigma^2 \mathbf{I}_p \textnormal{ for some } \sigma^2 >0 
	  \nonumber \\
	\textnormal{ vs. }  & H_1:  \bfSigma_1 =  \ldots = \bfSigma_{\lfloor nt_1^\star \rfloor} = \sigma^2 \mathbf{I}_p,~ \bfSigma_{\lfloor nt_1^\star \rfloor + 1} =  \ldots = 
	\bfSigma_{ n  } \neq  \sigma^2 \mathbf{I}_p,  \label{hypothesis}
	\end{align}
for some  $0< t_1^\star < % t_2^\star \leq 
1$.
 For the construction of a test we consider a sequential version of the statistic proposed by 
 \cite{john1971}, that is 
	\begin{align} \label{det1}
		U_{n,t} = \frac{ \frac{1}{p} \tr ( \mathbf{\hat \Sigma}_{n,t}^2)}{\big ( \frac{1}{p} \tr \mathbf{\hat \Sigma}_{n,t} \big ) ^2 }~,
	\end{align}
	and investigate the asymptotic behaviour of the stochastic process 
	$U_n= (U_{n,t})_{t\in[t_0,1]}$  under  the null hypothesis.
	%using the results about sequential processes of linear spectral statistics of the previous section.
Here, $\mathbf{\hat \Sigma}_{n,t}$ denotes the sequential sample covariance matrix corresponding to the sample $\mathbf{y}_1, \ldots, \mathbf{y}_{\lfloor n t \rfloor } $, that is,
	\begin{align} \label{def_tilde_B}
	    \mathbf{\hat \Sigma}_{n,t} = \frac{1}{n} \sum\limits_{i=1}^{\ntn} \mathbf{y}_i \mathbf{y}_i^\top 
	    = \frac{1}{n} \sum\limits_{i=1}^{\ntn} \bfSigma_i\sq \mathbf{x}_i \mathbf{x}_i^\top \bfSigma_i\sq .
	\end{align}
	Note that in contrast to tests  based on the likelihood ratio principle the dimension may exceed the sample size.
	Moreover, under the null hypothesis, 
	we have $\mathbf{\Sigma}_i =\sigma^2 \mathbf{I}_p$
	($i=1,\ldots , n$), and a simple calculation 
	shows that 
	the statistic $U_{n,t}$ is independent of the concrete proportionality constant $\sigma^2$.
	The following theorem deals with the weak convergence of $(U_n)_{n\in\N}$ considered as a sequence in the space of bounded functions $\ell^\infty ([t_0,1])$ 
	and its proof is postponed to Section \ref{sec4}.   
In the following discussion the symbol $\rightsquigarrow$ denotes 
weak convergence of  processes and the symbol $	 \stackrel{\mathcal{D}}{\to}$ weak convergences of a real-valued random variables.

	\begin{theorem} \label{thm:u}
Let $y\in(0,\infty)$, $t_0>0$ and define $y_t = y/t$ for $t\in[t_0,1]$. If  the random variables $\mathbf{x}_1, \ldots, \mathbf{x}_n$ satisfy the assumptions (a) and (1)  of Theorem \ref{thm},
it follows under the null hypothesis  \eqref{hypothesis} that 
	\begin{align*}
		p \lb U_{n,t} - 1 - y_{\ntn} \rb_{t\in[t_0,1]  } 
		 \rightsquigarrow
		 (U_t)_{t\in[t_0,1]} ~~~\textnormal{ in } \ell^\infty([t_0,1]),
	\end{align*}	 
as $n\to\infty$, where $(U_t)_{t\in[t_0,1]}$ denotes a Gaussian process with mean function
$ 	\E [ U_t ]  = y_t $ and covariance kernel 
	\begin{align*}
	\cov (U_{t_1}, U_{t_2} )& = 4 y_{\max(t_1,t_2)}^2  
	% = \HD{ 4 \big (	\min  \big \{ {y \over t_1}, {y \over t_2} 	\big  \} \big ) ^2}
~,~~t_1,t_2 \in [t_0,1].
\end{align*}	 
	\end{theorem}	
	
\begin{remark} \label{remhd1}
{\rm  ~~
\begin{itemize}
\item[(1)] 	
	To obtain a test for the  hypotheses in \eqref{hypothesis} we note that the continuous 
	mapping theorem implies under the null hypothesis
	\begin{align} \label{max_u}
		 \sup\limits_{t\in[t_0,1]}  p \lb U_{n,t} - 1 - y_{\ntn} \rb 
		 \stackrel{\mathcal{D}}{\to} \sup\limits_{ t \in [t_0,1] } U_t  , ~n\to\infty~.
	\end{align}
Therefore we propose to  reject  the null hypothesis in \eqref{hypothesis} whenever 
	\begin{align} \label{test1}
		 \sup\limits_{t\in[t_0,1]}  p \lb U_{n,t} - 1 - y_{\ntn} \rb 
		 > c_{\alpha} , 
	\end{align}
	where $c_{\alpha}$ denotes the $(1-\alpha)$-quantile of the statistic $\sup_{t\in [t_0,1]} U_t $. Thus, we have by \eqref{max_u}
	\begin{align*}
		\lim\limits_{n\to\infty} \PR_{H_0} \Big (  \sup\limits_{t\in[t_0,1]} p \big (  U_{n,t} - 1 - y_{\ntn} \big )  > c_{\alpha} \Big )  =  \PR \Big (   \sup\limits_{t\in[t_0,1]} U_t > c_{\alpha} \Big )    \leq \alpha ,
	\end{align*}	 
	which means, that the test keeps a  nominal level $\alpha$ (asymptotically).   
	\item[(2)]	
In order to investigate the consistency of the test \eqref{test1} assume that the matrices 
${\bf \Sigma}_i $ in \eqref{hypothesis} satisfy 
\begin{eqnarray*}
  \mathbf{\Sigma}_i = \left \{ 
  \begin{array}{lll}
    \sigma^2  \mathbf{I}_p & \mbox{if} & 0 \leq i  \leq  \lfloor n  t_1^\star \rfloor,  \\
    {\bf  \Sigma} & \mbox{if} & \lfloor n  t_1^\star \rfloor  < i \leq n,
  \end{array}
  \right .  ~
  \end{eqnarray*}
  where $\sigma^2>0$ and $ {\bf  \Sigma}$ is a $p \times p $ nonnegative definite matrix.
   We also assume that  $\frac{1}{p}{\rm tr } {\bf \Sigma}$ 
and $\frac{1}{p} {\rm tr }({\bf \Sigma}^2)$ converge to   $g>0$ and 
$h>0$, respectively. Furthermore,  for the matrix   $\mathbf{H} = \bfSigma\sq  = (H_{ij})_{i,j=1, \ldots , p}$ we have
\begin{align*}
    \Big (  \frac{1}{p} \sum\limits_{j,l=1}^p
 H_{jl}^2  \Big )^2 = \Big (  \frac{1}{p} {\rm tr } \bfSigma  \Big ) ^2 \to g^2. 
 \end{align*}
 A straightforward calculation then shows that for $t\in(t_1^\star, 1)$
\begin{eqnarray*}
  \frac{1}{p}  \E \left[ \tr \lb \hat{\mathbf{\Sigma}}_{n,t} \rb \right] & { \stackrel{\mathbb{P}}{\longrightarrow} }  &  t_1^\star \sigma^2 +  ( t - t_1^\star ) g , \\
\frac{1}{p}  \E \left[ \tr \lb \hat{\mathbf{\Sigma}}_{n,t}^2 \rb \right] & { \stackrel{\mathbb{P}}{\longrightarrow} } &  \lb t_1^\star\rb^2 \sigma^4 
+ 2   t_1^\star \sigma^2 ( t-t_1^\star ) g   
  +  ( t-t_1^\star )^2 h
+ y t_1^\star \sigma^4 
+ y ( t-t_1^\star ) g.
\end{eqnarray*}
 Using a martingale decomposition and the estimate (9.9.3) in \cite{bai2004}, one can show that for fixed $t \in (t_1^\star, 1)$
 \begin{align*}
   \E | s_{F^{\hat \bfSigma_{n,t}}} (z) 
     -  \E [ s_{F^{\hat \bfSigma_{n,t}}} (z) ] |^2 \to 0,
 \end{align*}
 if we assume that the spectral norm $|| \bfSigma||$ is  uniformly bounded with respect to  $n\in\N$.  
Using \eqref{cauchy}, this implies
\begin{align*}
   \frac{1}{p}  \tr \lb f( \hat{\mathbf{\Sigma}}_{n,t} ) \rb  -  \frac{1}{p}  \E \left[ \tr \lb f( \hat{\mathbf{\Sigma}}_{n,t} ) \rb \right]
    \conp 0
\end{align*}
for $f(x) = x$ and $f(x) = x^2. $
Consequently,
\begin{align*}
 U_{n,t}   {\stackrel{\mathbb{P}}{\longrightarrow} } & 
\frac{ (t_1^\star)^2 \sigma^4  + 2    t_1^\star \sigma^2  ( t-t_1^\star ) g   
  +  ( t-t_1^\star )^2 h
+ y t_1^\star \sigma^4
+ y ( t-t_1^\star ) g^2}
{ (t_1^\star)^2 \sigma^4
+  \lb ( t-t_1^\star ) g \rb^2
+ 2 t_1^\star \sigma^2  ( t-t_1^\star ) g } \\
= & 1 + y_t + \Delta_{1,t} + \Delta_{2,t}
\end{align*}
where
$$
\Delta_{1,t} = \frac{ ( t-t_1^\star )^2 (h-g^2) }
{ (t_1^\star)^2 \sigma^4
+  \lb ( t-t_1^\star ) g \rb^2
+ 2 t_1^\star \sigma^2  ( t-t_1^\star ) g }
\geq 0 
$$
by construction, and 
\begin{align*}
\Delta_{2,t} = &  \frac{  y t_1^\star \sigma^4
+ y ( t-t_1^\star ) g^2}
{ (t_1^\star)^2 \sigma^4
+  \lb ( t-t_1^\star ) g\rb^2
+ 2 t_1^\star \sigma^2  ( t-t_1^\star ) g }  - y_t \\
= & \frac{  y t_1^\star \sigma^4
+ y ( t-t_1^\star ) g^2 
 - y_t \big  \{    (t_1^\star)^2 \sigma^4 
+  \big  (   ( t-t_1^\star ) g \big  )^2
+ 2 t_1^\star \sigma^2  ( t-t_1^\star ) g \big  \}  }
{ (t_1^\star)^2 \sigma^4
+  \big  ( ( t-t_1^\star ) g \big )^2
+ 2 t_1^\star \sigma^2  ( t-t_1^\star ) g }  \\
= & \frac{ y_t t_1^\star 
( t-t_1^\star ) \lb \sigma^2 - g \rb^2 }{(t_1^\star)^2 \sigma^4
+  \big  ( ( t-t_1^\star ) g \big )^2
+ 2 t_1^\star \sigma^2  ( t-t_1^\star ) g}  \geq 0.
\end{align*} 

Note that under the alternative in \eqref{hypothesis} two types  of structural breaks in
the covariance structure corresponding to the terms $\Delta_{1,t}$ and $\Delta_{2,t}$ may occur. 
On the one hand, the diagonal elements in the matrices  $\bfSigma_1, \ldots , \bfSigma_n$ might shift from $\sigma^2$ to a different variance while the matrices still remain spherical.
This structural break is captured by the term $\Delta_{2,t}$. On the other hand, 
the change in the matrices  could violate the sphericity assumption, which corresponds to 
the term $\Delta_{1,t}$. \\ 
Consequently, whenever there exists a parameter $\tilde t \in (t_1^\star, 1) $ such that $\Delta_{1, \tilde t} > 0$ or $\Delta_{2,\tilde t} > 0$,
%(which corresponds to the alternative in \eqref{hypothesis}), 
it follows under the additional assumption $y - y_n = o\lb p\inv \rb$ that
$$
\sup_{t \in [t_0,1]} p (U_{n,t}-1 - y_{\lfloor nt \rfloor} ) \geq p (U_{n,\tilde t}-1 - y_{\lfloor n\tilde t \rfloor} )
{\stackrel{\mathbb{P}}{\longrightarrow} }\ \infty,
$$
and in this case the test \eqref{test1} rejects the null hypothesis with a probability converging to $1$ as $p,n \to \infty $, $p/n \to y \in (0, \infty ) $.
This is in particular the case for the alternative considered in \eqref{hypothesis}.
\end{itemize}
}
\end{remark}

\cite{fisher2010} consider several generalizations of the  classical test  introduced by \cite{john1971}.
Motivated by this work an  alternative test for the hypothesis \eqref{hypothesis}  could be based on  the test statistic 
	\begin{align*}
		U_{n,t}^{(2)} = \frac{ \frac{1}{p} \tr ( \mathbf{\hat \Sigma}_{n,t}^4)}{\lb \frac{1}{p} \tr \mathbf{\hat \Sigma}_{n,t}^2\rb^2 }, 
	\end{align*}	
where the matrix $\mathbf{\hat \Sigma}_{n,t}$ is defined  in \eqref{def_tilde_B}.
For $t=1$, the asymptotic properties  of an appropriately  centered   version of $U_{n,1}^{(2)}$ have been  investigated by  \cite{fisher2010}   assuming that all arithmetic means of the eigenvalues of the sample covariance up to order 16 converge to the corresponding arithmetic means of the eigenvalues of the population covariance. 
%	More generally, one could consider the ratio of arithmetic means of the sample eigenvalues, namely for $r\in\N$,
%	\begin{align*}
%		U_{n,t}^{(r)} =  \frac{ \frac{1}{p} \tr ( \mathbf{B}_{n,t}^{2r})}{\lb \frac{1}{p} \tr \mathbf{B}_{n,t}^r\rb^2 }.
%	\end{align*}
%	Then, we recover our first test statistic as $U_{n,t}=U_{n,t}^{(1)}$. 
The following results provides the   weak convergence of  the corresponding stochastic process  $U_n^{(2)}= (U_{n,t}^{(2)})_{t\in[t_0,1]}$ under the null hypothesis.  A corresponding asymptotic level-$\alpha$ test and a discussion of its power properties can be obtained by similar arguments as given  for the process $ \big ( U_{n,t}^{(1)}
\big )_{t \in [t_0,1]}$  in Remark \ref{remhd1} and the details are omitted for the sake of brevity.

	\begin{theorem} \label{thm:u2}
Under the assumptions of Theorem \ref{thm:u} we have
	\begin{align*}
		 p \Big (  U_{n,t}^{(2)}  -\frac{  1 + 6    y_{\ntn} + 6  y_{\ntn}^2 +  y_{\ntn}^3
		}{ ( 1 + y_{\ntn}  ) ^2}  \Big )_{t\in[t_0,1] } 
		 \rightsquigarrow
		 (U_t^{(2)})_{t\in[t_0,1]}~~~ \textnormal{ in } \ell^\infty([t_0,1]),
	\end{align*}	 
where $(U_t^{(2)})_{t\in[t_0,1]}$ denotes a Gaussian process with mean function 
	\begin{align*}
	\E [ U_t^{(2)} ] & =  \frac{ y (4 t^2 + 7 t y + 4 y^2)}{t (t + y)^2} ~,~~
t	\in [t_0,1],
	\end{align*}
and covariance  kernel 
	\begin{align*}
	\cov (U_{t_1}^{(2)}, U_{t_2}^{(2)} )& = \frac{8 y^2 \Big\{ 4 t_1^2 (2 t_2^2 + 3 t_2 y + 2 y^2) + 
   6 t_1 y (4 t_2^2 + 5 t_2 y + 2 y^2) + 
   y^2 (21 t_2^2 + 24 t_2 y + 8 y^2)\Big\}}{t_1^2 (t_1 + y)^2 (t_2 + y)^2}
\end{align*}
      for $t_0 \leq t_2 \leq t_1 \leq 1$.
	\end{theorem}

\medskip 

	\begin{example} 
	{\rm 
	We conclude this section with a small simulation study illustrating the finite-sample properties of the
	test \eqref{test1}. For this purpose, we generated centered $p$-dimensional normally distributed data with various covariance structures. 
		To be precise, we consider the  the alternatives 
	\begin{align}
	    & \bfSigma_1 = \ldots = \bfSigma_{\lfloor nt^\star \rfloor} = \mathbf{I}_p,  ~
	    \bfSigma_{\lfloor nt^\star \rfloor+ 1} = \ldots
	    = \bfSigma_n = \mathbf{I}_p + 
	    \textnormal{diag} (\underbrace{ 0, \ldots, 0}_{\substack{p/2}}, \underbrace{ \delta, \ldots, \delta}_{\substack{p/2}}) ,
	    \label{alternative} \\
	      & \bfSigma_1 = \ldots = \bfSigma_{\lfloor nt^\star \rfloor} = \mathbf{I}_p,  ~
	    \bfSigma_{\lfloor nt^\star \rfloor+ 1} = \ldots
	    = \bfSigma_n = \mathbf{I}_p + 
	    \textnormal{diag} (\underbrace{ 0, \ldots, 0}_{\substack{p/2}}, \underbrace{ \delta, \ldots, \delta}_{\substack{p/2}}) + \tilde{\mathbf{S}}(\delta) ,
	    \label{alternative11} \\
	    & \bfSigma_1 = \ldots = \bfSigma_{\lfloor nt^\star \rfloor} = \mathbf{I}_p,  ~
	    \bfSigma_{\lfloor nt^\star \rfloor+ 1} = \ldots
	    = \bfSigma_n = ( 1+ \varepsilon) \mathbf{I}_p  ,
	    \label{alternative2} \\
	     & \bfSigma_1 = \ldots = \bfSigma_{\lfloor nt^\star \rfloor} = \mathbf{I}_p,  ~
	    \bfSigma_{\lfloor nt^\star \rfloor+ 1} = \ldots
	    = \bfSigma_n = ( 1+ \varepsilon) \mathbf{I}_p + \mathbf{S}(\varepsilon) ,
	    \label{alternative21}
	\end{align} 
	where $\delta, \varepsilon \geq 0$ determine the "deviation" from  the null hypothesis (note that the choice $\delta = 0$ and $\varepsilon=0$ correspond to the null hypothesis 
	\eqref{hypothesis}).  Here, the entries of the $p \times p$ matrix $\mathbf{S}(\varepsilon) $ in \eqref{alternative21} are given by $S_{j,j-1}(\varepsilon) =S_{j-1,j} (\varepsilon) =\varepsilon, ~ 1\leq j \leq p,$ and all other entries are 0. 
	Similarly, the $p \times p$ matrix $\tilde{\mathbf{S}}(\varepsilon)$ in \eqref{alternative11} has the entries $\tilde{S}_{j,j-1}(\delta) =\tilde{S}_{j-1,j} (\delta) =\delta, ~ p/2 <j \leq p,$ and all other entries are 0.

    	 \begin{figure}[!ht]
    \centering
             \includegraphics[width=0.49\columnwidth, height=0.35\textheight]{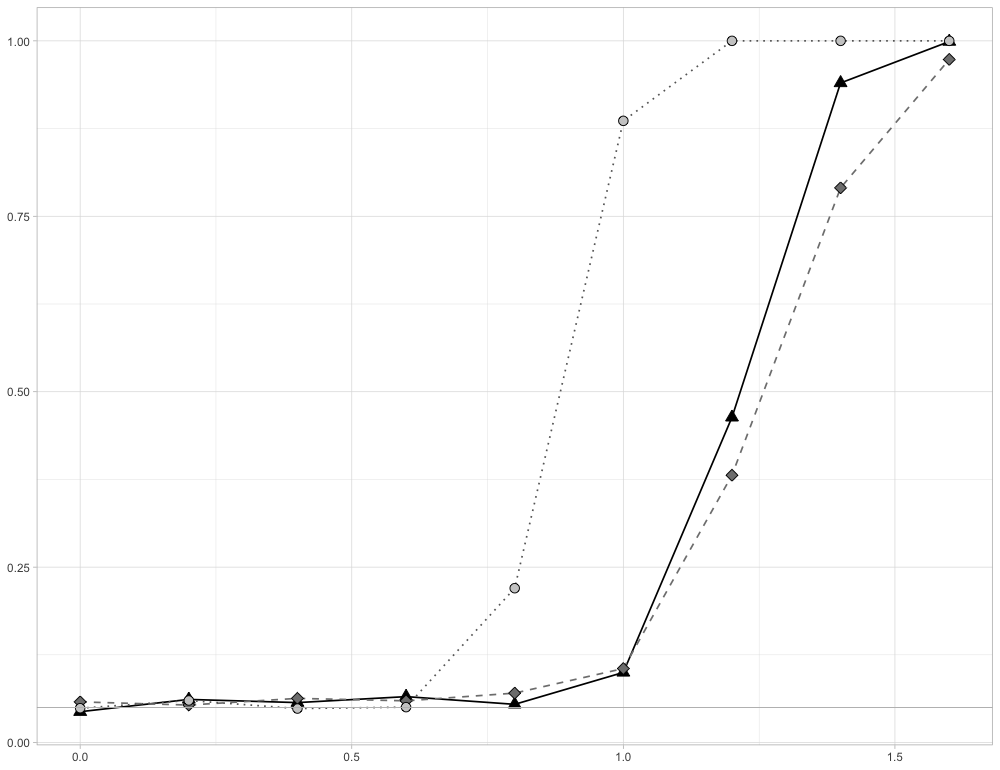}
             \includegraphics[width=0.49\columnwidth, height=0.35\textheight]{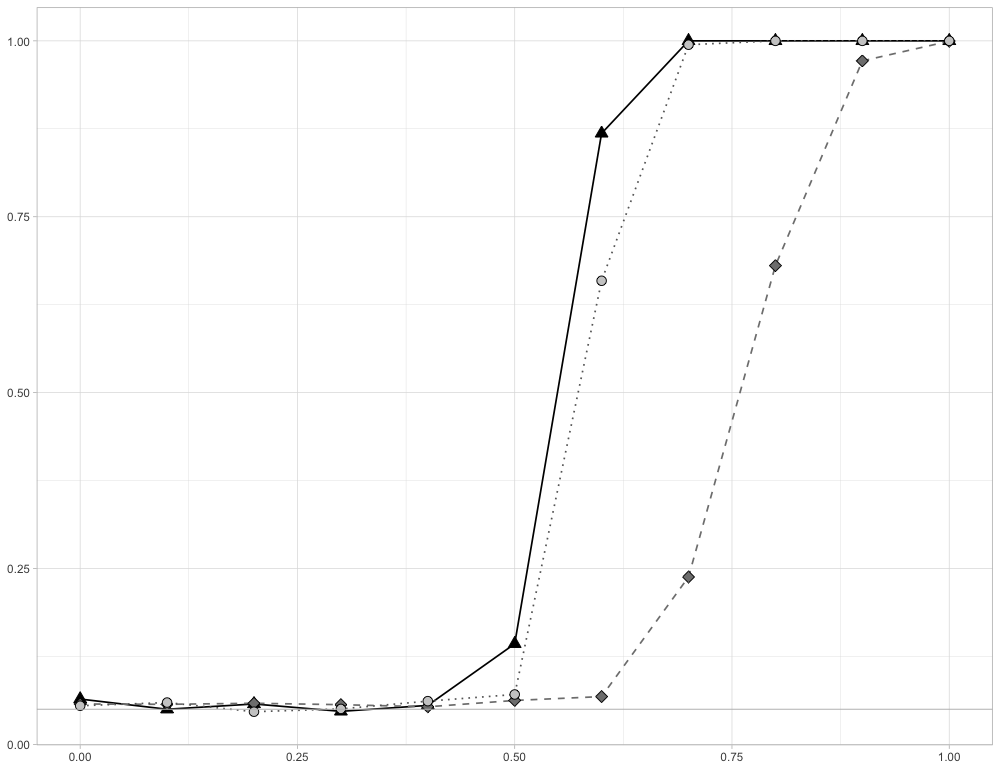}

         \caption{   \it  Simulated rejection probabilities of the test \eqref{test1}  under the null hypothesis ($\delta=0$) and  the different alternatives  in \eqref{alternative} (left)   and 
          \eqref{alternative11} (right)  for $\delta >0$ . The circle indicates $n=200, p=300$, the triangle $n=200, p=120$ and the square $n=150, p=300$.   }
    \label{fig1}
         
    \end{figure}

    	 \begin{figure}[!ht]
    \centering
          \includegraphics[width=0.49\columnwidth, height=0.35\textheight]{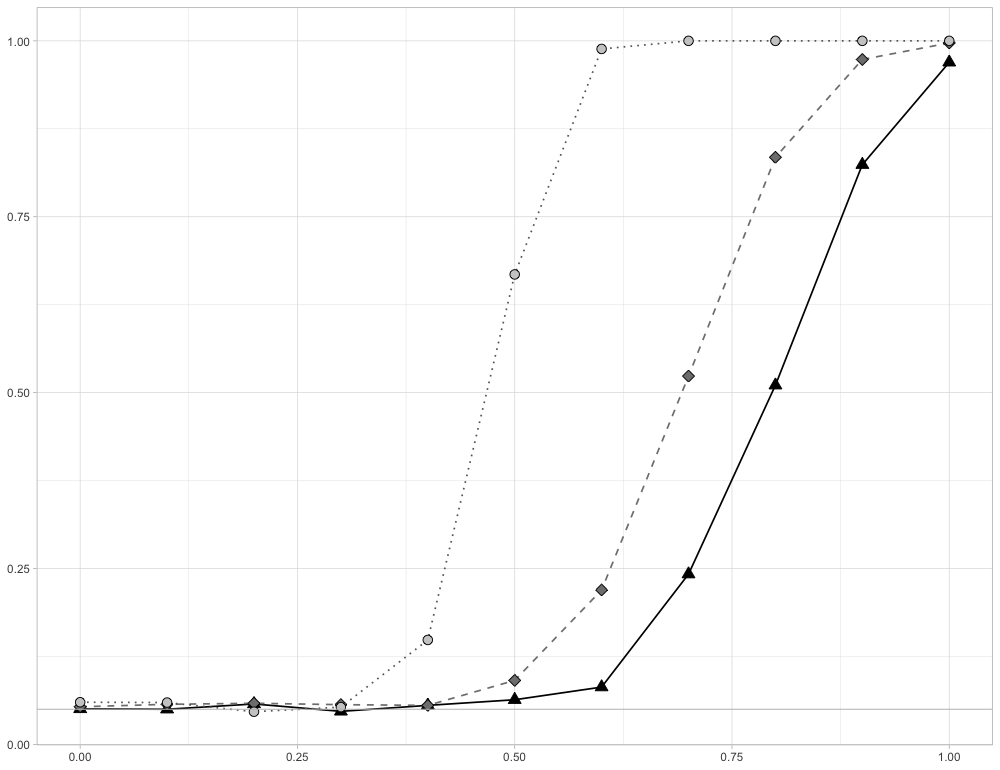}
         \includegraphics[width=0.49\columnwidth, height=0.35\textheight]{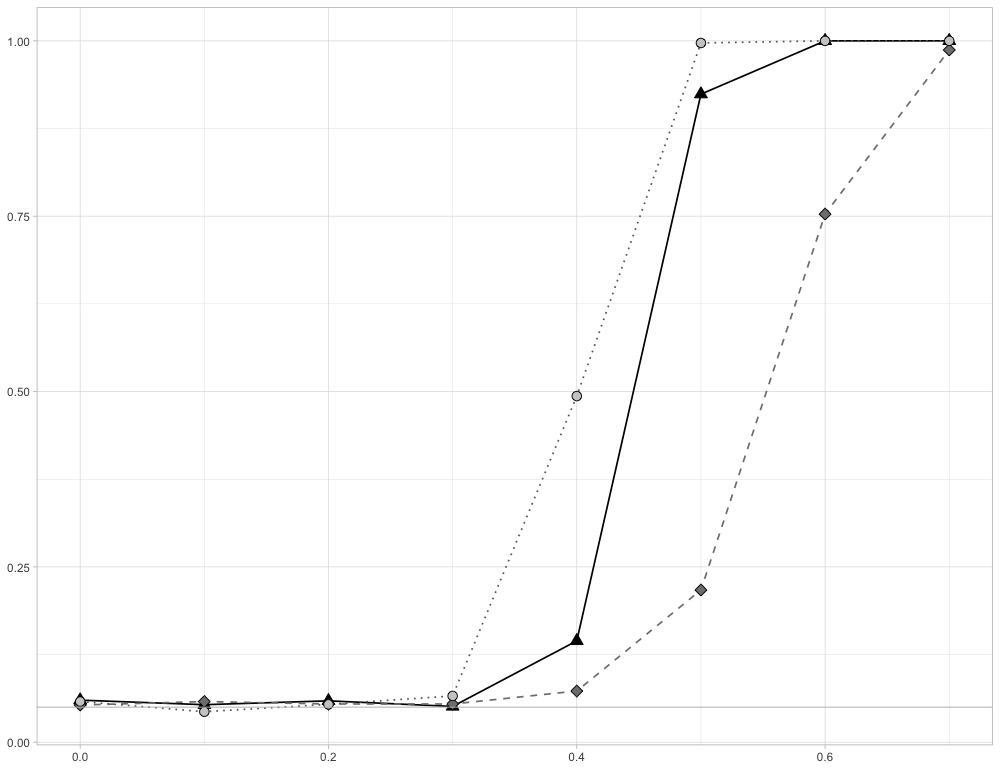}

         \caption{  \it   Simulated rejection 
         probabilities of the test   \eqref{test1} under the null hypothesis ($\varepsilon=0$) and  the different alternatives  in  \eqref{alternative2} (left) and  \eqref{alternative21} (right)  for $\varepsilon >0$. The circle indicates $n=200, p=300$, the triangle $n=200, p=120$ and the square $n=150, p=300$.   }
    \label{fig2}
    \end{figure}
	In Figure \ref{fig1} and Figure \ref{fig2}, we display the the empirical rejection of the test \eqref{test1} for  the  different alternatives and different values of $n$ and $p$, where 
	the change point is given by $t^\star = 0.6$. 
	For the  parameter  $t_0$, we always use  $t_0 = 0.2, $
	and all results are based on $2,000$  simulation runs. The vertical grey line in each figure defines the nominal level $\alpha = 5 \%$. 
	
	Note that the choices $ \delta=0$ and $\varepsilon =0$ correspond to the null hypothesis in 
	\eqref{alternative}, \eqref{alternative11}, \eqref{alternative2} and \eqref{alternative21}, respectively.
	We observe a good approximation of the nominal level in all cases under consideration. Moreover, the test has power under all considered alternatives, even if the dimension $p$ is substantially larger than
	the sample size. Note that the test performs better for alternatives of the form \eqref{alternative11} compared to the alternatives  in \eqref{alternative}. This reflects the intuition that the alternative in  \eqref{alternative} is somehow closer to sphericity than the alternative \eqref{alternative11}. A similar observation can be made for the alternatives \eqref{alternative2} and  \eqref{alternative21}.

%	Note that due to the invariance of our test statistic under $H_0$, we can assume w.l.o.g. that 
%	\begin{align*}
%	    \bfSigma_1 = \ldots = \bfSigma_n = %\mathbf{I}_p.
%	\end{align*}
%	The quantile of the discretized limiting distribution is simulated from 1,000,000 realizations of 
%	$$ \sup\limits_{\substack{t= \frac{j}{n}, \\ 1 \leq j \leq n }} U_t $$
%	for the corresponding value of $n$. 
	}
	\end{example}

\section{Proof of Theorem \ref{thm}} 
\label{sec40}

	\subsection{Outline of the proof of Theorem \ref{thm}} \label{sec34}
 
	A frequently used powerful tool in random matrix theory is the Stieltjes transform. This 
 is partially explained by the formula
	\begin{align} \label{cauchy}
		\int f(x) dG(x)  &= \frac{1}{2\pi i} \int \int_\mathcal{C} \frac{f(z)}{z-x} dz dG(x)  %\\
	%	&= \frac{1}{2 \pi i} \int\limits_{\mathcal{C}} f(z) \int \frac{1}{z-x} d G(x) dz
		= - \frac{1}{2 \pi i} \int_\mathcal{C} f(z) s_G(z) dz,  % \nonumber
	\end{align}
	where $G$ is an arbitrary cumulative distribution function (c.d.f.) with  a compact support, $f$ is an arbitrary analytic function on an open set, say $O$, containing the support of $G$, $\mathcal{C}$ is a positively oriented contour in $O$ enclosing the support of $G$ and 
	$$s_G (z)= \int \frac{1}{x-z} d G(x)  $$ 
	denotes the Stieltjes transform of $G $. Note  that \eqref{cauchy} follows from 
	Cauchy’s integral  and Fubini’s theorem. Thus invoking the continuous mapping theorem, it may suffice to prove weak convergence for the sequence $(M_n)_{n\in\N}$, where
	\begin{align} \label{def_M_n}
		M_{n}(z,t) = p \lb s_{F^{\mathbf{B}_{n,t}}} (z) - s_{\tilde{F}^{y_{\ntn}, H_{n}}} (z) \rb, ~~~  {z \in \mathcal{C}}.
	\end{align} 
 Here, $s_{\tilde{F}^{y_{\ntn}, H_{n}}}$ denotes the Stieltjes transform of $\tilde{F}^{y_{\ntn}, H_{n}}$ given in \eqref{def_gen_MP} characterized through the equation
	\begin{align} \label{a50}
		s_{\tilde{F}^{y_{\ntn}, H_{n} }}(z) = \int \frac{1}{\lambda \frac{\ntn}{n} \lb 1 - y_{\ntn} - y_{\ntn} z s_{\tilde{F}^{y_{\ntn}, H_{n} }}(z) \rb - z } dH_n (\lambda),
	\end{align}
and the contour  $\mathcal{C}$ in \eqref{def_M_n} has to be constructed in such a way that it encloses the support of  $\tilde{F}^{y_{\ntn}, H_n}$ and $F^{\mathbf{B_{n,t}}}$ with probability $1$ for  all $n\in\N, t\in[t_0,1]$. 
	This idea is formalized in the proof of Theorem \ref{thm} in Section \ref{sec_proof_thm}.

	In  order to  prove the weak convergence of \eqref{def_M_n}
% we set $M_n(z,t)= 0$ if $t\in[t_0,1]$ and  
define a contour $\mathcal{C}$ as follows. Let $x_r$ be any number greater than the right endpoint of the interval \eqref{interval} and $v_0 >0 $ be arbitrary. Let $x_l$ be any negative number if the left endpoint of the interval  \eqref{interval} is zero. Otherwise, choose 
	\begin{align*}
		x_l \in \lb 0, \liminf\limits_{n\to \infty} \lambda_{\min}(\T_n) I_{(0,1)} (y_{t_0}) t_0 (1-\sqrt{y_{t_0}})^2 \rb.
	\end{align*}		
	Let
$
		\mathcal{C}_u = \{ x + i v_0 : x\in [x_l , x_r] \} ~, 
$
	\begin{align*}
		\mathcal{C}^+ = \{ x_l + i v : v \in [0,v_0] \} ~ \cup ~ \mathcal{C}_u ~\cup ~\{ x_r + i v : v \in [0,v_0] \}. 
	\end{align*}
and define   $\mathcal{C} = \mathcal{C}^+ ~\cup ~   \overline{\mathcal{C}^+}$, where $\overline{\mathcal{C}^+}$ contains all elements of $\mathcal{C}^+$ complex conjugated. Next,
% we construct subsets $\mathcal{C}_n$ of $\mathcal{C}^+$ 
% on which $M_n$ coincides with $\hat{M}_n$. 
% For this purpose 
consider a sequence $(\varepsilon_n)_{n \in \N}$  converging to zero such that for some $\alpha \in (0,1)$
	\begin{align*}
	\varepsilon_n \geq n ^{-\alpha}, 
	\end{align*}
  define
	\begin{align*}
		\mathcal{C}_l &=
		\{ x_l + iv : v \in [n\inv \varepsilon_n, v_0] \} \\
		\mathcal{C}_r & = \{ x_r + i v : v \in [ n \inv \varepsilon_n , v_0 ] \},
	\end{align*}
and consider the  set $\mathcal{C}_n = \mathcal{C}_l \cup \mathcal{C}_u \cup \mathcal{C}_r $.
We  define an approximation 
	 $\hat{M}_n$ of  the   process ${M}_n$ for $z=x + iv\in\mathcal{C}^+, t\in[t_0,1] $ by 
	\begin{align} \label{def_hat_m}
		\hat{M}_n (z,t) = 		
		\begin{cases}
			M_n(z,t) & \textnormal{ if } z \in \mathcal{C}_{n}, \\
			M_n(x_r + i n \inv \varepsilon_{n} ,t) & \textnormal{ if } x=x_r,~v\in [0,n\inv \varepsilon_{n} ], \\
			M_n(x_l + i n \inv \varepsilon_{n} ,t) & \textnormal{ if } x=x_l,~v\in [0,n\inv \varepsilon_{n} ].
		\end{cases}
	\end{align}
	%Note that for $t\in[0,t_0)$, we have $\hat{M}_n(z,t) = M_n(z,t)=0$. 
	In Lemma \ref{a36} in the Appendix, it is shown that $(\hat{M}_n)_{n\in\N}$ approximates $(M_n)_{n\in\N}$ appropriately in the sense that the corresponding linear spectral statistics
		\begin{align*}
			- \frac{1}{2 \pi i} \int_\mathcal{C} f(z) M_n(z,t) dz ~~~~ 
			\textnormal{ and } 
			- \frac{1}{2 \pi i} \int_\mathcal{C} f(z) \hat{M}_n(z,t) dz ~
		\end{align*}
 in \eqref{cauchy}
 coincide asymptotically. As a consequence the weak convergence of the process \eqref{def_M_n} follows from that of  $\hat M_n$, which is established in the following theorem.  The  proof  is given in Section \ref{sec_proof_lem}.

	\begin{theorem}[Weak convergence for the process of Stieltjes transforms]

	\label{lem}
	Under the assumptions of Theorem \ref{thm}, the sequence $(\hat{M}_{n})_{n \in \N}$ defined in \eqref{def_hat_m}
	converges weakly to a Gaussian process $(M(z,t))_{z\in\mathcal{C}^+, t\in [t_0,1]}$ in the space $\ell^\infty(\mathcal{C}^+ \times [t_0,1])$.
	
The mean  of the limiting process $M$ is given by 
	\begin{align} \label{mean}
	 \E M(z, t) = 
	 \begin{cases} 
	\frac{t y  \int \frac{\sut_t^3(z)\lambda^2}{(t \sut_t(z) \lambda + 1)^3 } dH(\lambda) }
			{\lb 1 - t y  \int \frac{\sut_t^2(z)\lambda^2}{( t \sut_t(z) \lambda + 1 )^2}  dH(\lambda) \rb^2} 
			& 
		 \textnormal{ for the real case,} \\
		0& \textnormal{ for the complex case}
			\end{cases}
	 \end{align} 
	 ($z\in\mathcal{C}^+$, $t\in[t_0,1]$). In the complex case the covariance kernel  of the limiting process $M$ 
is given by 
	\begin{align*}
		\cov (M (z_1, t_1), M (z_2, t_2))  =   \sigma_{t_1,t_2}^2(z_1,\overline{z_2})     , 
		%~t_1 \geq t_2,
	\end{align*}
	where $ \sigma_{t_1,t_2}^2(z_1,z_2) $ is defined in \eqref{def_sigma}. In  the real case, we have
	\begin{align} \label{det60} 
		\cov (M (z_1, t_1), M(z_2, t_2))  =  2 \sigma_{t_1,t_2}^2(z_1,\overline{z_2}). 
		%~t_1 \geq t_2.
	\end{align} 
	\end{theorem}

\subsection{Proof of Theorem \ref{thm} using Theorem \ref{lem}} \label{sec_proof_thm}
	
%	\begin{proof}[Proof of Theorem \ref{thm}]
From \eqref{cauchy} we obtain
%	We use Cauchy's integral formula to obtain
%	\begin{align} \label{int1}
%		\int f(x) dG(x) = \frac{1}{2\pi i} \int \int_\mathcal{C} \frac{f(z)}{z-x} dz dG(x) = \frac{1}{2 \pi i} \int\limits_{\mathcal{C}} f(z) \int \frac{1}{z-x} d G(x) dz
%		= - \frac{1}{2 \pi i} \int_\mathcal{C} f(z) s_G(z) dz
%	\end{align}
%	and consequently,
	\begin{align} \label{int2} - \frac{1}{2 \pi i} \int\limits_{\mathcal{C}} f(z) \E s_G (z) dz = - \frac{1}{2 \pi i} \E \int\limits_{\mathcal{C}} f(z) s_G(z) dz =  \E \int f(x) d G(x).
	\end{align}
%	where G is an arbitrary c.d.f. obtaining a compact support, $f$ is an arbitrary analytic function on an open set, say $O$, containing the support of $G$ and $s_G$ denotes the Stieltjes transform of $G$. The complex integral on the right hand side is over any positively oriented contour $\mathcal{C}$ in $O$ enclosing the support of $G$. \\
	We choose $v_0, x_r, x_l$ so that $f_1$ and $f_2$ given in Theorem \ref{thm} are analytic on and inside the resulting contour $\mathcal{C}$
	and define 
		$$\mathbf{S}_{n,t} = \frac{1}{n} \mathbf{X}_{n,t} \mathbf{X}_{n,t}^\star. $$
	The almost sure convergence 
		\begin{align*}
		\lim\limits_{n\to\infty} \lambda_{\min} (\mathbf{S}_{n,t}) 
		& 
		%= \lim\limits_{n\to\infty} \left\{ \frac{\ntn}{n}\lambda_{\min} \lb \frac{1}{\ntn}\mathbf{X}_{n,t} \mathbf{X}_{n,t}^\star\rb \right\}
		= t ( 1 - \sqrt{y_{t}} )^2 I_{(0,1)} (y_t)
		= ( \sqrt{t} - \sqrt{y} )^2 I_{(0,1)} (y_t), \\
		~\lim\limits_{n\to\infty} \lambda_{\max} (\mathbf{S}_{n,t})
		 &
		 %=  \lim\limits_{n\to\infty} \left\{ \frac{\ntn}{n}\lambda_{\max} \lb \frac{1}{\ntn}\mathbf{X}_{n,t} \mathbf{X}_{n,t}^\star \rb \right\}
		= t ( 1 + \sqrt{y_{t}} )^2 = ( \sqrt{t} + \sqrt{y} )^2
	\end{align*}
	of the extreme eigenvalues 	 \cite[see, e.g., Theorem 1.1 in][] {baizhou2008} 
	and the inequalities
	\begin{align*}
		\lambda_{\max}(\mathbf{AB}) \leq \lambda_{\max}(\mathbf{A}) \lambda_{\max}(\mathbf{B}), ~ \lambda_{\min}(\mathbf{AB}) \geq \lambda_{\min}(\mathbf{A}) \lambda_{\min}(\mathbf{B})
	\end{align*}
	(valid for quadratic Hermitian nonnegative definite matrices $\mathbf{A}$ and $\mathbf{B}$) imply 
	\begin{align*}
		\limsup\limits_{n\to \infty} \lambda_{\max}(\mathbf{B}_{n,t})
		& \leq \limsup\limits_{n\to\infty} \lambda_{\max}(\mathbf{T}_n) \cdot  \limsup\limits_{n\to\infty} \lambda_{\max}(\mathbf{S}_{n,t}) \
		= \limsup\limits_{n\to\infty} \lambda_{\max}(\T_n) t \lb 1 + \sqrt{y_t} \rb^2 \\
		& \leq \limsup\limits_{n\to\infty} \lambda_{\max}(\T_n) \lb 1 + \sqrt{y_{t_0}} \rb^2 
		< x_r
		\end{align*}
		for each $t\in [t_0,1]$	 with probability 
		$1$.
	Similar calculations for $x_l$ show that it holds for all $t\in [t_0, 1] $  with probability 1	
		\begin{align} \label{a33}
		\liminf\limits_{n\to\infty} \min \big (  x_r - \lambda_{\max}(\mathbf{B}_{n,t}) , \lambda_{\min}(\mathbf{B}_{n,t}) - x_l \big ) > 0,
	\end{align} 
	which implies that for sufficiently large $n$  the contour  $\mathcal{C}$ encloses the support of $F^{\mathbf{B}_{n,t}}$, $t\in [t_0, 1] $, with probability $1$ for (note that the null set depends on $n$ and $t$).   For every $n$, there exist only finitely many $t_1, t_2 \in [t_0,1]$ such that $\nt \neq \ntt$. Since the countable union of null sets is again a null set, we may choose the above nullset in such a way that $\mathcal{C}$ encloses the support of $F^{\mathbf{B}_{n,t}}$ for sufficiently large $n$ with probability 1 (this  null set independent of $n$ and $t\in [t_0,1]$).  
	%zu dieser Überlegung habe ich Notizen
	From Lemma \ref{a37} in the Appendix, it follows that the support of $\tilde{F}^{y_{\ntn},H_{n}}$, $t \in [t_0, 1]$, is contained in the interval 
	\begin{align*}
		\Big [  \frac{\lfloor nt_0 \rfloor}{n} \lambda_{\min}(\T_n) I_{(0,1)} (y_{\lfloor n t_0 \rfloor})  ( 1 - \sqrt{y_{\lfloor n t_0 \rfloor }} )^2,  \lambda_{\max}(\T_n) (1+\sqrt{y_{\lfloor n t_0 \rfloor}})^2 \Big ],
	\end{align*} 
 which is enclosed by the contour  $\mathcal{C}$ for sufficiently large $n$.
	Therefore, using \eqref{cauchy} and \eqref{int2}, we have almost surely
	$$\Big (  \Big ( - \frac{1}{2 \pi i} \int_\mathcal{C} f_i (z) M_{n} (z,t) dz  \Big )_{i=1,2}  \Big)_{t \in [t_0, 1]} =  \big ( ( X_n(f_i,t)
	)_{i=1,2} \big )_{t \in [t_0, 1]}$$
	for sufficiently large $n$. 
	%for equality in distribution : since the finite-dimensional distributions of both processes follow the same law. \\
	%and both processes are D[t_0,1]-valued.
	Moreover, we have with probability 1 (see Lemma \ref{a36} in the Appendix)
	\begin{align*}
		\Big| \int\limits_{\mathcal{C}} f_i(z) (M_n(z,t) - \hat{M}_n(z,t) ) dz \Big| = o(1), ~i=1,2,
	\end{align*}
	uniformly with respect to $t\in[t_0,1]$. 
	Let $C(\mathcal{C} \times [t_0,1] )$ and $C([t_0,1])$ denote the spaces of continuous functions defined on $\mathcal{C} \times [t_0,1]$ and $[t_0,1]$, respectively, then the mapping
	\begin{align*}
		C ( \mathcal{C} \times [t_0, 1])  \to \lb C([t_0,1])\rb^2, ~  
		h \mapsto  
		\lb  I_{f_1}(h) , I_{f_2}(h) \rb 
	\end{align*}
	is  continuous, where
	\begin{align*}
	I_{f_i}(h)(\cdot) = - \frac{1}{2 \pi i} \int_\mathcal{C} f(z) h(z, \cdot) dz \in C ([t_0,1]), i=1,2.
	\end{align*}
	By Corollary \ref{m1_con} stated in Section \ref{sec_proof_tight} below and \eqref{mean}, the limiting process  $M$ in Theorem \ref{lem} satisfies  $M \in C ( \mathcal{C}^+ \times [t_0,1])$.   
	Invoking the continuous mapping theorem  \citep[see Theorem 1.3.6 in][]{vandervaart1996} and noting that $ \overline{M_n(z,t)} = M_n(\overline{z},t)$, we have 
%	that the weak limit of 
	\begin{align*}
	\big(  I_{f_1}(\hat{M}_n), I_{f_2}(\hat{M}_n)\big ) 
	&\rightsquigarrow
	%	\lb \lb - \frac{1}{2 \pi i} \int_\mathcal{C} f(z) \hat{M}_{n} (z, t) dz \rb_{t \in [t_0, 1]},
%		\lb - \frac{1}{2 \pi i} \int_\mathcal{C} g(z) \hat{M}_{n} (z, t) dz \rb_{t \in [t_0, 1]}
%		\rb 
%	\end{align*}
%	is given by 
%	\begin{align*}
	\lb  I_{f_1}(M), I_{f_2}(M)\rb 
		=
	\Big (\Big (   - \frac{1}{2 \pi i} \int_\mathcal{C} f_i(z) M(z, t) dz \Big )_{i=1,2} ~
		 \Big )_{t \in [t_0, 1]}  .
	\end{align*}
	The fact that this random variable is a Gaussian process follows from the observation  that the  Riemann sums corresponding to these integrals are multivariate Gaussian and therefore   integral must  be Gaussian as well. The limiting expression for the mean and the covariance follow immediately from Theorem \ref{lem}. For example, we have for the real case  observing \eqref{det60}
	\begin{align*}
	 & \cov \Big(
	  - \frac{1}{2 \pi i} \int_\mathcal{C} f_1(z) M(z, t_1) dz,
	    - \frac{1}{2 \pi i} \int_\mathcal{C} f_2(z) M(z, t_2) dz
	\Big ) \\
	= &  \frac{1}{4 \pi^2 } \int_{\mathcal{C}_1} \int_{\mathcal{C}_2} f_1(z_1) \overline{f_2(z_2)} \cov \lb 
	      M(z_1, t_1) ,
	      M(z_2, t_2) 
	\rb \overline{dz_2} dz_1 \\
	= &  \frac{1}{2 \pi^2 } \int_{\mathcal{C}_1} \int_{\mathcal{C}_2} f_1(z_1) \overline{f_2(z_2)} 
	 \sigma_{ t_1, t_2}^2 (z_1,\overline{z_2})	
	\overline{dz_2} dz_1. \\
	\end{align*}
	
%	\end{proof}	
	
\subsection{Proof of Theorem \ref{lem}} \label{sec_proof_lem}
   We begin with the usual ``truncation'' and 
    replace the entries of the matrix  $\mathbf{X}_n =(x_{ij})_{i=1,...,p, j=1,...,n } $  by truncated variables [see Section 9.7.1, \cite{bai2004}]. More precisely, without loss of generality  we assume that 
	\begin{align*}
	|x_{ij}| < \eta_n \sqrt{n}, ~ \E [ x_{ij} ] = 0, ~ \E |x_{ij}|^2 = 1, ~ \E |x_{ij} |^4 < \infty.
	\end{align*}
	Additionally, for the real case (part (1) of Theorem \ref{thm}) we may assume that
	\begin{align*}
		\E |x_{ij} |^4 = 3 + o(1)
	\end{align*}
	uniformly in $i\in\{1, \ldots, p\}, j\in\{1, \ldots, n\}$, and for the complex case (part (2) of Theorem \ref{thm})
	\begin{align*}
		\E x_{ij}^2 = o \Big ( \frac{1}{n} \Big ), ~ \E |x_{ij} |^4 = 2 + o(1)
	\end{align*}
	uniformly in $i\in\{1, \ldots, p\}, j\in\{1, \ldots, n\}$.
	Here, $(\eta_n)_{n\in\N}$ denotes a sequence converging to zero with the property
	\begin{align*}
	\eta_n n^{{1}/{5}} \to \infty.
	\end{align*}

   We  now give a brief outline for the proof of Theorem \ref{lem}
   describing the important steps, which are carried out in the following sections and the online appendix. 
   	We consider the stochastic processes  $(M_n)_{n\in\N}$ and $(\hat{M}_n)_{n\in\N}$ (which is defined in \eqref{def_hat_m}) as sequences in the space $\ell^\infty(\mathcal{C}^+ \times [t_0,1])$
and     use the decomposition 
 	\begin{equation} \label{det61}
 	    M_n = M_n^1 + M_n^2 ~,
 	\end{equation}  
 	where the   random part $M_n^1$ and the  deterministic part $M_n^2$ are given by 
\begin{align} 
\label{def_m_n1} 
		M_{n}^1(z,t) =&  p \lb s_{F^{\mathbf{B}_{n,t}}} (z) - \E \left[ s_{F^{\mathbf{B}_{n,t}}} (z) \right] \rb, \\
		M_n^2(z,t) = &  p \lb\E \left[  s_{ F^{\mathbf{B}_{n,t}}  } (z)\right]-  s_{\tilde{F}^{y_{\ntn}, H_n}} (z) \rb, ~
		\label{def_m_n2}
 	\end{align}
 the Stieltjes transform $s_{\tilde{F}^{y_{\ntn}, H_n}}$ is defined in \eqref{a50} and $s_{F^{\mathbf{B}_{n,t}}}$ denotes the Stieltjes transform of the empirical spectral distribution $F^{\mathbf{B}_{n,t}}.$ 
	\\ 
Our first result  provides the convergence of the finite-dimensional distributions of $({M}_n^1)_{n\in\N}$. Its proof relies on a central limit theorem for martingale difference schemes and  is carried out in Section \ref{sec_proof_fidis}.
 		\begin{theorem} \label{thm_fidis}
	Under the assumption (1) for the real case or assumption (2) for the complex case from Theorem \ref{thm}, it holds
	for all $k\in\N,t_1, t_2 \in [0,1]$, $ z_1,...,z_k \in\mathbb{C}$, $\im (z_i) \neq 0$
	\begin{align}
		& ( M_n^1(z_1, t_1), M_{n}^1(z_1, t_2),..., M_{n}^1(z_k, t_1), M_{n}^1(z_k, t_2) )^\top \nonumber \\
		& \cond ( M^1(z_1, t_1), M^1(z_1, t_2),..., M^1(z_k, t_1), M^1(z_k, t_2) )^\top~, 
		\label{conv_fidis}
 	\end{align}
 	where $M^1(z,t)= M(z,t) - \mathbb{E[}M(z,t)] $ is the centered version of the Gaussian process defined in Theorem  \ref{lem}.
	\end{theorem}	
\smallskip

Next, we define the process 
$\hat{M}_n^1$ in the same way  as $\hat{M}_n$  in \eqref{def_hat_m}
replacing $M_n$ by $M_n^1$ 
and show the following  tightness 
result. 	The main argument in in its  proof  consists in establishing delicate 
 	moment inequalities for the increments of the process $(\hat{M}_n^1)_{n\in\N}$, see Lemma \ref{hat}
 	and its proof in Appendix \ref{sec_proof_lemma_hat}.
%of $(\hat{M}_n^1)_{n\in\N}$, we refer the reader to the following theorem. 
 		\begin{theorem} \label{asympt_tight}
		Under the assumptions of Theorem \ref{thm}, the sequence $(\hat{M}_n^1)_{n\in\N}$ is asymptotically tight in the space $ \ell^\infty (\mathcal{C}^+\times [t_0,1])$.
		\end{theorem}

 	The third step is an investigation of 
  the deterministic part. In particular we  show that the bias   $(M_n^2)_{n\in\N}$ converges in the space $\ell^\infty(\mathcal{C}^+ \times [t_0,1])$ to the limit 
 given in \eqref{mean}. Note that the space of bounded function is equipped with the sup-norm, which demands an uniform convergence of the Stieltjes transform $\E [ s_{F^{\mathbf{B}_{n,t}}}(z) ]$ with respect to the arguments $t\in[t_0,1], z\in\mathcal{C}^+.$ The latter result is provided in Theorem \ref{thm_stieltjes} in Section \ref{sec_proof_bias}. 
 		\begin{theorem} \label{thm_bias}
	Under the assumptions of Theorem \ref{thm}, it holds 
	\begin{align*}
	   \lim\limits_{n\to\infty }
	   \sup\limits_{\substack{z\in\mathcal{C}_n, \\ t\in[t_0,1] }} \left| M_n^2(z,t) 
	   - \E [ M(z,t) ] \right| 
	   = 0.
	\end{align*}
	\end{theorem}
	The proofs of Theorem \ref{thm_fidis}, \ref{asympt_tight} and \ref{thm_bias} are postponed to Section \ref{sec_proof_fidis}, \ref{sec_proof_tight} and \ref{sec_proof_bias}, respectively. 
    Using these results, we are now in the position to prove Theorem \ref{lem}.
    \subsubsection{Proof of Theorem \ref{lem}}
        %and see that the weak convergence of a sequence in the space of bounded functions can be characterized by the convergence of its finite-dimensional distribution towards the finite-dimensional distribution of the limiting process $(M^1(z,t))_{z\in\mathcal{C}^+, t\in[t_0,1]}$ and the asymptotic tightness of the sequence $(\hat{M}_n^1)_{n\in\N}$. \\
		Theorem \ref{thm_fidis} yields the convergence of the finite-dimensional distributions of $M_n^1(z,t)$ for $t\in[t_0,1]$ and $z\in \mathcal{C}$ with $\im(z) \neq 0$ towards the corresponding finite-dimensional distributions of the centered process $M^1(z,t)= M(z,t) - \mathbb{E}[M(z,t) ] $. By the  definition  in equation \eqref{def_hat_m}, this implies the convergence of the finite-dimensional distributions of $\hat{M}_n^1(z,t)$ for $t\in[t_0,1]$ and $z\in \mathcal{C}$ with $\im(z) \neq 0$ towards the corresponding finite-dimensional distributions of $M^1$.
		Since the limiting process $(M^1(z,t))_{z\in\mathcal{C}^+, t\in[t_0,1]}$ is continuous as proven later in this section (see Corollary \ref{m1_con} in Section \ref{sec_proof_tight})
	and $\lb \mathcal{C}^+ \setminus \{x_l,x_r\} \rb \times [t_0,1]  $ is a dense subset of $\mathcal{C}^+ \times [t_0,1]$, this is sufficient in order to ensure uniqueness of the limiting process. 
	As  Theorem \ref{asympt_tight} 
establishes tightness, Theorem \ref{lem} follows from the decomposition \eqref{det61},
 Theorem 1.5.6 in \cite{vandervaart1996}
	and Theorem \ref{thm_bias}.

	\subsubsection{Proof of Theorem \ref{thm_fidis}}	
%	\subsection{Convergence of the finite-dimensional distribution of $(\hat{M}_n^1)_{n\in\N}$}
	\label{sec_proof_fidis}
%	\begin{proof}[Proof of Theorem \ref{thm_fidis}]
	 We start by performing some preparations and by introducing notations which will remain crucial for the rest of this work. 
	Using the Cramér–Wold device, the convergence in \eqref{conv_fidis} is equivalent to the weak convergence
	\begin{align} \label{f1}
		\sum\limits_{i=1}^k \lb \alpha_{i,1} M_{n}^1(z_i, t_1) + \alpha_{i,2} (M_{n}^1(z_i, t_2) \rb \cond 
		\sum\limits_{i=1}^k \lb \alpha_{i,1} M^1(z_i, t_1) + \alpha_{i,2} (M^1(z_i, t_2) \rb
	\end{align}
	for all $\alpha_{1,1} , \ldots  , \alpha_{k,1}, \alpha_{1,2} , \ldots  , \alpha_{k,2} \in \mathbb{C}$. 
 We want to show that the limiting random variable on the right hand side of the display above 
	follows a Gaussian distribution under the assumption (1) or (2) of Theorem \ref{thm}. \\
		Recalling assumption (b) in Theorem \ref{thm}, we may assume $||\T_n|| \leq 1$, $n\in\N$. 
Define for $k,j=1,..,\ntn$, $k\neq j$, $t\in (0,1]$, $z\in\mathbb{C}$ with Im$(z) \neq 0$
	\begin{align*}
		\mathbf{r}_{j} &= \frac{1}{\sqrt{n}} \mathbf{T}_n^{\frac{1}{2}} \mathbf{x}_j \\
		\mathbf{B}_{n,t} & = \sum\limits_{j=1}^{\ntn} \rd_j \rd_j^\star ,  \\ 
		\mathbf{D}_t (z) &= \mathbf{B}_{n,t} - z\mathbf{I} , \\
		\mathbf{D}_{j,t}(z) &= \mathbf{D}_t(z) - \mathbf{r}_{j} \mathbf{r}_{j}^\star, \\
		\mathbf{D}_{k,j,t}  (z) & = \mathbf{D}_{j,t} (z) -  \mathbf{r}_{k} \mathbf{r}_{k}^\star = \mathbf{D}_{t} (z) -  \mathbf{r}_{k} \mathbf{r}_{k}^\star - \rd_j \rd_j^\star,\\
		\alpha_{j,t} (z) &= \mathbf{r}_{j}^\star \mathbf{D}_{j,t}^{-2} (z) \mathbf{r}_{j} 
		- n^{-1} \operatorname{tr} ( \mathbf{D}_{j,t}^{-2} (z) \mathbf{T}_n ),  \\
		\gamma_{j,t} (z) &= \mathbf{r}_{j}^\star \mathbf{D}_{j,t}^{-1} (z) \mathbf{r}_{j} - n^{-1} \E \operatorname{tr} (\mathbf{D}_{j,t}^{-1} (z) \T_n ), \\
		\gamma_{k,j,t} (z) & = \rd_k^\star \D_{k,j,t}\inv(z) \rd_k 
		- n\inv \E \left[ \tr \lb \T_n \D_{k,j,t}\inv(z) \rb \right]  \\
		\hat{\gamma}_{j,t} (z) &= \mathbf{r}_{j}^\star \mathbf{D}_{j,t}^{-1} (z) \mathbf{r}_{j} - n^{-1} \operatorname{tr} (\mathbf{D}_{j,t}^{-1} (z) \T_n ), \\
		\beta_{j,t} (z) &= \frac{1}{1+\mathbf{r}_{j}^\star \mathbf{D}_{j,t}^{-1} (z) \mathbf{r}_{j} }, \\
		\beta_{k,j,t} (z) &= \frac{1}{1+\mathbf{r}_{k}^\star \mathbf{D}_{k,j,t}^{-1} (z) \mathbf{r}_{k} }, \\
		\overline{\beta}_{j,t} (z) &= \frac{1}{1+ n^{-1} \operatorname{tr}(\mathbf{T}_n \mathbf{D}_{j,t} \inv (z) ) } ,\\
		b_{j,t} (z) & = \frac{1}{1+ n^{-1} \E \operatorname{tr} (\mathbf{T}_n \mathbf{D}_{j,t} \inv (z))} ,\\
		b_t (z) &= \frac{1}{1+ n^{-1} \E \operatorname{tr} (\mathbf{T}_n \mathbf{D}_{t}^{-1} (z) )}.
	\end{align*}
Note that the terms  $\beta_{j,t} (z) , \beta_{k,j,t} (z), \overline{\beta}_{j,t} (z) , b_{j,t}(z) $ and $b_t(z) $ are  bounded in absolute value by $|z|/v$, where $v= {\rm  Im } (z)$ is assumed to be positive \citep[see (6.2.5) in][]{bai2004}. 
By the Sherman–Morrison formula we obtain the representation
	\begin{align} \label{sher_mor}
		\mathbf{D}_t^{-1} (z) - \mathbf{D}_{j,t}^{-1} (z) = - \mathbf{D}_{j,t}^{-1}(z) \mathbf{r}_{j} \mathbf{r}_{j}^\star \mathbf{D}_{j,t}^{-1}(z)
		 \beta_{j,t}(z).
	\end{align} 
	% \bigskip
 	In order to prove asymptotic normality of the random variable appearing in \eqref{f1}, we show that it can be represented  as a suitable martingale difference scheme plus some negligible remainder, which allows us to apply a central limit theorem. 
 	\\
 	For $j=1 \ldots , n$
	let  $\E_{j}$ denote the conditional expectation 
	with respect to the  filtration $\mathcal{F}_{nj}=\sigma( \{ \mathbf{r}_{1},...,\mathbf{r}_{j} \} )$
	(by $\E_{0}$ we denote the common expectation). 
		Recalling the definition \eqref{def_m_n1} and
	using the martingale decomposition, we have
	\begin{align}
		M_n^1  (z, t) & %= p ( s_{F^{  \mathbf{B}_{n,t}}} (z) - \E s_{F^{  \mathbf{B}_{n,t}}} (z) ) 
		= \tr ( \mathbf{D}_t^{-1} (z) - \E \mathbf{D}_t^{-1} (z) ) \nonumber \\
		&= \sum\limits_{j=1}^{\ntn} \lb \tr \E_j \mathbf{D}_t^{-1} (z) - \tr \E_{j-1} \mathbf{D}_t^{-1}(z) \rb \nonumber \\
		&= \sum\limits_{j=1}^{\ntn} \lb \tr \E_j \left[ \mathbf{D}_t^{-1}(z) - \mathbf{D}_{t,j}^{-1}(z) \right] 
		- \tr \E_{j-1} \left[ \mathbf{D}_t^{-1}(z) - \mathbf{D}_{j,t}^{-1}(z) \right]\rb \nonumber \\
		&= - \sum\limits_{j=1}^{\ntn} (\E_{j} - \E_{j-1} ) \beta_{j,t}(z) \mathbf{r}_{j}^\star \mathbf{D}_{j,t}^{-2}(z) \mathbf{r}_{j}. \label{mn_t}
	\end{align}
	Since, from now on, the proof is similar in spirit to Section 9.9 in \cite{bai2004}, we restrict ourselves to an overview explaining  the main steps and important differences. 
	By similar arguments as given in this monograph  it can be shown, that it is sufficient to prove asymptotic normality for
	the quantity 
	\begin{equation*} % \label{yj}
%		\sum\limits_{i=1}^k \Big (  \alpha_{i,1} \sum\limits_{j=1}^{\nt} Y_{j,t_1} (z_i) + \alpha_{i,2} \sum\limits_{j=1}^{\ntt} 
		% Y_{j,t_2} (z_i) \Big ) 	= 
		\sum\limits_{j=1}^{\max(\nt,\ntt) }
		Z_{nj}^{t_1,t_2}~,
		%\sum\limits_{i=1}^k \lb \alpha_{i,1} Y_{j,t_1} (z_i) + \alpha_{i,2} Y_{j,t_2} (z_i) \rb,
	\end{equation*}
	where
		\begin{align*} %\label{hol2}
		Z_{nj}^{t_1,t_2} &= \sum\limits_{i=1}^k \lb \alpha_{i,1} Y_{j,t_1}(z_i) + \alpha_{i,2} Y_{j,t_2}(z_i) \rb ~, \\ 
		Y_{j,t} (z) & = - \E_{j} \Big  [ \overline{\beta}_{j,t}(z) \alpha_{j,t}(z) - \overline{\beta}_{j,t}^2(z)\hat{\gamma}_{j,t}(z) \frac{1}{n} 
		\tr (\mathbf{T}_n \mathbf{D}_{j,t}^{-2}(z) ) \Big ] %\nonumber \\& 
		= - \E_{j} \frac{d}{dz} \overline{\beta}_{j,t}(z) \hat{\gamma}_{j,t}(z) ~
		\nonumber 
		%\label{deriv}
	\end{align*}
if  $j \leq \ntn$ and $Y_{j,t}(z) = 0 $ if  $j > \ntn$.

	For this purpose we verify conditions (5.29) - (5.31)  of  the central limit theorem for complex-valued martingale difference schemes given in Lemma 5.6 of  \cite{najimyao2016}. 
	   It is straightforward to show that 
%	\begin{align*}
$		Z_{nj}^{t_1,t_2}$
%= \sum\limits_{i=1}^k \lb \alpha_{i,1} Y_{j,t_1}(z_i) + \alpha_{i,2} Y_{j,t_2}(z_i) \rb
%	\end{align*}
	forms
	a martingale difference scheme with respect to to
	the filtration $\mathcal{F}_{nj}=\sigma( \{ \mathbf{r}_{1},...,\mathbf{r}_{j} \} )$ and  we can prove that (5.31) in this reference holds true
	by deriving bounds for  the $4$th moment of $Y_{j,t}(z)$. 
For a proof of condition (5.30), we note that
	\begin{align*}
		\sum\limits_{j=1}^{\max (\nt, \ntt)} \E_{j-1} \left[\lb Z_{nj}^{t_1,t_2} \rb ^2\right]
		= & \sum\limits_{i,l=1}^{k} \Big( \sum\limits_{j=1}^{\nt}  \alpha_{i,1} \alpha_{l,1} \E_{j-1} [ Y_{j,t_1}(z_i) Y_{j,t_1} (z_l) ] \\
		 & + \sum\limits_{j=1}^{\min(\nt,\ntt)} \alpha_{i,1} \alpha_{l,2} \E_{j-1} [ Y_{j,t_1}(z_i) Y_{j,t_2} (z_l) ] \\
		& + \sum\limits_{j=1}^{\min(\nt,\ntt)}  \alpha_{i,2} \alpha_{l,1} \E_{j-1} [ Y_{j,t_2}(z_i) Y_{j,t_1} (z_l) ] \\
		& + \sum\limits_{j=1}^{\ntt} \alpha_{i,2} \alpha_{l,2} \E_{j-1} [ Y_{j,t_2}(z_i) Y_{j,t_2} (z_l) ] \Big) ~.
%		\conp &  \sigma^2
	\end{align*}
	As all terms have the same form,  it is sufficient to show that   for all $z_1, z_2\in\mathbb{C}$ with $\textnormal{Im}(z_1), \textnormal{Im}(z_2) \neq 0$ and $t_1,t_2\in (0,1]$
	\begin{align} \label{a1}
		V_n(z_1,z_2,t_1,t_2 ) = \sum\limits_{j=1}^{\min(\nt, \ntt)} \E_{j-1} \left[ Y_{j,t_1} (z_1) Y_{j,t_2} (z_2) \right] 
	~~	\conp ~\sigma_{t_1,t_2}^2(z_1,{z_2})
	\end{align}
for an appropriate function 
$\sigma_{t_1,t_2}^2(z_1,z_2)$ (see equation \eqref{def_sigma} below for a precise definition).  Note that this convergence implies condition (5.29), since 	
	\begin{align*}
		\sum\limits_{j=1}^{\min(\nt, \ntt)} \E_{j-1} \big[ Y_{j,t_1} (z_1) \overline{Y_{j,t_2} (z_2)} \big] 
		= \sum\limits_{j=1}^{\min(\nt, \ntt)} \E_{j-1} \left[ Y_{j,t_1} (z_1) Y_{j,t_2} (\overline{z_2})\right] \conp ~\sigma_{t_1,t_2}^2(z_1,\overline{z_2}),
	\end{align*}
where the equality  follows from the fact that 
the matrices $\T_n, \mathbf{B}_{n,t}, \rd_j\rd_j^\star$ are Hermitian and $(\overline{\D_{j,t}\inv (z)})^T  = \D_{j,t}\inv (\overline{z}). $ 
 Consequently, 	Lemma 5.6 in \cite{najimyao2016} combined with the Cramér–Wold device 
yields the weak convergence of the finite-dimensional distributions to a multivariate normal distribution 
with covariance $\sigma_{t_1,t_2}^2(z_1,\overline{z_2}) = \cov (M^1(z_1,t_1),M^1(z_2,t_2)) $.
	
	\medskip	\medskip
	
	 Hence, it is remains  to show   
	 \eqref{a1} in order to establish the convergence of the finite dimensional distributions. 
For this purpose, we introduce 	the quantity 
	 \begin{align*}
	     V_n^{(2)} (z_1, z_2, t_1, t_2) & =\frac{1}{n^2} \sum\limits_{j=1}^{\min(\nt,\ntt)} b_{j,t_1}(z_1) b_{j,t_2}(z_2) \tr \lb  \E_j\left[\D_{j,t_1}^{-1}(z_1) \right] \T_n \E_j \left[ \D_{j,t_2}^{-1} (z_2)\right] \T_n \rb, 
	 \end{align*}
	 and note that it can be shown by similar arguments as on p. $273$ in  \cite{bai2004}
	 that 
	 \begin{align} \label{beh1}
	     \frac{\partial^2}{\partial z_1 \partial z_2} V_n^{(2)} (z_1, z_2, t_1, t_2) = V_n(z_1, z_2, t_1, t_2) + o_{\PR} (1)~.
	 \end{align}
Moreover,  we can show
	 \begin{align} \label{beh2}
	     V_n^{(2)} (z_1, z_2, t_1, t_2) 
	     \conp  a(z_1,z_2, t_1, t_2) \int\limits_0^{\min(t_1, t_2)} \frac{1}{1-\lambda a(z_1,z_2, t_1, t_2)} d \lambda,~ n\to\infty,
	 \end{align}
	  where 
	 \begin{align*}
	     a(z_1, z_2, t_1, t_2) & =  \frac{  \sut_{t_2} (z_2) - \sut_{t_1}(z_1) + (z_1 - z_2) \tilde{\su}_{t_1}(z_1) \tilde{\su}_{t_2}(z_2)}{t_2 \tilde{\su}_{t_2}(z_2) - t_1 \tilde{\su}_{t_1}(z_1) }
	 \end{align*}
	 and the Stieltjes transform $\tilde{\su}_{t}(z)$ is  defined  in \eqref{def_sut}.
From \eqref{beh1} and \eqref{beh2}  it follows that  
	\begin{align} 
	  \sigma_{t_1,t_2}^2 (z_1,z_2) & = \frac{\partial^2}{\partial z_1 \partial z_2} \int\limits_0^{ \min(t_1, t_2) a(z_1,z_2, t_1, t_2)} \frac{1}{1 - \lambda} d\lambda 
	= \frac{\partial}{\partial z_2} \Big ( \frac{ \min(t_1, t_2) \frac{\partial}{\partial z_1} a(z_1,z_2, t_1, t_2)}{1 - \min(t_1, t_2) a(z_1,z_2, t_1, t_2) }\Big ) \nonumber \\
	& =  \frac{\textnormal{numerator}}
	{\textnormal{denominator}}
	\label{def_sigma},
	\end{align}
	where
	\begin{align*}
	\textnormal{numerator} = &  \min(t_1,t_2)\Big\{ - t_2 (t_2 - \min(t_1,t_2) ) \sut_{t_2}^2(z_2)\sut_{t_1}'(z_1)
	\Big[ t_2 \sut_{t_2}^2(z_2) + (t_1 - t_2) \sut_{t_2}'(z_2) \Big]
	\Big\} \\
	& - t_1^2 \sut_{t_1}^4(z_1) \Big\{ \min(t_1,t_2) \sut_{t_2}^2(z_2) + (t_1 - \min(t_1,t_2) \sut_{t_2}'(z_2))
	\Big\}
	\\
	& + 2 t_1 t_2 \sut_{t_1}^3(z_1) \sut_{t_2}(z_2) \Big\{ \min(t_1,t_2)\sut_{t_2}^2(z_2) + (t_1 - \min(t_1,t_2)) \sut_{t_2}'(z_2) \Big\}
	\\
	& + 2 t_1 t_2 (t_2 - \min(t_1,t_2)) \sut_{t_1}(z_1) \sut_{t_2}^2(z_2) \sut_{t_1}'(z_1) \Big\{ \sut_{t_2}(z_2) + (-z_1 + z_2) \sut_{t_2}'(z_2) \Big\}
	\\
	& + \sut_{t_1}^2(z_1) \Big\{
		- t_2^2 \min(t_1,t_2) \sut_{t_2}^4(z_2) 
		+ t_1 (t_1 - t_2) (t_1 - \min(t_1,t_2)) \sut_{t_1}'(z_1) \sut_{t_2}'(z_2) 
		\\ & + 2 t_1 t_2 (t_1 - \min(t_1,t_2) ) (z_1 - z_2) \sut_{t_2}(z_2) \sut_{t_1}'(z_1) \sut_{t_2}'(z_2) 
		\\ & + \sut_{t_2}^2(z_2) \Big[ 
			t_2^2 (-t_1 + \min(t_1,t_2)) \sut_{t_2}'(z_2) 
			\\ & + t_1 \sut_{t_1}'(z_1)   \Big( 
				t_1(-t_2 + \min(t_1,t_2) ) 
				+ t_2 \min(t_1,t_2) (z_1 - z_2)^2 \sut_{t_2}'(z_2) 
				\Big)
			\Big]
	\Big\}
	\\
	\textnormal{denominator} = & \lb t_1 \sut_{t_1}(z_1) - t_2 \su_{t_2}(z_2) \rb^2 \Big\{ (-t_2 + \min(t_1,t_2) ) \sut_{t_2}(z_2) 
	\\ & + \sut_{t_1}(z_1) (t_1 - \min(t_1,t_2) + \min(t_1,t_2) (z_1 -  z_2) \sut_{t_2}(z_2) ) \Big\}^2.
	\end{align*}
	The  proofs of \eqref{beh1} and \eqref{beh2} are very similar to \cite{bai2004} and omitted for the sake of brevity. 
	Note also, that for the special case $t_1 = t_2=1$, this covariance structure coincides with formula (9.8.4) in   this monograph.

%	\end{proof}

%	\subsection{Asymptotic tightness of $(\hat{M}_n^1)_{n\in\N}$ in $\ell^\infty$} 
	
		\subsubsection{Proof of Theorem \ref{asympt_tight} and continuity of the limiting process} \label{sec_proof_tight}
% 		\begin{proof}[Proof of Theorem \ref{asympt_tight}]
		%rechne (A1) von tomecki nach
		We will show that the assumptions of Corollary A.4 in \cite{tomecki} are satisfied, where we identify the curve $\mathcal{C}^+$ with the compact interval $[0,1]$. For this purpose, we define the increments for the first and second coordinate of $\hat{M}_n^1$ by
		\begin{align}
		m^1 (z, t, z', z'') = & 
		\min \{ 
		| \hat{M}_n^1(z, t) - \hat{M}_n^1(z', t) |, | \hat{M}_n^1(z, t) - \hat{M}_n^1(z'', t) |
		\}, \label{m1} \\
		m^2 (z, t, t', t'') = & 
		\min \{ 
		| \hat{M}_n^1(z, t) - \hat{M}_n^1(z, t') |, | \hat{M}_n^1(z, t) - \hat{M}_n^1(z, t'') |
		\}, \label{m2}
		\end{align}
		where $t,t',t'' \in [t_0,1]$ and $z,z',z''\in\mathcal{C}^+$. 
		In order to find estimates for the tails of \eqref{m1} and \eqref{m2}, we establish 
		in the following lemma  
		estimates on the moments of the increments of $\hat{M}_n^1(z,t)$, which are  proved in Appendix \ref{sec_proof_lemma_hat}. 
		For this purpose note that it follows from \eqref{mn_t} that 
			\begin{align*}
			\hat{M}_n^1(z,t_1) - \hat{M}_n^1(z,t_2) 
			= \hat{Z}_n^1(z,t_1,t_2) + \hat{Z}_n^2(z,t_1,t_2). 
		\end{align*}
	where $\hat{Z}_n^1$ and  $\hat{Z}_n^1$ are the processes obtained 
from
\begin{align}
\label{det7}
	Z_n^1(z,t_1,t_2) = & \sum\limits_{j=1}^{\nt} (\E_j - \E_{j - 1} ) \lb  \beta_{j,t_2}(z) \rd_j^\star \D_{j,t_2}^{-2}(z) \rd_j 
	- \beta_{j,t_1}(z) \rd_j^\star \D_{j,t_1}^{-2}(z) \rd_j \rb, \\
	Z_n^2 (z,t_1,t_2) = &  \sum\limits_{j=\nt + 1}^{\ntt} (\E_j - \E_{j - 1}) \beta_{j,t_2}(z) \rd_j^\star \D_{j,t_2}^{-2}(z) \rd_j 
\label{det8}	
	\end{align}
	using the definition \eqref{def_hat_m}.

		\begin{lemma} \label{hat}
		For $t\in[t_0,1], z_1,z_2\in\mathcal{C}^+$, it holds for sufficiently large $n\in\N$ under the assumptions of Theorem \ref{asympt_tight}
		\begin{align} \label{ineq_hat1}
			\E | \hat{M}_n^1(z_1,t) - \hat{M}_n^1(z_2,t) |^{2+\delta} \leq K |z_1 - z_2|^{2+\delta},
		\end{align}
		where $K>0$ is some universal constant independent of $n,t,z_1,z_2$. 
		We also have for $t_1,t_2\in[t_0,1], z \in\mathcal{C}^+$
		\begin{align}
			\E |\hat{Z}_n^1(z,t_1,t_2) |^4 & \leq K  \Big ( \frac{\ntt - \nt}{n} \Big )^4,  \label{ineq_hat2}\\
	 		\E |\hat{Z}_n^2(z,t_1,t_2) |^{4 + \delta} & \leq K \Big ( \frac{\ntt - \nt}{n} \Big )^{2 + \delta /2}~. \label{ineq_hat3}
		\end{align}
	
		% where $\hat{Z}_n^1$ and $\hat{Z}_n^2$ are defined through $Z_n^1$ and $Z_n^2$, respectively, in the same way as $\hat{M}_n^1$ is defined through $M_n^1$. 
		\end{lemma}
		In order to simplify notation, we write $ a\lesssim b$ for $a\leq K b$, where $a,b\geq 0$ and $K>0$ denote some universal constant independent of $n, t, t_1, t_2, z, z_1, z_2$.  
		We continue with the proof of Theorem \ref{asympt_tight} by using results from Lemma \ref{hat}.  \\
		
		We observe that for $t' \leq t \leq t''$ and $\lambda > 0$
		\begin{align*}
			 \PR \lb m^2 (z, t, t', t'') > \lambda\rb 
			\leq & \PR \lb  |\hat{M}_n^1(z,t) - \hat{M}_n^1(z,t') | 
			|\hat{M}_n^1(z,t) - \hat{M}_n^1(z,t'') | 
			> \lambda^2 \rb \\
			= & \PR \lb  |\hat{Z}_n^1(z,t',t) + \hat{Z}_n^2(z,t',t) | |\hat{Z}_n^1(z,t,t'') + \hat{Z}_n^2(z,t,t'') | 
			> \lambda^2 \rb \\
			\leq &  \PR \lb  |\hat{Z}_n^1(z,t',t) + \hat{Z}_n^2(z,t',t) | 
			> \lambda \rb			
			+ \PR \lb   |\hat{Z}_n^1(z,t,t'') + \hat{Z}_n^2(z,t,t'') | 
			> \lambda \rb \\
%			\leq &   \PR \lb  |\hat{Z}_n^1(z,t',t) | > \lambda/2 \rb 
%			+\PR \lb  |\hat{Z}_n^2(z,t',t) | 
%			> \lambda/2 \rb			
%			+ \PR \lb   |\hat{Z}_n^1(z,t,t'') | > \lambda / 2 \rb 
%			+ \PR \lb  | \hat{Z}_n^2(z,t,t'') | 
%			> \lambda/2 \rb \\
			\leq &  
			\sum_{k=1}^2 \Big\{
			\PR \lb  |\hat{Z}_n^k(z,t',t) | > \lambda/2 \rb 
			+ \PR \lb   |\hat{Z}_n^k(z,t,t'') | > \lambda / 2 \rb  \Big\}
		 \\
			\leq & \Big ( \frac{2}{\lambda}\Big )^4 \E |\hat{Z}_n^1(z,t',t) |^4 
			+ \Big ( \frac{2}{\lambda} \Big )^{4+\delta} \E |\hat{Z}_n^2(z,t',t) |^{4+\delta}
			+ \Big ( \frac{2}{\lambda}\Big )^4 \E |\hat{Z}_n^1(z,t,t'') |^4 \\
			& + \Big ( \frac{2}{\lambda} \Big )^{4+\delta} \E | \hat{Z}_n^2(z,t,t'') | ^{4+\delta}. 
		\end{align*}
		In the case $t'' - t' \geq 1/n$, we use Lemma \ref{hat} and obtain 
		\begin{align*}
			\E |\hat{Z}_n^1(z,t',t)|^4 
			\lesssim &  \Big ( \frac{\ntn - \lfloor nt' \rfloor}{n} \Big )^4
			\lesssim  \Big ( t - t' + \frac{1}{n} \Big )^4
			\leq   \Big ( t'' - t' + \frac{1}{n} \Big )^4 
			\leq    2^4 (t'' - t')^4 \\
			\lesssim &   (t'' - t')^4, \\
				\E |\hat{Z}_n^2(z,t,t'') |^{4+\delta}
			\lesssim &   \Big ( \frac{\lfloor nt'' \rfloor - 
			\lfloor nt \rfloor}{n} \Big )^{2+\delta /2}
			\lesssim  \Big ( t'' - t + \frac{1}{n} \Big )^{2 + \delta /2} 
			\leq  \Big ( t'' - t' + \frac{1}{n} \Big )^{2 + \delta / 2} \\
			\leq & K 2^{2 + \delta /2}  (t'' - t')^{2 + \delta /2} \lesssim (t'' - t')^{2 + \delta /2}.
		\end{align*}
		The remaining terms can be treated  similarly in this case,
		which gives
		\begin{align*}
			\PR \lb m^2 (z, t, t', t'') > \lambda\rb 
			\lesssim  \max ( \lambda^{-4}, \lambda^{-( 4 + \delta)})  (t'' - t')^{ 2 + \delta /2}
		\end{align*}
		for $ t'' - t' \geq 1 /n$. 
		In the other case $t'' - t' < 1/n$, we have $\ntn = \lfloor nt'' \rfloor$ or $\ntn = \lfloor nt' \rfloor$ and consequently,
		\begin{align*}
			\hat{M}_n^1(z,t) - \hat{M}_n^1(z,t') = 0 \textnormal{ or } \hat{M}_n^1(z,t'') - \hat{M}_n^1(z,t) = 0.
		\end{align*}
Therefore we obtain for $t' \leq t \leq t'' \leq 1$
		\begin{align*}
			\PR \lb m^2 (z, t, t', t'') > \lambda\rb 
			\lesssim \max ( \lambda^{-4}, \lambda^{-( 4 + \delta)}) (t'' - t')^{ 2 + \delta /2}.
		\end{align*}
	 In order to derive a similar estimate for  the term $m^1$, we note that it follows  for $z,z',z''\in\mathcal{C}_n$
		\begin{align*}
			\PR \lb m^1 (z, t, z', z'') > \lambda\rb 
			\leq & \PR \lb  |\hat{M}_n^1(z,t) - \hat{M}_n^1(z',t) | | \hat{M}_n^1(z,t) - \hat{M}_n^1(z'',t) | > \lambda^2 \rb \\
			\leq & \lambda^{-( 2+ \delta  ) } \E [ |\hat{M}_n^1(z,t) - \hat{M}_n^1(z',t) | |\hat{M}_n^1(z,t) - \hat{M}_n^1(z'',t) |  ]^{1+\delta/2} \\
			\leq & \lambda^{-( 2 + \delta  ) } \lb \E |\hat{M}_n^1(z,t) - \hat{M}_n^1(z',t) |^{ 2 + \delta}
			\E |\hat{M}_n^1(z,t) - \hat{M}_n^1(z'',t) |^{ 2 + \delta} \rb^{1/2} \\
			\lesssim &  \lambda^{-( 2 + \delta  ) } \lb |z - z'|^{2+ \delta} |z - z''|^{2 + \delta} \rb^{1/2} 
			\leq  \lambda^{-( 2 + \delta  ) } |z' - z''|^{2+ \delta},
		\end{align*}
		where we used Lemma \ref{hat} in the last line.
	Moreover, we have 
	\begin{align*}
		& \PR \lb |\hat{M}_n^1(z_1,t_1) - \hat{M}_n^1 (z_2, t_2) | > \lambda \rb  \\
		\leq & \PR \lb |\hat{M}_n^1(z_1,t_1) - \hat{M}_n^1 (z_2, t_1) | > \frac{\lambda}{2} \rb 
		+ \PR \lb |\hat{M}_n^1(z_2,t_1) - \hat{M}_n^1 (z_2, t_2) | > \frac{\lambda}{2}  \rb \\
		\leq & \PR \Big ( |\hat{M}_n^1(z_1,t_1) - \hat{M}_n^1 (z_2, t_1) | > \frac{\lambda}{2} \Big ) 
		+ \sum\limits_{k=1}^2 \PR \Big ( | \hat{Z}_n^k(z_2, t_1, t_2) |  > \frac{\lambda}{4}  \Big )
		%+  \PR \lb |\hat{Z}_n^2(z_2, t_1, t_2) |  > \frac{\lambda}{4}  \rb 
		\\
		\leq & \Big ( \frac{2}{\lambda} \Big )^{2 + \delta} \E |\hat{M}_n^1(z_1,t_1) - \hat{M}_n^1 (z_2, t_1) |^{2 + \delta}
		+ \Big ( \frac{2}{\lambda} \Big )^4 
		 \E | \hat{Z}_n^1(z_2,t_1,t_2) |^4
		 +  \Big ( \frac{2}{\lambda} \Big )^{4 + \delta} \E | \hat{Z}_n^2(z_2,t_1,t_2) |^{4 + \delta}
		 \\ 
		 \lesssim &  \Big ( \frac{2}{\lambda} \Big )^{2 + \delta}  |z_1 - z_2|^{2 + \delta}
		+  \Big ( \frac{2}{\lambda} \Big )^4 \Big ( \frac{\ntt - \nt}{n} \Big ) ^4
		+   \Big ( \frac{2}{\lambda} \Big )^{4 + \delta} \Big ( \frac{\ntt - \nt}{n} \Big )^{2 + \delta /2} 
		 \\ 
		\lesssim & C_{1, \lambda}  \Big [   \Big  | \frac{\ntt - \nt}{n}  \Big  | ^{2 + \delta /2} + | z_1 - z_2|^{2 + \delta}  \Big ] \\
		\leq &  C_{1,\lambda}  \Big  [  \Big ( \left| t_2 - t_1 \right| + \frac{1}{n}  \Big )^{ 2 +\delta /2} + | z_1 - z_2|^{2 + \delta}  \Big ] \\
		\leq &  C_{1, \lambda}
		 \Big  (  \Big \| \lb z_1, t_1 \rb^\top - \lb z_2, t_2 \rb^\top   \Big \|_{\infty} + \frac{1}{n}  \Big  )^{2+\delta /2},
	\end{align*}
	where
	\begin{align*}
		C_{1,\lambda} = \max ( \lambda^{-4}, \lambda^{-( 2 + \delta) }, \lambda^{-(4+ \delta)})
		.
	\end{align*}
	Let $m\in\N$ and define for $j=(j_1,j_2) \in \{1, \ldots, m\}^2$ the set
	\begin{align*}
		K_j =  \Big [ \frac{j_1 - 1}{m}, \frac{j_1}{m}  \Big ] \times
		 \Big [ \frac{j_2 - 1}{m} \wedge t_0 , \frac{j_2}{m} \wedge t_0  \Big ].
	\end{align*}
	Combining the three inequalities above, we are able to apply Corollary A.4 in \cite{tomecki} with the parameters
	$\varepsilon=1/m, \delta' = 2 + \delta /2$ and get
	\begin{align*}
		& \PR  \Big  (  \sup\limits_{(z_1,t_1),(z_2,t_2)\in K_j} 
		|\hat{M}_n^1(z_1, t_1) - \hat{M}_n^2(z_2,t_2) | > \lambda
	 \Big ) 
		\lesssim   C_{2, \lambda}  \Big  (\frac{1}{m}  \Big ) ^{2 +\delta /2}
		+ C_{1, \lambda} \Big (\frac{1}{m} + \frac{1}{n}   \Big  )^{2 + \delta /2} ,
	\end{align*}		
	where 
$
		C_{2,\lambda} = \max(\lambda^{-4}, \lambda^{-(4+\delta)}, \lambda^{-(2+\delta)} ). 
$
	This implies 
	\begin{align*}
		& \limsup\limits_{n\to \infty} \PR  \Big  ( \sup\limits_{j\in \{1, \ldots, m\}^2} \sup\limits_{(z_1,t_1),(z_2,t_2)\in K_j} 
		|\hat{M}_n^1(z_1, t_1) - \hat{M}_n^2(z_2,t_2) | > \lambda
		 \Big  ) \\
		\leq & \limsup\limits_{n\to \infty}  \sum\limits_{j\in \{1, \ldots, m\}^2}
		\PR  \Big  (  \sup\limits_{(z_1,t_1),(z_2,t_2)\in K_j} 
		|\hat{M}_n^1(z_1, t_1) - \hat{M}_n^2(z_2,t_2) | > \lambda
		 \Big  )\\
		\lesssim &  \limsup\limits_{n\to \infty} m^2  \Big [  C_{2,\lambda}  \Big ( \frac{1}{m}  \Big ) ^{2 +\delta /2}
		+ C_{1,\lambda} \Big  (  \frac{1}{m} + \frac{1}{n}   \Big )^{2 + \delta /2}  \Big  ] \\
		\lesssim &     m^2 \frac{1}{m^{2 + \delta /2}}
		\to 0, \textnormal{ as } m\to \infty.
	\end{align*}
 Theorem 1.5.7 in \cite{vandervaart1996} finally implies the asymptotic tightness of the sequence $(\hat{M}_n^1)_{n\in\N}$,
 which completes the proof of Theorem  
 \ref{asympt_tight}. 
 
 \bigskip

% 	\end{proof} 
	\begin{corollary} \label{m1_con}
	There exists a version of the process $(M^1(z,t))_{z\in\mathcal{C}^+,t\in[t_0,1]}$ with continuous sample paths.
	\end{corollary}
	\begin{proof}
	By Addendum 1.5.8 in \cite{vandervaart1996}, almost all paths $(z,t,\omega) \in (\mathcal{C}^+\setminus \{x_l, x_r\} ) \times [t_0,1] \times \Omega \mapsto \hat{M}^1(z,t) ( \omega)$ are continuous. Since $(\mathcal{C}^+ \setminus \{ x_l, x_r\} ) \times [t_0,1] \subset \mathcal{C}^+ \times [t_0,1]$ is a dense set, we conclude that almost all paths $(z,t,\omega) \in \mathcal{C}^+\times [t_0,1] \times \Omega \mapsto \hat{M}^1(z,t) ( \omega)$ are continuous.
	\end{proof}

	\subsubsection{Proof of Theorem \ref{thm_bias} %Convergence of the bias $(\hat{M}_n^2)_{n\in\N}$
	} \label{sec_proof_bias}

%	In this section, we will show that $(\hat{M}_n^2)_{n\in \N}$ converges  uniformly to the mean function in \eqref{mean}. 
%For  a useful representation of the random variable ${M}_n^2(z,t)$, 
Let 
		\begin{align*}
		\tilde{\underline{s}}_{n,t}(z)  &= s_{F^{ \underline{\mathbf{B}}_{n,t}}} (z) = - \frac{1 - y_{\ntn}}{z} + y_{\ntn} \tilde{s}_{n,t}(z)
	\end{align*}
	be the Stieltjes transform of  the empirical spectral distribution  $F^{ \underline{ \mathbf{B}}_{n,t}}$ of 
	the matrix  $\underline{\mathbf{B}}_{n,t}$ defined in \eqref{comp}, and let 
	\begin{align*}
		\tilde{\su}_{n,t}^0 (z)  &= s_{\underline{\tilde{F}}^{y_{\ntn},H_n}(z)}
		= - \frac{1 - y_{\ntn}}{z} + y_{\ntn} \tilde{s}_{n,t}^0(z)
	\end{align*}
	be the Stieltjes transform of the distribution  
	%$\underline{\tilde{F}}^{y_{\ntn},H_n}$.}
	%, that is
%	\begin{align*}
%		\tilde{\underline{s}}_{n,t} &= s_{F^{ \underline{\mathbf{B}}_{n,t}}}, \\
%		\tilde{\su}_{n,t}^0 &= %s_{\underline{\tilde{F}}^{y_{\ntn},H_n}}.
%	\end{align*}
%	
	$$
 	\tilde{\underline{F}}^{y_{\ntn},H_n} ( \cdot) = \underline{F} ^{y_{\ntn}, H_n} \Big ( \frac{n}{\ntn} \cdot \Big )
 	$$
%	 Here, as an empirical version of $\underline{F}^{y_t,H}$, the distribution $\underline{F}^{y_{\ntn},H_n}$ is defined through
with
	$ \underline{F}^{y_{\ntn}, H_n} - y_{\ntn} F^{y_{\ntn}, H_n} = ( 1 - y_{\ntn} ) I_{[0,\infty)} $.
	Recalling the definition \eqref{def_m_n2}  we have 
	\begin{align}
		M_n^2(z,t) 
		%& = p \lb \E \left[s_{  F^{\mathbf{B}_{n,t}}  } (z)\right]-  s_{\tilde{F}^{y_{\ntn}, H_n}} (z) \rb \\
		& = p \lb \E [ \tilde{s}_{n,t}(z) ] - \tilde{s}_{n,t}^0(z) \rb % \\	& 
		= \ntn \lb \E \left[\tilde{\su}_{n,t} (z)\right]-  \tilde{\su}_{n,t}^0(z) \rb .
		\label{hol5}
	\end{align}

We begin with  a lemma  which can be used to derive   an alternative representation of $M_n^2(z,t)$. 
Note that this Lemma  corrects  an error in formula (9.11.1) in \cite{bai2004} and is proved  in Section  \ref{pra73} of the online supplement.

\begin{lemma} \label{a73}
	\begin{align*}
		& \lb \E [ \tilde{\su}_{n,t} (z) ] - \tilde{\su}_{n,t}^0 (z) \rb 
		\lb 1 - \frac{y_{n} \frac{\ntn}{n} \int \frac{\lambda^2 \tilde{\su}_{n,t}^0(z) dH_n(\lambda)}{( 1 + \lambda \frac{\ntn}{n}\E [\tilde{\su}_{n,t} (z) ] ) ( 1 + \lambda \frac{\ntn}{n} \tilde{\su}_{n,t}^0(z) ) }   }
	{  - z + y_{n} \int \frac{\lambda dH_n(\lambda) }{1 + \lambda \frac{\ntn}{n}\E [ \tilde{\su}_{n,t} (z) ] } - \frac{\ntn}{n} R_{n,t}(z)  } \rb 		
		\\
		= &  \frac{\ntn}{n}  R_{n,t}(z)  \E [ \tilde{\su}_{n,t}(z) ] \tilde{\su}_{n,t}^0(z) ,
	\end{align*}
	where  
	\begin{align*}
	R_{n,t}(z)   % &  = R_{n,t}(z)(z) \\
&	=     y_{\ntn} \ntn\inv \sum\limits_{j=1}^{\ntn} \E [ \beta_{j,t}( z) d_{j,t} (z) ] \lb \E [ \tilde{\su}_{n,t}(z) ] \rb\inv \\
&	=  y_{\ntn} n\inv \sum\limits_{j=1}^{\ntn} \E [ \beta_{j,t}( z) d_{j,t} (z) ] \Big ( \frac{\ntn}{n}\E [ \tilde{\su}_{n,t}(z) ]\Big ) \inv, \\
	d_{j,t} (z)  & = - \mathbf{q}_{j}^\star \T_n\sq \D_{j,t}\inv(z) ( \frac{\ntn}{n} \E [ \tilde{\su}_{n,t}(z) ] \T_n + \mathbf{I} )\inv \T_n\sq \mathbf{q}_j \\
& 	 +  \frac{1}{p} \E \Big[ \tr ( \frac{\ntn}{n}\E [ \tilde{\su}_{n,t} (z) ] \T_n + \mathbf{I} )\inv \T_n \D_{t}\inv(z) \Big], \\
	 \mathbf{q}_j & = \frac{1}{\sqrt{p}} \mathbf{x}_j.
	\end{align*}
		
	\end{lemma}

The next main step is the following 
	%provide the uniform convergence of the Stieltjes transform $(\E [ \sut_{n,t}(z) ])$ given in the following theorem 
	result, which  is  proved in Section \ref{sec_proof_stieltjes}. 
	
	\begin{theorem} \label{thm_stieltjes}
	Under the assumptions of Theorem \ref{thm}, we have 
	\begin{align*} %\label{a59}
	%	\lim\limits_{n\to\infty} \sup_{\substack{z\in\mathcal{C}_n, \\ t \in [t_0,1]}} | \E [ \tilde{s}_{n,t}(z) ] 
	%	- \tilde{s}_t(z) |
	%	= 0  , \\ 
		\lim\limits_{n\to\infty} \sup_{\substack{z\in\mathcal{C}_n, \\ t \in [t_0,1]}} | \E [ \sut_{n,t}(z) ] 
		- \sut_t(z) |~
		= 0,
	\end{align*}
	where $\sut_t$ is defined in \eqref{def_sut}. %and $\tilde{s}_t$ denotes the Stieltjes transform of $\tilde{F}^{y_t,H}$ given in \eqref{def_gen_MP_lim}. 
	% where
	%$$\tilde{s}_{n,\cdot}(\cdot): \mathcal{C}^+ \times [0,1] \to \mathbb{C}, ~
	%	(z,t) \mapsto \tilde{s}_{n,t}(z) = \frac{1}{p} \tr \lb \D_t\inv(z) \rb 
	% $$ 
	% and
	% $$\tilde{s}_\cdot(\cdot): \mathcal{C}^+ \times [0,1] \to \mathbb{C}, ~
	%	(z,t) \mapsto \tilde{s}_t(z)
	% $$ 
	% is the unique solution of the equation
	% \begin{align} \label{a60}
	%	\tilde{s}_t(z) = \int \frac{1}{\lambda t ( 1 - y_t - y_t z \tilde{s}_t(z) ) -z } dH(\lambda).
	%\end{align}
	\end{theorem}
	
%	\begin{proof}[Proof of Theorem \ref{thm_bias}]

The third  step in the proof of Theorem \ref{thm_bias} is the following result, which is  proved in Section  \ref{prsn0_conv} of the online supplement. 

		\begin{theorem}  \label{sn0_conv}
	Under the assumptions of Theorem \ref{thm}, we have 
		\begin{align*}
	&	\sup_{\substack{n\in\N, \\ z\in\mathcal{C}_n, \\ t\in[t_0,1]}} 
		| M_n^2(z,t) | \leq K~,~~~~
%		\end{align*}
%		and
%		\begin{align*}
			\lim\limits_{n\to\infty} \sup_{\substack{z\in\mathcal{C}_n, \\ t \in [t_0,1]}} |  \sut_{n,t}^0(z)  
		- \sut_t(z) |
		= 0.
		\end{align*}
	\end{theorem}

With these preparations  we   show in Section \ref{prconv2} that 
	\begin{align} \label{conv2}	
		\ntn R_{n,t}(z)  \E [ \tilde{\su}_{n,t}(z) ]   \to
		\begin{cases}
		 \frac{y  \int \frac{\sut_t^2(z)\lambda^2}{(t \sut_t(z) \lambda + 1)^3 } dH(\lambda) }
			{1 - t y  \int \frac{\sut_t^2(z)\lambda^2}{( t \sut_t(z) \lambda + 1 )^2}  dH(\lambda)} & 
		 \textnormal{ for the real case,} \\
		0& \textnormal{ for the complex case,}
		\end{cases}
	\end{align}
 uniformly with respect to  $z\in\mathcal{C}_n, t\in[t_0,1]$.
 Combining this result  with Theorem \ref{thm_stieltjes} and Lemma \ref{a80} yields 
		\begin{align} \label{R_to_0}
		\lim\limits_{n\to\infty} 
			\sup\limits_{\substack{z\in\mathcal{C}_n \\ t\in [t_0,1] }} |R_{n,t}(z)| 
			= 0 
		\end{align}
 This result and Lemma \ref{a80}, Theorem \ref{thm_stieltjes}, 
		Theorem \ref{sn0_conv}, Proposition \ref{a61} and the equation \eqref{repl_a47}
		show that 
	\begin{align*} %\label{conv1}
		\frac{y_{n} \frac{\ntn}{n} \int \frac{\lambda^2 \tilde{\su}_{n,t}^0(z) dH_n(\lambda)}{( 1 + \lambda \frac{\ntn}{n}\E [\tilde{\su}_{n,t} (z) ] ) ( 1 + \lambda \frac{\ntn}{n} \tilde{\su}_{n,t}^0(z) ) }   }
	{  - z + y_{n} \int \frac{\lambda dH_n(\lambda) }{1 + \lambda \frac{\ntn}{n}\E [ \tilde{\su}_{n,t} (z) ] } - R_{n,t}(z)  }
	\to t y \int \frac{\lambda^2 \tilde{\su}_t^2(z) dH(\lambda) }{(1 + \lambda t \sut_t(z) )^2}~.
	\end{align*}		
 Observing the representation in \eqref{hol5},  Lemma \ref{a73} and Theorem \ref{sn0_conv}, this implies 
	\begin{align*}
	M_n^2(z,t) \to 
	\begin{cases} 
	\frac{t y  \int \frac{\sut_t^3(z)\lambda^2}{(t \sut_t(z) \lambda + 1)^3 } dH(\lambda) }
			{\lb 1 - t y  \int \frac{\sut_t^2(z)\lambda^2}{( t \sut_t(z) \lambda + 1 )^2}  dH(\lambda) \rb^2} 
			& 
		 \textnormal{ for the real case,} \\
		0& \textnormal{ for the complex case.}
			\end{cases} 
			%\label{aim_cov}
	\end{align*}
	 uniformly with respect  $z\in\mathcal{C}_n, t\in[t_0,1]$,  which completes the proof of Theorem \ref{thm_bias}.

\section{Proof of Theorem \ref{thm:u} } \label{sec4}
   
     Due to the invariance of $U_{n,t}$ under $H_0$, we may assume w.l.o.g. that
      that $\bfSigma_1 = \ldots \bfSigma_n = \mathbf{I}$, which implies	$\mathbf{\hat \Sigma}_{n,t} = \mathbf{B}_{n,t}.$
	We apply Theorem \ref{thm} for the special case $f_1(x) = x$, $f_2(x)=x^2, \T_n = \mathbf{I}$, that is
		\begin{align*}
		X_n(f_1,t) & =  \tr \lb \mathbf{B}_{n,t} \rb  - \ntn y_n , ~ \\
		X_n(f_2,t) & =  \tr \lb \mathbf{B}_{n,t}^2 \rb  -  \ntn y_n \Big (  \frac{\ntn}{n} + y_n \Big ) , 
		~t \in [t_0,1].
	\end{align*}
	Note that all 
	conditions from Theorem \ref{thm} are satisfied, and therefore 
	\begin{align*} %\label{det2}
	\lb (X_{n}(f_1,t))_{t\in[t_0,1]} , (X_{n}(f_2,t))_{t\in[t_0,1]}  \rb_{n \in \N}  \rightsquigarrow 
	 \lb (X(f_1,t))_{t\in[t_0,1]} , (X(f_2,t))_{t\in[t_0,1]}  \rb 
	\end{align*}
	in the space $\lb \ell^\infty([t_0,1]) \rb^2$, where $ \lb (X(f_1,t))_{t\in[t_0,1]} , (X(f_2,t))_{t\in[t_0,1]}  \rb$ is a Gaussian process. 
	Thus, it is left to calculate mean, covariance and the centering term appearing in Theorem \ref{thm}. A tedious calculation in Section \ref{sec_51}  shows 
     \begin{align}\label{det3}
		\E[ X(f_1,t)] = 0, ~
		\E [X(f_2,t)] = t y,
	\end{align}
	and  
	\begin{align} \nonumber 
		& \cov (X(f_1,t_1), X(f_1,t_2)) 
		=  2 y \min(t_1,t_2), \\
		& \cov (X(f_2,t_1), X(f_2,t_2)) 
		=  4 \min(t_1, t_2) y \left\{ 2 t_1 t_2 + \left[ \min(t_1, t_2) + 2 (t_1 + t_2)\right] y + 2 y^2 \right\},
		\label{det4}\\
	& 	\cov (X(f_1,t_1) , X(f_2,t_2) )  = 4 \min(t_1,t_2) y (t_2 + y). 
	 \nonumber 
	\end{align}	
	In Section \ref{sec_51}, we also calculate the centering terms for $X_n(f_1,t)$ and $X_n(f_2,t).$
With the definition 
$
		\phi(x,y) = \frac{y}{x^2} , 
$
we obtain the representation 
	\begin{align*}
		U_{n,t} = \phi\Big (  \frac{1}{p} \tr ( \mathbf{B}_{n,t}), \frac{1}{p} \tr ( \mathbf{B}_{n,t}^2) \Big ) . 
	\end{align*}
	for the process $	U_{n,t}$ in  \eqref{det1}.  
	Consequently, the assertion can be proved by the functional delta method.

	To be precise, note that it follows from $y_n =p/n$
		\begin{align} \label{det6} 
		& p \begin{pmatrix}
		  \frac{1}{p} \tr ( \mathbf{B}_{n,t}) - \frac{\ntn}{n}  \\
		   \frac{1}{p}  \tr ( \mathbf{B}_{n,t}^2) - \frac{\ntn}{n} \lb \frac{\ntn}{n} + y_n\rb  
		\end{pmatrix}_{t \in [t_0,1] }
%		= 
%		\begin{pmatrix}
%		&  \lb  \tr ( \mathbf{B}_{n,t}) - \ntn y_n \rb_{t \in [t_0,1] } \\
%		 & \lb  \tr ( \mathbf{B}_{n,t}^2) - \ntn y_n \lb \frac{\ntn}{n} + y_n\rb   \rb_{t \in [t_0,1] }
%		\end{pmatrix}_{n\in\N} \\https://www.overleaf.com/project/6040a54ea7fcff10f338571f
		 \rightsquigarrow
		\begin{pmatrix}
		X(f_1,t)  \\
		   X(f_2,t) 
		\end{pmatrix}_{t \in [t_0,1] }
	\end{align}
in $ \lb \ell^\infty([t_0,1])\rb^2.$
	 Let $a_{n}(t)= \frac{\ntn}{n}$ and $b_{n}(t)= \frac{\ntn}{n} \big (  \frac{\ntn}{n} + y_n \big )$, such that 
	\begin{align*}
		\lim\limits_{n\to\infty} a_{n}(t) = t = a(t), ~
		\lim\limits_{n\to\infty} b_{n}(t) = t(t+y) = b(t)
	\end{align*}
	uniformly in $t\in [t_0,1]$.
 For a sequence $(h_{n,1},h_{n,2} )_{n\in\N}$ in $ (\ell^\infty([t_0,1] ))^2$ converging to $0$, a straightforward calculation shows that 
	\begin{align}
		&   p \left\{ \phi \lb a_{n} + p\inv h_{n,1} , b_{n} + p\inv h_{n,2} \rb - 
		\phi \lb a_{n}  , b_{n}  \rb
		\right\} 
		%= & p \left\{ 
		%\frac{b_{n}(t) + p\inv h_{n,2}(t)}{\lb a_{n}(t) + p \inv h_{n,1}(t) \rb^2 }
		%- \frac{b_{n}(t) }{a_{n}(t)^2 }
		%\right\}
		%= \frac{h_{n,2}(t)}{\lb a_{n}(t) + p \inv h_{n,1}(t) \rb^2} 
% 		+ \frac{b_{n}(t) \lb p a_{n}^2(t) - p \lb a_{n}(t) + p\inv h_{n,1}(t) \rb^2 \rb }{a_{n}^2(t) \lb a_{n}(t) + p \inv h_{n,1}(t) \rb^2} \nonumber \\
% 		= & \frac{h_{n,2}(t)}{\lb a_{n}(t) + p \inv h_{n,1}(t) \rb^2} 
% 		+ \frac{ b_{n}(t) \lb  - 2 a_{n}(t) h_{n,1}(t) - p\inv h_{n,1}^2 (t)  \rb }
% 		{a_{n}^2(t) \lb a_{n}(t) + p \inv h_{n,1}(t) \rb^2} \nonumber %\label{phi_deriv}
% 		\\
		\to 
% 		\frac{h_2(t)}{t^2} - 2  \frac{ t(t+y) t h_1 (t) }{t^4} 
% 		= \frac{h_2(t)}{t^2} - 2h_1(t) \frac{t+y}{t^2}
% 		 \frac{h_2 - 2h_1  (t+y)}{t^2} 
% 		 = 
		 \frac{h_2}{a^2} -  \frac{2 b h_1}{a^3}
		 = \phi'_{(a,b)}(h_1 , h_2 )
		\nonumber 
	\end{align}
	in $l^\infty([t_0,1])$, as $ n \to \infty$. 
	Moreover, we have
	\begin{align*}
		\phi \lb a_{n}(t)  , b_{n}(t)  \rb
%		\phi \lb \frac{\ntn}{n} , \frac{\ntn}{n} \lb \frac{\ntn}{n} + y_n \rb \rb 
		= \frac{\frac{\ntn}{n} \big ( \frac{\ntn}{n} + y_n \big )}
		{\big ( \frac{\ntn}{n} \big )^2}
		= \frac{n}{\ntn} \Big ( \frac{\ntn}{n} + y_n \Big ) 
		= 1 + y_{\ntn} .
	\end{align*}
	Thus, it follows from \eqref{det6} and  Theorem 3.9.5 in \cite{vandervaart1996} that
	\begin{align*}
	& p \left\{ U_{n,t} - 1 - y_{\ntn} \right\}_{t\in[t_0,1]  }
	= p \Big\{  \phi\Big ( \frac{1}{p} \tr ( \mathbf{B}_{n,t}), \frac{1}{p} \tr ( \mathbf{B}_{n,t}^2) \Big )
	 - 	\phi \lb a_{n}(t)  , b_{n}(t) \rb 
	 %\phi\lb \frac{\ntn}{n} , \frac{\ntn}{n} \lb \frac{\ntn}{n} + y_n\rb \rb 
	 \Big\}_{t\in[t_0,1]  }
	 %\\ 	& 
	 \rightsquigarrow	 
%	\lb \phi'_{(t,t(t+y))}(X(f,t), X(g,t)) \rb_{t\in[t_0,1]}
%	= 
	(U_t)_{t\in[t_0,1]}
	\end{align*}
	in $\ell^\infty([t_0,1])$,
	where
	\begin{align*}
		U_t =  \frac{X(f_2,t) - 2X(f_1,t) (t+y)}{t^2}, ~ t\in[t_0,1] 
	\end{align*}
is a Gaussian process.
	Recalling \eqref{det3} and \eqref{det4} 
	we obtain  for $t,t_1,t_2\in[t_0,1]$  with $t_2 \leq t_1$  by straightforward calculations 
	\begin{align*}
		\E [ U_t ]  = & \frac{1}{t^2} \lb \E [ X(f_2,t) ]  - 2 (t+y) \E [ X(f_1,t) ]  \rb 
		=  \frac{t y}{t^2}  = y_t , \\
		\cov (U_{t_1} , U_{t_2} ) = & \frac{1}{t_1^2 t_2^2}\cov \lb X(f_2,t_1) - 2(t_1+y)X(f_1,t_1) , X(f_2,t_2) - 2 (t_2+y)X(f_1,t_2) \rb \\
%		= &  \frac{1}{t_1^2 t_2^2} \Big\{ 
%		\cov ( X(g,t_1), X(g,t_2) )  
%		- 2 (t_2 +y) \cov ( X(g,t_1), X(f,t_2) ) \\
%		& - 2 (t_1 +y) \cov(X(f,t_1) , X(g,t_2) ) 
%		+ 4 (t_1 + y) (t_2 + y) \cov ( X(f,t_1) , X(f,t_2) ) 
%		\Big\} \\
		= & \frac{1}{t_1^2 t_2^2} \Big\{ 
		4 t_2 y \left\{ 2 t_1 t_2 + \left[ t_2 + 2 (t_1 + t_2)\right] y + 2 y^2 \right\}
		- 2 (t_2 +y) 4 \min(t_1,t_2) y (t_1 + y) \\
		& - 2 (t_1 +y) 4 \min(t_1,t_2) y (t_2 + y)
		+ 4 (t_1 + y) (t_2 + y) 2 y \min(t_1,t_2) 
		\Big\} \\
		= & \frac{1}{t_1^2 t_2^2} \Big\{ 
		4 t_2 y \left\{ 2 t_1 t_2 + \left[ t_2 + 2 (t_1 + t_2)\right] y + 2 y^2 \right\}
		- 8 (t_2 +y)  \min(t_1,t_2) y (t_1 + y) 
		\Big\} \\
		=& 4 \frac{y^2}{t_1^2}  = 4 y_{\max(t_1,t_2)}^2.
	\end{align*}
	which proves the assertion of Theorem \ref{thm:u}.

\bigskip
\bigskip

{\bf Acknowledgements.} 
The work of 
N.  D\"ornemann was 
supported by the Deutsche
Forschungsgemeinschaft
(RTG 2131). We are grateful to Zhidong Bai and Jack Silverstein for the discussion of the proof of  formula (9.11.1) in \cite{bai2004}.

\setlength{\bibsep}{1pt}
\begin{small}

\end{small}

\newpage 

\begin{appendices} % Name dieser section kann man in Zeile 42 ändern 

\section{Details for the arguments in Section \ref{sec_proof_thm}} 

	\begin{lemma} \label{a37}
	Let $\Gamma_{F}$ denote the support of a cdf $F$. Then it holds
	\begin{align*}
		\Gamma_{\tilde{F}^{y_{\ntn}, H_n}} \subset 
		\Big[ \frac{\lfloor n t_0 \rfloor}{n} \lambda_{\min}(\T_n) I_{(0,1)} (y_{\lfloor n t_0 \rfloor}) ( 1 - \sqrt{y_{\lfloor n t_0 \rfloor }} )^2,  \lambda_{\max}(\T_n) (1+\sqrt{y_{\lfloor n t_0 \rfloor}})^2  \Big].
	\end{align*} 
	\end{lemma}	
	The proof of Lemma \ref{a37} follows from  Lemma 6.1, \cite{bai2004} or Proposition 2.17, \cite{yao2015} and is therefore omitted.

% 	\begin{proposition}[Lemma 6.1, \cite{bai2004}; Proposition 2.17, \cite{yao2015}]
% 	If $\lambda \notin \Gamma_{\underline{F}^{y_{\ntn}, H_n}}$, then $\underline{s}_{n,t}(\lambda) \neq 0$ and $\alpha= - 1/ \su_{n,t}(\lambda)$ satisfies 
% 	\begin{enumerate}
% 	\item $\alpha \notin \Gamma_{H_n}$ and $\alpha \neq 0$,
% 	\item $\psi '(\alpha) > 0 $.
% 	\end{enumerate}
% 	Conversely, if $\alpha$ satisfies 1-2, then $\lambda = \psi(\alpha) \notin \Gamma_{\underline{F}^{y_{\ntn}, H_n}}$.
% 	\end{proposition} 
	
	The following lemma ensures that the process $(\hat{M}_n(z,t))_{z\in\mathcal{C}^+, t\in[t_0,1]}$ defined in \eqref{def_hat_m} provides an appropriate approximation  for  the  process $(M_n(z,t))_{z\in\mathcal{C}^+, t\in[t_0,1]}$.
	
	\begin{lemma} \label{a36}
	Let $i\in\{1,2\}.$
	It holds for all large $n$ and for all $z \in\mathcal{C}^+, t \in [t_0,1]$ with probability 1
% 	\begin{align*}
% 		\left| \int \limits_{\mathcal{C}} f_i(z) \lb M_n(z,t) - \hat{M}_n(z,t) \rb dz \right| 
% 		& \leq 4 K  \varepsilon_n \Bigg\{ \left| x_r - \max \lb \lambda_{\max}^{\mathbf{B}_{n,t}} , 
% 		\lambda_{\max}^{\mathbf{T}_n} ( 1 + \sqrt{y_{\lfloor nt \rfloor }} )^2 \rb \right|\inv \\
% 		& + 	\left| x_l - \min \lb  \lambda_{\min}^{\mathbf{B}_{n,t}} , 
% 		\lambda_{\min}^{\mathbf{T}_n} I_{(0,1)}  (y_{ \lfloor nt \rfloor }) t_0 (1 - \sqrt{ y_{ \lfloor nt \rfloor } } )^2 \rb \right|\inv
% 		\Bigg\}.
% 	\end{align*}
% 	This implies with probability 1 
	(uniformly in $t \in [t_0, 1]$)
	\begin{align*}
		\Big| \int \limits_{\mathcal{C}} f_i(z) \lb M_n(z,t) - \hat{M}_n(z,t) \rb dz \Big|  = o(1), \textnormal{ as } n\to\infty. 
	\end{align*}
	\end{lemma}
	\begin{proof}[Proof of Lemma \ref{a36}]
	For convenience, we write $f_i = f.$
	Since $\mathcal{C} = \mathcal{C}^+ \cup \overline{\mathcal{C}^+}$ and $M_n(\overline{z},t)=\overline{M_n(z,t)}$ for all $z = x + iv \in\mathcal{C}^+$, we have
(using also the definition of $\hat{M}_n$)
		% für x_l <0 stimmen M_n und \hat{M}_n noch auf einer größeren Menge überein. 
		\begin{align*}
			\Big| \int \limits_{\mathcal{C}} f(z) \lb M_n(z,t) - \hat{M}_n(z,t) \rb dz \Big|
			& \leq K \int\limits_{[0, n\inv \varepsilon_n]} \Big\{ | M_n( x_r + iv, t ) - M_n (x_r + i n\inv \varepsilon_n, t ) | 
			\\ & ~~~~~~ ~~~~~~ + | M_n ( x_l + iv, t) - M_n ( x_l+i n\inv \varepsilon_n, t ) | \Big\} dv.
		\end{align*}
	Let $\Gamma_F$ denote the support of a c.d.f. $F$, then it follows by Proposition 2.4 in \cite{yao2015} that
		\begin{align}\label{det70}
			\left| s_F(z) \right| \leq \frac{1}{\textnormal{dist} (z, \Gamma_F)},
		\end{align}
		where $z \in \mathbb{C} \setminus \Gamma_F$ and $s_F$ is the Stieltjes transform of $F$.
		Using \eqref{a33} and Lemma \ref{a37}, we have for $v\in [0,n\inv \varepsilon_n]$ and sufficiently large $n$   
		\begin{align*}
			\textnormal{dist} \lb x_r + iv, \Gamma_{F^{\mathbf{B}_{n,t}}} \rb 
			& \geq \big| x_r -  \lambda_{\max}(\mathbf{B}_{n,t}) \big| 
			\geq \big| x_r - \max \big( \lambda_{\max}(\mathbf{B}_{n,t}),
			 \lambda_{\max}(\mathbf{T}_n) ( 1 + \sqrt{ y_{\lfloor nt \rfloor}} )^2  \big) \big|, \\
			\textnormal{dist} \lb x_l + iv, \Gamma_{\tilde{F}^{y_{\ntn}}, H_n} \rb 
			& \geq \big| x_l - \frac{\lfloor nt_0 \rfloor}{n}\lambda_{\min}(\mathbf{T}_n) I_{(0,1)} (y_{ \lfloor nt \rfloor })  (1 - \sqrt{ y_{ \lfloor nt \rfloor } } )^2  \big| \\
			& \geq \big| x_l - \min\big(  \lambda_{\min}(\mathbf{B}_{n,t}) , 
			\frac{\lfloor nt_0 \rfloor}{n} \lambda_{\min}(\mathbf{T}_n) I_{(0,1)} (y_{ \lfloor nt \rfloor })  (1 - \sqrt{ y_{ \lfloor nt \rfloor } } )^2  \big) \big|.
		\end{align*}
		Similarly, one can show that for sufficiently large n
		\begin{align*}
		\textnormal{dist} \lb x_r + iv, \Gamma_{\tilde{F}^{y_{\ntn}}, H_n} \rb 
		&\geq \left| x_r - \max \lb  \lambda_{\max}(\mathbf{B}_{n,t}),
		 \lambda_{\max}(\mathbf{T}_n) ( 1 + \sqrt{ y_{\lfloor nt \rfloor} } )^2 \rb \right|, \\
		 \textnormal{dist} \lb x_l + iv, \Gamma_{F^{\mathbf{B}_{n,t}}} \rb 
		 & \geq \big| x_l - \min\big(  \lambda_{\min}(\mathbf{B}_{n,t}) , 
		\frac{\lfloor nt_0 \rfloor}{n} \lambda_{\min}(\mathbf{T}_n) I_{(0,1)} (y_{ \lfloor nt \rfloor })  (1 - \sqrt{ y_{ \lfloor nt \rfloor } } )^2  \big) \big|.
		\end{align*}
	Recall   the definition of $M_n$, 
		then \eqref{det70}    implies
	\begin{align*}
		\Big| \int \limits_{\mathcal{C}} f(z) \lb M_n(z,t) - \hat{M}_n(z,t) \rb dz \Big| 
		& \leq 4 K  \varepsilon_n \Big\{ \left| x_r - \max \lb \lambda_{\max}(\mathbf{B}_{n,t}) , 
		\lambda_{\max}(\mathbf{T}_n) ( 1 + \sqrt{y_{\lfloor nt \rfloor }} )^2 \rb \right|\inv \\
		& + 	\big| x_l - \min \big( \lambda_{\min}(\mathbf{B}_{n,t}) , 
		\lambda_{\min}(\mathbf{T}_n) I_{(0,1)} (y_{ \lfloor nt \rfloor }) \frac{\lfloor nt_0 \rfloor}{n} (1 - \sqrt{ y_{ \lfloor nt \rfloor } } )^2 \big) \big|\inv
		\Big\}.
	\end{align*}
	Due to \eqref{a33}, for every $t \in [t_0,1]$, the denominators are bounded away from 0 for sufficiently large $n$ with probability 1 (nullset may depend on $t$). Note that for every $n \in \N$, there are only finitely many $t_1,t_2 \in [t_0,1]$ such that $\nt \neq \ntt$. That is, since the countable union of nullsets is again a nullset, we find that with probability 1 (uniformly in $t$)
	\begin{align*}
		& \limsup\limits_{n \to \infty} \Big| \int \limits_{\mathcal{C}} f(z) \lb M_n(z,t) - \hat{M}_n(z,t) \rb dz \Big| \\
		& \leq 4 K \lim\limits_{n \to \infty} \varepsilon_n   
		 \Big\{ \Big( x_r - \limsup\limits_{n\to\infty} \max \lb \lambda_{\max}(\mathbf{B}_{n,t}) , 
		\lambda_{\max}(\mathbf{T}_n) ( 1 + \sqrt{y_{\lfloor nt \rfloor }} )^2 \rb \Big) \inv \\
		 & + 	\Big( \liminf\limits_{n\to\infty} \min \Big( \lambda_{\min}(\mathbf{B}_{n,t}) , 
		\lambda_{\min}(\mathbf{T}_n) I_{(0,1)} (y_{ \lfloor nt \rfloor }) \frac{\lfloor nt_0 \rfloor}{n} (1 - \sqrt{ y_{ \lfloor nt \rfloor } } )^2 \Big) - x_l \Big) \inv
		\Big\} \\
		& \leq 4 K \lim\limits_{n \to \infty}  \varepsilon_n  \Big\{ \Big( x_r - \limsup\limits_{n\to\infty} 
		\lambda_{\max}(\mathbf{T}_n) ( 1 + \sqrt{y_{\lfloor nt_0 \rfloor }} )^2 \Big) \inv \\
		 & + 	\Big( \liminf\limits_{n\to\infty} 
		\lambda_{\min}(\mathbf{T}_n) I_{(0,1)} (y_{ \lfloor nt_0 \rfloor }) \frac{\lfloor nt_0 \rfloor}{n} (1 - \sqrt{ y_{ \lfloor nt_0 \rfloor } } )^2 - x_l \Big) \inv
		\Big\} = 0.
	\end{align*}
	% mit Wahrscheinlichkeit 1, da x_r - limsup... >0 bzw liminf.. - x_l >0 mit Wahrscheinlichkeit 1 (folgt aus der fast sicheren Konvergenz der Eigenwerte)
	\end{proof}

\section{More details for the proof of  Theorem \ref{lem} }

In this section we provide the remaining arguments in the proof of  Theorem \ref{lem} in Section  \ref{sec_proof_lem}.
Several further very technical results are given in Section \ref{aux}.

\subsection{Proof of Lemma \ref{hat}} \label{sec_proof_lemma_hat}

To be precise, recall the definition of  $Z_n^1$ and $Z_n^2$
in \eqref{det7} and \eqref{det8} and define $\hat{Z}_n^1$ and $\hat{Z}_n^2$ by  $Z_n^1$ and $Z_n^2$, respectively, in the same way as $\hat{M}_n^1$ is defined by  $M_n^1$ in equation  \eqref{def_hat_m}.  
		The bounds \eqref{ineq_hat2} and \eqref{ineq_hat3} for the moments of $\hat{Z}_n^1$ and $\hat{Z}_n^2$ follow directly from corresponding bounds \eqref{tight_t1} and \eqref{tight_t2} in Lemma \ref{no_hat}. 
		
		We continue by proving the first assertion \eqref{ineq_hat1}.
		If $z_1$ and $z_2$ are both contained in $\mathcal{C}_n$, the assertion directly follows from \eqref{tight_z}. Otherwise we assume that $N\in\N$ is sufficiently large so that for all $n\geq N$
		\begin{align*}
			v_0 > \varepsilon_n n\inv.
		\end{align*}
		Let $z_1\in \mathcal{C}_n$ and $z_2 \notin \mathcal{C}_n$, that is,
		$0 \leq \textnormal{Im} (z_2) \leq \varepsilon_n n\inv 
		\leq \textnormal{Im}(z_1).$
		With the notation Re$(z_2)= x\in\{x_l,x_r\}$ we have from \eqref{tight_z}
		\begin{align*}
			\E | \hat{M}_n^1(z_1,t) - \hat{M}_n^1(z_2,t) |^{2+\delta}
			= & \E | M_n^1(z_1,t) - M_n^1(x + i\varepsilon_n n\inv,t) |^{2+\delta}
			\lesssim  |z_1 - (x + i\varepsilon_n n\inv ) |^{2 + \delta} \\
			\leq &   \left[ ( \textnormal{Re} (z_1) - x )^2 
			+  (\textnormal{Im}(z_1) - \varepsilon_n n\inv)^2
			\right]^{(2+\delta)/2} \\
			\leq &  \left[ ( \textnormal{Re} (z_1) - x )^2 
			+  (\textnormal{Im}(z_1) - \textnormal{Im}(z_2))^2 \right]^{(2 + \delta)/2} \\
			= &  |z_1 - z_2|^{2 + \delta}.
		\end{align*}
		Finally, if both  $z_1,z_2 \in\mathcal{C}^+ \setminus \mathcal{C}_n$, it follows from \eqref{tight_z} that
		\begin{align*}
			\E | \hat{M}_n^1(z_1,t) - \hat{M}_n^1(z_2,t) |^{2+\delta}
			= & \E | M_n^1( \textnormal{Re} (z_1) + i \varepsilon_n n\inv ) 
			- M_n^1( \textnormal{Re} (z_2) + i \varepsilon_n n\inv ) |^{2 + \delta} \\
			\lesssim &  |\textnormal{Re}(z_1) - \textnormal{Re}( z_2 )|^{ 2 + \delta} 
			\leq   |z_1 - z_2| ^{2 + \delta},
		\end{align*}
which completes the proof of Lemma \ref{hat}.		

\subsection{Proof of Lemma \ref{a73}} \label{pra73}

	We begin by deriving an alternative form for $R_{n,t}(z)$. By 
	\begin{align*}
		\E [ \tilde{s}_{n,t} (z) ] = \frac{1}{y_{\ntn}} \E [ \tilde{\su}_{n,t} (z)] + \frac{1}{z y_{\ntn} } - \frac{1}{z}
	\end{align*}		
	and Lemma \ref{a52}, we have 
	\begin{align*} 
		 -\frac{\ntn}{n} \E \tilde{\su}_{n,t}(z) \Big( - z- & \frac{1}{\E [ \tilde{\su}_{n,t}(z) ] } + y_{n} \int \frac{\lambda d H_n(\lambda)}{1+ \lambda \frac{\ntn}{n}\E [ \tilde{\su}_{n,t} (z) ]}\Big) \\
		= & y_{n} \int \frac{dH_n(\lambda)}{1 + \lambda \frac{\ntn}{n} \E [ \tilde{\su}_{n,t}(z)] } + z y_{n} \E [ \tilde{s}_{n,t}(z ) ] \\
		 = & - n\inv \sum\limits_{j=1}^{\ntn} \E \Big[ \beta_{j,t}( z ) 
		 \Big( \rd_{j}^\star \D_{j,t}\inv(z ) ( \frac{\ntn}{n} \E \tilde{\su}_{n,t}(z) \T_n + \mathbf{I} )\inv \rd_{j}  \\
		 & ~~~~~
		 - \frac{1}{n}  \E [ \tr ( \frac{\ntn}{n} \E \tilde{\su}_{n,t}(z) \T_n + \mathbf{I} )\inv \T_n \D_t\inv(z ) ] \Big) \Big] \\
		  = & - y_{n} n\inv \sum\limits_{j=1}^{\ntn} \E \Big[ \beta_{j,t}( z ) 
		 \Big( \mathbf{q}_j^\star \T_n\sq \D_{j,t}\inv(z) ( \frac{\ntn}{n}\E \tilde{\su}_{n,t}(z) \T_n + \mathbf{I} )\inv \T_n\sq \mathbf{q}_j
		 \\
		 & ~~~~~
		 - \frac{1}{p}  \E [ \tr ( \frac{\ntn}{n} \E \tilde{\su}_{n,t}(z) \T_n + \mathbf{I} )\inv \T_n \D_t\inv(z) ] \Big) \Big] \\
		 = & - y_{n} n\inv \sum\limits_{j=1}^{\ntn} \E \left[ \beta_{j,t}( z ) 
		 d_{j,t}(z ) \right] 
		 = - \frac{\ntn}{n} y_{\ntn} n\inv \sum\limits_{j=1}^{\ntn} \E \left[ \beta_{j,t}( z ) 
		 d_{j,t}(z ) \right]. \\
	\end{align*}
	This implies
	\begin{align*}
		\frac{\ntn}{n} R_{n,t}(z) = -  z - \frac{1}{ \E [ \tilde{\su}_{n,t}(z) ] } + y_{n} \int \frac{\lambda d H_n(\lambda)}{1 + \lambda \frac{\ntn}{n} \E [ \tilde{\su}_{n,t}(z) ] },
	\end{align*}
	and we can conclude
	\begin{align}
	   \E [ \tilde{\su}_{n,t} (z) ] 	& - \tilde{\su}_{n,t}^0 (z) \nonumber \\
	= & \frac{1}{- z + y_{n} \int \frac{\lambda dH_n(\lambda) }{1 + \lambda \frac{\ntn}{n} \E [ \tilde{\su}_{n,t} (z) ] } - \frac{\ntn}{n} R_{n,t}(z) }
	- \frac{1}{- z + y_{n} \int \frac{\lambda dH_n(\lambda) }{1 + \lambda \frac{\ntn}{n} \tilde{\su}_{n,t}^0 (z) }  } \nonumber \\
	= & \frac{y_{n} \Big(  \int \frac{\lambda dH_n(\lambda) }{1 + \lambda \frac{\ntn}{n} \tilde{\su}_{n,t}^0 (z) } -  \int \frac{\lambda dH_n(\lambda) }{1 + \lambda \frac{\ntn}{n}\E [ \tilde{\su}_{n,t} (z) ] }  \Big) + \frac{\ntn}{n} R_{n,t}(z)  }
	{ \Big( - z + y_{n} \int \frac{\lambda dH_n(\lambda) }{1 + \lambda \frac{\ntn}{n} \E [ \tilde{\su}_{n,t} (z) ]  } -\frac{\ntn}{n}  R_{n,t}(z) \Big) \Big( - z + y_{n} \int \frac{\lambda dH_n(\lambda) }{1 + \lambda \frac{\ntn}{n} \tilde{\su}_{n,t}^0 (z) } \Big)  } \nonumber \\
	= & \frac{y_{n} \frac{\ntn}{n}\lb \E [ \tilde{\su}_{n,t}(z) ] - \tilde{\su}_{n,t}^0(z) \rb \int \frac{\lambda^2 dH_n(\lambda)}{( 1 + \lambda \frac{\ntn}{n}\E [ \tilde{\su}_{n,t} (z) ] ) ( 1 + \lambda \frac{\ntn}{n}\tilde{\su}_{n,t}^0(z) ) }  + \frac{\ntn}{n} R_{n,t}(z)  }
	{ \Big( - z + y_{n} \int \frac{\lambda dH_n(\lambda) }{1 + \lambda \frac{\ntn}{n} \E [ \tilde{\su}_{n,t} (z) ] } - \frac{\ntn}{n} R_{n,t}(z) \Big) \Big( - z + y_{n} \int \frac{\lambda dH_n(\lambda) }{1 + \lambda \frac{\ntn}{n}\tilde{\su}_{n,t}^0 (z) } \Big)  } \label{a90}\\
	= & \frac{y_{n} \frac{\ntn}{n}\lb \E [ \tilde{\su}_{n,t}(z) ] - \tilde{\su}_{n,t}^0(z) \rb \int \frac{\lambda^2 \tilde{\su}_{n,t}^0(z) dH_n(\lambda)}{( 1 + \lambda \frac{\ntn}{n}\E [\tilde{\su}_{n,t} (z) ] ) ( 1 + \lambda \frac{\ntn}{n} \tilde{\su}_{n,t}^0(z) ) }   }
	{  - z + y_{n} \int \frac{\lambda dH_n(\lambda) }{1 + \lambda \frac{\ntn}{n}\E [ \tilde{\su}_{n,t} (z) ] } - \frac{\ntn}{n}  R_{n,t}(z)  } \nonumber \\
	& + \frac{\ntn}{n} R_{n,t}(z)  \E [ \tilde{\su}_{n,t}(z) ] \tilde{\su}_{n,t}^0(z) . \nonumber
	\end{align}

	\subsection{Proof of Theorem \ref{thm_stieltjes}} \label{sec_proof_stieltjes}
%	\subsection{Uniform convergence of the Stieltjes transform} \label{sec_proof_stieltjes}

 In this section, $D[0,1]^2$ denotes the Skorokhod space  on $[0,1]^2$ \citep[see][for a formal definition]{bickel1971,neuhaus1971}. 
We  will identify the set $\mathcal{C}^+ \times [0,1]$ with the square  $[0,1]^2$ and  proceed in several steps. First, we will show a uniqueness condition,  second we prove the existence of a Skorokhod-limit of $(\E [\tilde{s}_{n,\cdot}(\cdot) ])_{n\in\N}$.  We conclude by proving that the Skorokhod-limit is in fact an uniform limit.

		\begin{lemma} \label{uni}
			Let $(\E [ \tilde{s}_{k(n), t}(z) ])_{n\in\N}$ and $(\E [ \tilde{s}_{l(n), t}(z) ])_{n\in\N}$ be two subsequences of $(\E [ \tilde{s}_{n, t}(z) ])_{n\in\N}$
			and $m_1$ and $m_2$ be functions on $\mathcal{C}^+ \times [t_0,1]$. 
			If for $z\in\mathcal{C}^+, t\in[t_0,1]$, 
			\begin{align*}
				\lim\limits_{n\to\infty} \E [ \tilde{s}_{k(n), t}(z) ] = m_1(z,t)
		\text{~~and ~~~}
				\lim\limits_{n\to\infty} \E [ \tilde{s}_{j(n), t}(z) ] = m_2(z,t),
			\end{align*}
			then we have for $z\in\mathcal{C}^+, t\in[t_0,1]$
			\begin{align*}
			    m_1(z,t) = m_2(z,t) = \tilde{s}_t(z),
			\end{align*}
			where $\tilde{s}_t$ denotes the Stieltjes transform of $\tilde{F}^{y_t,H}$ given in \eqref{def_gen_MP_lim}
		\end{lemma}

	\begin{proof}[Proof of Lemma \ref{uni}]
	We will start   showing that a potential limit of the sequence  $(\E [ \tilde{s}_{n,\cdot}(\cdot)])_{n\in\N}$  satisfies an equation which admits a unique solution. 
	%obeys equation \eqref{a60}. 
	For this purpose we will  adapt ideas from \cite{baizhou2008} and also correct some arguments  in step 2 in the proof of their Theorem 1.1. To be precise, define for  $z\in\mathcal{C}^+$ and $t\in[t_0,1]$   
	\begin{align*}
		\mathbf{K} =b_t(z) \T_n ,
		% \mathbf{K} =\overline{\beta}_{k,t}(z) \T_n , Fehler bei Bai und Zhou!!!
	\end{align*}
	% Fehler bei Bai und Zhou!} \\ 
	and note that  
	\begin{align*}
	\D_t(z) - \Big( \frac{\ntn}{n}\mathbf{K} - z\mathbf{I} \Big) = \sum\limits_{k=1}^{\ntn} \rd_k \rd_k^\star - \frac{\ntn}{n}\mathbf{K}.
	\end{align*}

Multiplying  with $( (\ntn/ n)\mathbf{K} - z \mathbf{I})\inv$  and $\D_t\inv(z)$  from the left and from the  right, respectively, 
and using  identity (6.1.11) from \cite{bai2004} yields
	\begin{align*}
	& \Big( \frac{\ntn}{n}\mathbf{K} - z \mathbf{I} \Big) \inv - \D_t\inv(z) \\
	= &  \sum\limits_{k=1}^{\ntn} \Big( \frac{\ntn}{n}\mathbf{K} - z \mathbf{I}\Big) \inv \rd_k \rd_k^\star \D_t\inv(z) 
	- \frac{\ntn}{n} \Big( \frac{\ntn}{n}\mathbf{K} - z \mathbf{I}\Big) \inv \mathbf{K} \D_t\inv(z) \\
	= & \sum\limits_{k=1}^{\ntn} \beta_{k,t}(z) \Big( \frac{\ntn}{n}\mathbf{K} - z \mathbf{I}\Big) \inv \rd_k \rd_k^\star \D_{k,t}\inv(z)
	- \frac{\ntn}{n} \Big( \frac{\ntn}{n}\mathbf{K} - z \mathbf{I}\Big) \inv \mathbf{K} \D_t\inv(z). \\
	\end{align*}
	This implies for $l\in\{0,1\}$ 
%	\begin{align*}
%		& \T_n^l \lb \frac{\ntn}{n}\mathbf{K} - z \mathbf{I} \rb \inv - \T_n^l \D_t\inv(z) \\
%		=& \sum\limits_{k=1}^{\ntn} \beta_{k,t}(z) \T_n^l \lb \frac{\ntn}{n}  \mathbf{K} - z\mathbf{I} \rb\inv \rd_k \rd_k^\star \D_{k,t}\inv(z) - \frac{\ntn}{n}\T_n^l \lb \frac{\ntn}{n}\mathbf{K} - z\mathbf{I} \rb \inv \mathbf{K} \D_t\inv(z). 
%	\end{align*}
%	Taking traces and dividing by $p$, we conclude
	\begin{align*}
		& \frac{1}{p} \tr  \T_n^l \Big( \frac{\ntn}{n} \mathbf{K} - z \mathbf{I} \Big) \inv 
		- \frac{1}{p} \tr \T_n^l \D_t\inv(z) \\
		= & \frac{1}{p}\sum\limits_{k=1}^{\ntn} \beta_{k,t}(z) \rd_k^\star \D_{k,t}\inv(z) \T_n^l \Big( \frac{\ntn}{n}\mathbf{K} - z\mathbf{I} \Big) \inv \rd_k 
		 - \frac{1}{p} \tr \frac{\ntn}{n}\T_n^l \Big( \frac{\ntn}{n}\mathbf{K} - z\mathbf{I} \Big)\inv \mathbf{K} \D_t\inv(z) \\
		 = & \frac{1}{p} \sum\limits_{k=1}^{\ntn} \beta_{k,t}(z) \varepsilon_k,
	\end{align*}
	where 
	\begin{align*}
		\varepsilon_k 
		= & \rd_k^\star \D_{k,t}\inv(z) \T_n^l \Big( \frac{\ntn}{n}\mathbf{K} - z\mathbf{I} \Big)\inv \rd_k 
		 - n\inv \beta_{k,t}\inv(z) \tr \T_n^l \Big( \frac{\ntn}{n}\mathbf{K} - z\mathbf{I} \Big) \inv \mathbf{K} \D_t\inv(z) 
		\\
		= &   \rd_k^\star \D_{k,t}\inv(z) \T_n^l \Big( \frac{\ntn}{n}\mathbf{K} - z\mathbf{I} \Big)\inv \rd_k 
		 - n\inv \tr \T_n^l \Big( \frac{\ntn}{n}\mathbf{K} - z\mathbf{I} \Big) \inv \mathbf{K} \D_t\inv(z) 
		 \lb 1 + \rd_k^\star \D_{k,t}\inv(z) \rd_k \rb.
	\end{align*}
	We decompose $\varepsilon_k = \varepsilon_{k1} + \varepsilon_{k2} + \varepsilon_{k3}$, where
	\begin{align*}
		\varepsilon_{k1} = &
		n\inv \tr \T_n^{l+1} \Big( \frac{\ntn}{n} \mathbf{K} - z \mathbf{I} \Big) \inv \D_{k,t}\inv (z) 
		- n\inv \tr \T_n^{l+1} \Big( \frac{\ntn}{n} \mathbf{K} - z \mathbf{I} \Big)\inv \D_{t}\inv (z) \\
		\varepsilon_{k2} = &
		\rd_k^\star \D_{k,t}\inv(z) \T_n^l \Big( \frac{\ntn}{n} \mathbf{K} - z \mathbf{I} \Big)\inv \rd_k
		- n\inv \tr \T_n^{l+1} \Big( \frac{\ntn}{n} \mathbf{K} - z \mathbf{I} \Big)\inv \D_{k,t}\inv(z) \\
		\varepsilon_{k3} = & 
		 - n\inv \tr \T_n^l \Big( \frac{\ntn}{n}\mathbf{K} - z\mathbf{I} \Big) \inv \mathbf{K} \D_t\inv(z) 
		 \lb( 1 + \rd_k^\star \D_{k,t}\inv(z) \rd_k \rb)  \\
		 & 
		 +n\inv \tr \T_n^{l+1} \Big( \frac{\ntn}{n} \mathbf{K} - z \mathbf{I} \Big)\inv \D_{t}\inv (z) \\
		 = & - n\inv \tr \T_n^{l+1} \Big( \frac{\ntn}{n} \mathbf{K} - z \mathbf{I} \Big)\inv \D_t\inv(z) 
		 \left\{ b_{t}(z)\lb \rd_k^\star \D_{k,t}\inv(z) \rd_k + 1\rb- 1 \right\},
	\end{align*}
	%hier war ein Fehler bei Bai und Zhou (b_t)
	and we have used the fact that the matrices $\T_n$ and $( (\ntn/n) \mathbf{K} - z\mathbf{I} ) \inv$ commute.
    Similar arguments as given by \cite{baizhou2008} for their estimate (3.4)
yield
	\begin{align*}
		\Big| \Big| \Big( \frac{\ntn}{n} \mathbf{K} - z \mathbf{I} \Big)\inv \Big| \Big| \leq K 
		%\textnormal{ and }
		%\left| \left| \lb \frac{\ntn}{n} \mathbf{\tilde{K}} - z \mathbf{I} \rb\inv \right| \right| \leq K 
		,
	\end{align*}
and this estimate  can  be used to show   
	\begin{align*} 
	%\label{a54}
	\E |\varepsilon_{ki} |^2 \to 0, ~ n\to\infty~~, i \in \{ 1,2,3\} .
	\end{align*}
%     More precisely, from  Lemma 2.6 in \cite{silversteinbai1995} it follows that 
% 	\begin{align*}
% 		| \varepsilon_{k1} | \to 0 ,
% 	\end{align*}
% 	as $n\to\infty$. The second moment of $\varepsilon_{k2}$ can be bounded by using (9.9.6) in \cite{bai2004}. For $\varepsilon_{k3},$ we note that 
% 	\begin{align*}
% 		| \varepsilon_{k3} |
% 		= & \Big| n\inv \tr \T_n^{l+1} \Big( \frac{\ntn}{n} \mathbf{K} - z \mathbf{I} \Big)\inv \D_t\inv(z) 
% 		 \Big\{ b_{t}(z)\lb \rd_k^\star \D_{k,t}\inv(z) \rd_k 
% 		 - \E \left[ n\inv \tr \T_n \D_{t}\inv(z) \right] \rb \Big\} \Big| \\
% 		 \leq & K \lb | \varepsilon_{k31} | + | \varepsilon_{k32} | + | \varepsilon_{k33} | \rb ,
% 	\end{align*}
% 	where 
% 	\begin{align*}
% 		\varepsilon_{k31} = & \rd_k^\star \D_{k,t}\inv(z) \rd_k - n\inv \tr \T_n \D_{k,t}\inv(z), \\
% 		\varepsilon_{k32} = & n\inv \tr \T_n \D_{k,t}\inv(z) - n\inv \tr \T_n \D_{t}\inv(z), \\
% 		\varepsilon_{k33} = & n\inv \tr \T_n \D_{t}\inv(z) - \E \left[ n\inv \tr \T_n \D_{t}\inv(z) \right] .
% 	\end{align*}
% 	Since similar terms are studied intensively in the proof of Lemma \ref{no_hat}, %in Appendix \ref{sec_mom_ineq},
% 	%in Sections \ref{fidis_1} and \ref{sec_tight}, 
% 	we omit a detailed proof of \eqref{a54}. 
	This implies for $l \in \{0,1\}$
	\begin{align} \label{a55}
	\frac{1}{p} \Big( \E \tr  \T_n^l \Big( \frac{\ntn}{n} \mathbf{K} - z \mathbf{I} \Big) \inv 
		- \E \tr \T_n^l \D_t\inv(z) \Big) \to 0, ~n\to\infty.
	\end{align}

	Using \eqref{a55} with $l=0$ for the first line and $l=1$ for the second one, we have
	\begin{align}
	&\frac{1}{p} \E \tr \Big( \frac{ \frac{\ntn}{n}\T_n}{1 + y_{\ntn} a_{n,t}(z)} - z \mathbf{I}\Big)\inv- \E \tilde{s}_{n,t} (z) \to 0	, \label{a56}\\
	& \frac{1}{p} \E \tr \frac{\ntn}{n} \T_n \Big( \frac{ \frac{\ntn}{n}\T_n}{1 + y_{\ntn} a_{n,t}(z)} - z \mathbf{I}\Big)\inv- a_{n,t}(z) \to 0	, \label{a57}
	\end{align}
	where $a_{n,t}(z) = ( \ntn /n ) p\inv \E \tr \T_n \D_t\inv(z)$, so that $1+y_{\ntn} a_{n,t} (z) = b_t(z)$. 
	We use
	\begin{align*}
		\Big| \frac{1}{1 + y_{\ntn} a_{n,t}(z) } \Big| \leq \frac{|z|}{v}
	\end{align*}
	to conclude from \eqref{a57}
	\begin{align*}
		1 + \frac{z}{p} \E \tr \Big( \frac{ \frac{\ntn}{n} \T_n}{1 + y_{\ntn} a_{n,t}(z) } - z \mathbf{I} \Big) \inv 
		- \frac{a_{n,t}(z)}{1 + y_{\ntn} a_{n,t}(z)} \to 0.
	\end{align*}
	Combining this with \eqref{a56} yields
	\begin{align*}
		1 + z \E \tilde{s}_{n,t}(z) - \frac{a_{n,t}(z)}{1 + y_{\ntn} a_{n,t}(z)} \to 0
	\end{align*}
	and, by rearranging terms and multiplying with $y_{\ntn}$,
	\begin{align*}
		\frac{1}{1 + y_{\ntn} a_{n,t}(z) }  
		= 1 - y_{\ntn} ( 1 + z \E \tilde{s}_{n,t}(z) ) + o(1).
	\end{align*}
	Substituting this in \eqref{a56}, we get
	\begin{align} \label{a58}
		\frac{1}{p} \E \tr \Big (  \frac{\ntn}{n}\T_n \lb 1 - y_{\ntn} ( 1 + z \E\tilde{s}_{n,t}(z) ) \rb  - z \mathbf{I}\Big )\inv- \E \tilde{s}_{n,t} (z) \to 0.
	\end{align}
	Due to \eqref{a58}, any potential limit $\tilde{s}_\cdot (\cdot)$ of $(\E \tilde{s}_{n,\cdot}(\cdot) )_{n\in\N}$ satisfies
	\begin{align*}
		\tilde{s}_t(z) = \int \frac{1}{ \lambda t ( 1 - y_t( 1 + z \tilde{s}_t(z))) - z } dH(\lambda).
	\end{align*}	 
	It follows from Theorem 1.1 in \cite{baizhou2008}, that this equation admits a unique solution $\tilde{s}_\cdot (\cdot)$.  

\end{proof}
	In the following lemma, we consider for technical reasons the functions $\hat{s}_{n,\cdot} (\cdot) : \mathcal{C}^+ \times [0,1] \to \mathbb{C}$ with $\hat{s}_{n,t}(z) = 0$ for $t<t_0$ and for $t\in[t_0,1], z=x+iv\in\mathcal{C}^+$
	\begin{align*}
		\hat{s}_{n,t}(z)
		= \begin{cases}
		\tilde{s}_{n,t}(z) & : z \in\mathcal{C}_n \\
		\tilde{s}_{n,t}(x_r + i  n\inv \varepsilon_n ) 
		& : x=x_r, ~v\in[0,n\inv\varepsilon_n] \\ 		
		\tilde{s} _{n,t}(x_l + i  n\inv \varepsilon_n )
		& : x=x_l, ~v\in[0,n\inv\varepsilon_n] 
		\end{cases}
	\end{align*}
	and for $t\in[0,1], z\in\mathcal{C}^+$
	\begin{align*}
	\hat{m}_{n,t}(z) =
		\begin{cases}
			 \lim\limits_{t\to 1} \hat{s}_{n,t}(z) 
			 = \hat{s}_{n,\frac{n - 1}{n} } (z) &: t=1 \\
			 \hat{s}_{n,t}(z) &: t\in [0,1).
		\end{cases}
	\end{align*}
	Note that for $t\in[0,1]$, the functions $\hat{s}_{n,t}(\cdot)$ and $\tilde{s}_{n,t}(\cdot)$ coincide on $\mathcal{C}_n$, $n\in\N$ and that for $z\in\mathcal{C}^+$, the functions $\hat{s}_{n,t}(z)$ and $\hat{m}_{n,t}(z)$ differ only in the point  $t=1$. 
	\begin{lemma} \label{lem_cp_closure}
	The set $\{ \E [ \hat{m}_{n, \cdot} (\cdot) ] : n \in \N \}$ has a compact closure in
	the Skorokhod space $D[0,1]^2$. 
	\end{lemma}

	\begin{proof}[Proof of Lemma \ref{lem_cp_closure}]
	The sequence $(\E \hat{m}_{n,\cdot}(\cdot))_{n\in \N}$ is  bounded, since by Lemma \ref{mom_Dinv}, we get uniformly with respect to $t \in [t_0,1], z\in\mathcal{C}_n, n\in\N$
	\begin{align*}
		\left| \E \tilde{s}_{n,t}(z) \right| = \frac{1}{p} \left| \E \tr \D_t\inv(z) \right| \leq \E || \D_t\inv(z) || \leq K.
	\end{align*}
	 We observe for $t_2 \geq t_1$ and $z\in\mathcal{C}_n$
	\begin{align*}
		| \E [ \tilde{s}_{n,t_1}(z) ] - \E [ \tilde{s}_{n,t_2} (z) ] | 
		= & \Big| \frac{1}{p} \E \left[ \tr \lb \D_{t_1}\inv(z) - \D_{t_2}\inv(z) \rb \right] \Big|
		= \Big| \frac{1}{p} \sum\limits_{j= \nt + 1 }^{\ntt} \E \left[ \rd_j^\star \D_{t_1}\inv(z) \D_{t_2}\inv(z) \rd_j \right] \Big| 
		\nonumber \\
		= & \Big| \frac{1}{p} \sum\limits_{j= \nt + 1 }^{\ntt} \Big\{ \E \left[ \rd_j^\star \D_{t_1}\inv(z) \D_{j,t_2}\inv(z) \rd_j \right] \nonumber \\
		& - \E \left[ \beta_{j,t_2}(z) \rd_j^\star \D_{t_1}\inv(z)  \D_{j,t_2}\inv(z) \rd_j \rd_j^\star \D_{j,t_2}\inv(z) \rd_j \right] \Big\} \Big| \nonumber \\
		\leq & K y_n \frac{\ntt - \nt}{n} \leq K \frac{\ntt - \nt }{n}, 
	\end{align*}
	where the constant $K$ is independent of $t_1,t_2,z,n$.
	Thus, we have
	\begin{align*} % \label{a64}
		\big| \E [ \hat{s}_{n,t_1}(z) ] - \E [ \hat{s}_{n,t_2} (z) ]\big | 
		\leq K \frac{\ntt - \nt }{n}.
	\end{align*}
	We aim to show
	\begin{align} \label{w2'}
		\lim\limits_{\delta \to 0} \sup\limits_{n\in \N} \sup\limits_{\substack{(t_1,t_2,t) \in \mathcal{A}_\delta, \\ z \in\mathcal{C}^+ } } \min 
		\lb | \E [ \hat{s}_{n,t}(z) ] - \E [ \hat{s}_{n,t_1} (z) ] | , | \E [ \hat{s}_{n,t}(z) ] - \E [ \hat{s}_{n,t_2} (z) ] | \rb
		= 0,
	\end{align}
	where 
	\begin{align*} 
		\mathcal{A}_\delta = \{ (t_1,t_2,t)~ : ~ t_1 \leq t \leq t_2, ~ t_2 - t_1 \leq \delta \}.
	\end{align*}
	Let $\varepsilon >0$ be given. We choose $N\in \N$ sufficiently large such that 
$
		\frac{1}{N} < \varepsilon
$
	and $\delta > 0$ sufficiently small such that $\delta < \varepsilon$ and for all $n\in \{1,\ldots,N\}$
	\begin{align*}
		\ntn - \ntt = 0 \textnormal{ or } \ntn - \nt = 0,
	\end{align*}
	where $(t_1,t_2,t)\in\mathcal{A}_\delta$. Then, it holds
	\begin{align*}
		\sup\limits_{n \leq N} \sup\limits_{\substack{(t_1,t_2,t) \in \mathcal{A}_\delta, \\ z \in\mathcal{C}^+ } } \min 
		\lb | \E [ \hat{s}_{n,t}(z) ] - \E [ \hat{s}_{n,t_1} (z) ] | , | \E [ \hat{s}_{n,t}(z) ] - \E [ \hat{s}_{n,t_2} (z) ] | \rb
		= 0.
	\end{align*}
	For $n \geq N$ we conclude
	\begin{align*}
		\big | \E [ \hat{s}_{n,t}(z) ] - \E [ \hat{s}_{n,t_1} (z) ] \big | 
		\leq & K \frac{\ntn - \nt}{n} \\
		\leq & K \Big ( \Big | \frac{\ntn - nt }{n} \Big| + |t_1 - t| + \Big| \frac{\nt - nt_1}{n} \Big | \Big ) \leq 3  \varepsilon K 
	\end{align*}
	and obtain
	\begin{align*}
		\sup\limits_{n\geq N} \sup\limits_{\substack{(t_1,t_2,t) \in \mathcal{A}_\delta, \\ z \in\mathcal{C}^+ } } \min 
		\big \{ | \E [ \hat{s}_{n,t}(z) ] - \E [ \hat{s}_{n,t_1} (z) ] | , | \E [ \hat{s}_{n,t}(z) ]  - \E [ \hat{s}_{n,t_2} (z) ] | \big \}
		\leq 3 \varepsilon K.
	\end{align*}
	Thus, \eqref{w2'} holds true. Similarly, one can show 
	\begin{align*}
		\lim\limits_{\delta \to 0} \sup\limits_{n\in \N} \sup\limits_{\substack{t_1,t_2 \in [1 - \delta, 1), \\ z\in\mathcal{C}^+ }}
		| \E[ \hat{s}_{n,t_1}(z) ] - \E [ \hat{s}_{n,t_2}(z) ] | = 0. 
	\end{align*}
	By definition, this implies
	\begin{align*} %\label{m_hat}
		\lim\limits_{\delta \to 0} \sup\limits_{n\in \N} \sup\limits_{\substack{t_1,t_2 \in [1 - \delta, 1], \\ z\in\mathcal{C}^+ }}
		| \E[ \hat{m}_{n,t_1}(z) ] - \E [ \hat{m}_{n,t_2}(z) ] | = 0.
	\end{align*}
	Since $\hat{s}_{n,t} (z) = 0 $ for $t < t_0$, we also have for $\delta < t_0$
	\begin{align*}
		 \sup\limits_{n\in \N} \sup\limits_{\substack{t_1,t_2 \in [0, \delta), \\ z \in\mathcal{C}^+ }} 
		| \E [ \hat{s}_{n,t_1}(z) ] - \E [ \hat{s}_{n,t_2}(z) ] | = 0. 
	\end{align*}
	Therefore, it follows from the proof of Theorem 14.4 in \cite{billingsley1968} that 
	\begin{align}
	0 = & \lim\limits_{\delta\to 0} ~ \sup\limits_{n\in\N} ~ \sup\limits_{z \in \mathcal{C}^+} ~ 
	\inf\limits_{\substack{(t_0, \ldots, t_r) \in \mathcal{B}_{\delta,r}, \\ r \in \N}} ~ \max\limits_{0 \leq i \leq r } ~ 
	\sup\limits_{t, t' \in [t_{i-1}, t_i) } 
	| \E [ \hat{s}_{n,t}(z) ] - \E [ \hat{s}_{n,t'}(z) ] | \nonumber \\
	= & \lim\limits_{\delta\to 0} ~ \sup\limits_{n\in\N} ~ \sup\limits_{z \in \mathcal{C}^+} ~ 
	\inf\limits_{\substack{(t_0, \ldots, t_r) \in \mathcal{B}_{\delta,r}, \\ r \in \N}} ~ \max\limits_{0 \leq i \leq r } ~ 
	\sup\limits_{t, t' \in [t_{i-1}, t_i\rangle } 
	| \E [ \hat{m}_{n,t}(z) ] - \E [ \hat{m}_{n,t'}(z) ] | , \label{g1}
	\end{align}
	where $[t_{i-1}, t_i\rangle$ is defined as $[t_{i-1}, t_i]$ if $t_i=1$ and as $[t_{i-1}, t_i)$ otherwise and we set
	\begin{align*}
	\mathcal{B}_{\delta,r} = \{ (t_0, \ldots, t_r): 0 = t_0 < t_1 < \ldots < t_r = 1, ~ 
	t_i - t_{i - 1} > \delta \textnormal{ for } i\in \{1,\ldots, r\} \}.
	\end{align*}
	For the next step, we have for $z_1,z_2 \in\mathcal{C}_n$
	\begin{align*}
		| \E [ \tilde{s}_{n,t} (z_1) ] - \E [ \tilde{s}_{n,t}(z_2) ] | = & \frac{1}{p} | z_1 - z_2 | | \E \tr \D_{t}\inv(z_1) \D_{t}\inv(z_2) | 
		\leq K  | z_1 - z_2 | \E || \D_{t}\inv(z_1) \D_{t}\inv(z_2) || \\
		\leq & K |z_1 - z_2|
	\end{align*}
	uniformly in $t\in[t_0,1]$,
	which implies for $z_1, z_2 \in\mathcal{C}_n$ or $z_1,z_2 \notin \mathcal{C}_n$ that
	\begin{align*}
		| \E [ \hat{s}_{n,t} (z_1) ] - \E [ \hat{s}_{n,t}(z_2) ] | \leq K |z_1 - z_2|.
	\end{align*}
	In the case $z_1 = x_1 + iv_1 \in\mathcal{C}_n$ and $z_2 = x_2 + iv_2 \notin\mathcal{C}_n$, we conclude
	\begin{align*}
		| \E [ \hat{s}_{n,t} (z_1) ] - \E [ \hat{s}_{n,t}(z_2) ] |
		& = | \E [ \tilde{s}_{n,t} (z_1) ] 
		- \E [ \tilde{s}_{n,t}(x_2 + i n\inv \varepsilon_n) ] |
		\leq K | z_1 - (x_2 + i n\inv \varepsilon_n) | \\
		= & \left\{ (x_1 - x_2)^2 + (v_1 - n\inv \varepsilon_n)^2 \right\} \sq 
		\leq \left\{ (x_1 - x_2)^2 + (v_1 - v_2)^2 \right\} \sq \\
		\leq & K |z_1 - z_2|,
	\end{align*}
	since
$
		v_2 \leq n\inv \varepsilon_n \leq  v_1. 
$
	Thus, we have
	\begin{align}
		\lim\limits_{\delta \to 0} ~ \sup\limits_{n\in\N} ~  \sup\limits_{\substack{ t\in [0,1], \\ z_1,z_2 \in\mathcal{C}^+, \\ |z_1 - z_2| < \delta }} | \E [ \hat{s}_{n,t} (z_1) ] - \E [ \hat{s}_{n,t} (z_2) ] | =0. \label{f2}
	\end{align}
	Since for $A\times B \subset (\mathcal{C}^+)^2, C \times D \subset [0,1]^2$
	\begin{align*}
		& \sup\limits_{\substack{(z,z')\in A\times B, \\ (t,t')\in C \times D  }}
		| \E [ \hat{s}_{n,t}(z) ] - \E [ \hat{s}_{n,t'}(z') ] | \\
		\leq  &
		\sup\limits_{\substack{(z,z')\in A\times B, \\ t\in C  }}
		| \E [ \hat{s}_{n,t}(z) ] - \E [ \hat{s}_{n,t}(z') ] |
		+ \sup\limits_{\substack{z'\in B, \\ (t,t')\in C \times D  }}
		 | \E [ \hat{s}_{n,t}(z') ] - \E [ \hat{s}_{n,t'}(z') ] |,
	\end{align*}
 	we conclude from \eqref{g1} and \eqref{f2} 
	\begin{align} 
		& \lim\limits_{\delta\to 0} ~ \sup\limits_{n\in\N} ~
	\inf\limits_{\substack{((t_0, z_0), \ldots, (t_r,z_r)) \in \mathcal{B}_{\delta,r}^{(2)}, \\ r \in \N}} ~ 
	\max\limits_{0 \leq i \leq r } ~ 
	\sup\limits_{\substack{t, t' \in [t_{i-1}, t_i ) \\ z, z' \in [z_{i-1}, z_i\rangle}} 
	| \E [ \hat{s}_{n,t}(z) ] - \E [ \hat{s}_{n,t'}(z') ] | \nonumber \\
	= & \lim\limits_{\delta\to 0} ~ \sup\limits_{n\in\N} ~
	\inf\limits_{\substack{((t_0, z_0), \ldots, (t_r,z_r)) \in \mathcal{B}_{\delta,r}^{(2)}, \\ r \in \N}} ~ 
	\max\limits_{0 \leq i \leq r } ~ 
	\sup\limits_{\substack{t, t' \in [t_{i-1}, t_i \rangle \\ z, z' \in [z_{i-1}, z_i\rangle}} 
	| \E [ \hat{m}_{n,t}(z) ] - \E [ \hat{m}_{n,t'}(z') ] | \nonumber \\
	= &  0, \label{a67}
	\end{align}
	where
	\begin{align*}
	B_{\delta,r}^{(2)} = \{ ((t_0,z_0), \ldots, (t_r,z_r)) ~ : ~& 
	0 = t_0 < t_1 < \ldots < t_r = 1, ~ 
	0 = z_0 < z_1 < \ldots < z_r = 1, ~ \\ 
	& t_i - t_{i - 1} > \delta, ~ z_i - z_{i - 1} > \delta\textnormal{ for } i\in \{1,\ldots, r\}  \}.
	\end{align*}
	Note that in this definition, an element $z\in\mathcal{C}^+$ is identified with its representative in $[0,1]$.  \\
	One can observe that \eqref{a67} is equivalent to
	\begin{align*}
		\lim\limits_{\delta\to 0} \sup\limits_{n\in\N} \omega_{\E [ \hat{m}_{n,\cdot}(\cdot)]} ' (\delta)
		=0,
	\end{align*}
	where the modulus $\omega'$ is defined in \cite{neuhaus1971}.
	Applying Theorem 2.1 in this reference, we conclude that $\{ \E [ \hat{m}_{n, \cdot} (\cdot) ] : n \in \N \}$ has a compact closure in $D[0,1]^2.$ 
	% zu diesem Beweis habe ich noch Notizen als Fotos
	% erster Teil + Beweis von Billingsley zeigt: w'->0 bzgl. t-Koordinate und gleichmäßig in z
	% zweiter Teil : w - > 0 bzgl. z-Koordinate und gleichmäßig in t => w' ->0 bzgl. z-Koordinate und gleichmäßig in t
	% => (mit einer Dreiecksungleichung) w' -> 0 bzgl. z und t-Koordinate, also können wir Theorem 2.1 in Neuhaus anwenden.

\end{proof}

\begin{proof}[Proof of  Theorem \ref{thm_stieltjes}]
%	With this setup, we are able to show Theorem \ref{thm_stieltjes}.
%	\begin{proof}[Proof of Theorem \ref{thm_stieltjes}]
	From Lemma \ref{uni} and Lemma \ref{lem_cp_closure}, we conclude that 
	\begin{align*} 
		\lim\limits_{n\to\infty} d_2 |_{\mathcal{C}_n \times [t_0,1)} 
		( \E [\tilde{s}_{n,\cdot}(\cdot) ] , \tilde{s}_{\cdot}(\cdot) ) 
		= \lim\limits_{n\to\infty} d_2 |_{\mathcal{C}_n \times [t_0,1)} 
		( \E [\hat{m}_{n,\cdot}(\cdot) ] , \tilde{s}_{\cdot}(\cdot) )
		= 0,
	\end{align*}
	where $d_2 |_{A} $ for some set $A \subset \mathcal{C}^+ \times [0,1]$ denotes the Skorokhod metric restricted to functions on $A$. Observe that for $t=1$
	\begin{align*}
		\lim\limits_{n\to\infty} \sup\limits_{z \in\mathcal{C}_n} | \E [ \tilde{s}_{n,1}(z) ] - \tilde{s}_1(z)| =0.
	\end{align*}
	Then it is straightforward to show that 
	\begin{align} \label{skor}
		\lim\limits_{n\to\infty} d_2 |_{\mathcal{C}_n \times [t_0,1]} 
		( \E [\tilde{s}_{n,\cdot}(\cdot) ] , \tilde{s}_{\cdot}(\cdot) ) 
		= 0. 
	\end{align}
    The considerations in the proof of Lemma \ref{lem_cp_closure} reveal that $\E [ \tilde{s}_{\cdot}(\cdot) ] \in C (  \mathcal{C}^+ \times [t_0,1] ) $. 
		In this case, the convergence in the Skorokhod  space  in \eqref{skor} implies the uniform convergence
	\begin{align*}
		\lim\limits_{n\to\infty} \sup_{\substack{z\in\mathcal{C}_n, \\ t \in [t_0,1]}} | \E [ \tilde{s}_{n,t}(z) ] 
		- \tilde{s}_t(z) |
		= 0. 
	\end{align*}
	A similar convergence result with respect to the sup-norm can be shown for the Stieltjes transform $\sut_{n,t}(z)$. More precisely,  since
	\begin{align*}
		\sut_t (z) & = - \frac{1 - y_t}{z} + y_t \tilde{s}_t (z), \\
%	\end{align*}
%	$\sut$ inherits the continuity from $\tilde{s}_\cdot(\cdot)$, that is, $\sut_\cdot (\cdot) \in C (  \mathcal{C}^+ \times [t_0,1] ) $.
%	Similarly, we have
%and similarly,
%		\begin{align*}
		\sut_{n,t} (z) &= - \frac{1 - y_{\ntn}}{z} + y_{\ntn} \tilde{s}_{n,t} (z),
	\end{align*}
	we also have
	\begin{align*}
		\lim\limits_{n\to\infty} \sup_{\substack{z\in\mathcal{C}_n, \\ t \in [t_0,1]}} | \E [ \sut_{n,t}(z) ] 
		- \sut_t(z) |
		= 0. 
	\end{align*}
	\end{proof}

	\subsection{Proof of Theorem \ref{sn0_conv}} \label{prsn0_conv}
		The second assertion directly follows from the first one combined with Theorem \ref{thm_stieltjes}. Therefore, it is sufficient to show that $(M_n^2)_{n\in\N}$ is uniformly bounded. For this purpose, we use the following lemma.
		\begin{lemma} \label{a72}
			We have
			\begin{align*}
				\sup\limits_{\substack{n\in\N, \\ z\in\mathcal{C}_n, \\ t\in[t_0,1]}} 
				\Big | \frac{\textnormal{Im} ( \sut_{n,t}^0(z) )}{\textnormal{Im}(z) } \Big | 
				\leq K. 
			\end{align*}
		\end{lemma}
		\begin{proof}[Proof of Lemma \ref{a72}]
		We have for sufficiently large $n$ 
		\begin{align*}
			\textnormal{Im} ( \sut_{n,t}^0(z) ) 
			=& \int \textnormal{Im} \Big ( \frac{1}{\lambda - z} \Big ) d \tilde{\underline{F}}^{y_{\ntn},H_n} (\lambda)
			= \int \frac{- \textnormal{Im} ( \lambda - z)}{|\lambda - z|^2} 
			d \tilde{\underline{F}}^{y_{\ntn},H_n} (\lambda) \\
			= &  \int \frac{ \textnormal{Im} ( z)}{ (\lambda - \textnormal{Re} (z))^2 + \textnormal{Im}^2(z)  } 
			d \tilde{\underline{F}}^{y_{\ntn},H_n} (\lambda)
			\leq  K \textnormal{Im} (z),
		\end{align*}
		since for $z\in\mathcal{C}_l \cup \mathcal{C}_r$, $\textnormal{Re}(z) \in \{x_l, x_r\}$ is uniformly bounded away from the support of $\tilde{\underline{F}}^{y_{\ntn},H_n}$ for sufficiently large $n$ (Lemma \ref{a37}). If $z\in\mathcal{C}_u$, then $\textnormal{Im}(z) =v_0$ is constant and hence, the denominator is also uniformly bounded away from 0. 
		
		\end{proof}
	To  continue with the proof of Theorem \ref{sn0_conv}, we note that it follows 
 from \eqref{a90} in the proof of Lemma \ref{a73} in Section \ref{pra73} that 
		\begin{align*}
			& \ntn ( \E [ \sut_{n,t} (z) ] - \sut_{n,t}^0(z) )  -  \ntn \frac{\ntn }{n}R_{n,t}(z)  \E [ \tilde{\su}_{n,t}(z) ] \tilde{\su}_{n,t}^0(z)  \\
		&~~~~~~~~~~~~~~~~~~	= 
			\frac{y_{n} \frac{\ntn}{n} \ntn \lb \E [ \tilde{\su}_{n,t}(z) ] - \tilde{\su}_{n,t}^0(z) \rb \int \frac{\lambda^2 dH_n(\lambda)}{( 1 + \lambda \frac{\ntn}{n}\E [\tilde{\su}_{n,t} (z) ] ) ( 1 + \lambda \frac{\ntn}{n} \tilde{\su}_{n,t}^0(z) ) }  }
			{{\big ( - z + y_{n} \int \frac{\lambda dH_n(\lambda) }{1 + \lambda \frac{\ntn}{n} \E [ \tilde{\su}_{n,t} (z) ] } - R_{n,t}(z) \big ) \big ( - z + y_{n} \int \frac{\lambda dH_n(\lambda) }{1 + \lambda \frac{\ntn}{n}\tilde{\su}_{n,t}^0 (z) } \big )   }}, 
% + \ntn \frac{\ntn }{n}R_{n,t}(z)  \E [ \tilde{\su}_{n,t}(z) ] \tilde{\su}_{n,t}^0(z),
		\end{align*}
	which is  equivalent to 
		\begin{align*}
			& \ntn \big ( \E [ \sut_{n,t} (z) ] - \sut_{n,t}^0(z)\big )
			= 
			\frac{\ntn \frac{\ntn }{n} R_{n,t}(z) \E [ \tilde{\su}_{n,t}(z) ] \tilde{\su}_{n,t}^0(z) }
			{ 1 - \frac{y_{n} \frac{\ntn}{n} 
			\int \frac{\lambda^2 dH_n(\lambda)}{( 1 + \lambda \frac{\ntn}{n}\E [\tilde{\su}_{n,t} (z) ] ) ( 1 + \lambda \frac{\ntn}{n} \tilde{\su}_{n,t}^0(z) ) } } 
			{\big ( - z + y_{n} \int \frac{\lambda dH_n(\lambda) }{1 + \lambda \frac{\ntn}{n} \E [ \tilde{\su}_{n,t} (z) ] } - R_{n,t}(z) \big ) \big ( - z + y_{n} \int \frac{\lambda dH_n(\lambda) }{1 + \lambda \frac{\ntn}{n}\tilde{\su}_{n,t}^0 (z) } \big )   }   }.
		\end{align*}
		Note that $\sut_{n,t}^0(z)$ is uniformly bounded which follows by a similar argument as given in the proof of  Lemma \ref{a72}.
	 In order to show that the the sequence $(M_n^2)_{n\in\N}$ is uniformly bounded, by using \eqref{conv2},   it is sufficient 
		to show that the denominator is uniformly bounded away from 0 for sufficiently large $n$. For this aim, it is sufficient to prove that
		\begin{align*}
			\Big| \frac{y_{n} \frac{\ntn}{n} 
			\int \frac{\lambda^2 \tilde{\su}_{n,t}^0(z) \E [ \sut_{n,t}(z) ] dH_n(\lambda)}{( 1 + \lambda \frac{\ntn}{n}\E [\tilde{\su}_{n,t} (z) ] ) ( 1 + \lambda \frac{\ntn}{n} \tilde{\su}_{n,t}^0(z) ) } } 
			{\Big ( - z + y_{n} \int \frac{\lambda dH_n(\lambda) }{1 + \lambda \frac{\ntn}{n} \E [ \tilde{\su}_{n,t} (z) ] } - R_{n,t}(z) \Big ) \Big ( - z + y_{n} \int \frac{\lambda dH_n(\lambda) }{1 + \lambda \frac{\ntn}{n}\tilde{\su}_{n,t}^0 (z) } \Big )   } \Big| <1
		\end{align*}
		holds uniformly.
			Similarly to the proof of Lemma \ref{a80}, we conclude that for any bounded subset $S \subset \mathbb{C}^+$
		\begin{align*}
			\inf\limits_{\substack{n\in\N, \\ z \in S, \\ t\in [t_0,1]}} 
			| \sut_{n,t}^0(z) | > 0.
		\end{align*}
		Using this, Hölder's inequality, Lemma \ref{a72} and the identities
		%Lemma \ref{a20}, 
		\begin{align*}
	\frac{y_{n} \frac{\ntn}{n} \int \frac{ \lambda^2 dH_n(\lambda)}{|1+\lambda\frac{\ntn}{n}\tilde{\su}_{n,t}^0(z)|^2}}
		{  \Big | - z + y_{n} \int \frac{\lambda dH_n(\lambda)}{1+\lambda \frac{\ntn}{n}\tilde{\su}_{n,t}^0(z)}\Big |^2}
		= &  \frac{ \frac{\ntn}{n}\im ( \tilde{\su}_{n,t}^0(z)) y_{n} \int \frac{\lambda^2 dH_n(\lambda)}{|1+\lambda \frac{\ntn}{n}\tilde{\su}_{n,t}^0(z)|^2}}
		{\im (z) + \frac{\ntn}{n} \im ( \tilde{\su}_{n,t}^0(z)) y_{n} \int \frac{\lambda^2 dH_n(\lambda)}{|1+\lambda \frac{\ntn}{n} \tilde{\su}_{n,t}^0(z)|^2}}, \\
	\frac{y_{n} \frac{\ntn}{n} \int \frac{ \lambda^2 dH_n(\lambda)}{|1+\lambda\frac{\ntn}{n} \E[ \tilde{\su}_{n,t}(z)]  |^2}}
		{  \left| - z + y_{n} \int \frac{\lambda dH_n(\lambda)}{1+\lambda \frac{\ntn}{n}\E[ \tilde{\su}_{n,t}(z)]} + R_{n,t} \right|^2}
		= & \frac{ \frac{\ntn}{n}\im ( \E[ \tilde{\su}_{n,t}(z)] ) y_{n} \int \frac{\lambda^2 dH_n(\lambda)}{|1+\lambda \frac{\ntn}{n}\E[ \tilde{\su}_{n,t}(z)]|^2}}
		{\im (z) + \frac{\ntn}{n} \im ( \E[ \tilde{\su}_{n,t}(z)]) y_{n} \int \frac{\lambda^2 dH_n(\lambda)}{|1+\lambda \frac{\ntn}{n} \E[ \tilde{\su}_{n,t}(z)]|^2} + \im(R_{n,t} ) },
	\end{align*}
		we obtain for sufficiently large $n$
		\begin{align*}
			& \Bigg| y_n \frac{\int \frac{\lambda^2  \frac{\ntn}{n}  dH_n(\lambda)}{( 1 + \lambda \frac{\ntn}{n}\E [\tilde{\su}_{n,t} (z) ] ) ( 1 + \lambda \frac{\ntn}{n} \tilde{\su}_{n,t}^0(z) ) } }
			{\Big ( - z + y_{n} \int \frac{\lambda dH_n(\lambda) }{1 + \lambda \frac{\ntn}{n} \E [ \tilde{\su}_{n,t} (z) ] } - R_{n,t}(z) \Big ) \Big ( - z + y_{n} \int \frac{\lambda dH_n(\lambda) }{1 + \lambda \frac{\ntn}{n}\tilde{\su}_{n,t}^0 (z) } \Big )   }			
			\Bigg|^2 \\
			\leq & 
			 \frac{ y_n \frac{\ntn}{n}\int \frac{\lambda^2  dH_n(\lambda) }{\left|1 + \lambda \frac{\ntn}{n} \tilde{\su}_{n,t}^0(z) \right |^2} }
			{\Big | - z + y_{n} \int \frac{\lambda dH_n(\lambda) }{1 + \lambda \frac{\ntn}{n}\tilde{\su}_{n,t}^0 (z) } \Big |^2 }  
			   \frac{y_n\frac{\ntn}{n}\int \frac{  \lambda^2 dH_n(\lambda) }{\left|1 + \lambda \frac{\ntn}{n} \E [\tilde{\su}_{n,t}(z) ] )\right |^2}}
			{\Big | - z + y_{n} \int \frac{\lambda dH_n(\lambda) }{1 + \lambda \frac{\ntn}{n} \E [ \tilde{\su}_{n,t} (z) ] } - R_{n,t}(z) \Big |^2}  \\
		=&   \frac{ \frac{\ntn}{n}\im ( \tilde{\su}_{n,t}^0(z)) y_{n} \int \frac{\lambda^2 dH_n(\lambda)}{|1+\lambda \frac{\ntn}{n}\tilde{\su}_{n,t}^0(z)|^2}}
		{\im (z) + \frac{\ntn}{n} \im ( \tilde{\su}_{n,t}^0(z)) y_{n} \int \frac{\lambda^2 dH_n(\lambda)}{|1+\lambda \frac{\ntn}{n} \tilde{\su}_{n,t}^0(z)|^2}}  \\
		& \times   \frac{ \frac{\ntn}{n}\im ( \E [ \tilde{\su}_{n,t}(z)]) y_{n} \int \frac{\lambda^2 dH_n(\lambda)}{|1+\lambda \frac{\ntn}{n}\E [ \tilde{\su}_{n,t}(z)]|^2}}
		{\im (z) + \frac{\ntn}{n} \im ( \E [ \tilde{\su}_{n,t}(z)] ) y_{n} \int \frac{\lambda^2 dH_n(\lambda)}{|1+\lambda \frac{\ntn}{n} \E [ \tilde{\su}_{n,t}(z)]|^2} + \im (R_{n,t}(z))} \\
		\leq &  1 - \frac{\im(z)} {\im (z) + \frac{\ntn}{n} \im ( \tilde{\su}_{n,t}^0(z)) y_{n} \int \frac{\lambda^2 dH_n(\lambda)}{|1+\lambda \frac{\ntn}{n} \tilde{\su}_{n,t}^0(z)|^2}}  \\
		\leq &  1 - \frac{\im(z)}{\im(z) + \frac{\ntn}{n} K \im(z) y_{n} \int \frac{\lambda^2 dH_n(\lambda)}{|1+\lambda \frac{\ntn}{n} \tilde{\su}_{n,t}^0(z)|^2}} 
		\leq   1 - \frac{1}{1 + K }  <1 , 
		\end{align*}
 	where we used the fact that $\im(R_{n,t}(z)) + \im (z) \geq  0$ for sufficiently large $n$, which follows from Lemma \ref{aux_im_z}. 
		This finishes the proof of Theorem \ref{sn0_conv}.

\subsection{Proof of the statement  \eqref{conv2}}   \label{prconv2}

	As a preparation, we need the following proposition. The proof is omitted for the sake of brevity. 
	\begin{proposition}\label{a61}
	\begin{align*} 
		\sup\limits_{n\in\N, z\in\mathcal{C}_n, t\in [t_0,1]} \Big \| ( \frac{\ntn}{n} \E [ \tilde{\su}_{n,t} (z) ] \T_n + \mathbf{I} )\inv \Big \| \leq K.
	\end{align*}
	\end{proposition}
	
	Using \eqref{sher_mor} and the representation  		
		\eqref{beta_quer} 
%	\begin{align} \label{beta_quer}
%		\beta_{j,t}(z) = \overline{\beta}_{j,t}(z) - \overline{\beta}_{j,t}^2(z) \hat{\gamma}_{j,t}(z) 
%	+ \overline{\beta}_{j,t}^2(z) \beta_{j,t}(z) \hat{\gamma}_{j,t}^2(z),
% 	\end{align} 
	we obtain
	\begin{align*}
		& \ntn R_{n,t}(z)  \E [ \tilde{\su}_{n,t}(z) ]  = 
		y_{\ntn} \sum\limits_{j=1}^{\ntn} \E [ \beta_{j,t}(z) d_{j,t}(z) ] \nonumber \\
		= & - y_{\ntn} \sum\limits_{j=1}^{\ntn} \E \Big[  \beta_{j,t}(z) \Big\{ \mathbf{q}_{j}^\star \T_n\sq \D_{j,t}\inv(z) ( \frac{\ntn}{n}\E [ \tilde{\su}_{n,t}(z) ] \T_n + \mathbf{I} )\inv \T_n\sq \mathbf{q}_j \nonumber \\
	&  - \frac{1}{p} \E \Big[ \tr ( \frac{\ntn}{n} \E [  \tilde{\su}_{n,t} (z) ] \T_n + \mathbf{I} )\inv \T_n \D_{t}\inv(z) \Big]  \Big\} \Big] \nonumber \\
	 = & - y_{\ntn} \sum\limits_{j=1}^{\ntn} \E \Big[  \beta_{j,t}(z) \Big\{ \mathbf{q}_{j}^\star \T_n\sq \D_{j,t}\inv(z) 
	 ( \frac{\ntn}{n}\E [ \tilde{\su}_{n,t}(z) ] \T_n + \mathbf{I} )\inv \T_n\sq \mathbf{q}_j \nonumber \\
	 & - \frac{1}{p} \E \Big [ \tr (\frac{\ntn}{n} \E [ \tilde{\su}_{n,t} (z) ] \T_n + \mathbf{I} )\inv \T_n \D_{j,t}\inv(z) \Big]  \Big\} \Big] \nonumber \\
	 & +  \frac{1}{\ntn} \sum\limits_{j=1}^{\ntn}  \E \Big[ \beta_{j,t}(z) \tr (\frac{\ntn}{n} \E [ \tilde{\su}_{n,t} (z) ] \T_n + \mathbf{I} )\inv \T_n \E \left[ \D_{t}\inv(z) - \D_{j,t}\inv(z)  \right] \Big] \nonumber \\
	 = & T_{n,1}(z,t) 
	 +  T_{n,2}(z,t) + o(1) 
	\end{align*}
		%Summand mit overline beta_j verschwindet, da r_j unabhängig von overline beta_j
	% Summand mit hat gamma_j,t^2 verschwindet nach Anwendung von Hölder-Ungleichung
  uniformly with respect to $z\in\mathcal{C}_n, t\in [t_0,1]$,	where the terms $T_{n,1}$ and $T_{n,2} $ are defined by  
	\begin{align}
		T_{n,1} (z,t) & =  y_{\ntn} \sum\limits_{j=1}^{\ntn} \E \Big[  \overline{\beta}_{j,t}^2(z) \Big\{ \mathbf{q}_{j}^\star \T_n\sq \D_{j,t}\inv(z) ( \frac{\ntn}{n} \E [ \tilde{\su}_{n,t}(z) ] \T_n + \mathbf{I} )\inv \T_n\sq \mathbf{q}_j \nonumber
	\\  & - \frac{1}{p} \E \Big[ \tr (\frac{\ntn}{n} \E [ \tilde{\su}_{n,t} (z) ] \T_n + \mathbf{I} )\inv \T_n \D_{j,t}\inv(z) \Big]  \Big\}  \hat{\gamma}_{j,t}(z)\Big] , \label{I}\\
		T_{n,2} (z,t) &= - \frac{1}{\ntn} \sum\limits_{j=1}^{\ntn} \E \left[ \beta_{j,t}(z) \right] 
	 \E\Big [ \beta_{j,t}(z) \rd_j^\star \D_{j,t}\inv(z) 
	 \Big (  \frac{\ntn}{n} \E \sut_{n,t}(z) \T_n + \mathbf{I} \Big) \inv \T_n \D_{j,t}\inv(z) \rd_j \Big ] \label{II},
	\end{align}
    For this argument we used the fact
	\begin{eqnarray*}
		&& \E \Big[  \overline{\beta}_{j,t}(z) \Big\{ \mathbf{q}_{j}^\star \T_n\sq \D_{j,t}\inv(z) ( \frac{\ntn}{n} \E [ \tilde{\su}_{n,t}(z) ] \T_n + \mathbf{I} )\inv \T_n\sq \mathbf{q}_j  \\
	&& ~~~~~~~~~~~~~~~~~~~~~~~ - \frac{1}{p} \E \Big[ \tr (\frac{\ntn}{n} \E [ \tilde{\su}_{n,t} (z) ] \T_n + \mathbf{I} )\inv \T_n \D_{j,t}\inv(z) \Big]  \Big\}  \Big]  
  = 0
	\end{eqnarray*}
	and  that by the estimate (9.10.2) in \cite{bai2004}
	\begin{align*}
		& \E \Big|  \overline{\beta}_{j,t}^2(z) \beta_{j,t}(z) \Big\{ \mathbf{q}_{j}^\star \T_n\sq \D_{j,t}\inv(z) ( \frac{\ntn}{n} \E [ \tilde{\su}_{n,t}(z) ] \T_n + \mathbf{I} )\inv \T_n\sq \mathbf{q}_j \\
	& - \frac{1}{p} \E \Big[ \tr (\frac{\ntn}{n} \E [ \tilde{\su}_{n,t} (z) ] \T_n + \mathbf{I} )\inv \T_n \D_{j,t}\inv(z) \Big]  \Big\} \hat{\gamma}_{j,t}(z)^2 \Big| \\
	 \leq & 
	 \E \sq \Big|  \overline{\beta}_{j,t}^2(z) \beta_{j,t}(z) \Big\{ \mathbf{q}_{j}^\star \T_n\sq \D_{j,t}\inv(z) ( \frac{\ntn}{n} \E [ \tilde{\su}_{n,t}(z) ] \T_n + \mathbf{I} )\inv \T_n\sq \mathbf{q}_j \\
	& - \frac{1}{p} \E \Big[ \tr (\frac{\ntn}{n} \E [ \tilde{\su}_{n,t} (z) ] \T_n + \mathbf{I} )\inv \T_n \D_{j,t}\inv(z) \Big]  \Big|^2 \E\sq |\hat{\gamma}_{j,t}(z) |	^4 \Big| \\
	\leq & K n\inv \eta_n^2 = o\lb n\inv \rb. 
	\end{align*}
	
	For the term in \eqref{I} we obtain the representation 
	\begin{align*}
	T_{n,1}(z, t)
	=& y_{\ntn} \sum\limits_{j=1}^{\ntn} \E \Big[  \overline{\beta}_{j,t}^2(z) \Big\{ \mathbf{q}_{j}^\star \T_n\sq \D_{j,t}\inv(z) ( \frac{\ntn}{n} \E [ \tilde{\su}_{n,t}(z) ] \T_n + \mathbf{I} )\inv \T_n\sq \mathbf{q}_j \nonumber
	\\  & - \frac{1}{p} \tr (\frac{\ntn}{n} \E [ \tilde{\su}_{n,t} (z) ] \T_n + \mathbf{I} )\inv \T_n \D_{j,t}\inv(z)   \Big\}  \hat{\gamma}_{j,t}(z)\Big] \\
	& -  y_{\ntn} \sum\limits_{j=1}^{\ntn} \E \Big[  \overline{\beta}_{j,t}^2(z)  \frac{1}{p} \tr (\frac{\ntn}{n} \E [ \tilde{\su}_{n,t} (z) ] \T_n + \mathbf{I} )\inv \T_n \E [ \D_{j,t}\inv(z) ]   \hat{\gamma}_{j,t}(z)\Big] \\
	&  +  y_{\ntn} \sum\limits_{j=1}^{\ntn} \E \Big[  \overline{\beta}_{j,t}^2(z)  \frac{1}{p} \tr (\frac{\ntn}{n} \E [ \tilde{\su}_{n,t} (z) ] \T_n + \mathbf{I} )\inv \T_n  \D_{j,t}\inv(z)  \hat{\gamma}_{j,t}(z)\Big] \\
	= &  y_{\ntn} \sum\limits_{j=1}^{\ntn} \E \Big[  \overline{\beta}_{j,t}^2(z) \Big\{ \mathbf{q}_{j}^\star \T_n\sq \D_{j,t}\inv(z) ( \frac{\ntn}{n} \E [ \tilde{\su}_{n,t}(z) ] \T_n + \mathbf{I} )\inv \T_n\sq \mathbf{q}_j \nonumber
	\\  & - \frac{1}{p} \tr (\frac{\ntn}{n} \E [ \tilde{\su}_{n,t} (z) ] \T_n + \mathbf{I} )\inv \T_n \D_{j,t}\inv(z)   \Big\}  \hat{\gamma}_{j,t}(z) \Big] \\
		= &  y_{\ntn} \sum\limits_{j=1}^{\ntn} z^2 \sut_{t}^2(z) \E \Big[   \Big\{ \mathbf{q}_{j}^\star \T_n\sq \D_{j,t}\inv(z) ( \frac{\ntn}{n} \E [ \tilde{\su}_{n,t}(z) ] \T_n + \mathbf{I} )\inv \T_n\sq \mathbf{q}_j \nonumber
	\\  
	& - \frac{1}{p} \tr (\frac{\ntn}{n} \E [ \tilde{\su}_{n,t} (z) ] \T_n + \mathbf{I} )\inv \T_n \D_{j,t}\inv(z) \Big\}  \hat{\gamma}_{j,t}(z)\Big] + o(1),
	\end{align*}
    where in the last step we used the inequality (9.10.2) in \cite{bai2004}, 
	%and Lemma \ref{su} (similarly to \eqref{b_s}), 
to replace  all of  the terms $\beta_{j,t}(z) , \overline{\beta}_{j,t}(z) , b_{j,t}(z)$ and similarly defined quantities  by $-z \sut_{t} (z)$. 
	 This argument also implies for  the term $	T_{n,2}$ defined in \eqref{II} 
	\begin{align*} %\label{u2}
		T_{n,2}(z,t) = - \frac{z^2 \sut_t^2(z)}{\ntn n} \sum\limits_{j=1}^{\ntn} \E \Big[ \tr  \D_{j,t}\inv(z) 
	 \Big ( \frac{\ntn}{n} \E \sut_{n,t}(z) \T_n + \mathbf{I} \Big)\inv \T_n \D_{j,t}\inv(z) \T_n \Big ] + o(1). 
	\end{align*}
	We now consider the the complex case, where we have from equation  (9.8.6) in \cite{bai2004}  
	%(similarly to the observations following \eqref{a5})
	\begin{align*} 
		T_{n,1}(z,t)
	%	= & y_{\ntn} \sum\limits_{j=1}^{\ntn} z^2 \sut_{t}^2(z) \E \Bigg[   \Bigg\{ \mathbf{q}_{j}^\star \T_n\sq \D_{j,t}\inv(z) ( \frac{\ntn}{n} \E [ \tilde{\su}_{n,t}(z) ] \T_n + \mathbf{I} )\inv \T_n\sq \mathbf{q}_j \nonumber
%	\\  
%	& - \frac{1}{p} \tr (\frac{\ntn}{n} \E [ \tilde{\su}_{n,t} (z) ] \T_n + \mathbf{I} )\inv \T_n \D_{j,t}\inv(z) \Bigg\}  \hat{\gamma}_{j,t}(z)\Bigg] + o(1) % \label{u3}
%	\\
	& =  \frac{z^2 \sut_t^2(z)}{\ntn n} \sum\limits_{j=1}^{\ntn} \E \Big[ \tr  \D_{j,t}\inv(z) \Big ( \frac{\ntn}{n} \E [ \sut_{n,t}(z) ] \T_n + \mathbf{I} \Big)\inv \T_n \D_{j,t}\inv(z) \T_n  \Big] + o(1).\nonumber 
	\end{align*}
	which yields 
$
	T_{n,1}(z,t) + T_{n,2}(z,t) = o(1) ,
$
and as a consequence  \eqref{conv2} in  this case.  \\

Next, we consider the real case  using  again equation (9.8.6) in \cite{bai2004}, which gives 
	%(again similarly to the observations following \eqref{a5}),
	\begin{align}
	\nonumber 
	 \ntn R_{n,t}(z)  \E [ \tilde{\su}_{n,t}(z) ]  & = 
%		y_{\ntn} \sum\limits_{j=1}^{\ntn} \E [ \beta_{j,t}(z) d_{j,t}(z) ] \nonumber \\ 
% &		=  
		T_{n,1}(z,t) + T_{n,2}(z,t) + o(1)  \\
	& 	= \frac{z^2 \sut_t^2(z)}{\ntn n} \sum\limits_{j=1}^{\ntn} 
		\E \Big[ \tr  \D_{j,t}\inv(z) 
	 \Big( \frac{\ntn}{n} \E [ \sut_{n,t}(z)] \T_n + \mathbf{I} \Big)\inv \T_n \D_{j,t}\inv(z) \T_n \Big] + o(1) \nonumber \\
&	 =  \frac{z^2 \sut_t^2(z)}{\ntn n} \sum\limits_{j=1}^{\ntn} 
	  \E \big[ \tr  \D_{j,t}\inv(z) 
	 \big( t  \sut_{t}(z) \T_n + \mathbf{I} \big)\inv \T_n \D_{j,t}\inv(z) \T_n \big] + o(1).
	 \label{a71}
	\end{align}

For a detailed analysis of the  random variable in \eqref{a71}  we use the  decomposition  
	\begin{align} 
		\D_{j,t}^{-1}(z) 
		= & - \Big (  z\mathbf{I} - \frac{\ntn-1}{n} b_{j,t}(z) \T_n \Big) \inv 
		+ b_{j,t}(z) \mathbf{A}_t(z) + \mathbf{B}_t(z) + \mathbf{C}_t(z),\label{d_inv}
	\end{align}
	where  
	\begin{align}
	 	\mathbf{A}_t(z) &=  \sum\limits_{i \neq j, 1 \leq i \leq \ntn } \Big (  z\mathbf{I} - \frac{\ntn-1}{n} b_{j,t}(z) \T_n \Big) \inv \lb \rd_{i} \rd_{i}^\star - n\inv \T_n \rb \D_{i,j,t}\inv(z) , \label{def_A}\\
	 	\mathbf{B}_t(z) &= \sum\limits_{i \neq j, 1 \leq i \leq \ntn }  ( \beta_{i,j,t} (z) - b_{j,t}(z) ) \Big (   z\mathbf{I} - \frac{\ntn-1}{n} b_{j,t}(z) \T_n \Big ) \inv 
	 	\rd_{i} \rd_{i}^\star \D_{i,j,t}\inv (z), %\label{def_B} 
	 	\nonumber \\
	 	\mathbf{C}_t(z) &= b_{j,t}(z) \Big ( z\mathbf{I} - \frac{\ntn-1}{n} b_{j,t}(z) \T_n  \Big) \inv \T_n n \inv \sum\limits_{i \neq j} 
	 	\lb \D\inv_{i,j,t} (z) - \D\inv_{j,t} (z) \rb. \nonumber %\label{def_C} 
	\end{align}
	(here, we do not reflect the dependence on  index $j$ in our notation). We now investigate these terms in more detail.
	\\
	\smallskip
% 	\HD{{\bf Brauchen wir das hier ?} nein
% For the term $\mathbf{A}_t(z)$  note  that $\rd_{i}$ is independent of $\rd_{1},...,\rd_{j}$ and of $\D_{i,j,t}(z)$ for $i>j$, which implies  
% 	\begin{align*}
% 	 \E_j \Big [  \Big (  z\mathbf{I} - \frac{\ntn-1}{n} b_{j,t}(z) \T_n \Big ) \inv  \lb \rd_{i} \rd_{i}^\star - n\inv \T_n \rb \D_{i,j,t}\inv(z) \Big ] 
% 	 % \\ 	& = \Big (  z\mathbf{I} - \frac{\ntn-1}{n} b_{j,t}(z) \T_n \Big ) \inv  \lb \E \left[ \rd_{i} \rd_{i}^\star \right] - n\inv \T_n \rb \E_j \left[ \D_{i,j,t}\inv(z) \right] 
% 	 = 0~, 
% 	\end{align*}
% and consequently we only have to consider therm with index  $i<j$ 	 in the definition of $\mathbf{A}_t$. }\\
	Let $\mathbf{M}$ be a $p\times p$ (random) matrix and let $||\mathbf{M}||$ denote a nonrandom bound on the spectral norm of $\mathbf{M}$ for all parameters governing $\mathbf{M}$ and all realizations of $\mathbf{M}$. 
	Then, one can show the following bounds (similarly to the inequalities  (9.9.14) and (9.9.15) in \cite{bai2004})
	\begin{align} 
		\E | \tr ( \mathbf{B}_t(z) \mathbf{M} ) | 
		& \leq K ||\mathbf{M}|| n^{\frac{1}{2}}, \label{a7} \\
		\left| \tr ( \mathbf{C}_t(z) \mathbf{M} ) \right| & \leq K || \mathbf{M} ||. \label{a8}
	\end{align}
	Moreover, we have for any  nonrandom $\mathbf{M}$ 
	%, \eqref{mom}, and \eqref{a6}
	\begin{align*}
		\E \left| \tr \mathbf{A}_t (z) \mathbf{M} \right| 
		 & \leq K ||\mathbf{M}||  , %\label{A}
	\end{align*}
	which follows using formula  (9.9.6) in \cite{bai2004}.

	Using the decomposition given in \eqref{d_inv},  the estimates \eqref{a7} and \eqref{a8}  (which shows that all  terms involving $\mathbf{B}_t(z) $ and $\mathbf{C}_t(z) $ are negligible)   and the fact 
	\begin{align*}
			\E \big[ \tr  \big( z \mathbf{I} - \frac{\ntn - 1}{n} b_{j,t}(z) \T_n \big)\inv 
	 \lb t  \sut_{t}(z) \T_n + \mathbf{I} \rb\inv \T_n \mathbf{A}_t(z)  \T_n \big]
	 =0,
	\end{align*}		
	we obtain
		\begin{align}
		& y_{\ntn} \sum\limits_{j=1}^{\ntn} \E [ \beta_{j,t}(z) d_{j,t}(z) ] \nonumber \\
		= & \frac{z^2 \sut_t^2(z)}{\ntn n} \sum\limits_{j=1}^{\ntn} 
	  \E \Big [ \tr  \Big  (  z \mathbf{I} - \frac{\ntn - 1}{n} b_{j,t}(z) \T_n \Big  )\inv 
	 \lb t  \sut_{t}(z) \T_n + \mathbf{I} \rb\inv \T_n \nonumber \\ 
	 & \times \Big  (  z \mathbf{I} - \frac{\ntn - 1}{n} b_{j,t}(z) \T_n \Big ) \inv  \T_n \Big  ] \nonumber \\
	 & + \frac{z^2 \sut_t^2(z)}{\ntn n} \sum\limits_{j=1}^{\ntn} 
	  b_{j,t}^2(z)\E \left[ \tr  \mathbf{A}_t(z)
	 \lb t  \sut_{t}(z) \T_n + \mathbf{I} \rb\inv \T_n \mathbf{A}_t(z)  \T_n \right]
	 + o(1) \nonumber \\
	 = & \frac{ \sut_t^2(z)}{ n} 
	  \E \left[ \tr  
	 \lb t  \sut_{t}(z) \T_n + \mathbf{I} \rb^{-3} \T_n^2 \right] \nonumber \\
	  & + \frac{z^4 \sut_t^4(z)}{\ntn n} \sum\limits_{j=1}^{\ntn} 
	  \E \left[ \tr  \mathbf{A}_t(z)
	 \lb t  \sut_{t}(z) \T_n + \mathbf{I} \rb\inv \T_n \mathbf{A}_t(z)  \T_n \right]
	 + o(1). \label{a68}
	\end{align}
For the term    $\mathbf{A}_t(z)$  in \eqref{def_A} (which actually depends on $j$)   we have 
	\begin{align*}
		 \mathbf{A}_t(z) 
		 &=  \sum\limits_{i \neq j, 1 \leq i \leq \ntn } \Big( z\mathbf{I} - \frac{\ntn-1}{n} b_{j,t}(z) \T_n \Big)\inv \lb \rd_{i} \rd_{i}^\star - n\inv \T_n \rb \D_{i,j,t}\inv(z) \\
		 = & \sum\limits_{i \neq j, 1 \leq i \leq \ntn }\D_{i,j,t}\inv(z)
		 \lb \rd_{i} \rd_{i}^\star - n\inv \T_n \rb 
		  \Big( z\mathbf{I} - \frac{\ntn-1}{n} b_{j,t}(z) \T_n \Big)\inv,
	\end{align*}
	which follows from $\mathbf{A}_t(z) = \lb \mathbf{A}_t(\overline{z}) \rb^\star$.
	Substituting the first and second  expression for the term $\mathbf{A}_t(z)$ on the left  and  on the right in   \eqref{a68}, respectively,  yields
	\begin{align}
		& \frac{z^4 \sut_t^4(z)}{\ntn n} \sum\limits_{j=1}^{\ntn} 
	  \E \left[ \tr  \mathbf{A}_t(z)
	 \lb t \sut_{t}(z) \T_n + \mathbf{I} \rb\inv \T_n \mathbf{A}_t(z)  \T_n \right] \nonumber \\
	 = & \frac{z^4 \sut_t^4(z)}{\ntn n} \sum\limits_{j=1}^{\ntn} \sum\limits_{i,l \neq j}
	  \E \Big[ \tr  \Big( z\mathbf{I} - \frac{\ntn-1}{n} b_{j,t}(z) \T_n \Big)\inv \lb \rd_{i} \rd_{i}^\star - n\inv \T_n \rb \D_{i,j,t}\inv(z)
	 \lb t \sut_{t}(z) \T_n + \mathbf{I} \rb\inv \T_n \nonumber \\
	 & \times \D_{l,j,t}\inv(z)
		 \lb \rd_{l} \rd_{l}^\star - n\inv \T_n \rb 
		  \Big( z\mathbf{I} - \frac{\ntn-1}{n} b_{l,t}(z) \T_n \Big)\inv   \T_n \Big] \nonumber \\
		  = & \frac{z^2 \sut_t^4(z)}{\ntn n}
		  \sum\limits_{j=1}^{\ntn}
		  \sum\limits_{i,l \neq j} 
		  A_{i,l,j} (z,t) + o(1), \label{a69}
	\end{align}
	where 
	\begin{align*}
	    A_{i,l,j}(z,t) & =	  \E \Big[ \tr \lb t \sut_t(z) \T_n + \mathbf{I} \rb^{-2} \T_n 
		  \lb \rd_i \rd_i^\star - n\inv \T_n\rb 
		  \D_{i,j,t}\inv(z)
	 \lb t \sut_{t}(z) \T_n + \mathbf{I} \rb\inv \T_n \nonumber \\
	 & \times \D_{l,j,t}\inv(z)
		 \lb \rd_{l} \rd_{l}^\star - n\inv \T_n \rb		
		  \Big].
	\end{align*}
	In the following, we will show that the sum of the cross terms $A_{i,l,j}(z,t) $ (i.e. $l\neq i$) in \eqref{a69} vanishes asymptotically. For this purpose we  use the formulas for $l\neq i$
	\begin{align*}
		\D_{i,j,t}\inv(z)  = & \D_{l,i,j,t}\inv(z) - \beta_{l,i,j,t}(z) \D_{l,i,j,t}\inv(z) \rd_l \rd_l^\star \D_{l,i,j,t}\inv(z), % \\
%		\D_{l,j,t}\inv(z)  = & \D_{i,l,j,t}\inv(z) - \beta_{i,l,j,t}(z) \D_{i,l,j,t}\inv(z) \rd_i \rd_i^\star \D_{i,l,j,t}\inv(z),
	\end{align*}
 where
	\begin{align*}
		\beta_{l,i,j,t}(z) = \frac{1}{1 + \rd_l^\star \D_{l,i,j,t}\inv(z) \rd_l }.
	\end{align*}
	Note that that the expectation appearing in the cross term $A_{i,l,j}(z,t)$ will be $0$ if $\D_{i,j,t}\inv(z)$ or $\D_{l,j,t}\inv(z)$ are replaced by $\D_{l,i,j,t}\inv(z).$ 
	Hence, it remains to bound for $i\neq l$  (use also \eqref{beta_quer} ) 
	\begin{align*}
	& | A_{i,l,j}(z,t) | \\ = & \Big| \E \Big[ \tr \lb t \sut_t(z) \T_n + \mathbf{I} \rb^{-2} \T_n 
		  \lb \rd_i \rd_i^\star - n\inv \T_n\rb 
		  \lb \D_{l,i,j,t}\inv(z) - \D_{i,j,t}\inv(z) \rb
	 \lb t \sut_{t}(z) \T_n + \mathbf{I} \rb\inv \T_n \nonumber \\
	 & \times \lb \D_{i,l,j,t}\inv(z) - \D_{l,j,t}\inv(z) \rb 
		 \lb \rd_{l} \rd_{l}^\star - n\inv \T_n \rb		
		  \Big] \Big| \\
		= & \Big| \E \Big[ \tr \lb t \sut_t(z) \T_n + \mathbf{I} \rb^{-2} \T_n 
		  \lb \rd_i \rd_i^\star - n\inv \T_n\rb 
		  \beta_{l,i,j,t}(z) \D_{l,i,j,t}\inv(z) \rd_l \rd_l^\star \D_{l,i,j,t}\inv(z) 
	 \lb t \sut_{t}(z) \T_n + \mathbf{I} \rb\inv \T_n \nonumber \\
	 & \times \beta_{i,l,j,t}(z) \D_{i,l,j,t}\inv(z) \rd_i \rd_i^\star \D_{i,l,j,t}\inv(z)
		 \lb \rd_{l} \rd_{l}^\star - n\inv \T_n \rb		
		  \Big] \Big| \\
		  		  = & o \lb n\inv \rb,
	\end{align*}
	which is shown in Lemma \ref{lem_cross_terms} and corrects a wrong statement on p. 260 in
	the monograph of  \cite{bai2004}. 
	
Summarizing, we have shown that 
	\begin{align}
	&  y_{\ntn} \sum\limits_{j=1}^{\ntn} \E [ \beta_{j,t}(z) d_{j,t}(z) ] \nonumber \\
	 = & 
	 \frac{ \sut_t^2(z)}{ n} 
	  \E \left[ \tr  
	 \lb t  \sut_{t}(z) \T_n + \mathbf{I} \rb^{-3} \T_n^2 \right] \nonumber \\
	 & + \frac{z^2 \sut_t^4(z)}{\ntn n}
		  \sum\limits_{j=1}^{\ntn}
		  \sum\limits_{i \neq j} 
		  \E \Big[ \tr \lb t \sut_t(z) \T_n + \mathbf{I} \rb^{-2} \T_n 
		  \lb \rd_i \rd_i^\star - n\inv \T_n\rb 
		  \D_{i,j,t}\inv(z)
	 \lb t \sut_{t}(z) \T_n + \mathbf{I} \rb\inv \T_n \nonumber \\
	 & \times \D_{i,j,t}\inv(z)
		 \lb \rd_{i} \rd_{i}^\star - n\inv \T_n \rb		
		  \Big] + o(1) \nonumber \\
		  = & 
	 \frac{ \sut_t^2(z)}{ n} 
	  \E \left[ \tr  
	 \lb t  \sut_{t}(z) \T_n + \mathbf{I} \rb^{-3} \T_n^2 \right] \nonumber \\
	 & + \frac{z^2 \sut_t^4(z)}{\ntn n}
		  \sum\limits_{j=1}^{\ntn}
		  \sum\limits_{i \neq j} 
		  \E \Big[ \tr \lb t \sut_t(z) \T_n + \mathbf{I} \rb^{-2} \T_n 
		   \rd_i \rd_i^\star 
		  \D_{i,j,t}\inv(z)
	 \lb t \sut_{t}(z) \T_n + \mathbf{I} \rb\inv \T_n  \D_{i,j,t}\inv(z)
		  \rd_{i} \rd_{i}^\star 
		  \Big] \nonumber
		  \\ & + o(1) \nonumber\\
		  = & 
		   \frac{ \sut_t^2(z)}{ n} 
	  \E \left[ \tr  
	 \lb t  \sut_{t}(z) \T_n + \mathbf{I} \rb^{-3} \T_n^2 \right] \nonumber \\
	 & + \frac{z^2 \sut_t^4(z)}{\ntn n^3}
		  \sum\limits_{j=1}^{\ntn}
		  \sum\limits_{i \neq j} 
		  \E \Big[ \tr \left\{ \lb t \sut_t(z) \T_n + \mathbf{I} \rb^{-2} \T_n^2 \right\}   
		  \tr \left\{ 
		  \D_{i,j,t}\inv(z)
	 \lb t \sut_{t}(z) \T_n + \mathbf{I} \rb\inv \T_n  \D_{i,j,t}\inv(z)
		  \T_n \right\}
		  \Big] \nonumber
		  \\ & + o(1) \nonumber \\
		  = & 
		   \frac{ \sut_t^2(z)}{ n} 
	  \E \left[ \tr  
	 \lb t  \sut_{t}(z) \T_n + \mathbf{I} \rb^{-3} \T_n^2 \right] \nonumber \\
	 & + \frac{z^2 \sut_t^4(z)}{ n^3}
		  \sum\limits_{j=1}^{\ntn}
		  \E \Big[ \tr \left\{ \lb t \sut_t(z) \T_n + \mathbf{I} \rb^{-2} \T_n^2 \right\}   
		  \tr \left\{ 
		  \D_{j,t}\inv(z)
	 \lb t \sut_{t}(z) \T_n + \mathbf{I} \rb\inv \T_n  \D_{j,t}\inv(z)
		  \T_n \right\}
		  \Big]
		  + o(1). \label{a70}
	\end{align}
Here  we used for the last equality  the fact 
	\begin{align*}
		& \Big|
		 \E \Big[ 
		  \tr \left\{ 
		  \D_{j,t}\inv(z)
	 \lb t \sut_{t}(z) \T_n + \mathbf{I} \rb\inv \T_n  \D_{j,t}\inv(z)
		  \T_n \right\}
		  -  
		  \tr \left\{ 
		  \D_{i,j,t}\inv(z)
	 \lb t \sut_{t}(z) \T_n + \mathbf{I} \rb\inv \T_n  \D_{i,j,t}\inv(z)
		  \T_n \right\}
		  \Big]
		  \Big| \\
		  \leq & 
		  \E \Big| \tr \lb \D_{i,j,t}\inv(z) - \D_{j,t}\inv(z) \rb 
		  \lb t \sut_{t}(z) \T_n + \mathbf{I} \rb\inv \T_n \D_{j,t}\inv(z) \T_n \Big| \\
		  & +  \E \Big| \tr \D_{i,j,t}\inv(z) 
		  \lb t \sut_{t}(z) \T_n + \mathbf{I} \rb\inv \T_n \lb \D_{i,j,t}\inv(z) - \D_{j,t}\inv(z) \rb  \T_n \Big| \\
		  = & 
		  \E \Big| \beta_{i,j,t}(z)  \rd_i^\star \D_{i,j,t}\inv(z) 
		  \lb t \sut_{t}(z) \T_n + \mathbf{I} \rb\inv \T_n \D_{j,t}\inv(z) \T_n \D_{i,j,t}\inv(z) \rd_i \Big| \\
		  & +  \E \Big| \beta_{i,j,t} (z)  \rd_i^\star \D_{i,j,t}\inv(z)   \T_n 
		  \D_{i,j,t}\inv(z) 
		  \lb t \sut_{t}(z) \T_n + \mathbf{I} \rb\inv \T_n  \D_{i,j,t}\inv(z) \rd_i
		  \Big| \\
		  \leq & K 
		  +   \E \Big| \beta_{i,j,t}(z)  \rd_i^\star \D_{i,j,t}\inv(z) 
		  \lb t \sut_{t}(z) \T_n + \mathbf{I} \rb\inv \T_n 
		   \lb \D_{i,j,t}\inv(z) - 
		  \beta_{i,j,t}(z)\D_{i,j,t}\inv(z) \rd_i \rd_i^\star 
		  \D_{i,j,t}\inv(z)
		  \rb 
		   \T_n \D_{i,j,t}\inv(z) \rd_i \Big| \\
		   \leq & K. 
	\end{align*}
		Hence, 
		\begin{align*}
			 & \Big| 
			 	\frac{z^2 \sut_t^4(z)}{\ntn n^3}
		  \sum\limits_{j=1}^{\ntn}
		  \sum\limits_{i \neq j} 
		\tr \left\{ \lb t \sut_t(z) \T_n + \mathbf{I} \rb^{-2} \T_n^2 \right\}   
			 \E \Big[ 
		  \tr \left\{ 
		  \D_{j,t}\inv(z)
	 \lb t \sut_{t}(z) \T_n + \mathbf{I} \rb\inv \T_n  \D_{j,t}\inv(z)
		  \T_n \right\} \\
		  & -  
		  \tr \left\{ 
		  \D_{i,j,t}\inv(z)
	 \lb t \sut_{t}(z) \T_n + \mathbf{I} \rb\inv \T_n  \D_{i,j,t}\inv(z)
		  \T_n \right\}
		  \Big]
		  \Big| 
		   =  o(1). 
		\end{align*}
		We now apply \eqref{a71} for \eqref{a70} and obtain 
		\begin{align*}
		 y_{\ntn} \sum\limits_{j=1}^{\ntn} \E [ \beta_{j,t}(z) d_{j,t}(z) ] 
		= & 
		\frac{ \sut_t^2(z)}{ n} 
	  \E \left[ \tr  
	 \lb t  \sut_{t}(z) \T_n + \mathbf{I} \rb^{-3} \T_n^2 \right] \nonumber \\
	 & + \frac{ \sut_t^2(z) \ntn }{ n^2}
		  \tr \left\{ \lb t \sut_t(z) \T_n + \mathbf{I} \rb^{-2} \T_n^2 \right\}   
		  y_{\ntn} \sum\limits_{j=1}^{\ntn} \E [ \beta_{j,t}(z) d_{j,t}(z) ] 
		  + o(1).
		\end{align*}
		This implies \eqref{conv2} for the real case, namely, 
		\begin{align*}
			y_{\ntn} \sum\limits_{j=1}^{\ntn} \E [ \beta_{j,t}(z) d_{j,t}(z) ] 
			= & \frac{\frac{ \sut_t^2(z)}{ n} 
	  \tr \left\{  
	 \lb t  \sut_{t}(z) \T_n + \mathbf{I} \rb^{-3} \T_n^2 \right\} }
			{1 - \frac{ \sut_t^2(z) \ntn }{ n^2}
		  \tr \left\{ \lb t \sut_t(z) \T_n + \mathbf{I} \rb^{-2} \T_n^2 \right\}   }
			+ o(1) \\ 
			= & \frac{y  \int \frac{\sut_t^2(z)\lambda^2}{(t \sut_t(z) \lambda + 1)^3 } dH(\lambda) }
			{1 - t y  \int \frac{\sut_t^2(z)\lambda^2}{( t \sut_t(z) \lambda + 1 )^2}  dH(\lambda)} + o(1).
		\end{align*}

\subsection{Further auxiliary results} \label{aux} 
    \begin{lemma} \label{mom_Dinv}
	We have uniformly in $n\in\N, t\in[t_0,1], z\in\mathcal{C}_n$
	\begin{align} \label{c23}
		\E ||\D_t\inv(z) ||^q \leq K,
	\end{align}
	where $K>0$ is a constant depending on $q\in\N$.
	Similarly, for pairwise different integers $i,j,k \in \{ 1, \ldots, \ntn\}$
	\begin{align*} 
		\max \lb  \E ||\D_{t}\inv(z) ||^q, \E ||\D_{j,t}\inv(z) ||^q , \E ||\D_{j,k,t}\inv(z) ||^q  \rb 
		\leq K.
	\end{align*}
	It also holds that 
	\begin{align} \label{c22}
		|| \D_t\inv (z) || \leq & K
		+ n \varepsilon_n\inv I \{ ||\mathbf{B}_{n,t} || \geq \eta_{r,t} \textnormal{ or }
		\lambda_{\min}(\mathbf{B}_{n,t}) \leq \eta_{l,t} \}.
	\end{align}
	\end{lemma}
	\begin{proof}[Proof of Lemma \ref{mom_Dinv}]
	We restrict ourselves to the first assertion. 
	Let first $z\in \mathcal{C}_u$, that is, $z= x + i v_0$ for some $x \in [x_l , x_r ]$. Then, 
	\begin{align*}
		|| \D_t\inv(z) || 
		= \frac{1}{ \min (  | \lambda_{\min}(\mathbf{B}_{n,t}) - z | , | \lambda_{\max}(\mathbf{B}_{n,t}) - z | ) } 
		\leq \frac{1}{v_0} = K.
	\end{align*}
	 This implies
$
	 	\E || \D_t \inv (z) ||^q \leq K.
$
	Next, assume $z\in \mathcal{C}_l \cup \mathcal{C}_r$, that is, $z = x_r + iv$ or $z = x_l + i v$ for some $v\in [n\inv \varepsilon_n, v_0]$. 
	By formula  (9.7.8) and (9.7.9) in \cite{bai2004}	 we have for $t\in [t_0,1]$ and any $m>0$
	\begin{align}
		\PR \lb || \mathbf{B}_{n,t} || > \eta_{r,t} 
		\textnormal{ or }
		\lambda_{\min}(\mathbf{B}_{n,t}) < \eta_{l,t} 
		\rb
		= o \lb \ntn^{-m} \rb 
		= o \lb n^{-m} \rb,
		\label{c24}
	\end{align}
	where $\eta_{r,t}$ denotes a fixed number between 
	\begin{align*}
		\limsup\limits_{n \to \infty} || \T_n|| ( 1 + \sqrt{y_t} )^2 t
	\end{align*}
	and $x_r$ and $\eta_{l,t} $ between
	\begin{align*}
		\liminf\limits_{n\to \infty}  \lambda_{\min}(\T_n)  ( 1 - \sqrt{y_t} )^2 I_{(0,1)} (y_t) t
	\end{align*}
	and $x_l$.
	We estimate 
	\begin{align*}
		 \E ||\D_t\inv(z) ||^q  
		\leq & K \E \left[ ||\D_t\inv(z) || 
		I \{ ||\mathbf{B}_{n,t} || \leq \eta_{r,t} \textnormal{ and } \lambda_{\min}(\mathbf{B}_{n,t}) \geq \eta_{l,t} \} \right] ^q \\
		& +K  \E \left[ ||\D_t\inv(z) ||
		 I \{ ||\mathbf{B}_{n,t} || > \eta_{r,t} \textnormal{ or } \lambda_{\min}(\mathbf{B}_{n,t}) < \eta_{l,t} \} \right] ^q \\
		 \leq & K + K n^q \varepsilon_n^{-q} n^{-m}  \leq K.
	\end{align*}
	To derive a bound for the first term, we distinguish the cases $z\in \mathcal{C}_r$ and $z \in \mathcal{C}_l$. For the sake of brevity, we only consider the first one. It holds
	\begin{align}
	&  ||\D_t\inv(z) || 
		I \{ ||\mathbf{B}_{n,t} || \leq \eta_{r,t} \textnormal{ and } \lambda_{\min}(\mathbf{B}_{n,t}) \geq \eta_{l,t} \}  \nonumber \\
		= &   \frac{1}{ \min (  | \lambda_{\min}(\mathbf{B}_{n,t}) - (x_r + iv) | , | \lambda_{\max}(\mathbf{B}_{n,t}) - (x_r+iv) | ) } 
		I \{ ||\mathbf{B}_{n,t} || \leq \eta_{r,t} \textnormal{ and } \lambda_{\min}(\mathbf{B}_{n,t}) \geq \eta_{l,t} \} \nonumber \\
		\leq &  \frac{1}{ x_r -  \lambda_{\max}(\mathbf{B}_{n,t})    } 
		I \{ ||\mathbf{B}_{n,t} || \leq \eta_{r,t} \textnormal{ and } \lambda_{\min}(\mathbf{B}_{n,t}) \geq \eta_{l,t} \}  \nonumber \\
		\leq &  \frac{1}{ x_r -  \eta_{r,t}  } 
		\leq \frac{1}{ x_r - \limsup\limits_{n\to\infty} ||\T_n|| (1 + \sqrt{y_{t_0}} )^2  } = K .
		\label{c20}
	\end{align}
	For the second summand, we conclude
	\begin{align}
		&  ||\D_t\inv(z) ||
		 I \{ ||\mathbf{B}_{n,t} || > \eta_{r,t} \textnormal{ or } \lambda_{\min}(\mathbf{B}_{n,t}) < \eta_{l,t} \} \nonumber \\
		 \leq &   \frac{1}{ \min (  | \lambda_{\min}(\mathbf{B}_{n,t}) - z | , | \lambda_{\max}(\mathbf{B}_{n,t}) - z | ) } 
		 I \{ ||\mathbf{B}_{n,t} || > \eta_{r,t} \textnormal{ or } \lambda_{\min}(\mathbf{B}_{n,t}) < \eta_{l,t} \}  \nonumber \\
		 \leq & n \varepsilon_n^{-1} I \left\{ ||\mathbf{B}_{n,t} || > \eta_{r,t} \textnormal{ or } \lambda_{\min}(\mathbf{B}_{n,t}) < \eta_{l,t} \right\}.
		 \label{c21}
	\end{align}
	The bounds in \eqref{c20} and \eqref{c21} show that \eqref{c22} holds true. 
	The assertion in \eqref{c23} follows by applying \eqref{c24}.  
	\end{proof}

The bounds for the increments of $M_n(z,t)$, $z\in\mathcal{C}_n, t\in[t_0,1]$ are given in the following lemma, which will be proven later. 
	
	\begin{lemma} \label{no_hat}
	 	For $t\in[t_0,1], z_1,z_2\in\mathcal{C}_n$, it holds for sufficiently large $n\in\N$ under the assumptions of Theorem \ref{asympt_tight}
		\begin{align}
			\E | M_n^1(z_1,t) - M_n^1(z_2,t) |^{2+\delta} \lesssim |z_1 - z_2|^{2+\delta}. \label{tight_z}
		\end{align} 
		We also have for $t_1,t_2\in[t_0,1], z \in\mathcal{C}_n$
		\begin{align}
			\E |Z_n^1(z,t_1,t_2) |^4 & \lesssim  \Big  ( \frac{\ntt - \nt}{n} \Big  )^4, \label{tight_t1} \\
	 		\E |Z_n^2(z,t_1,t_2) |^{4 + \delta} & \lesssim  \Big  ( \frac{\ntt - \nt}{n} \Big  )^{2 + \delta /2}, \label{tight_t2}
		\end{align} 
		where 
\begin{equation} \label{aa1}
			M_n^1(z,t_1) - M_n^1(z,t_2) 
			= Z_n^1(z,t_1,t_2) + Z_n^2(z,t_1,t_2)
			,
			\end{equation}
	and $Z_n^1$ and $Z_n^2$ are defined in \eqref{det7} and \eqref{det8}, respectively. 
	\end{lemma}

	The proof of  Lemma \ref{no_hat} requires  some preparations. 
	Note that while a fourth moment condition is sufficient for proving the convergence of the finite-dimensional distribution of $(\hat{M}_n^1)_{n\in\N}$ (Theorem \ref{thm_fidis}) and the convergence of the non-random part $(\hat{M}_n^2)_{n\in\N}$ (Theorem \ref{thm_bias}), we need the stronger moment assumption from Theorem \ref{thm}, namely
	\begin{align} \label{mom_cond}
		\sup\limits_{i,j,n} \E | x_{ij}|^{12} < \infty,
	\end{align}
	exclusively for a proof of the asymptotic tightness of $(\hat{M}_n^1)_{n\in\N}$. \\
	Under this assumption, by Lemma B.26 in \cite{bai2004},
	%Lemma 2.7 in \cite{baisilverstein1998} for iid case
	the following estimates for moments of quadratic forms hold true for $q\geq 2$
	\begin{align*}
		\E | \mathbf{x}_j^\star \mathbf{A} \mathbf{x}_j - \tr \mathbf{A} |^q
		\lesssim &  \lb \tr \mathbf{A} \mathbf{A}^\star \rb ^{q/2}  + \eta_n^{(2q-12) \vee 0} n^{(q-6) \vee 0} \tr (\mathbf{A} \mathbf{A}^\star )^{q/2}  \\
		\lesssim &
		\begin{cases}
		&  \lb \tr \mathbf{A} \mathbf{A}^\star \rb^{q/2} \lb 1 + n^{(q-6) \vee 0} \rb, \\
		&  n^{q/2} ||\A||^q +  n n^{(q-6) \vee 0} ||\A||^q.
		\end{cases} 
		\end{align*}
		Thus, we have for $q \geq 2$
		\begin{align}
		\E | \mathbf{r}_j^\star \mathbf{A} \mathbf{r}_j - n\inv \tr \T_n \mathbf{A} |^q
		\lesssim &
		\begin{cases}
		&   \lb \tr \mathbf{A} \mathbf{A}^\star \rb^{q/2}  n^{-(q \wedge 6)}, \\
		&   ||\A||^q n^{-((q/2) \wedge 5)}.
		\end{cases}
		\label{h1}
		\end{align}

	Furthermore, combining \eqref{h1} with arguments given in the proof of (9.9.6) in \cite{bai2004}, we obtain  the following lemma. 
	
	\begin{lemma} \label{h1a}
		Let $j,m\in\N_0$, $ q \geq 2$ and $\mathbf{A}_l$, $l\in\{1,\ldots,m+1\}$ be $p \times p$ (random) matrices independent of $\rd_j$ which obey for any $\tilde{q} \geq 2$
		\begin{align*}
			\E || \mathbf{A}_l ||^{\tilde{q}} < \infty, ~ l\in \{1, \ldots, m+1\}. 
		\end{align*}
		Then, it holds
		\begin{align*}
			\E \Big| \Big( \prod\limits_{k=1}^m \rd_j^\star \mathbf{A}_k \rd_j \Big)
		\lb \rd_j^\star \mathbf{A}_{m+1} \rd_j - n\inv \tr \T_n \mathbf{A}_{m+1} \rb \Big|^q
		\lesssim n^{-((q/2) \wedge 5)} .
		\end{align*}
		If even for any $l\in \{1, \ldots, m+1\}$, $\tilde{q} \geq 2$
		\begin{align*}
			\E \left[  \tr \mathbf{A} \mathbf{A}_l^\star  \right]^{\tilde{q}} < \infty ,
		\end{align*}
		holds true, then we have
		\begin{align*}
				\E \Big| \Big( \prod\limits_{k=1}^m \rd_j^\star \mathbf{A}_k \rd_j \Big)
		\lb \rd_j^\star \mathbf{A}_{m+1} \rd_j - n\inv \tr \T_n \mathbf{A}_{m+1} \rb \Big|^q
		\lesssim 
		 n^{-(q \wedge 6)}.
		\end{align*}
	\end{lemma}
	
		\begin{remark}
{\rm 
		In fact, as the proof of Lemma \ref{h1a} reveals, we could impose a less restrictive condition on the spectral moments of $\mathbf{A}_l$, $l\in\{1, \ldots, m+1\}$. For our purpose, it is sufficient to state the previous lemma in this form, since, when applying Lemma \ref{h1a}, the involved matrices will have bounded spectral moments of any order.
	\\ In particular, the second assertion will be useful if $\mathbf{B}_l$ involves a term like $\rd_k \rd_k^\star$ for some $k\neq j$ among other matrices like $\D_{j,t}\inv(z)$,  while we will make use of the first assertion in case that $\mathbf{B}_l$ only involves matrices like $\D_{j,t}\inv(z)$. In the latter case, contrary to the first one, we are not able to control moments of $\tr \mathbf{B}_l \mathbf{B}_l^\star$ uniformly in $n$. 
	}
	\end{remark}
	
	\begin{proof}[Proof of Lemma \ref{h1a}]
		For $m=0$, the assertion of the lemma follows directly from \eqref{h1} for any $q\geq 2$. 
		We continue the proof by an induction over the integer $m$ for some fixed $q\geq 2$. 
		\begin{align*}
			& \E \Big| \Big( \prod\limits_{k=1}^m \rd_j^\star \mathbf{A}_k \rd_j \Big)
		\lb \rd_j^\star \mathbf{A}_{m+1} \rd_j - n\inv \tr \T_n \mathbf{A}_{m+1} \rb \Big|^q \\
		\lesssim & \E \Big| \Big( \prod\limits_{k=1}^{m - 1} \rd_j^\star \mathbf{A}_k \rd_j \Big)
		\lb \rd_j^\star \mathbf{A}_{m} \rd_j - n\inv \tr \T_n \mathbf{A}_{m} \rb
		\lb \rd_j^\star \mathbf{A}_{m+1} \rd_j - n\inv \tr \T_n \mathbf{A}_{m+1} \rb \Big|^q \\
		& + \E \Big| \Big( \prod\limits_{k=1}^{m - 1} \rd_j^\star \mathbf{A}_k \rd_j \Big)
		 n\inv \tr \T_n \mathbf{A}_{m} 
		\lb \rd_j^\star \mathbf{A}_{m+1} \rd_j - n\inv \tr \T_n \mathbf{A}_{m+1} \rb \Big|^q \\
		\leq & \Big(  \E \Big| \Big( \prod\limits_{k=1}^{m - 1} \rd_j^\star \mathbf{A}_k \rd_j \Big)
		\lb \rd_j^\star \mathbf{A}_{m} \rd_j - n\inv \tr \T_n \mathbf{A}_{m} \rb \Big|^{2q} 
		\E \left| \lb \rd_j^\star \mathbf{A}_{m+1} \rd_j - n\inv \tr \T_n \mathbf{A}_{m+1} \rb \right|^{2q} \Big)^{\frac{1}{2}}\\
		& + \E \Big| \Big( \prod\limits_{k=1}^{m - 1} \rd_j^\star \mathbf{A}_k \rd_j \Big)
		 n\inv \tr \T_n \mathbf{A}_{m} 
		\lb \rd_j^\star \mathbf{A}_{m+1} \rd_j - n\inv \tr \T_n \mathbf{A}_{m+1} \rb \Big|^q.
		\end{align*}
		By applying the induction hypothesis to these three terms, we get the desired result for each case.
	\end{proof}
	
	Adapting the proof of (9.10.5) in \cite{bai2004}, we obtain under the strong moment condition \eqref{mom_cond} for $q \geq 2$
	\begin{align} \label{gamma_mom}
	 \E|\gamma_{j,t}(z)|^q \lesssim n^{-( (q/2) \wedge 5)}.
	\end{align}
    We need an estimate for moments of complex martingale difference schemes. We refer to Lemma 2.1 in \cite{li2003}, which is an corollary from Burkholder's inequality and can easily be extended to the complex case. 
	We are now in the position to give a proof of Lemma \ref{no_hat}.
		
		\begin{proof}[Proof of Lemma \ref{no_hat}]
			In the following, we will often make use of the decompositions
	\begin{align}
		\D_t\inv(z) & = \D_{j,t}\inv(z) - \beta_{j,t}(z) \D_{j,t}\inv(z) \rd_j \rd_j^\star \D_{j,t}\inv(z), \label{aa2}\\
		\beta_{j,t}(z) &= b_{j,t}(z) - \beta_{j,t}(z) b_{j,t}(z) \gamma_{j,t}(z).
		\nonumber
	\end{align}
%	Observing the representation in \eqref{mn_t}, we remind the reader of the decomposition
%	\begin{align} \label{aa1}
%	M_n^1(z,t_1) - M_n^1(z,t_2)
%	= & Z_n^1(z,t_1,t_2) + Z_n^2 (z,t_1,t_2), ~ t_2 > t_1, 
%	\end{align}
%	where $Z_n^1$ and $Z_n^2$ are defined in \eqref{det7} and \eqref{det8}. 
%	\begin{align*}
% 	Z_n^1(z,t_1,t_2) = & \sum\limits_{j=1}^{\nt} (\E_j - \E_{j - 1} ) \lb  \beta_{j,t_2}(z) \rd_j^\star \D_{j,t_2}^{-2}(z) \rd_j 
% 	- \beta_{j,t_1}(z) \rd_j^\star \D_{j,t_1}^{-2}(z) \rd_j \rb, \\
% 	Z_n^2 (z,t_1,t_2) = &  \sum\limits_{j=\nt + 1}^{\ntt} (\E_j - \E_{j - 1}) \beta_{j,t_2}(z) \rd_j^\star \D_{j,t_2}^{-2}(z) \rd_j .
% 	\end{align*}
Observing the decomposition \eqref{aa1}, our  aim is to show  the inequalities in \eqref{tight_t1} and \eqref{tight_t2}, where we assume $t_2 > t_1$ w.l.o.g.
	\\ \bigskip \\ 
	\textbf{Step 1:} \textit{ Analysis of $Z_n^2$}  \\ 
	Beginning with the proof of \eqref{tight_t2} for $Z_n^2$, we are able to show that (using Lemma 2.1 in \cite{li2003} with $q=4+\delta$)
	\begin{align*}
		\E | Z_n^2(z, t_1, t_2) |^{4+\delta} & =  \E \Big| \sum\limits_{j=\nt + 1}^{\ntt} (\E_j - \E_{j - 1}) \beta_{j,t_2}(z) \rd_j^\star \D_{j,t_2}^{-2}(z) \rd_j \Big|^{4+\delta} \\
		\lesssim  &  \lb \ntt - \nt \rb^{1+\delta/2} \sum\limits_{j=\nt + 1}^{\ntt} \E  \left|(\E_j - \E_{j - 1}) \beta_{j,t_2}(z) \rd_j^\star \D_{j,t_2}^{-2}(z) \rd_j \right|^{4+\delta} \\
		\lesssim &  \Big(\frac{\ntt - \nt}{n} \Big)^{2+\delta/2},
		\end{align*}
		since we can bound
		\begin{align}
		\E  \left|(\E_j - \E_{j - 1}) \beta_{j,t_2}(z) \rd_j^\star \D_{j,t_2}^{-2}(z) \rd_j \right|^{4+\delta} 
		\lesssim & 
		 \E  \left|(\E_j - \E_{j - 1}) b_{j,t_2}(z) \rd_j^\star \D_{j,t_2}^{-2}(z) \rd_j \right|^{4+\delta} \nonumber 
		\\ & + \E  \left|(\E_j - \E_{j - 1}) \beta_{j,t_2}(z) b_{j,t_2}(z) \rd_j^\star \D_{j,t_2}^{-2}(z) \rd_j \gamma_{j,t_2} (z) \right|^{4+\delta} \nonumber \\
		\lesssim & 
		\E  \left|(\E_j - \E_{j - 1})  \left\{ \rd_j^\star \D_{j,t_2}^{-2}(z) \rd_j 
		- n\inv \tr \T_n \D_{j,t_2}^{-2}(z) \right\} \right|^{4+\delta} \nonumber 
		\\ & +  \E  \left| \beta_{j,t_2}(z) b_{j,t_2}(z) \rd_j^\star \D_{j,t_2}^{-2}(z) \rd_j \gamma_{j,t_2} (z) \right|^{4+\delta} \label{gamma}
		\\
		\lesssim &  n^{-(2+\delta/2)} . \nonumber
	\end{align}
	% hier genügt viertes Moment nicht für Größenordnung n^{-3}
	We should explain the bound for \eqref{gamma} in more detail: First, note that we are able to bound the moments of $||\D_{j,t}\inv(z)||$ independent of $n,z,t$ (see Lemma \ref{mom_Dinv}). As a further preparation, we observe for $z\in\mathcal{C}_n, t\in [t_0,1]$ from Lemma  \ref{mom_Dinv} 
	\begin{align}
		|| \D_t\inv (z) ||
		\lesssim & 1
		+ n \varepsilon_n\inv I \{ ||\mathbf{B}_{n,t} || \geq \eta_{r,t} \textnormal{ or }
		\lambda_{\min}(\mathbf{B}_{n,t}) \leq \eta_{l,t} \} \nonumber \\
		\leq & 1 
		+ n^2 I \{ ||\mathbf{B}_{n,t} || \geq \eta_{r,t} \textnormal{ or }
		\lambda_{\min}(\mathbf{B}_{n,t}) \leq \eta_{l,t} \}, \label{d_inv_norm}
	\end{align}
	where we used the fact that $\varepsilon_n \geq n^{-\alpha}$ for some $\alpha \in (0,1)$. 
	Thus, since  $| \rd_j |^2 \leq n$, we obtain
	\begin{align}
		| \beta_{j,t}(z) | = &  | 1 - \rd_j^\star \D_t\inv(z) \rd_j | \leq 1 + |\rd_j|^2 || \D_t\inv(z) || 
		\nonumber \\
		\lesssim & 1 +  | \rd_j|^2+ n^3  I \{ ||\mathbf{B}_{n,t} || \geq \eta_{r,t} \textnormal{ or }
		\lambda_{\min}(\mathbf{B}_{n,t}) \leq \eta_{l,t} \}. \label{beta}
	\end{align}
	It is easy to see that the inequality (9.10.6) in \cite{bai2004} also holds for $\beta_{j,t}(z)$ and by the same arguments following (9.10.6), we obtain
	\begin{align} \label{h2}
		  |b_{j,t}(z) | \leq K.
	\end{align}
	Similarly to these bounds, using \eqref{c22} in Lemma \ref{mom_Dinv} for the matrix $\D_{j,t}\inv(z)$, we get for any $m\geq 1$
	\begin{align*}
	| \gamma_{j,t}(z) | 
	= & | \rd_j^\star \D_{j,t}\inv(z) \rd_j - n\inv \E [ \tr \T_n \D_{j,t}\inv(z) ]|    
	\lesssim  |\rd_j|^2 || \D_{j,t}\inv(z) || +    \E || \D_{j,t}\inv(z) || \nonumber \\
		\lesssim &   |\rd_j|^2
		 +     | \rd_j|^2 n \varepsilon_n\inv 
		 I \{ ||\mathbf{B}_{n,t}^{(-j)} || \geq \eta_{r,t} \textnormal{ or }
		\lambda_{\min}(\mathbf{B}_{n,t}^{(-j)}) \leq \eta_{l,t} \} \nonumber \\
		& +  | \rd_j|^2 n \varepsilon_n\inv 
		 \mathbb{P} \{ ||\mathbf{B}_{n,t}^{(-j)} || \geq \eta_{r,t} \textnormal{ or }
		\lambda_{\min}(\mathbf{B}_{n,t}^{(-j)}) \leq \eta_{l,t} \} \nonumber \\
		\leq &  | \rd_j|^2 +    n^3 I \{ ||\mathbf{B}_{n,t}^{(-j)} || \geq \eta_{r,t} \textnormal{ or }
		\lambda_{\min}(\mathbf{B}_{n,t}^{(-j)}) \leq \eta_{l,t} \}
		+  o\lb n^{-m} \rb, %\label{gamma_det_bound}
	\end{align*}
	where we used the fact that for any $m>0$
	 \begin{align}
	  & \PR \{ ||\mathbf{B}_{n,t_2}^{(-j)} || \geq \eta_{r,t_2} \textnormal{ or }
		\lambda_{\min}(\mathbf{B}_{n,t_2}^{(-j)}) \leq \eta_{l,t_2} \} = o \lb n^{-m} \rb, \nonumber \\
		& \PR  \{ ||\mathbf{B}_{n,t_2} || \geq \eta_{r,t_2} \textnormal{ or }
		\lambda_{\min}(\mathbf{B}_{n,t_2}) \leq \eta_{l,t_2} \}
		= o \lb n^{-m} \rb \label{ind_bound}
	 \end{align}
	and the notation 
	 \begin{align*}
	 \mathbf{B}_{n,t}^{(-j)} = \mathbf{B}_{n,t} - \rd_j \rd_j^\star.
	 \end{align*}
	
	Using \eqref{h1} and \eqref{gamma_mom}, we can also bound 
	\begin{align*}
		\E \Big|  | \rd_j | ^{2} \gamma_{j,t}(z) \Big|^{4+\delta} 
		= & \E | \rd_j^\star \rd_j  \gamma_{j,t}(z) |^{4+\delta}  
		\lesssim  \E | \lb \rd_j^\star \rd_j - n\inv \tr \T_n \rb \gamma_{j,t}(z) |^{4+\delta} 
		+  \E | n\inv \tr ( \T_n )\gamma_{j,t}(z) |^{4+\delta}  \\
		\leq &  \lb  \E | \rd_j^\star \rd_j - n\inv \tr \T_n |^{8 + 2\delta} \E| \gamma_{j,t}(z) |^{8+2\delta} \rb\sq
		+  \E | \gamma_{j,t}(z) |^{4+\delta}  \lesssim  n^{-(2+\delta/2)} .
	\end{align*}
	By induction, one can show for some $q \in\N_0$ and $\delta \geq 0$
	\begin{align} \label{gamma_r}
		\E \Big|  | \rd_j | ^{2q} \gamma_{j,t}(z) \Big|^{4+\delta}  \lesssim n^{-(2 + \delta /2) }. 
	\end{align}
	Combining these inequalities, we conclude
	\begin{align}
	&  \E  \left| \beta_{j,t_2}(z) b_{j,t_2}(z) \rd_j^\star \D_{j,t_2}^{-2}(z) \rd_j \gamma_{j,t_2} (z) \right|^{4+\delta}  \nonumber \\
	 \lesssim &  \E \Big| 
	 \Big  ( 1 +  | \rd_j|^2  + n^3  I  \{ ||\mathbf{B}_{n,t_2} || \geq \eta_{r,t_2} \textnormal{ or }
		\lambda_{\min}(\mathbf{B}_{n,t_2}) \leq \eta_{l,t_2} \} \Big  )
	 | \rd_j |^2 \nonumber \\
	 & \times \Big  ( 1 
		+ n^2 I \{ ||\mathbf{B}_{n,t_2}^{(-j)} || \geq \eta_{r,t_2} \textnormal{ or }
		\lambda_{\min}(\mathbf{B}_{n,t_2}^{(-j)}) \leq \eta_{l,t_2} \} \Big  ) ^2
	 \gamma_{j,t_2} (z) \Big|^{4+\delta} . \label{r1}
	 \end{align}

	 The expectation in \eqref{r1} can now be estimated by multiplying these terms out and using the inequalities \eqref{h2} and \eqref{ind_bound}.

	 Thus, we conclude that
	 \begin{align*}
	 	\E  \left| \beta_{j,t_2}(z) b_{j,t_2}(z) \rd_j^\star \D_{j,t_2}^{-2}(z) \rd_j \gamma_{j,t_2} (z) \right|^{4+\delta}  
	 	\lesssim n^{-(2+\delta /2)}. 
	 \end{align*} 
	% Note that for the analysis of $Z_n^2$, it is sufficient to assume uniformly bounded moments of order $6+2\delta$. 
	 \\ \bigskip \\
	 \textbf{Step 2:} \textit{Analysis of $ M_n^1(z_1,t) - M_n^1(z_2,t) $} \\
	 Before investigating the term $Z_n^1$ in the decomposition \eqref{aa1}, we show that \eqref{tight_z} holds true in a similar fashion to the considerations above. We write for $z_1,z_2\in\mathcal{C}_n, t\in[t_0,1]$
	 \begin{align*}
	 	M_n^1(z_1,t) - M_n^1(z_2,t) 
	 	& = \sum\limits_{j=1}^{\ntn} ( \E_j - \E_{j - 1} ) \tr \lb \D_t\inv(z_1) - \D_t\inv(z_2) \rb \\
	 	& = \sum\limits_{j=1}^{\ntn} ( \E_j - \E_{j - 1} ) 
	 	( z_1 - z_2) \tr \D_t\inv(z_1)\D_t\inv(z_2) \\
	 	& = G_{n1} + G_{n2} + G_{n3}, 
	 \end{align*}
	 where
	 \begin{align*}
	 	G_{n1} & = ( z_1 - z_2)  \sum\limits_{j=1}^{\ntn} ( \E_j - \E_{j - 1} )
	 	\beta_{j,t}(z_1) \beta_{j,t}(z_2) \lb \rd_j^\star \D_{j,t}\inv(z_1) \D_{j,t}\inv(z_2)
	 	\rd_j \rb^2, \\
	 	G_{n2} &=  - ( z_1 - z_2)  \sum\limits_{j=1}^{\ntn} ( \E_j - \E_{j - 1} )
	 	\beta_{j,t}(z_1) \rd_j^\star \D_{j,t}^{-2}(z_1) \D_{j,t}\inv(z_2) \rd_j, \\
	 	G_{n3} &= - ( z_1 - z_2)  \sum\limits_{j=1}^{\ntn} ( \E_j - \E_{j - 1} )
	 	\beta_{j,t}(z_2) \rd_j^\star \D_{j,t}^{-2}(z_2) \D_{j,t}\inv(z_1) \rd_j.
	 \end{align*}
	 The terms  $G_{n2}$ and $G_{n3}$ can be estimated using similar arguments as given  in the proof of \eqref{tight_t2}. More precisely, we obtain for the second term
	 \begin{align*}
	 	\E | G_{n2} |^{2 + \delta} \lesssim |z_1 - z_2|^{2 + \delta},
	 \end{align*}
	 and a similar inequality holds for the third term. 
	 For the first summand, we have
	 \begin{align*}
	 G_{n1} = G_{n11} + G_{n12} + G_{n13},
	 \end{align*}
	 where 
	 \begin{align*}
	 	G_{n11} 
	 	= & ( z_1 - z_2)   \sum\limits_{j=1}^{\ntn} ( \E_j - \E_{j - 1} )
	 	b_{j,t}(z_1) b_{j,t}(z_2) \lb \rd_j^\star \D_{j,t}\inv(z_1) \D_{j,t}\inv(z_2)
	 	\rd_j \rb^2, \\
	 	G_{n12}
	 	= & -  ( z_1 - z_2)  \sum\limits_{j=1}^{\ntn} ( \E_j - \E_{j - 1} )
	 	b_{j,t}(z_2) \beta_{j,t}(z_1) \beta_{j,t}(z_2) \lb \rd_j^\star \D_{j,t}\inv(z_1) \D_{j,t}\inv(z_2)
	 	\rd_j \rb^2 \gamma_{j,t}(z_2), \\
	 	G_{n13}
	 	= & -( z_1 - z_2)   \sum\limits_{j=1}^{\ntn} ( \E_j - \E_{j - 1} )
	 	b_{j,t}(z_1) b_{j,t}(z_2)  \beta_{j,t}(z_1) \lb \rd_j^\star \D_{j,t}\inv(z_1) \D_{j,t}\inv(z_2)
	 	\rd_j \rb^2 \gamma_{j,t}(z_1). \\
	 \end{align*}
	 Here, the terms $G_{n12}$ and $G_{n13}$ can be treated by similar arguments as in the derivation of \eqref{gamma} using Lemma 2.1 in \cite{li2003}, which gives for $l\in\{1,2\}$
	 \begin{align*}
	 	\E |G_{n1l}|^{2 + \delta} \lesssim |z_1 - z_2|^{2 + \delta}.
	 \end{align*}
	 Therefore, it remains to investigate the term $G_{n11}$:
	 \begin{align*}
	 	\E |G_{n11}|^{ 2 + \delta} 
	 	\lesssim &   |z_1 - z_2|^{2+\delta} n^{\delta / 2} \sum\limits_{j=1}^{\ntn} 
		\E \left| ( \E_j - \E_{j - 1} )
	 	b_{j,t}(z_1) b_{j,t}(z_2) \lb \rd_j^\star \D_{j,t}\inv(z_1) \D_{j,t}\inv(z_2)
	 	\rd_j \rb^2 \right|^{2+\delta}. 
	 \end{align*}
	 We obtain for the summands in $\E |G_{n11}|^{2+ \delta}$ observing \eqref{h2}
	 \begin{align*}
	 	& \E \left | ( \E_j - \E_{j - 1} )
	 	b_{j,t}(z_1) b_{j,t}(z_2) \lb \rd_j^\star \D_{j,t}\inv(z_1) \D_{j,t}\inv(z_2)
	 	\rd_j \rb^2 \right|^{2+\delta}  \\
	 	\lesssim &  \E \left | ( \E_j - \E_{j - 1} )
	 \lb \rd_j^\star \D_{j,t}\inv(z_1) \D_{j,t}\inv(z_2)
	 	\rd_j \rb^2 \right|^{2+\delta} \\
	 	= &  \E \left | ( \E_j - \E_{j - 1} ) \left[ 
	 \lb \rd_j^\star \D_{j,t}\inv(z_1) \D_{j,t}\inv(z_2)
	 	\rd_j \rb^2 - \lb n\inv\tr \T_n \D_{j,t}\inv(z_1) \D_{j,t}\inv(z_2) \rb^2 
	 	\right] \right|^{2+\delta} \\
	 	= &  \E \Big | ( \E_j - \E_{j - 1} ) \Big[ 
	  \lb \rd_j^\star \D_{j,t}\inv(z_1) \D_{j,t}\inv(z_2)
	 	\rd_j  
	 	-  n\inv\tr \T_n \D_{j,t}\inv(z_1) \D_{j,t}\inv(z_2)  \rb \\
	  & \times \lb \rd_j^\star \D_{j,t}\inv(z_1) \D_{j,t}\inv(z_2)
	 	\rd_j  
	 	+  n\inv\tr \T_n \D_{j,t}\inv(z_1) \D_{j,t}\inv(z_2) \rb 
	 	\Big] \Big|^{2+\delta} \\
	 	\lesssim &  \E \Big | ( \E_j - \E_{j - 1} ) \Big[ 
	  \lb \rd_j^\star \D_{j,t}\inv(z_1) \D_{j,t}\inv(z_2)
	 	\rd_j  
	 	-  n\inv\tr \T_n \D_{j,t}\inv(z_1) \D_{j,t}\inv(z_2)  \rb 
	 \rd_j^\star \D_{j,t}\inv(z_1) \D_{j,t}\inv(z_2)
	 	\rd_j  
	 	\Big] \Big|^{2+\delta} \\
	 	+ &  \E \Big | ( \E_j - \E_{j - 1} ) \Big[ 
	  \lb \rd_j^\star \D_{j,t}\inv(z_1) \D_{j,t}\inv(z_2)
	 	\rd_j  
	 	-  n\inv\tr \T_n \D_{j,t}\inv(z_1) \D_{j,t}\inv(z_2)  \rb 
	 	n\inv \tr \T_n \D_{j,t}\inv(z_1) \D_{j,t}\inv(z_2)
	 	\Big] \Big|^{2+\delta} \\
	 	\lesssim & n^{-(1+\delta /2)},
	 \end{align*}
	 where we used Lemma \ref{h1a} with $q=2+\delta$ and $m=1$ and Lemma \ref{mom_Dinv} for the last inequality.
	These considerations show that \eqref{tight_z} holds true. 
	% Note that for this analysis, it is sufficient to assume uniformly bounded moments of order $4+2\delta$. 
	\bigskip \\
	 \textbf{Step 3:} \textit{Analysis of $ Z_n^1 $} \\
	Next, we  show the estimate  \eqref{tight_t1} for the term $Z_n^1$. Doing so, we will need condition \eqref{mom_cond} on the moments of $x_{ij}$. For the following calculation, we will write $\beta_t$ instead of $ \beta_t(z)$, $\D_t\inv $ instead of $\D_t\inv(z)$ and further omit the $z$-argument for similar quantities. 
	We have for $j \leq \nt$
	\begin{align*}
	    & \beta_{j,t_2} \rd_j^\star \D_{j,t_2}^{-2}\rd_j
		- \beta_{j,t_1} \rd_j^\star \D_{j,t_1}^{-2} \rd_j
		=  ( \beta_{j,t_2} - \beta_{j,t_1} ) \rd_j^\star \D_{j,t_2}^{-2} \rd_j
		+ \beta_{j,t_1} \rd_j^\star \lb \D_{j,t_2}^{-2} - \D_{j,t_1}^{-2}\rb \rd_j \\
		= &  ( \beta_{j,t_2} - \beta_{j,t_1} ) \rd_j^\star \D_{j,t_2}^{-2} \rd_j
		+ \beta_{j,t_1} \rd_j^\star \lb \D_{j,t_2}^{-1} - \D_{j,t_1}^{-1}\rb \D_{j,t_1}\inv  \rd_j 
		+  \beta_{j,t_1} \rd_j^\star \D_{j,t_2}\inv \lb \D_{j,t_2}^{-1} - \D_{j,t_1}^{-1}\rb   \rd_j \\
		= & \lb \rd_j^\star \D_{t_1}\inv \rd_j - \rd_j^\star \D_{t_2}\inv \rd_j \rb 
		\rd_j^\star \D_{j,t_2}^{-2} \rd_j 
		- \beta_{j,t_1} \rd_j^\star \D_{j,t_1}\inv \Big( \sum\limits_{k=\nt + 1}^{\ntt} \rd_k \rd_k^\star \Big) \D_{j,t_2}\inv \D_{j,t_1}\inv \rd_j \\
		& - \beta_{j,t_1} \rd_j^\star \D_{j,t_2}\inv \D_{j,t_1}\inv \Big( \sum\limits_{k=\nt + 1}^{\ntt} \rd_k \rd_k^\star \Big) \D_{j,t_2}\inv  \rd_j \\
		= &  \rd_j^\star \D_{t_2}\inv \Big( \sum\limits_{k=\nt + 1}^{\ntt} \rd_k \rd_k^\star \Big) \D_{t_1}\inv \rd_j \rd_j^\star \D_{j,t_2}^{-2} \rd_j 
		- \beta_{j,t_1} \rd_j^\star \D_{j,t_1}\inv \Big( \sum\limits_{k=\nt + 1}^{\ntt} \rd_k \rd_k^\star \Big) \D_{j,t_2}\inv \D_{j,t_1}\inv \rd_j \\
		& - \beta_{j,t_1} \rd_j^\star \D_{j,t_2}\inv \D_{j,t_1}\inv \Big( \sum\limits_{k=\nt + 1}^{\ntt} \rd_k \rd_k^\star \Big) \D_{j,t_2}\inv  \rd_j \\
		= & \sum\limits_{k=\nt + 1}^{\ntt} \Big\{
		 \rd_j^\star \D_{t_2}\inv  \rd_k \rd_k^\star  \D_{t_1}\inv \rd_j \rd_j^\star \D_{j,t_2}^{-2} \rd_j 
		- \beta_{j,t_1} \rd_j^\star \D_{j,t_1}\inv  \rd_k \rd_k^\star  \D_{j,t_2}\inv \D_{j,t_1}\inv \rd_j 
		 - \beta_{j,t_1} \rd_j^\star \D_{j,t_2}\inv \D_{j,t_1}\inv  \rd_k \rd_k^\star \D_{j,t_2}\inv  \rd_j 
		\Big\}. 
	\end{align*}
	Hence, using the identity \eqref{aa2}, we obtain the representation
	\begin{align*}
		& Z_n^1(z,t_1,t_2) \\
		= & \sum\limits_{j=1}^{\nt} \sum\limits_{k=\nt + 1}^{\ntt} ( \E_j - \E_{j - 1})  \Big\{
		- \beta_{j,t_1} \rd_j^\star \D_{j,t_1}\inv  \rd_k \rd_k^\star  \D_{j,t_2}\inv
		\D_{j,t_1}\inv \rd_j 
		 - \beta_{j,t_1} \rd_j^\star \D_{j,t_2}\inv \D_{j,t_1}\inv  \rd_k \rd_k^\star \D_{j,t_2}\inv  \rd_j \\
		& + \rd_j^\star \D_{j,t_2}\inv  \rd_k \rd_k^\star  \D_{j,t_1}\inv \rd_j \rd_j^\star \D_{j,t_2}^{-2} \rd_j 
		- \beta_{j,t_1} \rd_j^\star \D_{j,t_2}\inv  \rd_k \rd_k^\star  \D_{j,t_1}\inv \rd_j \rd_j^\star \D_{j,t_1}\inv \rd_j \rd_j^\star \D_{j,t_2}^{-2} \rd_j \\
		& - \beta_{j,t_2} \rd_j^\star \D_{j,t_2}\inv \rd_j \rd_j^\star \D_{j,t_2}\inv  \rd_k \rd_k^\star  \D_{j,t_1}\inv \rd_j \rd_j^\star \D_{j,t_2}^{-2} \rd_j 
		+ \beta_{j,t_1} \beta_{j,t_2} \rd_j^\star \D_{j,t_2}\inv \rd_j \rd_j^\star \D_{j,t_2}\inv \rd_k \rd_k^\star  \D_{j,t_1}\inv  \rd_j \rd_j^\star \D_{j,t_1}\inv \rd_j \rd_j^\star \D_{j,t_2}^{-2} \rd_j 
		\Big\}.
	\end{align*}
	% jetzt alle Matrizen durch von r_k unabhängige Version ersetzen, und beta_{j,t_2} ersetzen ; beobachten ,dass einige Terme wegfallen und ich mindestens zwei Paare r_k r_k^\star in jedem Summanden habe !! 
%	Let
%	\begin{align*}
%		\gamma_{k,j,t}(z) = \rd_k^\star \D_{k,j,t}\inv(z) \rd_k - n\inv \E [ \tr \T_n \D_{k,j,t}\inv(z) ].
%	\end{align*}

We use the  substitutions
	\begin{align}
		\D_{j,t_2}\inv = \D_{k,j,t_2}\inv  - \beta_{k,j,t_2} \D_{k,j,t_2}\inv \rd_k \rd_k^\star \D_{k,j,t_2}\inv \label{aa3}
	\end{align}
	and 
	\begin{align*}
		\beta_{j,t} = b_{j,t} - b_{j,t} \beta_{j,t} \gamma_{j,t}, ~ 
		\beta_{k,j,t_2} 	=  b_{k,j,t_2} - b_{k,j,t_2} \beta_{k,j,t_2} \gamma_{k,j,t_2},
	\end{align*}
	where
$
		\gamma_{k,j,t}(z) = \rd_k^\star \D_{k,j,t}\inv(z) \rd_k - n\inv \E [ \tr \T_n \D_{k,j,t}\inv(z) ].
$
This yields  the representation
	\begin{align*}
		Z_n^1 (z, t_1, t_2) = \sum\limits \sum\limits_{j=1}^{\nt} \sum\limits_{k = \nt + 1}^{\ntt} 
		(\E_j - \E_{j - 1} ) T_{j,k}.
	\end{align*}		
	
 Here, the first sum corresponds to the summation with respect to a finite number of  different terms  $T_{j,k}$, which 
	%is 
	%,for the sake of brevity, 
%	not further specified. (  
%	More precisely, the  terms $T_{j,k}$ 
are of the form 
	\begin{align*}
		&  \prod\limits_{l_1=1}^{q} \Big( \rd_j^\star \mathbf{A}_{l_1} \Big( \prod\limits_{l_2=1}^{q_{l_1}}  \rd_k \rd_k^\star \mathbf{B}_{l_1,l_2} \Big) \rd_j \Big) , \\
		& 	\lb \beta_{j,t_1} \gamma_{j,t_1} \rb^{X_1}
		\lb  \beta_{j,t_2} \gamma_{j,t_2} \rb^{X_2} 
		\prod\limits_{l_1=1}^{q} \Big( \rd_j^\star \mathbf{A}_{l_1} \Big( \prod\limits_{l_2=1}^{q_{l_1}}  \rd_k \rd_k^\star \mathbf{B}_{l_1,l_2} \Big) \rd_j \Big), \\
		& \lb  \beta_{k,j,t_2} \gamma_{k,j,t_2} \rb^X 
		\prod\limits_{l_1=1}^{q} \Big( \rd_j^\star \mathbf{A}_{l_1} \Big( \prod\limits_{l_2=1}^{q_{l_1}}  \rd_k \rd_k^\star \mathbf{B}_{l_1,l_2} \Big) \rd_j \Big),  \\
		& \lb  \beta_{k,j,t_2} \gamma_{k,j,t_2} \rb^X 
		\lb  \beta_{j,t_1} \gamma_{j,t_1} \rb^{X_1} 
		\lb  \beta_{j,t_2} \gamma_{j,t_2} \rb^{X_2}
		\prod\limits_{l_1=1}^{q} \Big( \rd_j^\star \mathbf{A}_{l_1} \Big( \prod\limits_{l_2=1}^{q_{l_1}}  \rd_k \rd_k^\star \mathbf{B}_{l_1,l_2} \Big) \rd_j \Big).
	\end{align*}
	Here, $q\in\N$, $q_{l_1} \in\N_0$, $l_1 \in\{1, \ldots, q\}$, there exists an index $l_1\in \{1, \ldots, q\}$ such that $q_{l_1} \geq 1$, and the matrices $\mathbf{A}_{l_1}$ and $\mathbf{B}_{l_1,l_2}$ are products of the matrices $\D_{j,t_1}\inv, \D_{k,j,t_2}\inv$ and $\T_n$ for $l_2 \in \{1, \ldots, q_{l_1} \}$, $l_1 \in \{ 1, \ldots, q \}$ and of the deterministic scalars $b_{j,t_1}, b_{j,t_2}, b_{k,j,t_2}$. We assume that $X\in\N$ and that one of the exponents  $X_1\in\N_0$ and $X_2\in\N_0$ is positive, that is, $X_1 + X_2 \geq 1$.  
	Since, again by Lemma 2.1 in \cite{li2003},
	\begin{align*}
	& \E | Z_n^1 (z, t_1, t_2) |^4 = 
	\E \Big| \sum \sum\limits_{j=1}^{\nt} \sum\limits_{k=\nt + 1}^{\ntt} 
	(\E_j - \E_{j - 1}) T_{j,k}\Big|^4
	\lesssim n \sum \sum\limits_{j=1}^{\nt}  \E  \Big|  \sum\limits_{k=\nt + 1}^{\ntt} 
	(\E_j - \E_{j - 1}) T_{j,k}\Big|^4 ,
	\end{align*}
	in order to prove \eqref{tight_t1}, it suffices to show that for $j\in\{1,\ldots,\nt\}$ and $k\in\{\nt + 1, \ldots, \ntt\}$
	\begin{align*}
		\E \left| 
		(\E_j - \E_{j - 1}) T_{j,k}\right|^4 \lesssim n^{-6}.
	\end{align*} 

In order to derive this estimate, we note that we can ignore the deterministic and bounded terms $b_{j,t_1}, b_{j,t_2}, b_{k,j,t_2}$ and denote by $\mathbf{A}_l,{l\in\N},$ a $p\times p$ (random) matrix which is a product of $\D_{j,t_1}\inv, \D_{k,j,t_2}\inv$ and $\T_n$. 
%The concrete definition of $\A_l$ may change from line to line. 
For the sake of brevity, we only consider terms of the type
	\begin{align}
		R_1 = & \E | ( \E_j - \E_{j - 1})  \beta_{j,t_1} \rd_j^\star \A_1 \rd_j 			 	
	 	 \rd_j^\star  \A_2 \rd_k \rd_k^\star \A_3 \rd_j \gamma_{j,t_1} |^4,
		\label{t0} \\ 
		R_2 = & \E| (\E_j - \E_{j - 1} ) \rd_j^\star \A_1 \rd_k \rd_k^\star \A_2 \rd_k \rd_k^\star \A_3 \rd_j |^4, \label{t1} \\
		R_3 = & \E| (\E_j - \E_{j - 1} ) \beta_{j,t_2} \rd_j^\star \A_1 \rd_k \rd_k^\star \A_2 \rd_j \gamma_{j,t_2} |^4, \label{t2} \\
		R_ 4 = & \E | (\E_j - \E_{j - 1} ) \beta_{k,j,t_2} \rd_j^\star \A_1 \rd_k \rd_k^\star \A_2
		\rd_k \rd_k^\star  \A_3 \rd_j \gamma_{k,j,t_2} |^4, \label{t3} \\
		R_5 = & \E | (\E_j - \E_{j - 1} ) \beta_{k,j,t_2} \beta_{j,t_2} \rd_j^\star \A_1 \rd_k \rd_k^\star  \A_2
		\rd_k \rd_k^\star \A_3 \rd_j \gamma_{k,j,t_2} \gamma_{j,t_2} |^4. \label{t4} 
	\end{align}
	For further investigations, we observe that 
	\begin{align} \label{aa20}
		(\E_j - \E_{j - 1} ) \tr  \Big( \prod_{l=1}^{q_1} \rd_k \rd_k^\star \A_l \Big) 
		= 0,
	\end{align}
	since $\rd_k \rd_k^\star \A_l $, $l\in\{1,\ldots, q_1\}$ does not depend on $\rd_j$. 
	In order to estimate the term in \eqref{t0}, we note that due to independence
	\begin{align*}
	 	\E | ( \E_j - \E_{j - 1}) \beta_{j,t_1} \rd_j^\star \A_1 \rd_j 		 	
	 	 \rd_j^\star \A_2 \lb \rd_k \rd_k^\star \A_3 - n\inv \T_n \A_3 \rb \rd_j \gamma_{j,t_1}  	  |^4 =0,	 
	\end{align*}
	so that we obtain, using similar arguments as in the derivation of \eqref{gamma}, in particular the bound in \eqref{gamma_r}, 
	\begin{align*}
		R_1  
	 	 \lesssim &   \E | ( \E_j - \E_{j - 1}) \beta_{j,t_1}   \rd_j^\star \A_1 \rd_j 
	 	 \rd_j^\star \A_2 \lb   \rd_k \rd_k^\star \A_3 - n\inv  \T_n \A_3 \rb \rd_j \gamma_{j,t_1}|^4 \\
	 	 & +  n^{-4}  \E | ( \E_j - \E_{j - 1}) \beta_{j,t_1}   \rd_j^\star \A_1 \rd_j 
	 	 \rd_j^\star \A_2 \T_n \A_3 \rd_j \gamma_{j,t_1}|^4 \\ 
	 	 = &    n^{-4}  \E | ( \E_j - \E_{j - 1}) \beta_{j,t_1}   \rd_j^\star \A_1 \rd_j 
	 	 \rd_j^\star \A_2 \T_n \A_3 \rd_j \gamma_{j,t_1}|^4  
	 	 \lesssim   n^{-6}. 
	\end{align*}
	For \eqref{t1}, we have using Lemma \ref{h1a} and \eqref{aa20}
	\begin{align*}
		R_2 = & 
		\E\Big| (\E_j - \E_{j - 1} ) \Big[  \rd_j^\star \A_1 \rd_k \rd_k^\star \A_2 \rd_k \rd_k^\star \A_3 \rd_j 
		- n\inv \tr \T_n \A_1 \rd_k \rd_k^\star \A_2 \rd_k \rd_k^\star \A_3 \Big] \Big|^4 \\
		\lesssim &  \E\Big| (\E_j - \E_{j - 1} ) \Big[  \rd_j^\star \A_1 \rd_k \lb  \rd_k^\star \A_2 \rd_k - n\inv \tr \T_n \A_2 \rb  \rd_k^\star \A_3 \rd_j 
		\\ & - n\inv \tr \T_n \A_1 \rd_k \lb \rd_k^\star \A_2 \rd_k - n\inv \tr \T_n \A_2 \rb \rd_k^\star \A_3 \Big] \Big|^4 \\
		& +  \E\Big| (\E_j - \E_{j - 1} ) \Big[  \rd_j^\star \A_1 \rd_k \lb   n\inv \tr \T_n \A_2 \rb  \rd_k^\star \A_3 \rd_j  - n\inv \tr \T_n \A_1 \rd_k \lb  n\inv \tr \T_n \A_2 \rb \rd_k^\star \A_3 \Big] \Big|^4 \\
		= &  \E\Big| (\E_j - \E_{j - 1} ) \Big[  \rd_j^\star \A_1 \rd_k \lb  \rd_k^\star \A_2 \rd_k - n\inv \tr \T_n \A_2 \rb  \rd_k^\star \A_3 \rd_j 
		\\ & - n\inv \tr \T_n \A_1 \rd_k \lb \rd_k^\star \A_2 \rd_k - n\inv \tr \T_n \A_2 \rb \rd_k^\star \A_3 \Big] \Big|^4 \\
		& + n^{-4} \E\Big| (\E_j - \E_{j - 1} ) \Big[  \rd_j^\star \A_1 \T_n \lb  n\inv  \tr \T_n \A_2 \rb   \A_3 \rd_j  - n\inv \tr \T_n \A_1 \T_n \lb  n\inv \tr \T_n \A_2 \rb  \A_3 \Big] \Big|^4 \\
		\lesssim & n^{-4} \E \lb \tr 
		\lb \A_1 \rd_k \lb  \rd_k^\star \A_2 \rd_k - n\inv \tr \T_n \A_2 \rb  \rd_k^\star \A_3 \rb 
		\lb \A_1 \rd_k \lb  \rd_k^\star \A_2 \rd_k - n\inv \tr \T_n \A_2 \rb  \rd_k^\star \A_3 \rb^\star 
		\rb^2  +   n^{-6} \\
		= &  n^{-4} \E \lb \tr 
		\lb \A_1 \rd_k \lb  \rd_k^\star \A_2 \rd_k - n\inv \tr \T_n \A_2 \rb  \rd_k^\star \A_3 \rb 
		\lb \A_3^\star \rd_k \overline{\lb  \rd_k^\star \A_2 \rd_k - n\inv \tr \T_n \A_2 \rb } \rd_k^\star \A_1^\star \rb 
		\rb^2   +  n^{-6}  \\
		= &n^{-4} \E |
		\lb \rd_k^\star \A_2 \rd_k - n\inv \tr \T_n \A_2  \rb^2
		 \rd_k^\star \A_3  
		\A_3^\star \rd_k  \rd_k^\star \A_1^\star \A_1 \rd_k
		|^2 +   n^{-6}  \\
		\lesssim &  n^{-6} .
	\end{align*}
	% 8. Momente
	Next, we have for the term $R_4$ defined in \eqref{t3} by similar arguments as in the derivation of \eqref{gamma}
	\begin{align*}
		R_4
		=  & \E | (\E_j - \E_{j - 1} ) \beta_{k,j,t_2} \{\rd_j^\star \A_1 \rd_k \rd_k^\star \A_2
		\rd_k \rd_k^\star \A_3  \rd_j 
		- n\inv \tr \T_n \A_1 \rd_k \rd_k^\star \A_2
		\rd_k \rd_k^\star \A_3
		\}\gamma_{k,j,t_2} |^4 \\
		\lesssim &  n^{-4} \E \Big[ \lb \tr \lb \A_1 \rd_k \rd_k^\star \A_2
		\rd_k \rd_k^\star \A_3 \rb \lb \A_1 \rd_k \rd_k^\star \A_2
		\rd_k \rd_k^\star \A_3 \rb^\star \rb^2  
		| \beta_{k,j,t_2} \gamma_{k,j,t_2} |^4 \Big] \\ 
		\lesssim & n^{-6},
	\end{align*}
	where we used the bound in Lemma \ref{h1a} and the fact that $\rd_j$ is independent of $\gamma_{k,j,t_2}$ and $\beta_{k,j,t_2}. $ \\
	% 8. Momente 
	Concerning the term $R_3$ in \eqref{t2}, we first decompose using \eqref{aa3}
	\begin{align*}
		\gamma_{j,t_2}(z) = &  \gamma_{j,k,t_2}(z) 
		- \lb \beta_{k,j,t_2}(z)  \rd_j^\star \D_{k,j,t_2}\inv(z) \rd_k \rd_k^\star \D_{k,j,t_2}\inv(z) \rd_j 
		- n\inv \E \left[ \beta_{k,j,t_2}(z) \tr  \T_n \D_{k,j,t_2}\inv(z) \rd_k \rd_k^\star \D_{k,j,t_2}\inv(z) \rb \right] \nonumber \\
		= & \gamma_{j,k,t_2}(z) 
		-  \beta_{k,j,t_2}(z) \lb  \rd_j^\star \D_{k,j,t_2}\inv(z) \rd_k \rd_k^\star \D_{k,j,t_2}\inv(z) \rd_j  - n\inv \tr  \T_n \D_{k,j,t_2}\inv(z) \rd_k \rd_k^\star \D_{k,j,t_2}\inv(z) \rb \\
		& - n\inv \lb\beta_{k,j,t_2}(z)  \rd_k^\star \D_{k,j,t_2}\inv(z) \T_n \D_{k,j,t_2}\inv(z) \rd_k
		-\E \left[ \beta_{k,j,t_2}(z) \tr  \T_n \D_{k,j,t_2}\inv(z) \rd_k \rd_k^\star \D_{k,j,t_2}\inv(z)  \right] \rb  \\
		= & \gamma_{j,k,t_2}(z) 
		-  \beta_{k,j,t_2}(z) \lb  \rd_j^\star \D_{k,j,t_2}\inv(z) \rd_k \rd_k^\star \D_{k,j,t_2}\inv(z) \rd_j  - n\inv \tr  \T_n \D_{k,j,t_2}\inv(z) \rd_k \rd_k^\star \D_{k,j,t_2}\inv(z) \rb \\
		& - n\inv b_{k,j,t_2}(z) \lb  \rd_k^\star \D_{k,j,t_2}\inv(z) \T_n \D_{k,j,t_2}\inv(z) \rd_k
		-n\inv \E \left[ \tr  \T_n \D_{k,j,t_2}\inv(z) \T_n \D_{k,j,t_2}\inv(z)  \right] \rb  \\
		& +  n\inv b_{k,j,t_2}(z) \Big( \beta_{k,j,t_2}(z)  \rd_k^\star \D_{k,j,t_2}\inv(z) \T_n \D_{k,j,t_2}\inv(z) \rd_k \gamma_{k,j,t_2}(z) \\
		& -\E \left[ \beta_{k,j,t_2}(z) \tr  \T_n \D_{k,j,t_2}\inv(z) \rd_k \rd_k^\star \D_{k,j,t_2}\inv(z)  \gamma_{k,j,t_2}(z) \right] \Big) \\
			= & \gamma_{j,k,t_2}(z) 
		-  \beta_{k,j,t_2}(z) \lb  \rd_j^\star \D_{k,j,t_2}\inv(z) \rd_k \rd_k^\star \D_{k,j,t_2}\inv(z) \rd_j  - n\inv \tr  \T_n \D_{k,j,t_2}\inv(z) \rd_k \rd_k^\star \D_{k,j,t_2}\inv(z) \rb \\
		& - n\inv b_{k,j,t_2}(z) \tilde{\gamma}_{k,j,t_2}(z) 
		 +  n\inv b_{k,j,t_2}(z) \Big( \beta_{k,j,t_2}(z)  \rd_k^\star \D_{k,j,t_2}\inv(z) \T_n \D_{k,j,t_2}\inv(z) \rd_k \gamma_{k,j,t_2}(z) \\
		& -\E \left[ \beta_{k,j,t_2}(z)  \rd_k^\star \D_{k,j,t_2}\inv(z)  \T_n \D_{k,j,t_2}\inv(z) \rd_k \gamma_{k,j,t_2}(z) \right] \Big)\\
		 = & \gamma_{j,k,t_2}(z) 
		-  b_{k,j,t_2}(z) \lb  \rd_j^\star \D_{k,j,t_2}\inv(z) \rd_k \rd_k^\star \D_{k,j,t_2}\inv(z) \rd_j  - n\inv \tr  \T_n \D_{k,j,t_2}\inv(z) \rd_k \rd_k^\star \D_{k,j,t_2}\inv(z) \rb \\
		& + b_{k,j,t_2}(z) \beta_{k,j,t_2}(z) \lb  \rd_j^\star \D_{k,j,t_2}\inv(z) \rd_k \rd_k^\star \D_{k,j,t_2}\inv(z) \rd_j  - n\inv \tr  \T_n \D_{k,j,t_2}\inv(z) \rd_k \rd_k^\star \D_{k,j,t_2}\inv(z) \rb \gamma_{k,j,t_2}(z) \\
		& - n\inv b_{k,j,t_2}(z) \tilde{\gamma}_{k,j,t_2}(z) 
		 +  n\inv b_{k,j,t_2}(z) \Big( \beta_{k,j,t_2}(z)  \rd_k^\star \D_{k,j,t_2}\inv(z) \T_n \D_{k,j,t_2}\inv(z) \rd_k \gamma_{k,j,t_2}(z) \\
		& -\E \left[ \beta_{k,j,t_2}(z)  \rd_k^\star \D_{k,j,t_2}\inv(z)  \T_n \D_{k,j,t_2}\inv(z) \rd_k \gamma_{k,j,t_2}(z) \right] \Big)
		,
	\end{align*}
	where
	\begin{align*}
	\tilde{\gamma}_{k,j,t_2} (z) = 
	\rd_k^\star \D_{k,j,t_2}\inv(z) \T_n \D_{k,j,t_2}\inv(z) \rd_k
		-n\inv \E \left[ \tr  \T_n \D_{k,j,t_2}\inv(z) \T_n \D_{k,j,t_2}\inv(z) \right] .
	\end{align*}
	Thus, we conclude for \eqref{t2}, using the notations $\mathbf{A}_3 = \D_{k,j,t_2}\inv$ and $\mathbf{A}_4 = \D_{k,j,t_2}\inv \T_n \D_{k,j,t_2}\inv $ and the fact that $b_{k,j,t_2}$ is deterministic and bounded,
	\begin{align}
		R_3 
		\lesssim & 
		 \E| (\E_j - \E_{j - 1} ) \beta_{j,t_2} \rd_j^\star \A_1 \rd_k \rd_k^\star \A_2 \rd_j \gamma_{j,k,t_2} |^4 \nonumber\\
		& +  \E| (\E_j - \E_{j - 1} ) \beta_{j,t_2} b_{k,j,t_2}
		 \rd_j^\star \A_1 \rd_k \rd_k^\star \A_2 \rd_j
		 \lb  \rd_j^\star \A_3 \rd_k \rd_k^\star \A_3 \rd_j  - n\inv \tr  \T_n  \A_3 \rd_k \rd_k^\star \A_3  \rb |^4 \nonumber\\
		& +  \E| (\E_j - \E_{j - 1} ) \beta_{j,t_2}
		\beta_{k,j,t_2} b_{k,j,t_2} \rd_j^\star \A_1 \rd_k \rd_k^\star \A_2 \rd_j
		\lb  \rd_j^\star \A_3 \rd_k \rd_k^\star \A_3 \rd_j  - n\inv \tr  \T_n  \A_3 \rd_k \rd_k^\star \A_3  \rb \gamma_{k,j,t_2} |^4 \nonumber\\
		& +  n^{-4} 	\E| (\E_j - \E_{j - 1} ) \beta_{j,t_2}
		b_{k,j,t_2} \rd_j^\star \A_1 \rd_k \rd_k^\star \A_2 \rd_j	
		\tilde{\gamma}_{k,j,t_2}  |^4 \nonumber\\
		& +  n^{-4} 	\E| (\E_j - \E_{j - 1} ) \beta_{j,t_2}
		\beta_{k,j,t_2} b_{k,j,t_2} \rd_j^\star \A_1 \rd_k \rd_k^\star \A_2 \rd_j
  \rd_k^\star \A_4 \rd_k \gamma_{k,j,t_2}  |^4 \nonumber \\
  		 & + n^{-4} \E| (\E_j - \E_{j - 1} ) \beta_{j,t_2}
		 \rd_j^\star \A_1 \rd_k \rd_k^\star \A_2 \rd_j
		|^4 \E | \beta_{k,j,t_2} b_{k,j,t_2}  \rd_k^\star \A_4\rd_k  \gamma_{k,j,t_2} |^4
		\nonumber \\
		\lesssim & R_{31} + R_{32} + R_{33} + n^{-6},  \nonumber 
		\end{align}
		where
		\begin{align}		
		R_{31} = &  \E| (\E_j - \E_{j - 1} ) \beta_{j,t_2} \rd_j^\star \A_1 \rd_k \rd_k^\star \A_2 \rd_j \gamma_{j,k,t_2} |^4, \label{t5}\\
		R_{32} = &   \E| (\E_j - \E_{j - 1} ) \beta_{j,t_2}
		 \rd_j^\star \A_1 \rd_k \rd_k^\star \A_2 \rd_j
		\lb  \rd_j^\star \A_3 \rd_k \rd_k^\star \A_3 \rd_j  - n\inv \tr  \T_n  \A_3 \rd_k \rd_k^\star \A_3  \rb |^4, %\label{t51}
		\nonumber \\
		R_{33} = &  \E| (\E_j - \E_{j - 1} ) \beta_{j,t_2}
		\beta_{k,j,t_2} \rd_j^\star \A_1 \rd_k \rd_k^\star \A_2 \rd_j
		\lb  \rd_j^\star \A_3  \rd_k \rd_k^\star \A_3 \rd_j  - n\inv \tr  \T_n \A_3 \rd_k \rd_k^\star \A_3 \rb \gamma_{k,j,t_2} |^4 \label{t6}
	\end{align}
	and we used an analogue of the estimate \eqref{gamma_r} for the terms $\gamma_{k,j,t_2}$  and $\tilde{\gamma}_{k,j,t_2}$ in the last step. 
	The term $R_{31}$ in \eqref{t5} can be bounded using the bounds in \eqref{d_inv_norm}, \eqref{beta} and \eqref{ind_bound} as follows:
	\begin{align*}
	R_{31} \lesssim R_{311} + R_{312} + o \lb n^{-l} \rb,
	\end{align*}
	where
	\begin{align*}
		R_{311} =  &  \E | \rd_j^\star \A_1 \rd_k \rd_k^\star \A_2 \rd_j \gamma_{j,k,t_2} |^4, \\
		R_{312} = &   \E | \rd_j^\star \rd_j \rd_j^\star \A_1 \rd_k \rd_k^\star \A_2 \rd_j \gamma_{j,k,t_2} |^4.
	\end{align*}
	Since $R_{312}$ can be handled similarly to $R_{311}$, we only consider $R_{311}$ and obtain by Lemma \ref{h1a}
	\begin{align*}
R_{311} 
	\lesssim &  \E | \lb  \rd_k^\star \A_2 \rd_j \rd_j^\star \A_1 \rd_k  
	- n\inv \tr \T_n \A_2 \rd_j \rd_j^\star \A_1\rb 
	\gamma_{j,k,t_2} |^4 
	 + n^{-4}  \E | \rd_j^\star \A_1  \T_n \A_2 \rd_j 
	\gamma_{j,k,t_2} |^4 \\
	\lesssim & n^{-6} 
	+  n^{-4} \E \left[ |\gamma_{j,k,t_2}|^4 \lb \tr \lb \A_2 \rd_j \rd_j^\star \A_1 \rb 
	\lb \A_2 \rd_j \rd_j^\star \A_1 \rb^\star  \rb^2 \right] \\
	\leq & n^{-6}.
	\end{align*}
	Note that the term $R_{33}$ defined in \eqref{t6} can be bounded similarly. 
	% 8. Momente 
	Similarly to $R_{31}$ given in \eqref{t5} we bound $|\beta_{j,t_2}|$ and get
	\begin{align*}
		R_{32} \lesssim R_{321} + R_{322}  + o \lb n^{-l} \rb,
	\end{align*}
	where 
	\begin{align*}
		R_{321} = & \E| 
		 \rd_j^\star \A_1 \rd_k \rd_k^\star \A_2 \rd_j
		\lb  \rd_j^\star \A_3 \rd_k \rd_k^\star \A_3 \rd_j  - n\inv \tr  \T_n  \A_3 \rd_k \rd_k^\star \A_3  \rb |^4, \\
		R_{322} = &  \E| \rd_j^\star \rd_j
		 \rd_j^\star \A_1 \rd_k \rd_k^\star \A_2 \rd_j
		\lb  \rd_j^\star \A_3 \rd_k \rd_k^\star \A_3 \rd_j  - n\inv \tr  \T_n  \A_3 \rd_k \rd_k^\star \A_3  \rb |^4 . \\
	\end{align*}
	For the sake of brevity, we shall limit ourselves to investigating the summand $R_{321}$. 

	\begin{align*}
		R_{321} 
		\lesssim & \E| 
		 \lb \rd_k^\star \A_2 \rd_j \rd_j^\star \A_1 \rd_k - n\inv \tr \T_n \A_2 \rd_j \rd_j^\star \A_1 \rb 
		\lb  \rd_j^\star \A_3 \rd_k \rd_k^\star \A_3 \rd_j  - n\inv \tr  \T_n  \A_3 \rd_k \rd_k^\star \A_3  \rb |^4 \\
		& + n^{-4}  \E| 
		  \rd_j^\star \A_1 \T_n \A_2 \rd_j
		\lb  \rd_j^\star \A_3 \rd_k \rd_k^\star \A_3 \rd_j  - n\inv \tr  \T_n  \A_3 \rd_k \rd_k^\star \A_3  \rb |^4 \\
		\lesssim & n^{-6} 
		+  \lb \E| 
		  \rd_k^\star \A_2 \rd_j \rd_j^\star \A_1 \rd_k - n\inv \tr \T_n \A_2 \rd_j \rd_j^\star \A_1 |^8 
		\E | \rd_j^\star \A_3 \rd_k \rd_k^\star \A_3 \rd_j  - n\inv \tr  \T_n  \A_3 \rd_k \rd_k^\star \A_3   |^8\rb\sq \\
		\lesssim & n^{-6}. 
	\end{align*}

	Finally, invoking Lemma \ref{h1a} and \eqref{gamma_r}, we can show for the term $R_5$ defined in \eqref{t4} that
		\begin{align*}
		 	R_5 \lesssim R_{51} + R_{52}, 
		\end{align*}
		where
		\begin{align*}
		R_{51}
		= &    \E | \beta_{k,j,t_2} \beta_{j,t_2} \lb \rd_j^\star \A_1 \rd_k \rd_k^\star \A_2
		\rd_k \rd_k^\star \A_3 \rd_j  
		-n\inv \tr \T_n \A_1 \rd_k \rd_k^\star \A_2
		\rd_k \rd_k^\star \A_3 \rb 
		\gamma_{k,j,t_2} \gamma_{j,t_2} |^4 \\
		\leq & \lb \E | \rd_j^\star \A_1 \rd_k \rd_k^\star \A_2
		\rd_k \rd_k^\star \A_3 \rd_j  
		-n\inv \tr \T_n \A_1 \rd_k \rd_k^\star \A_2
		\rd_k \rd_k^\star \A_3 |^8
		\E |  \beta_{k,j,t_2} \beta_{j,t_2} 
		\gamma_{k,j,t_2} \gamma_{j,t_2} |^8 \rb^{\frac{1}{2}}
		 \lesssim  n^{-6}, \\
			R_{52} = &  
			n^{-4} \E | \beta_{k,j,t_2} \beta_{j,t_2} 
		\rd_k^\star \A_2
		\rd_k \rd_k^\star \A_3 \T_n \A_1 \rd_k 
		\gamma_{k,j,t_2} \gamma_{j,t_2} |^4 
		 \lesssim  n^{-6}.
	\end{align*}

		Thus, the moment inequalities \eqref{tight_z}, \eqref{tight_t1} and \eqref{tight_t2} for $M_n^1$ hold true.	
		\end{proof}

	For the proof of  Lemma \ref{a73}, we need the following identity, which can be proved similarly to (5.2) in \cite{baisilverstein1998}.   
	
	\begin{lemma} \label{a52}
	\begin{align*}
	& y_n \int \frac{dH_n(\lambda)}{1 + \lambda \frac{\ntn}{n} \E  [\tilde{\su}_{n,t}(z)] } + z y_n \E [ \tilde{s}_{n,t}(z) ] \\
	= & \frac{1}{n} \sum\limits_{j=1}^{\ntn} \E \Big\{ 
	\beta_{j,t}(z) \Big[  \rd_j^\star \D_{j,t}\inv(z) \Big  ( \frac{\ntn}{n} \E [ \sut_{n,t}(z) ] \T_n -  \mathbf{I} \Big  )\inv \rd_j \\
	& - \frac{1}{n} \tr \Big  ( \frac{\ntn}{n} \E [ \sut_{n,t}(z) ] \T_n -  \mathbf{I} \Big  )\inv \T_n \E [ \D_t\inv(z) ]  
	\Big]
	\Big\}
	\end{align*}
	\end{lemma}

	\begin{lemma} \label{a80}
		For any bounded subset $S \subset \mathbb{C}^+$, we have
		\begin{align*}
			\inf\limits_{\substack{z \in S, t\in [t_0,1]}} | \sut_t(z) | > 0.
		\end{align*}
	\end{lemma}
	\begin{proof}[Proof of Lemma \ref{a80}]
		Let us assume that the assertion does not hold. In this case, there exists sequences $(z_n)_{n\in\N}$ in $S$ and $(t_n)_{n\in\N}$ in $[t_0,1]$ with the property
		\begin{align*}
			\lim\limits_{n\to\infty} \sut_{t_n}(z_n) = 0.
		\end{align*}
		 By choosing approriate subsequences, we assume without loss of generality that $(z_n)_{n\in\N}$ converges to limit in the closure of $S$ and $(t_n)_{n\in\N}$ converges to a limit in $[t_0,1]$.
		%Lemma \ref{a47} says that
		From \eqref{repl_a47}, we conclude
% 	\begin{align*}
% 		y \int \frac{\lambda \sut_{t_n}(z_n) }{1+ \lambda t_n \tilde{\su}_{t_n}(z_n)} dH(\lambda) = 1+ z_n \sut_{t_n}(z_n)
% 	\end{align*}
% 	and hence, 
	\begin{align*}
		\lim\limits_{n\to\infty} 
		y \int \frac{\lambda \sut_{t_n}(z_n) }{1+ \lambda t_n \tilde{\su}_{t_n}(z_n)} dH(\lambda)
		= 1.
	\end{align*}
	
		But, using the fact that $H$ is compactly supported, we see that the expression above tends to 0. Thus, we get a contradiction. 
	\end{proof}

	\begin{lemma} \label{lem_cross_terms}
	In the real case, it holds for $i\neq l$ ($i,l \in\{1, \ldots, n\} \setminus \{j\}$) 
	\begin{align}
	& \sup\limits_{\substack{z\in\mathcal{C}_n, \\ t\in [t_0,1] }}
	\Big| \E \Big[ \tr \lb t \sut_t(z) \T_n + \mathbf{I} \rb^{-2} \T_n 
		  \lb \rd_i \rd_i^\star - n\inv \T_n\rb 
		  \beta_{l,i,j,t}(z) \D_{l,i,j,t}\inv(z) \rd_l \rd_l^\star \D_{l,i,j,t}\inv(z) 
	 \lb t \sut_{t}(z) \T_n + \mathbf{I} \rb\inv \T_n \nonumber \\
	 & \times \beta_{i,l,j,t}(z) \D_{i,l,j,t}\inv(z) \rd_i \rd_i^\star \D_{i,l,j,t}\inv(z)
		 \lb \rd_{l} \rd_{l}^\star - n\inv \T_n \rb		
		  \Big] \Big| \label{mean_cross_term}\\
		  = & o \lb n \inv \rb. \nonumber
		  \end{align}
	
	\end{lemma}
	\begin{proof}[Proof of Lemma \ref{lem_cross_terms}]
	Denoting
	\begin{align*}
		\hat{\gamma}_{i,l,j,t}(z) 
		= & \rd_i^\star \D_{i,l,j,t}\inv(z) \rd_i - n\inv \tr \T_n \D_{i,l,j,t}\inv(z), \\
		\overline{\beta}_{i,j,l,t}^2(z) = & 
		\frac{1}{1 + n\inv \tr \T_n \D_{i,l,j,t}\inv(z) },
		\end{align*}
		we use the representation 
			\begin{align} \label{beta_quer}
		\beta_{j,t}(z) = \overline{\beta}_{j,t}(z) - \overline{\beta}_{j,t}^2(z) \hat{\gamma}_{j,t}(z) 
	+ \overline{\beta}_{j,t}^2(z) \beta_{j,t}(z) \hat{\gamma}_{j,t}^2(z),
	\end{align} 
in order to replace $\beta_{l,i,j,t}(z)$ and $\beta_{i,l,j,t}(z)$. Note that $\E [ || \D_{l,i,j,t}\inv(z) ||] \leq K$ and $||\lb t \sut_t(z) \T_n + \mathbf{I} \rb^{-1}||\leq K$
		which follows from Lemma \ref{mom_Dinv} and Proposition \ref{a61}. 
		% ausführlich: \eqref{a62} in anderem Dokument.
		By applying the triangle inequality, this gives us several summands for the mean in \eqref{mean_cross_term}. More precisely, we can write
		\begin{align*}
			& \Big| \E \Big[ \tr \lb t \sut_t(z) \T_n + \mathbf{I} \rb^{-2} \T_n 
		  \lb \rd_i \rd_i^\star - n\inv \T_n\rb 
		  \beta_{l,i,j,t}(z) \D_{l,i,j,t}\inv(z) \rd_l \rd_l^\star \D_{l,i,j,t}\inv(z) 
	 \lb t \sut_{t}(z) \T_n + \mathbf{I} \rb\inv \T_n \nonumber \\
	 & \times \beta_{i,l,j,t}(z) \D_{i,l,j,t}\inv(z) \rd_i \rd_i^\star \D_{i,l,j,t}\inv(z)
		 \lb \rd_{l} \rd_{l}^\star - n\inv \T_n \rb		
		  \Big] \Big| \\
		  \leq & \sum\limits_{\zeta_1, \zeta_2} | \E [ T(\zeta_1, \zeta_2) ]|,
		\end{align*}
		where $T(\zeta_1, \zeta_2)$ has the following form
		\begin{align*}
			& \tr \lb t \sut_t(z) \T_n + \mathbf{I} \rb^{-2} \T_n 
		  \lb \rd_i \rd_i^\star - n\inv \T_n\rb 
		  \zeta_1 \D_{l,i,j,t}\inv(z) \rd_l \rd_l^\star \D_{l,i,j,t}\inv(z) 
	 \lb t \sut_{t}(z) \T_n + \mathbf{I} \rb\inv \T_n \nonumber \\
	 & \times
	  \zeta_2 \D_{i,l,j,t}\inv(z) \rd_i \rd_i^\star \D_{i,l,j,t}\inv(z)
		 \lb \rd_{l} \rd_{l}^\star - n\inv \T_n \rb		, 
		\end{align*}		 
		and
		 $$\zeta_1 \in\{ \overline{\beta}_{l,i,j,t}(z), - \overline{\beta}_{l,i,j,t}^2(z) \hat{\gamma}_{l,i,j,t}(z), \beta_{l,i,j,t}(z)\overline{\beta}_{l,i,j,t}^2(z) \hat{\gamma}_{l,i,j,t}^2(z) \}, $$ 
		$$\zeta_2 \in \{ \overline{\beta}_{i,l,j,t}(z), - \overline{\beta}_{i,l,j,t}(z)^2 \hat{\gamma}_{i,l,j,t}(z), \beta_{i,l,j,t}(z)\overline{\beta}_{i,l,j,t}^2(z) \hat{\gamma}_{i,l,j,t}^2(z) \}.$$
		The assertion now follows, if we show that for all $\zeta_1, \zeta_2$
		\begin{align} \label{mean_t_zeta}
			| \E [ T(\zeta_1, \zeta_2)] |
			= o \lb n\inv \rb. 
		\end{align}
		In the following, we restrict ourselves to three different cases noting that the remaining cases can be handled similarly.
		 \\
		To begin with, let $\zeta_1= \overline{\beta}_{l,i,j,t}(z)$ and $\zeta_2 = \overline{\beta}_{i,l,j,t}(z)$. In this case, we have
		\begin{align*}
			 & \left| \E [ T (\zeta_1, \zeta_2 ) ] \right|  \\
			\leq  & K \Big| \E  
			 \tr \lb t \sut_t(z) \T_n + \mathbf{I} \rb^{-2} \T_n 
		  \lb \rd_i \rd_i^\star - n\inv \T_n\rb 
		  \D_{l,i,j,t}\inv(z) \rd_l \rd_l^\star \D_{l,i,j,t}\inv(z) 
	 \lb t \sut_{t}(z) \T_n + \mathbf{I} \rb\inv \T_n \nonumber \\
	 & \times
	   \D_{i,l,j,t}\inv(z) \rd_i \rd_i^\star \D_{i,l,j,t}\inv(z)
		 \lb \rd_{l} \rd_{l}^\star - n\inv \T_n \rb				
			\Big| \\
			\leq & 
			K  \Big| \E  
			 \tr \lb t \sut_t(z) \T_n + \mathbf{I} \rb^{-2} \T_n 
		   \rd_i \rd_i^\star 
		  \D_{l,i,j,t}\inv(z) \rd_l \rd_l^\star \D_{l,i,j,t}\inv(z) 
	 \lb t \sut_{t}(z) \T_n + \mathbf{I} \rb\inv \T_n 
	   \D_{i,l,j,t}\inv(z) \rd_i \rd_i^\star \D_{i,l,j,t}\inv(z)
		  \rd_{l} \rd_{l}^\star 		
			\Big| \\
			& + Kn\inv  \Big| \E  
			 \tr \lb t \sut_t(z) \T_n + \mathbf{I} \rb^{-2} \T_n^2
		  \D_{l,i,j,t}\inv(z) \rd_l \rd_l^\star \D_{l,i,j,t}\inv(z) 
	 \lb t \sut_{t}(z) \T_n + \mathbf{I} \rb\inv \T_n 
	   \D_{i,l,j,t}\inv(z) \rd_i \rd_i^\star \D_{i,l,j,t}\inv(z)
		  \rd_{l} \rd_{l}^\star				
			\Big| \\
			& +K  n\inv  \Big| \E  
			 \tr \lb t \sut_t(z) \T_n + \mathbf{I} \rb^{-2} \T_n 
		  \rd_i \rd_i^\star 
		  \D_{l,i,j,t}\inv(z) \rd_l \rd_l^\star \D_{l,i,j,t}\inv(z) 
	 \lb t \sut_{t}(z) \T_n + \mathbf{I} \rb\inv \T_n 
	   \D_{i,l,j,t}\inv(z) \rd_i \rd_i^\star \D_{i,l,j,t}\inv(z)
		 \T_n			
			\Big| \\
			& + Kn^{-2} 
			 \Big| \E  
			 \tr \lb t \sut_t(z) \T_n + \mathbf{I} \rb^{-2} \T_n^2
		  \D_{l,i,j,t}\inv(z) \rd_l \rd_l^\star \D_{l,i,j,t}\inv(z) 
	 \lb t \sut_{t}(z) \T_n + \mathbf{I} \rb\inv \T_n 
	   \D_{i,l,j,t}\inv(z) \rd_i \rd_i^\star \D_{i,l,j,t}\inv(z)
		\T_n 			
			\Big| \\
			= & K ( T_1 + T_2 + T_3 ) 
			+ o \lb n\inv \rb ,
		\end{align*}
		where
		\begin{align*}
			T_1 = & \Big| \E  \left[ 
			   \rd_i^\star 
		  \D_{l,i,j,t}\inv(z) \rd_l \rd_l^\star \D_{l,i,j,t}\inv(z) 
	 \lb t \sut_{t}(z) \T_n + \mathbf{I} \rb\inv \T_n 
	   \D_{i,l,j,t}\inv(z) \rd_i \rd_i^\star \D_{i,l,j,t}\inv(z)
		  \rd_{l} \rd_{l}^\star 	\lb t \sut_t(z) \T_n + \mathbf{I} \rb^{-2} \T_n 
		   \rd_i	\right] 
			\Big|, \\
			T_2 = & n\inv  \Big| \E  
			 \tr \lb t \sut_t(z) \T_n + \mathbf{I} \rb^{-2} \T_n^2
		  \D_{l,i,j,t}\inv(z) \rd_l \rd_l^\star \D_{l,i,j,t}\inv(z) 
	 \lb t \sut_{t}(z) \T_n + \mathbf{I} \rb\inv \T_n 
	   \D_{i,l,j,t}\inv(z) \rd_i \rd_i^\star \D_{i,l,j,t}\inv(z)
		  \rd_{l} \rd_{l}^\star				
			\Big| ,\\
			T_3 = & n\inv  \Big| \E  
			 \tr \lb t \sut_t(z) \T_n + \mathbf{I} \rb^{-2} \T_n 
		  \rd_i \rd_i^\star 
		  \D_{l,i,j,t}\inv(z) \rd_l \rd_l^\star \D_{l,i,j,t}\inv(z) 
	 \lb t \sut_{t}(z) \T_n + \mathbf{I} \rb\inv \T_n 
	   \D_{i,l,j,t}\inv(z) \rd_i \rd_i^\star \D_{i,l,j,t}\inv(z)
		 \T_n		
			\Big|.
		\end{align*}
		For the first summand, we obtain using (9.8.6) in \cite{bai2004} for the real case 
		% (similarly to \eqref{a5}) 
		\begin{align*}
			 T_1  \leq & 
			\Big| \E  \Big[ 
			   \big\{ \rd_i^\star 
		  \D_{l,i,j,t}\inv(z) \rd_l \rd_l^\star \D_{l,i,j,t}\inv(z) 
	 \lb t \sut_{t}(z) \T_n + \mathbf{I} \rb\inv \T_n 
	   \D_{i,l,j,t}\inv(z) \rd_i \\
	  &   - n\inv \tr \T_n \D_{l,i,j,t}\inv(z) \rd_l \rd_l^\star \D_{l,i,j,t}\inv(z) 
	 \lb t \sut_{t}(z) \T_n + \mathbf{I} \rb\inv \T_n 
	   \D_{i,l,j,t}\inv(z)
	   \big\} \\
	   & \times 
	   \big\{ \rd_i^\star \D_{i,l,j,t}\inv(z)
		  \rd_{l} \rd_{l}^\star 	\lb t \sut_t(z) \T_n + \mathbf{I} \rb^{-2} \T_n 
		   \rd_i	
		   - n\inv \tr \T_n \D_{i,l,j,t}\inv(z)
		  \rd_{l} \rd_{l}^\star 	\lb t \sut_t(z) \T_n + \mathbf{I} \rb^{-2} \T_n 
		   \big\} 
			\Big] \Big|
			+ o \lb n\inv \rb \\
			= & o\lb n\inv \rb. 
		\end{align*}
		With similar ideas, it can be shown that $T_2 = o\lb n\inv \rb$ and $T_3 = o\lb n \inv \rb$ and that \eqref{mean_t_zeta} holds true in the case $\zeta_1= \overline{\beta}_{l,i,j,t}(z)$ and $\zeta_2 = - \overline{\beta}_{i,l,j,t}^2(z) \hat{\gamma}_{i,l,j,t}(z).$

	Finally, we consider the case $\zeta_1= - \overline{\beta}_{l,i,j,t}^2(z) \hat{\gamma}_{l,i,j,t}(z)$ and $\zeta_2 = - \overline{\beta}_{i,l,j,t}^2(z) \hat{\gamma}_{i,l,j,t}(z).$ 
	Note that $\overline{\beta}_{i,l,j,t}(z) = \overline{\beta}_{l,i,j,t}(z)$. 
	We obtain \eqref{mean_t_zeta}, that is,
	\begin{align*}
		 &  | \E [ T(\zeta_1, \zeta_2) ] | \\
		  = & \Big| \E \Big[ \overline{\beta}_{i,j,l,t}^4(z) \tr \lb t \sut_t(z) \T_n + \mathbf{I} \rb^{-2} \T_n 
		  \lb \rd_i \rd_i^\star - n\inv \T_n\rb 
		  \D_{l,i,j,t}\inv(z) \rd_l \rd_l^\star \D_{l,i,j,t}\inv(z) 
	 \lb t \sut_{t}(z) \T_n + \mathbf{I} \rb\inv \T_n \hat{\gamma}_{l,i,j,t}(z) \nonumber \\
	 & \times  \D_{i,l,j,t}\inv(z) \rd_i \rd_i^\star \D_{i,l,j,t}\inv(z)
		 \lb \rd_{l} \rd_{l}^\star - n\inv \T_n \rb	\hat{\gamma}_{i,l,j,t}(z) 	
		  \Big] \Big| \\
		  \leq & \E \sq|E_1|^2 \E^{\frac{1}{4}}|E_2|^4 \E^{\frac{1}{4}}|E_3|^4= o \lb n\inv \rb,
	\end{align*}
	if
	\begin{align}
		\E \sq|E_1|^2 & \leq K n\inv , \label{est_e1}\\
		 \E^{\frac{1}{4}}|E_2|^4 & \leq K  \label{est_e2}, \\
		 \E^{\frac{1}{4}}|E_3|^4 & = o(1),
		 \label{est_e3}
	\end{align} 
	where
	\begin{align*}
	E_1 = & \overline{\beta}_{i,j,l,t}^4(z)  
		\hat{\gamma}_{i,l,j,t}(z) \hat{\gamma}_{l,i,j,t}(z)   ,\\ 
	E_2 = & \rd_l^\star \D_{l,i,j,t}\inv(z) 
	 \lb t \sut_{t}(z) \T_n + \mathbf{I} \rb\inv \T_n  \D_{i,l,j,t}\inv(z) \rd_i , \\
	 E_3 = & \tr \Big\{ \lb t \sut_t(z) \T_n + \mathbf{I} \rb^{-2} \T_n 
		  \lb \rd_i \rd_i^\star - n\inv \T_n\rb 
		 	\D_{l,i,j,t}\inv(z) \rd_l  \rd_i^\star \D_{i,l,j,t}\inv(z) 
		 	 \lb \rd_{l} \rd_{l}^\star - n\inv \T_n \rb \Big\}.
	\end{align*}
	We begin with a proof of \eqref{est_e1}. Note that $\hat{\gamma}_{l,i,j,t}(z)$ is independent of $\rd_i$ and $\overline{\beta}_{i,j,l,t}(z)$ is independent of $\rd_i$ and $\rd_j$. 
	% By using similar techniques as in the proof of \textcolor{red}{Theorem/Lemma \ref{no_hat}  } and 
	Using (9.9.6) in \cite{bai2004} twice, we obtain
	\begin{align*}
		 \E | E_1 |^2  
		  \leq  Kn^{-2},	  
	\end{align*}
	which proves \eqref{est_e1}. The estimate \eqref{est_e2} can be proven similarly to \cite{bai2004}, p. 290. 
	
	Finally, we will prove that \eqref{est_e3} holds true. We obtain 
	\begin{align*}
		E_3
		= & 
		\rd_i^\star \D_{i,l,j,t}\inv(z) 
		 	 \lb \rd_{l} \rd_{l}^\star - n\inv \T_n \rb 
		 	  \lb t \sut_t(z) \T + \mathbf{I} \rb^{-2} \T 
		  \lb \rd_i \rd_i^\star - n\inv \T\rb 
		 	\D_{l,i,j,t}\inv(z) \rd_l \\
		 	= & 
		E_{31}  + E_{32} + E_{33} + E_{34}, 
	\end{align*}
	where
	\begin{align*} 
		E_{31} = &
		 \rd_i^\star \D_{i,l,j,t}\inv(z) 
		 	 \rd_{l} \rd_{l}^\star 
		 	  \lb t \sut_t(z) \T + \mathbf{I} \rb^{-2} \T 
		   \rd_i \rd_i^\star 
		 	\D_{l,i,j,t}\inv(z) \rd_l \\
		 	 E_{32} = & - n\inv 
		 	\rd_i^\star \D_{i,l,j,t}\inv(z) 
		 	  \rd_{l} \rd_{l}^\star 
		 	  \lb t \sut_t(z) \T + \mathbf{I} \rb^{-2} \T^2 
		 	\D_{l,i,j,t}\inv(z) \rd_l , \\
		 	E_{33} = & - n\inv 
		 	\rd_i^\star \D_{i,l,j,t}\inv(z) 
		 	 \T_n 
		 	  \lb t \sut_t(z) \T + \mathbf{I} \rb^{-2} \T 
		   \rd_i \rd_i^\star 
		 	\D_{l,i,j,t}\inv(z) \rd_l, \\
		 		 E_{34} = & n^{-2} 
		 	\rd_i^\star \D_{i,l,j,t}\inv(z) 
		 	 \T_n 
		 	  \lb t \sut_t(z) \T + \mathbf{I} \rb^{-2} \T^2
		 	\D_{l,i,j,t}\inv(z) \rd_l.
	\end{align*}
	For $k\in\{2,3,4\}$, it holds
	\begin{align*}
		E | E_{3k} |^4 = o(1). 
	\end{align*}
	For the first summand, we conclude
	\begin{align*}
		E | E_{31} |^4
		\leq & K 
		\E\sq \Big| \rd_{l}^\star 
		 	  \lb t \sut_t(z) \T + \mathbf{I} \rb^{-2} \T 
		   \rd_i \rd_i^\star 
		 	\D_{l,i,j,t}\inv(z) \rd_l 
		 	 \Big|^8 \\
		 	 \leq & K \E\sq \Big| \rd_{l}^\star 
		 	  \lb t \sut_t(z) \T + \mathbf{I} \rb^{-2} \T 
		   \rd_i \rd_i^\star 
		 	\D_{l,i,j,t}\inv(z) \rd_l 
		 	- n\inv \tr \T \lb t \sut_t(z) \T + \mathbf{I} \rb^{-2} \T 
		   \rd_i \rd_i^\star 
		 	\D_{l,i,j,t}\inv(z)
		 	 \Big|^8 \\
		 	 & + K \E\sq \Big| 
		 	n\inv \tr \T \lb t \sut_t(z) \T + \mathbf{I} \rb^{-2} \T 
		   \rd_i \rd_i^\star 
		 	\D_{l,i,j,t}\inv(z)
		 	 \Big|^8 \\
		 	 \leq & K \E\sq \Big| \rd_{l}^\star 
		 	  \lb t \sut_t(z) \T + \mathbf{I} \rb^{-2} \T 
		   \rd_i \rd_i^\star 
		 	\D_{l,i,j,t}\inv(z) \rd_l 
		 	- n\inv \tr \T \lb t \sut_t(z) \T + \mathbf{I} \rb^{-2} \T 
		   \rd_i \rd_i^\star 
		 	\D_{l,i,j,t}\inv(z)
		 	 \Big|^8 \\
		 	 & + Kn^{-4} \\
		 	 \leq & Kn^{-\frac{1}{2}} + Kn^{-4} 
		 	 = o(1),
	\end{align*}
	which proves \eqref{est_e3}. 
	Hence, the proof of Lemma \ref{lem_cross_terms} is finished.
	\end{proof}

	\begin{lemma}\label{aux_im_z}
			It holds for sufficiently large $N \in\N$
			\begin{align*}
				\inf\limits_{n\geq N} \inf\limits_{\substack{ z\in\mathcal{C}_n, \\ t\in[t_0,1]}} 
				\lb \im (z) + \im (R_{n,t}(z)) \rb \geq 0. 
			\end{align*}
		\end{lemma}
		\begin{proof}[Proof of Lemma \ref{aux_im_z}]
		We start by investigating real and imaginary part of  $1/ \E [ \sut_{n,t}(z) ]$. 
		As a preparation for the latter, one can show similarly to Lemma \ref{a80} that $\re ( \sut_t(z) ) $ is uniformly bounded away from 0. Thus, due to Theorem \ref{thm_stieltjes}, we also have for some sufficiently large $N\in\N$
		\begin{align*}
			\inf\limits_{n \geq N} 
			\inf\limits_{\substack{z\in\mathcal{C}_n, \\
			t\in[t_0,1]}} 
			| \re \E [ \sut_{n,t}(z) ] | > 0
			\textnormal{ and }
			 \inf\limits_{n \geq N} 
			\inf\limits_{\substack{z\in\mathcal{C}_n, \\
			t\in[t_0,1]}} 
			|  \E [ \sut_{n,t}(z) ] | > 0.
		\end{align*}
		Using also $| \E [ \sut_{n,t}(z)] | \leq 1 / \im(z)$, this implies for the real part of the inverse for some $K_1 > 0$
		\begin{align*}
			\re \lb \E [ \sut_{n,t}(z) ] \rb \inv 
			= &  \frac{ \re \lb \E [ \sut_{n,t}(z) ] \rb }{| \E [ \sut_{n,t}(z) ] |^2 } 
			\geq  K_1 \im^2(z).
		\end{align*}
		For the imaginary part, we conclude for some $K_2>0$
			\begin{align*}
			\im \lb \E [ \sut_{n,t}(z) ] \rb\inv 
			= & \frac{- \im \lb \E [ \sut_{n,t}(z) ] \rb }{| \E [ \sut_{n,t}(z) ] |^2 }
			= \frac{1}{ | \E [ \sut_{n,t}(z) ] |^2 }  \im \Big( \int \frac{ - 1}{ \lambda - z } 
			d F^{\mathbf{\underline{B}}_{n,t}} (\lambda ) \Big) \\
			= & \frac{1}{ | \E [ \sut_{n,t}(z) ] |^2 }  \int \frac{ - \im (z) }{ |z - \lambda|^2 } 
			d F^{\mathbf{\underline{B}}_{n,t}} (\lambda )
			\geq K \E \int \frac{ - \im(z)}{| \lambda - z|^2 } 
			d F^{\mathbf{\underline{B}}_{n,t}} (\lambda ) \\
			\geq & - K_2 \im(z) .
		\end{align*}
		By definition of $R_{n,t}(z)$, we have for all $n\geq N$
		\begin{align*}
		 \im (R_{n,t}(z)) 
		= &
		y_{\ntn} \ntn\inv \sum\limits_{j=1}^{\ntn} 
		\im \lb \E [ \beta_{j,t}( z) d_{j,t} (z) ] \lb \E [ \tilde{\su}_{n,t}(z) ] \rb \inv \rb \\
		= & y_{\ntn} \ntn\inv \sum\limits_{j=1}^{\ntn} 
		\im \lb \E [ \beta_{j,t}( z) d_{j,t} (z) ] \rb 
		\re \lb \E [ \tilde{\su}_{n,t}(z) ] \rb\inv  \\
		& + y_{\ntn} \ntn\inv \sum\limits_{j=1}^{\ntn} 
		\re \lb \E [ \beta_{j,t}( z) d_{j,t} (z) ] \rb 
		\im \lb \E [ \tilde{\su}_{n,t}(z) ]  \rb\inv \\
		\geq & K_1  \im^2 (z) \ntn\inv \im \Big(  y_{\ntn}\sum\limits_{j=1}^{\ntn} 
		 \E [ \beta_{j,t}( z) d_{j,t} (z) ] \Big) \\
		 & - K_2  \im(z) \ntn\inv \re \Big(  y_{\ntn}\sum\limits_{j=1}^{\ntn} 
		 \E [ \beta_{j,t}( z) d_{j,t} (z) ] \Big),
		\end{align*}
		which implies that
		\begin{align*}
			& \im ( R_{n,t}(z) ) + \im(z) \\
			\geq  & \im(z) 
			+  K_1 \im^2(z) \ntn\inv \im \Big(  y_{\ntn}\sum\limits_{j=1}^{\ntn} 
		 \E [ \beta_{j,t}( z) d_{j,t} (z) ] \Big) 
		 -  K_2  \im(z) \ntn\inv \re \Big(  y_{\ntn}\sum\limits_{j=1}^{\ntn} 
		 \E [ \beta_{j,t}( z) d_{j,t} (z) ] \Big) \\
		 	\geq &  \im(z) \Big\{ 
		 1 +K_1 \im(z) \ntn\inv \im \Big(  y_{\ntn}\sum\limits_{j=1}^{\ntn} 
		 \E [ \beta_{j,t}( z) d_{j,t} (z) ] \Big) 
		 -  K_2 \ntn\inv \re \Big(  y_{\ntn}\sum\limits_{j=1}^{\ntn} 
		 \E [ \beta_{j,t}( z) d_{j,t} (z) ] \Big) \Big\}.
		\end{align*}
		The real and imaginary part of $\beta_{j,t}(z) d_{j,t}(z)$ might be negative, but, due to \eqref{R_to_0}, we have for some $N\in\N$
		\begin{align*}
			\sup\limits_{\substack{n\geq N}} \sup\limits_{\substack{z\in \mathcal{C}_n, \\ t\in[t_0,1] }}
			\Big| K_1
			\im(z) \ntn\inv \im \Big(  y_{\ntn}\sum\limits_{j=1}^{\ntn} 
		 \E [ \beta_{j,t}( z) d_{j,t} (z) ] \Big) 
		 -   K_2 \ntn\inv \re \Big(  y_{\ntn}\sum\limits_{j=1}^{\ntn} 
		 \E [ \beta_{j,t}( z) d_{j,t} (z) ] \Big)
			\Big| 
			< 1 
		\end{align*}
		Thus, we conclude that 
		\begin{align*}
				\inf\limits_{n\geq N} \inf\limits_{\substack{ z\in\mathcal{C}_n, \\ t\in[t_0,1]}} 
				\lb \im (z) + \im (R_{n,t}(z)) \rb \geq 0. 
			\end{align*}
		\end{proof}

\section{Details for the proof of  Theorem \ref{thm:u} % oder Section \ref{sec4}
}

 \subsection{How to calculate mean and covariance in Theorem \ref{thm}}

The following result  provides essential formulas for the calculation of the  mean and covariance structure in Theorem \ref{thm} in the case $\T_n = \mathbf{I}$. 
It  generalizes the formulas given in Proposition A.1 in \cite{wang2013} and Proposition 3.6 in \cite{yao2015}.

	\begin{proposition} \label{prop_formula}
	 Let $h_t = \sqrt{y_t}\in (0,\infty)$ and  $\T_n = \mathbf{I}$ and let $f_1$ and $f_2$ be functions which are analytic on an open region containing the interval in \eqref{interval}. For the random variable $\big( X(f_1,t_1), X(f_2,t)\big)_{t\in[t_0,1]}$  given in Theorem \ref{thm}, we have the following formulas  
		\begin{align*}
			 \E [ X(f_i,t) ]  = & \frac{1}{2\pi i} \lim\limits_{r \searrow 1} \oint\limits_{|\xi|=1} f ( t ( 1 + h_t r \xi + h_t r\inv \xi\inv + h_t^2 ) ) \Big( \frac{\xi }{   \xi^2 - r^{-2} } - \frac{1}{\xi} \Big) d\xi,  \\
			 \cov (X(f_1,t_1) , X(f_2,t_2) ) 
			= &  \frac{1}{2 \pi^2}
			\lim\limits_{\substack{r_2 > r_1, \\ r_1, r_2 \searrow 1}}		
			\oint\limits_{|\xi_1|=1} \oint\limits_{|\xi_2|=1} 
			f_1 ( t_1 ( 1 + h_{t_1} r_1 \xi_1 + h_{t_1} r_1\inv \xi_1\inv + h_{t_1}^2 ) ) \nonumber \\
			& ~ ~ ~ ~ \times  
			\overline{f_2 ( t_2 ( 1 + h_{t_2} r_2 \xi_2\inv + h_{t_2} r_2\inv \xi_2 + h_{t_2}^2 ) ) }
			\frac{g_1 (\xi_1, \xi_2) }{ g_2 (\xi_1, \xi_2) }
			d\xi_2 d\xi_1 
		\end{align*}
		where $t, t_1, t_2 \in[t_0,1]$ with $t_2 \leq t_1$ 
		\begin{align*}
		g_1 (\xi_1, \xi_2) = & - \Big( h_1 h_2 r_1 r_2 \Big\{ h_2^4 r_1^2 r_2^2 t_2^2 \xi_1^2 \xi_2^2 + 
       2 h_2^3 r_1^2 r_2 t_2 \xi_1^2 \xi_2 (r_2^2 t_1 
+ t_ 2 \xi_2^2) 
 \\
            &  	- 2 h_1 h_2 r_1 r_2 t_1 \xi_1 \xi_2 (r_2^2 t_1 (2 + 
             h_1 r_1 \xi_1) +  r_1 t_2 \xi_1 (h_1 + 
             2 r_1 \xi_1) \xi_2^2 \Big\} \\
             & + h_1^2 t_1 \xi_2^2 
			\Big\{r_2^2 t_1 (1 + 2 h_1 r_1 \xi_1 
			+  3 r_1^2 \xi_1^2 + h_1^2 r_1^2 \xi_1^2 + 
             2 h_1 r_1^3 \xi_1^3) 
              + r_1^2 t_ 2 \xi_1^2 \
(-1 + r_1^2 \xi_1^2) \xi_2^2 \Big\} \\
			& + h_2^2 \Big\{ r_2^4 t_1^2 
- r_2^2 t_1 t_2 (1 + 2 h_1 r_1 \xi_1 - 
             3 r_1^2 \xi_1^2 + 2 h_1^2 r_1^2 \xi_1^2 
             + 2 h_1 r_1^3 \xi_1^3) \xi_2^2 
              + r_1^2 t_2^2 \xi_1^2 \xi_2^4 \Big\} \Big) , \\
		g_2(\xi_1, \xi_2) = & \big(h_2 r_2 - 
       h_1 r_1 \xi_1 \xi_2 \big)^2 \big( h_2^2 r_1 r_2 t_2 \xi_1 
\xi_2 - h_1 r_2 t_1 (1 + 
           h_1 r_1 \xi_1 + r_1^2 \xi_1^2) \xi_2 + 
       h_2 r_1 \xi_1 (r_2^2 t_1 + t_2 \xi_2^2)\big)^2.
		\end{align*}
		% g1 / g2 entspricht in meinem Programm Integr ; dateinamme: cov_for_f_id
		In the complex case, we have $\E [X(f_i,t)]=0, ~ i=1,2,$ and the covariance structure is given by $1/2$ times the covariance structure for the real case. 
	\end{proposition}

	\begin{proof}[Proof of Proposition \ref{prop_formula}]
	If suffices to consider the real case.
		Since $H = \delta_{\{1\}}$, we obtain from Theorem \ref{thm} for $i\in \{1,2\}$
		\begin{align*}
			\E[ X(f_i,t)] = &  - \frac{1}{2\pi i} \oint\limits_{\mathcal{C}} f_i(z) 
		\frac{t y  \frac{\sut_t^3(z)}{(t \sut_t(z)  + 1)^3 }  }
			{\big( 1 - t y   \frac{\sut_t^2(z)}{( t \sut_t(z) + 1 )^2} \big)^2} 
			dz	 \\
		 \cov (X(f_1,t_1), X(f_2,t_2)) 
		= &   \frac{1}{2 \pi^2 } \oint_{\mathcal{C}_1} \oint_{\mathcal{C}_2} f_1(z_1) \overline{f_2(z_2)} 
	 \sigma_{ t_1, t_2} ^2(z_1,\overline{z_2})
	 \overline{dz_2} dz_1,
		\end{align*}
		where the contours $\mathcal{C}, \mathcal{C}_1, \mathcal{C}_2$ enclose the interval given in \eqref{interval} and $\mathcal{C}_1$ and $\mathcal{C}_2$ are assumed to be non-overlapping. \\
		\textbf{Step 1:} \textit{Specifying the contours}
		\\ We claim that 		 it suffices for $\mathcal{C}=\mathcal{C}_t$ to enclose the interval 
			$ \left[  t (1-\sqrt{y_{t}})^2 ,  
		t ( 1+ \sqrt{y_{t}} )^2 \right] $ and we will prove this assertion in a first step. Similar arguments hold true for contours $\mathcal{C}_1 = \mathcal{C}_{1,t_1}$ and $\mathcal{C}_2 = \mathcal{C}_{2,t_2}$.  
		
		\medskip
		
		The assertion is clear in the case $y_t < 1$. 
		In the case $y_t>1$, the transformed Mar\v{c}enko-Pastur distribution $\tilde{F}^{y_{\ntn}}$ has a discrete part at the origin for sufficiently large $n$. A priori, the contour should enclose the whole support   of $\tilde{F}^{y_t}$, including the origin.  However, by the exact separation theorem in \cite{baisilverstein1999}, we see that the mass at 0 of the spectral distribution $F^{\mathbf{B}_{n,t}}$ coincides with that of $\tilde{F}^{y_{\ntn}}$  for sufficiently large $n$. Thus, we can restrict the integration in \eqref{def_X} to the interval $ \left[  t (1-\sqrt{y_{t}})^2 ,  
		t ( 1+ \sqrt{y_{t}} )^2 \right] $ and neglect the discrete part at the origin. \\ 
		\textbf{Step 2:} \textit{Calculation of the mean} \\
	For calculation of the mean, we use a   change of variables
	\begin{align*}
		z (\xi) =z  =  t ( 1 + h_t r \xi + h_t r\inv \xi\inv + h_t^2 ),
	\end{align*}
	where $r>1$ is close to 1 and $|\xi| =1$. 
	It can be checked that when $\xi$ runs anticlockwise on the unit circle, $z$ will run a contour $\mathcal{C}$ enclosing the interval $ [ t (1 - h_t)^2, t (1+h_t)^2].$ 
Using the identity \eqref{repl_a47},
	%Lemma \ref{a47}, 
	we have for $z\in\mathcal{C}$ 
	\begin{align*}
		\sut_t(z) & = - \frac{1}{t ( 1 + h_t r \xi )}~,~ 
			\frac{ \sut_t(z)}{ t \sut_t(z) + 1} 
			= - \frac{1}{t h_t r \xi}~,~
		d z  = t h_t (r - r\inv \xi^{-2} ) d\xi. 
	\end{align*}
	%We calculate
	%\begin{align*}
%	= \frac{- \frac{1}{t (1 + h_t r \xi) } }{1 - t \frac{1}{t (1 + h_t r \xi) } }
%	\end{align*}
	Thus, we can write for $i \in \{1,2\}$
	\begin{align*}
			\E[ X(f_i,t)] = &
			\lim\limits_{r \searrow 1}  \frac{1}{2 \pi i} \oint\limits_{|\xi|=1} f_i(z(\xi)) t^2 h_t^2 \frac{\big( \frac{1}{t h_t r \xi}\big)^3}{\big( 1 - t^2 h_t^2 \big( \frac{1}{t h_t r \xi} \big)^2 \big)^2} t h_t \lb r - r\inv \xi^{-2} \rb d\xi \\
			& =  \lim\limits_{r \searrow 1}  \frac{t}{2 \pi i} \oint\limits_{|\xi|=1} f_i(z(\xi)) \frac{r^{-2}}{\xi t ( \xi^2 - r^{-2} ) }  d\xi \\
			& = - \lim\limits_{r \searrow 1}  \frac{t}{2 \pi i}  \oint\limits_{|\xi|=1}  f_i\lb z ( \xi ) \rb \big( \frac{1}{\xi t} - \frac{\xi}{ t (\xi^2 - r^{-2} ) } \big) d\xi \\
			& =  \lim\limits_{r \searrow 1}  \frac{1}{2 \pi i}  \oint\limits_{|\xi|=1}  f_i\lb z ( \xi ) \rb \big(  \frac{\xi}{  (\xi^2 - r^{-2} ) } - \frac{1}{\xi } \big) d\xi. \\
	\end{align*}
	% 	& =   - \lim\limits_{r \searrow 1}  \frac{1}{2 \pi i}  \oint\limits_{|\xi|=1}  f\lb t | 1 + h\xi|^2 \rb \lb \frac{1}{\xi t} - \frac{\xi }{ t (\xi^2 - r^{-2} ) } \rb d\xi . 
	% in den anderen Quellen gibt es noch die Vereinfachung: f(z(xi))= f(t|1+h_t xi|^2); bin mir nicht sicher, ob das rigoros stimmt! 
	\textbf{Step 3:} \textit{Calculation of the covariance function} \\ 
	In order to calculate the covariance structure, we define two non-overlapping contours through
	\begin{align*}
		z_j = z_j(\xi_j) = t \lb 1 + h_{t_j} \xi_j + h_{t_j} r_j\inv \overline{\xi}_j + h_{t_j}^2 \rb, ~j=1,2,
	\end{align*}		
	where $r_2 > r_1 > 1$.
	Thus, we have for $t_2 \leq t_1$ 
	\begin{align*}
		&  \cov (X(f_1,t_1), X(f_2,t_2)) 
		=  \lim\limits_{\substack{r_2 > r_1, \\ r_1, r_2 \searrow 1}}	 
		 \frac{1}{2 \pi^2 } \oint_{|\xi_1|=1} \oint_{|\xi_2|=1} f_1(z_1(\xi_1)) \overline{f_2(z_2(\xi_2))} \\
		 & ~~~~~~~~~~~~~~~~~ ~~~~~~~~~~~~~~ \times 
	 \sigma_{ t_1, t_2}^2 (z_1(\xi_1),\overline{z_2(\xi_2)})	
	 t_1 h_{t_1} (r_1 - r_1\inv \xi_1^{-2} )  t_2 h_{t_2} (r_2 - r_2\inv \xi_2^{-2} ) 
	 \overline{d\xi_2} d\xi_1\\
	 = &   \lim\limits_{\substack{r_2 > r_1, \\ r_1, r_2 \searrow 1}}	
	   \frac{1}{2 \pi^2 } \oint_{|\xi_1|=1} \oint_{|\xi_2|=1} f_1(z_1(\xi_1)) \overline{f_2(z_2(\xi_2\inv))} 
	    \\
		 & ~~~~~~~~~~~~~~~~~ ~~~~~~~~~~~~~~ \times 
	 \sigma_{ t_1, t_2}^2 (z_1(\xi_1),\overline{z_2(\xi_2\inv)}) 	
	 t_1 h_{t_1} (r_1 - r_1\inv \xi_1^{-2} )  t_2 h_{t_2} (r_2 - r_2\inv \xi_2^{2} ) 
	 d\xi_2 d\xi_1.
	\end{align*}
	By proceeding similarly as for the mean and additionally using
	\begin{align*}
			\sut_t(z) dz &  =  \frac{h_t r}{t (1 + h_t r \xi)^2 } d\xi,
	\end{align*}
	we get by straightforward but tedious algebra the desired formula for the covariance
	  (we partially used a computer algebra system). 
	\end{proof}

\subsection{Proof of  \eqref{det3} and  \eqref{det4}}
\label{sec_51}
    Recall that $f_1(x) =x$ and $f_2 (x) = x^2.$
	We begin determining  the centering term. Using the moments of the Mar\v{c}enko-Pastur distribution (e.g., see Example 2.12 in \cite{yao2015}), we get
	\begin{align*}
	\int f_1(x)  d \tilde{F}^{y_{\ntn}} (x) 
	=\int x d \tilde{F}^{y_{\ntn}} (x)
%	= \int x d F^{y_{\ntn}} \lb \frac{n}{\ntn} x \rb 
	= \frac{\ntn}{n} \int x d F^{y_{\ntn}} (x) = \frac{\ntn}{n},
	\end{align*}
	where $F^y$ denotes the Mar\v{c}enko-Pastur distribution with index parameter $y>0$ and scale parameter $\sigma^2 =1$. Similarly, we see that by using Proposition 2.13 in \cite{yao2015}
	\begin{align*}
		\int f_2(x)  d \tilde{F}^{y_{\ntn}} (x) 
	=  \int x^2 d \tilde{F}^{y_{\ntn}} (x) 
	%= \int x^2 d F^{y_{\ntn}} \lb \frac{n}{\ntn} x \rb 
	%= \lb \frac{\ntn}{n} \rb^2 \int x^2 d F^{y_{\ntn}} (x) \\
	=  \Big( \frac{\ntn}{n} \Big)^2 \lb 1 + y_{\ntn} \rb = \frac{\ntn}{n} \Big( \frac{\ntn}{n} + y_n \Big).
	\end{align*}
	We calculate the quantities given in Proposition \ref{prop_formula} by using the residue theorem. We find for the real case
	\begin{align}
	\E [ X(f_1,t) ]  
	%= & \frac{1}{2\pi i} \lim\limits_{r \searrow 1} %\oint\limits_{|\xi|=1}  t ( 1 + h_t r \xi + h_t %r\inv \xi\inv + h_t^2 )  \lb \frac{\xi }{   \xi^2 %- r^{-2} } - \frac{1}{\xi} \rb d\xi \nonumber \\ 
	 = & \frac{t}{2\pi i} \lim\limits_{r \searrow 1} \oint\limits_{|\xi|=1} 
	  \frac{ \xi + h_t r \xi^2 + h_t r\inv  + h_t^2 \xi  }{\xi} \Big( \frac{\xi }{   \xi^2 - r^{-2} } - \frac{1}{\xi} \Big) d\xi \nonumber \\ 
	   = & \frac{t}{2\pi i} \lim\limits_{r \searrow 1} \oint\limits_{|\xi|=1} 
	  \frac{ \xi + h_t r \xi^2 + h_t r\inv + h_t^2 \xi  }{  (\xi - r\inv) (\xi + r\inv) }   d\xi 
	 - \frac{t}{2\pi i} \oint\limits_{|\xi|=1} 
	  \frac{ \xi + h_t r \xi^2 + h_t r\inv + h_t^2 \xi  }{\xi^2}    d\xi \label{e1}\\ 
	  = & \lim\limits_{r \searrow 1} t\frac{ \xi + h_t r \xi^2 + h_t r\inv + h_t^2 \xi  }{  \xi + r\inv }  \Big|_{\xi = r\inv}
	 + \lim\limits_{r \searrow 1} t\frac{ \xi + h_t r \xi^2 + h_t r\inv + h_t^2 \xi  }{  \xi - r\inv }  \Big|_{\xi = - r\inv} \nonumber \\
	  & -  t \frac{\partial }{\partial \xi}\lb \xi + h_t r \xi^2 + h_t r\inv + h_t^2 \xi \rb      \Big|_{\xi = 0} \nonumber \\
	   %= & t \lim\limits_{r \searrow 1} \frac{r\inv + %h_t r^{-1} + h_t r\inv + h_t^2 r\inv}{2 r\inv} %
	   %- t \lim\limits_{r \searrow 1} \frac{- r\inv + %h_t r^{-1} + h_t r\inv - h_t^2 r\inv}{2 r\inv} % - t (1 + h_t^2) \nonumber \\
	   = &t  \lim\limits_{r \searrow 1} \frac{2 r\inv + 2 h_t^2 r\inv}{2 r\inv}  - t ( 1 + h_t^2) 
	   =  0. \nonumber 
	\end{align}
	Note that $\xi = \pm r\inv$ are poles of order 1 for the first integrand in \eqref{e1}, since $r>1$, while $\xi=0$ is a pole of order 2 for the second integrand in \eqref{e1}. For the complex case, we directly have $\E [ X(f_1,t) ] = \E [ X(f_2,t)] = 0.$
	\\ For $f_2(x) = x^2$, we have
	\begin{align*}
		\E [ X(f_2,t) ]  
		%= & \frac{1}{2\pi i} \lim\limits_{r \searrow 1} \oint\limits_{|\xi|=1}  
		 %t^2 ( 1 + h_t r \xi + h_t r\inv \xi\inv + h_t^2 )^2 \lb \frac{\xi }{   \xi^2 - r^{-2} } - %\frac{1}{\xi} \rb d\xi \\ 
	 %= & \frac{t^2}{2\pi i} \lim\limits_{r \searrow 1} \oint\limits_{|\xi|=1} 
	  %\frac{ \lb \xi + h_t r \xi^2 + h_t r\inv  + h_t^2 \xi \rb^2  }{\xi^2} \lb \frac{\xi }{   \xi^2 - r^{-2} } %- \frac{1}{\xi} \rb d\xi \\ 
	  =  I_1 - I_2,
	\end{align*}
	where
	\begin{align*}
		I_1 = & \frac{t^2}{2\pi i} \lim\limits_{r \searrow 1} \oint\limits_{|\xi|=1} 
	  \frac{ \lb \xi + h_t r \xi^2 + h_t r\inv  + h_t^2 \xi \rb^2  }{\xi \lb   \xi - r^{-1} \rb \lb \xi + r\inv \rb }  d\xi , \\
	   I_2 = &  \frac{t^2}{2\pi i} \lim\limits_{r \searrow 1} \oint\limits_{|\xi|=1} 
	  \frac{ \lb \xi + h_t r \xi^2 + h_t r\inv  + h_t^2 \xi \rb^2  }{\xi^3 }  d\xi. 
	\end{align*}
	The integrand in $I_1$ has  poles which are all of order 1  at the points $0, r\inv , - r\inv$. Thus, using the residue theorem,
	\begin{align*}
		I_1 = & t^2 \lim\limits_{r \searrow 1} 
	  \frac{ \lb \xi + h_t r \xi^2 + h_t r\inv  + h_t^2 \xi \rb^2  }{\lb   \xi - r^{-1} \rb \lb \xi + r\inv \rb }  \Big|_{\xi=0}
	  + t^2 \lim\limits_{r \searrow 1} 
	  \frac{ \lb \xi + h_t r \xi^2 + h_t r\inv  + h_t^2 \xi \rb^2  }{\xi  \lb \xi + r\inv \rb }  \Big|_{\xi=r \inv} \\
	  & + t^2 \lim\limits_{r \searrow 1} 
	  \frac{ \lb \xi + h_t r \xi^2 + h_t r\inv  + h_t^2 \xi \rb^2  }{\xi \lb   \xi - r^{-1} \rb }  \Big|_{\xi=- r\inv} \\
	  = & - t^2 h_t^2  
	  +  \frac{ t^2( 1 + h_t)^4}{2}
	  +   \frac{t^2( 1 - h_t)^4}{2}
	   =  -  t h^2  
	  +  \frac{ t^2( 1 + h_t)^4}{2}
	  +   \frac{t^2( 1 - h_t)^4}{2}.
	  \end{align*}
	  Using that the integrand in $I_2$ has a pole at $\xi=0$ of order 3, similar calculations   yield
	  $I_2 =  (1 + 4 h_t^2 + h_t^4) t^2$, which gives
	  \begin{align*}
	  	\E [ X(f_2,t)] = t h^2 = t y.
	  \end{align*}
	  % I_1 und I_2 habe ich mit Mathematica einmal nachgerechnet und mit dem Fall t=1 verglichen (siehe Datei mean_for_pow(x,k)  

	For the covariance function of $(X(f_1,t))_{t\in[t_0,1]}$, we have  for $t_2 \leq t_1$
	\begin{align*}
		 \cov (X(f_1,t_1) , X(f_1,t_2) )  
		= &  \frac{1}{2 \pi^2}
			\lim\limits_{\substack{r_2 > r_1, \\ r_1, r_2 \searrow 1}}		
			\oint\limits_{|\xi_1|=1} \oint\limits_{|\xi_2|=1} 
			 t_1 \lb 1 + h_{t_1} \xi_1 + h_{t_1} r_1\inv \xi_1\inv + h_{t_1}^2 \rb 
			 t_2  \\
			 & ~~~~~~~~~~~~~ \times  
			 \lb 1 + h_{t_2} \xi_2 + h_{t_2} r_2\inv \xi_2\inv + h_{t_2}^2 \rb
			\frac{g_1 (\xi_1, \xi_2) }{ g_2 (\xi_1, \xi_2) }
			d\xi_2 d\xi_1 \\
			= & - \frac{2 \pi i}{2 \pi^2 } \lim\limits_{\substack{r_2 > r_1, \\ r_1, r_2 \searrow 1}} \oint\limits_{|\xi_1|=1} 
			\frac{ h_{t_1} t_1 t_2 (h_1 + r_1 \xi_1 + h_1^2 r_1 \xi_1 + h_1 r_1^2 \xi_1^2)}{
 r_1^2 r_2^2 \xi_1^3}
			d\xi_1 \\
			= & - \frac{(2 \pi i)^2}{2 \pi^2} \lim\limits_{\substack{r_2 > r_1, \\ r_1, r_2 \searrow 1}} \frac{\partial^2}{\partial^2 \xi_1}
			\frac{ h_{t_1} t_1 t_2 (h_1 + r_1 \xi_1 + h_1^2 r_1 \xi_1 + h_1 r_1^2 \xi_1^2)}{
 r_1^2 r_2^2 } \Big|_{\xi_1 = 0} 
 \\ = & 2\lim\limits_{\substack{r_1 > r_2, \\ r_1, r_2 \searrow 1}} \frac{h_1^2 t_2}{r_2^2} 
 = 2 h_1^2 t_2 = 2 y t_2,
	\end{align*}
	where we used a computer algebra system for simplifying the first integrand and then applied the residue theorem twice. 
	%For the complex case, the covariance is given by
	%\begin{align*}
%		\cov (X(f_1,t_1) , X(f_1,t_2) )  = y t_2.
%	\end{align*}
	Considering the function $f_2$, we have for ($t_2 \leq t_1$)
	\begin{align}
		& \cov (X(f_2,t_1) , X(f_2,t_2) ) \nonumber \\
		= & \frac{1}{2 \pi^2}
			\lim\limits_{\substack{r_2 > r_1, \\ r_1, r_2 \searrow 1}}		
			\oint\limits_{|\xi_1|=1} \oint\limits_{|\xi_2|=1} 
			 t_1^2 \lb 1 + h_{t_1} \xi_1 + h_{t_1} r_1\inv \xi_1\inv + h_{t_1}^2 \rb^2 
			 t_2^2  \nonumber \\
			 &  ~~~~~~~~~~~~~ ~~~~~~~~~~~~~ ~~~~~~~~~~~~~ \times 
			 \lb 1 + h_{t_2} \xi_2 + h_{t_2} r_2\inv \xi_2\inv + h_{t_2}^2 \rb^2
			\frac{g_1 (\xi_1, \xi_2) }{ g_2 (\xi_1, \xi_2) }
			d\xi_2 d\xi_1  \nonumber \\
		= &  \frac{2 \pi i}{2 \pi^2 } \lim\limits_{\substack{r_2 > r_1, \\ r_1, r_2 \searrow 1}} \oint\limits_{|\xi_1|=1} 
		 \frac{2 h_1 t_1 t_2^2 }{r_1^4 r_2^4 \xi_1^5} 
		 \lb h_1 + r_1 x + h_1^2 r_1 x + h_1 r_1^2 \xi_1^2\rb^2 \nonumber \\
   & \times \lb -h_1 t_1 - h_1^2 r_1 t_1 \xi_1 - r_1 r_2^2 t_1 \xi_1 - 
   h_2^2 r_1 r_2^2 t_1 \xi_1 + h_2^2 r_1 t_2 \xi_1 + h_1^2 r_1^3 t_1 \xi_1^3 - 
   h_2^2 r_1^3 t_2 \xi_1^3\rb d \xi_1 \label{e2} \\
   		= &  \frac{ \lb 2 \pi i \rb^2}{2 \pi^2 } 
   		\lim\limits_{\substack{r_2 > r_1, \\ r_1, r_2 \searrow 1}}
   		 \frac{1}{(5-1)!} \frac{\partial^4}{\partial^4 \xi_1}  \Big\{
   		 \frac{2 h_1 t_1 t_2^2 }{r_1^4 r_2^4 } 
		 \lb h_1 + r_1 \xi_1 + h_1^2 r_1 \xi_1 + h_1 r_1^2 \xi_1^2\rb^2 \nonumber \\
    & \times \lb -h_1 t_1 - h_1^2 r_1 t_1 \xi_1 - r_1 r_2^2 t_1 \xi_1 - 
   h_2^2 r_1 r_2^2 t_1 \xi_1 + h_2^2 r_1 t_2 \xi_1 + h_1^2 r_1^3 t_1 \xi_1^3 - 
   h_2^2 r_1^3 t_2 \xi_1^3\rb d \xi_1 \Big\} \Big|_{\xi_1 = 0}
   		\nonumber \\
		= & 4 t_2 y \left\{ 2 t_1 t_2 + \left[ t_2 + 2 (t_1 + t_2)\right] y + 2 y^2 \right\}, ~ t_2 \leq t_1. \nonumber
	\end{align}
	Note that $\xi_1=0$ is a pole of order 5 for the integrand in \eqref{e2} and that 
%	For the complex case, we have
%	\begin{align*}
	%	\cov (X(g,t_1) , X(g,t_2) ) = 2 t_2 y \left\{ 2 t_1 t_2 + \left[ t_2 + 2 (t_1 + %t_2)\right] y + 2 y^2 \right\}, ~ t_2 \leq t_1.
%	\end{align*}
 in the special case $t_1 = t_2 = 1$ we recover the mean and covariance given in (9.8.14) and, respectively, (9.8.15) in \cite{bai2004}.  \\
	Finally, we want to calculate the dependence structure between $X(f_1,t_1)$ and $X(f_2,t_2)$. Using similar techniques as above, we obtain for $t_2 \leq t_1$ 
	\begin{align} 
		& \cov (X(f_1,t_1) , X(f_2,t_2) ) \nonumber \\
		= &  \frac{1}{2 \pi^2}
			\lim\limits_{\substack{r_2 > r_1, \\ r_1, r_2 \searrow 1}}		
			\oint\limits_{|\xi_1|=1} \oint\limits_{|\xi_2|=1} 
			 t_1 \lb 1 + h_{t_1} \xi_1 + h_{t_1} r_1\inv \xi_1\inv + h_{t_1}^2 \rb 
			 t_2^2 \nonumber  \\
			 &  ~~~~~~~~~~~~~~~~~~~~~~~~~ \times 
			 \lb 1 + h_{t_2} \xi_2 + h_{t_2} r_2\inv \xi_2\inv + h_{t_2}^2 \rb^2
			\frac{g_1 (\xi_1, \xi_2) }{ g_2 (\xi_1, \xi_2) }
			d\xi_2 d\xi_1  \label{integr1}\\ 
			= & \frac{2\pi i}{2 \pi^2} \lim\limits_{\substack{r_2 > r_1, \\ r_1, r_2 \searrow 1}}		
			\oint\limits_{|\xi_1|=1}
			\frac{1}{r_1^3 r_2^4 \xi_1^4}
   2 h_1 t_2^2 (h_1 + r_1 \xi_1 + h_1^2 r_1 \xi_1 + 
   h_1 r_1^2 \xi_1^2) \nonumber \\
   & \times 
   \left\{ -h_1 t_1 + h_1^2 r_1 t_1 \xi_1 (-1 + r_1^2 \xi_1^2) - 
   r_1 \xi_1 \left[ (1 + h_2^2) r_2^2 t_1 + h_2^2 t_2 (-1 + r_1^2 \xi_1^2)\right] \right\} d\xi_1 
			\label{integr2} \\
		= & 4 t_2 y (t_2 + y).\label{cov1}
	\end{align}
	After simplifying the integrand in \eqref{integr1} with a computer algebra program, we see that it has a pole at $\xi_2=0$ of order 2. Note that the pole at 
	\begin{align*}
		\xi_2 =\frac{h_2 r_2}{h_1 r_1 \xi_1} 
	\end{align*}
	is not relevant for an application of the residue theorem, since
	\begin{align*}
		\left| \frac{h_2 r_2}{h_1 r_1 \xi_1} \right| = \left| \frac{t_1 r_2}{ r_1 t_2 \xi_1} \right| 
		=   \frac{t_1 r_2}{ t_2 r_1  }  > 1.
	\end{align*}
	The integrand in \eqref{integr2} has a pole at $\xi_1 =0$ of order 4.
	Similarly, we have, again for $t_2 \leq t_1$,
	\begin{align} 
	  \cov ( X(f_2,t_1) , X(f_1,t_2)  )  
	%	= &  \frac{1}{2 \pi^2}
		%	\lim\limits_{\substack{r_2 > r_1, \\ r_1, r_2 \searrow 1}}		
		%	\oint\limits_{|\xi_1|=1} \oint\limits_{|\xi_2|=1} 
		%	 t_1^2 \lb 1 + h_{t_1} \xi_1 + h_{t_1} r_1\inv \xi_1\inv + h_{t_1}^2 \rb^2 
		%	 t_2 \nonumber \\
		%	 & ~~~~~~~~~~~~~~~~~~~~~~~~~~~~~~~~\times 
		%	 \lb 1 + h_{t_2} \xi_2 + h_{t_2} r_2\inv \xi_2\inv + h_{t_2}^2 \rb
		%	\frac{g_1 (\xi_1, \xi_2) }{ g_2 (\xi_1, \xi_2) }
		%	d\xi_2 d\xi_1  \label{integr3}\\ 
		%	= & \frac{2\pi i}{2 \pi^2} \lim\limits_{\substack{r_2 > r_1, \\ r_1, r_2 %\searrow 1}}		
		%	\oint\limits_{|\xi_1|=1}
		%	-\frac{h_1 t_1^2 t_2 (h_1 + r_1 \xi_1 + h_1^2 r_1 \xi_1 + h_1 r_1^2 %\xi_1^2)^2}
		%	{r_1^3 r_2^2 \xi_1^4}
		% d\xi_1 \label{integr4}\\
		=   4 t_2 y (t_1 + y).\label{cov2}
	\end{align}
	%After simplifying the integrand in \eqref{integr3} with a computer algebra program, %one can see that the corresponding integrand has a pole at $\xi_2 = 0$ of order 1, %while, as above, the pole at
	%\begin{align*}
	%	\xi_2 =\frac{h_2 r_2}{h_1 r_1 \xi_1} 
%	\end{align*}
%	is not relevant for the integration. 
%	The integrand in \eqref{integr4} has at pole at $\xi_1 = 0$ of order 4. 
	By combining \eqref{cov1} and \eqref{cov2}, we have for $t_1, t_2 \in[t_0,1]$
	\begin{align*}
		\cov (X(f_1,t_1) , X(f_2,t_2) ) = 4 \min(t_1,t_2) y (t_2 + y). 
	\end{align*}

\subsection{Proof of  Corollary \ref{thm:logdet} }

	We apply Theorem \ref{thm} for the choice $h(x) = \log (x)$. Note that, as $y \geq t_0$, the interval in \eqref{interval} contains the point 0. Thus, we have to impose $y< t_0$, since $h$ is not analytic in a neighborhood of 0. \\
	  Using Example 2.11 in \cite{yao2015}, we obtain for the centering term
	\begin{align*}
		\int \log x d \tilde{F}^{y_{\ntn}} (x)
		= &  \int \log x d F^{y_{\ntn}} \Big( \frac{n}{\ntn	} x \Big) 
		=  \int \log x d F^{y_{\ntn}} \lb  x \rb + \log \Big( \frac{\ntn}{n} \Big) \\
		= & 	\Big( - 1 + \frac{y_{\ntn} - 1}{y_{\ntn}} \log ( 1 - y_{\ntn} ) \Big) + \log \Big( \frac{\ntn}{n} \Big) \\ 
		= & - 1 - \frac{1}{y_{\ntn}} \log ( 1 - y_{\ntn} ) + \log \Big( \frac{\ntn}{n} - y_n \Big), 
	\end{align*}
	which implies
	\begin{align*}
		p \int \log x d \tilde{F}^{y_{\ntn}} (x) & = 
		- p  - \ntn \log ( 1 - y_{\ntn} ) + p \log \lb \frac{\ntn}{n} - y_n \rb.
	\end{align*}
	By Proposition \ref{prop_formula}, we have for the mean of the limiting process $\mathbb{D}$ in the real case
	\begin{align}\label{det11}
		\E [\mathbb{D}(t)] 
		% =  & \frac{1}{2\pi i} \lim\limits_{r \searrow 1} \oint\limits_{|\xi|=1} \log ( %t ( 1 + h_t r \xi + h_t r\inv \xi\inv + h_t^2 ) ) \lb \frac{\xi }{   \xi^2 - %r^{-2} } - \frac{1}{\xi} \rb d\xi \\
		 =   I_1 + I_2,
	\end{align}
	where 
	\begin{align}
		I_1= & \frac{1}{2\pi i} \lim\limits_{r \searrow 1} \oint\limits_{|\xi|=1} \log ( t ( 1 + h_t r \xi + h_t r\inv \xi\inv + h_t^2 ) )  \frac{\xi }{   \xi^2 - r^{-2} }  d\xi \nonumber \\
		= & \frac{1}{2\pi i} \lim\limits_{r \searrow 1} \oint\limits_{|\xi|=1} \log( t | 1 + h_t \xi |^2)  \frac{\xi }{   \xi^2 - r^{-2} }  d\xi, \label{e3} \\
		I_2  = & - \frac{1}{2\pi i} \lim\limits_{r \searrow 1} \oint\limits_{|\xi|=1} \log ( t ( 1 + h_t r \xi + h_t r\inv \xi\inv + h_t^2 ) )  \frac{1}{\xi}  d\xi \nonumber \\
		= & - \frac{1}{2\pi i} \lim\limits_{r \searrow 1} \oint\limits_{|\xi|=1} \log ( t | 1 + h_t \xi |^2 )  \frac{1}{\xi}  d\xi~ \nonumber 
	\end{align}
(see also \cite{wang2013} for a similar representation).
	Beginning with $I_1$, we further decompose (note that for  $|\xi|=1$, it holds $\xi\inv = \overline{\xi} $ )
	\begin{align*}
		I_1 = I_{11} + I_{12},
	\end{align*}
	where
	\begin{align*}
		I_{11} = & \frac{1}{2\pi i} \lim\limits_{r \searrow 1} \oint\limits_{|\xi|=1} \log( t ( 1 + h_t \xi ) )  \frac{\xi }{   (\xi - r\inv) (\xi + r\inv) }  d\xi , \\
	   = & \frac{1}{2} \lim\limits_{r \searrow 1}  \left\{ \log \lb t \lb 1 + h_t r\inv\rb \rb
		+ \log \lb t \lb 1 - h_t r\inv\rb \rb  \right\} 
		= \frac{1}{2} \log \lb t^2 \lb 1 - h_t^2 \rb \rb, \\
		I_{12} = &  \frac{1}{2\pi i} \lim\limits_{r \searrow 1} \oint\limits_{|\xi|=1} \log( t ( 1 + h_t \xi\inv ) )  \frac{\xi }{   (\xi - r\inv) (\xi + r\inv)}  d\xi \\
			= &  \frac{1}{2 \pi i } \lim\limits_{r \searrow 1} \oint\limits_{|z|=1} \log( t ( 1 + h_t z ) )  \frac{ r^2 }{  z (z -r ) (r + z )}  dz \\
		= & \lim\limits_{r \searrow 1} \log ( t( 1 + h_t z)) \frac{r^2}{(z - r) (z +r)} \Big|_{z=0}  
		= -   \log(t).
	\end{align*}
	
	\begin{align*}
	%	I_{12} = & 
	%	- \frac{1}{2 \pi i } \lim\limits_{r \searrow 1} \oint\limits_{|z|=1} \log( t ( 1 %+ h_t z ) )  \frac{ 1 }{  z^3 (z\inv - r\inv) (z\inv + r\inv)}  dz \\
	\end{align*}
	These calculations imply
	\begin{align}\label{I1}
		I_1 = \frac{1}{2} \log \lb t^2 \lb 1 - h_t^2 \rb \rb 
		-   \log(t).
	\end{align}
	The quantity $I_2$ in \eqref{det11} can be determined similarly using the decomposition
	\begin{align*}
		I_2 = I_{21} + I_{22},
	\end{align*}
	where
	\begin{align*}
		I_{21} = & - \frac{1}{2\pi i} \lim\limits_{r \searrow 1} \oint\limits_{|\xi|=1} \log ( t ( 1 + h_t \xi  ) )  \frac{1}{\xi}  d\xi = - \log t \\
		I_{22} = & - \frac{1}{2\pi i} \lim\limits_{r \searrow 1} \oint\limits_{|\xi|=1} \log ( t ( 1 + h_t \xi\inv  ) )  \frac{1}{\xi}  d\xi = \log t.
	\end{align*}
	This gives $I_2 = 0$, and by 
	  \eqref{I1} and \eqref{det11}, we obtain 
	\begin{align*}
		\E [ \mathbb{D}(t) ] = 
		\frac{1}{2} \log \lb t^2 \lb 1 - h_t^2 \rb \rb 
		-   \log(t)
		= \frac{1}{2} \log \lb 1 - h_t^2 \rb 
		= \frac{1}{2} \log \lb 1 -y_t \rb.
	\end{align*}		
	%For the complex case, we have $ \E [ \mathbb{D}(t) ] =  0$. 
	Next, we calculate the covariance structure. Similarly to \eqref{e3}, we obtain for $t_2 \leq t_1$ 
	\begin{align*}
	\cov (\mathbb{D}(t_1), \mathbb{D}(t_2) ) 
		= & \frac{1}{2 \pi^2}
			\lim\limits_{\substack{r_2 > r_1, \\ r_1, r_2 \searrow 1}}		
			\oint\limits_{|\xi_1|=1} \oint\limits_{|\xi_2|=1} 
			\log ( t_1 |1 + h_{t_1} \xi_1 |^2 ) 
			\overline{ \log ( t_2 |1 + h_{t_2} \xi_2|^2 ) }
			\frac{g_1 (\xi_1, \xi_2) }{ g_2 (\xi_1, \xi_2) }
			d\xi_2 d\xi_1 \\
			= &  \frac{1}{2 \pi^2}
			\lim\limits_{\substack{r_2 > r_1, \\ r_1, r_2 \searrow 1}}		
			\oint\limits_{|\xi_1|=1} \oint\limits_{|\xi_2|=1} 
			\log ( t_1 |1 + h_{t_1} \xi_1 |^2 ) 
			\log ( t_2 |1 + h_{t_2} \xi_2|^2 ) 
			\frac{g_1 (\xi_1, \xi_2) }{ g_2 (\xi_1, \xi_2) }
			d\xi_2 d\xi_1 \\
			= & I_3 + I_4,
			\end{align*}
			where (note that $|1 + h_{t_2} \xi_2|^2 = (1 + h_{t_2} \xi_2) (1 + h_{t_2} \xi_2\inv) $ )
			\begin{align*} 
			I_3 = &  \frac{1}{2 \pi^2}
			\lim\limits_{\substack{r_2 > r_1, \\ r_1, r_2 \searrow 1}}		
			\oint\limits_{|\xi_1|=1} \log ( t_1 |1 + h_{t_1} \xi_1 |^2 )
			\oint\limits_{|\xi_2|=1} 
			\log ( t_2 (1 + h_{t_2} \xi_2) ) 
			\frac{g_1 (\xi_1, \xi_2) }{ g_2 (\xi_1, \xi_2) }
			d\xi_2 d\xi_1, \\
			I_4 = &   \frac{1}{2 \pi^2}
			\lim\limits_{\substack{r_2 > r_1, \\ r_1, r_2 \searrow 1}}		
			\oint\limits_{|\xi_1|=1} \log ( t_1 |1 + h_{t_1} \xi_1 |^2 ) 
			\oint\limits_{|\xi_2|=1} 
			\log ( t_2 (1 + h_{t_2} \xi_2\inv ) ) 
			\frac{g_1 (\xi_1, \xi_2) }{ g_2 (\xi_1, \xi_2) }
			d\xi_2 d\xi_1. \\
	\end{align*}
	Using a computer algebra program for simplifying  $I_3$ and $I_4$, we see that $I_3=0$ and for $I_4$, and we perform the substitution $\xi_2 = z_2\inv$, which yields 
	\begin{align*}
		I_4 = \frac{2 \pi i}{2 \pi^2}
			\lim\limits_{\substack{r_2 > r_1, \\ r_1, r_2 \searrow 1}}		
			\oint\limits_{|\xi_1|=1} \log ( t_1 |1 + h_{t_1} \xi_1 |^2 )
			 \frac{h_1 r_1 }{r_2 + h_1 r_1 \xi_1} d\xi_1 
			 =   I_{41} + I_{42},
	\end{align*}
	where
	\begin{align*}
		I_{41} = & \frac{2 \pi i}{2 \pi^2}
			\lim\limits_{\substack{r_2 > r_1, \\ r_1, r_2 \searrow 1}}		
			\oint\limits_{|\xi_1|=1} \log ( t_1 ( 1 + h_{t_1} \xi_1 ) )
			 \frac{h_1 r_1 }{r_2 + h_1 r_1 \xi_1} d\xi_1, \\
		I_{42} = & \frac{2 \pi i}{2 \pi^2}
			\lim\limits_{\substack{r_2 > r_1, \\ r_1, r_2 \searrow 1}}		
			\oint\limits_{|\xi_1|=1} \log ( t_1 ( 1 + h_{t_1} \xi_1\inv ) )
			 \frac{h_1 r_1 }{r_2 + h_1 r_1 \xi_1} d\xi_1. \\
	\end{align*}
	It holds $I_{41} = 0$, since we have for the pole at $\xi_1 = -r_2/(h_1 r_1)$ that
 $	\left| \xi_1   \right|^2 > \frac{1}{h_{t_1}^2} = \frac{t_1}{y}\geq \frac{t_0}{y}  \geq 1. $

	As above, we perform for $I_{42}$ the substitution $\xi_1\inv =z_1$ and obtain
	\begin{align*}
	I_{42} = & - \frac{2 \pi i}{2 \pi^2}
			\lim\limits_{\substack{r_2 > r_1, \\ r_1, r_2 \searrow 1}}		
			\oint\limits_{|z_1|=1} \log(t_1 (1 + h_{t_1} z_1))
	\frac{h_{t_1} r_1 }{h_{t_1} r_1 z_1 + r_2 z_1^2} dz_1 \\
%	= & - \frac{2 \pi i}{2 \pi^2}
	%		\lim\limits_{\substack{r_2 > r_1, \\ r_1, r_2 \searrow 1}}		
	%		\oint\limits_{|z_1|=1} \log(t_1 (1 + h_{t_1} z_1))
%	\frac{h_{t_1} r_1 }{z_1 ( h_{t_1} r_1  + r_2 z_1) } dz_1 \\
	= & - \frac{( 2 \pi i )^2}{2 \pi } \lim\limits_{\substack{r_2 > r_1, \\ r_1, r_2 \searrow 1}}	 \Big\{ - \log(t_1) + \log \Big( t_1 \Big( 1 - \frac{h_{t_1}^2 r_1}{r_2} \Big) \Big) \Big\} \\
	= & - 2 \log ( 1 - h_{t_1}^2).
	\end{align*}	 
	Finally, we obtain for $t_2 \leq t_1$
	\begin{align*}
	\cov (\mathbb{D}(t_1), \mathbb{D}(t_2)  )  = I_3 + I_4
		= & - 2 \log (1 - h_{t_1}^2 )
			= - 2 \log (1 - y_{t_1} ).
	\end{align*}

	\end{appendices}


\begin{thebibliography}{}

\bibitem[Anderson, 1984]{anderson2003}
Anderson, T.~W. (1984).
\newblock {\em An Introduction to Multivariate Statistical Analysis}.
\newblock Wiley Series in Probability and Mathematical Statistics: Probability
  and Mathematical Statistics. John Wiley \& Sons, Inc., New York, second
  edition.

\bibitem[Aue et~al., 2009]{Aue2009b}
Aue, A., H{\"o}rmann, S., Horv{\'a}th, L., and Reimherr, M. (2009).
\newblock Break detection in the covariance structure of multivariate time
  series models.
\newblock {\em The Annals of Statistics}, 37(6B):4046--4087.

\bibitem[Bai et~al., 2009]{baietal2009}
Bai, Z., Jiang, D., Yao, J.-F., and Zheng, S. (2009).
\newblock Corrections to {LRT} on large-dimensional covariance matrix by {RMT}.
\newblock {\em Annals of Statistics}, 37:3822--3840.

\bibitem[Bai and Silverstein, 1998]{baisilverstein1998}
Bai, Z. and Silverstein, J.~W. (1998).
\newblock No eigenvalues outside the support of the limiting spectral
  distribution of large-dimensional sample covariance matrices.
\newblock {\em The Annals of Probability}, 26(1):316--345.

\bibitem[Bai and Silverstein, 1999]{baisilverstein1999}
Bai, Z. and Silverstein, J.~W. (1999).
\newblock Exact separation of eigenvalues of large dimensional sample
  covariance matrices.
\newblock {\em Annals of Probability}, pages 1536--1555.

\bibitem[Bai and Silverstein, 2004]{baisilverstein2004}
Bai, Z. and Silverstein, J.~W. (2004).
\newblock Clt for linear spectral statistics of large-dimensional sample
  covariance matrices.
\newblock {\em Annals of Probability}, 32(1):553--605.

\bibitem[Bai and Silverstein, 2010]{bai2004}
Bai, Z. and Silverstein, J.~W. (2010).
\newblock {\em Spectral analysis of large dimensional random matrices},
  volume~20.
\newblock Springer.

\bibitem[Bai and Zhou, 2008]{baizhou2008}
Bai, Z. and Zhou, W. (2008).
\newblock Large sample covariance matrices without independence structures in
  columns.
\newblock {\em Statistica Sinica}, pages 425--442.

\bibitem[Banna et~al., 2020]{banna2020}
Banna, M., Najim, J., and Yao, J. (2020).
\newblock A clt for linear spectral statistics of large random
  information-plus-noise matrices.
\newblock {\em Stochastic Processes and their Applications}, 130(4):2250--2281.

\bibitem[Bao et~al., 2015a]{baolinpanzhou2015}
Bao, Z., Lin, L.-C., Pan, G., and Zhou, W. (2015a).
\newblock {Spectral statistics of large dimensional Spearman's rank correlation
  matrix and its application}.
\newblock {\em The Annals of Statistics}, 43(6):2588 -- 2623.

\bibitem[Bao et~al., 2015b]{bao2015}
Bao, Z., Pan, G., and Zhou, W. (2015b).
\newblock The logarithmic law of random determinant.
\newblock {\em Bernoulli}, 21(3):1600--1628.

\bibitem[Bickel and Wichura, 1971]{bickel1971}
Bickel, P.~J. and Wichura, M.~J. (1971).
\newblock Convergence criteria for multiparameter stochastic processes and some
  applications.
\newblock {\em The Annals of Mathematical Statistics}, 42(5):1656--1670.

\bibitem[Billingsley, 1968]{billingsley1968}
Billingsley, P. (1968).
\newblock {\em Probability and Measure}.
\newblock Wiley Series in Probability and Statistics: Probability and
  Statistics. John Wiley \& Sons Inc., first edition.

\bibitem[Birke and Dette, 2005]{birke_dette}
Birke, M. and Dette, H. (2005).
\newblock A note on testing the covariance matrix for large dimension.
\newblock {\em Statistics \& Probability Letters}, 74(3):281--289.

\bibitem[Bodnar et~al., 2019]{botdetpar2019}
Bodnar, T., Dette, H., and Parolya, N. (2019).
\newblock {Testing for independence of large dimensional vectors}.
\newblock {\em The Annals of Statistics}, 47(5):2977 -- 3008.

\bibitem[Cai et~al., 2015]{cai2015}
Cai, T.~T., Liang, T., and Zhou, H.~H. (2015).
\newblock Law of log determinant of sample covariance matrix and optimal
  estimation of differential entropy for high-dimensional gaussian
  distributions.
\newblock {\em Journal of Multivariate Analysis}, 137:161--172.

\bibitem[Chen et~al., 2010]{chenzhangzhong2010}
Chen, S.~X., Zhang, L.-X., and Zhong, P.-S. (2010).
\newblock Testing high dimensional covariance matrices.
\newblock {\em Journal of the American Statistical Association}, 105:810--819.

\bibitem[D'Aristotile, 2000]{aristotile2000}
D'Aristotile, A. (2000).
\newblock An invariance principle for triangular arrays.
\newblock {\em Journal of Theoretical Probability}, 13(2):327--341.

\bibitem[D'Aristotile et~al., 2003]{aristotileetal2003}
D'Aristotile, A., Diaconis, P., and Newman, C.~M. (2003).
\newblock Brownian motion and the classical groups.
\newblock {\em IMS Lecture Notes-Monograph Series}, 41:97--116.

\bibitem[Dette and D{\"o}rnemann, 2020]{dettedoernemann2020}
Dette, H. and D{\"o}rnemann, N. (2020).
\newblock Likelihood ratio tests for many groups in high dimensions.
\newblock {\em Journal of Multivariate Analysis}, 178:104605.

\bibitem[Dette and G{\"o}smann, 2020]{dettegoesmann2020}
Dette, H. and G{\"o}smann, J. (2020).
\newblock A likelihood ratio approach to sequential change point detection for
  a general class of parameters.
\newblock {\em Journal of the American Statistical Association},
  115(531):1361--1377.

\bibitem[Dette and Tomecki, 2019]{tomecki}
Dette, H. and Tomecki, D. (2019).
\newblock Determinants of block hankel matrices for random matrix-valued
  measures.
\newblock {\em Stochastic Processes and their Applications},
  129(12):5200--5235.

\bibitem[Fan and Li, 2006]{fanli2006}
Fan, J. and Li, R. (2006).
\newblock Statistical challenges with high dimensionality: feature selection in
  knowledge discovery.
\newblock In Sanz-Sol{\'{e}}, M., Soria, J., Varona, J.~L., and Verdera, J.,
  editors, {\em Proceedings of the International Congress of Mathematicians (p.
  595-622).} European Mathematical Society, Madrid, Spain.

\bibitem[Fisher et~al., 2010]{fisher2010}
Fisher, T.~J., Sun, X., and Gallagher, C.~M. (2010).
\newblock A new test for sphericity of the covariance matrix for high
  dimensional data.
\newblock {\em Journal of Multivariate Analysis}, 101(10):2554--2570.

\bibitem[Gupta and Xu, 2006]{guptaxu2006}
Gupta, A.~K. and Xu, J. (2006).
\newblock On some tests of the covariance matrix under general conditions.
\newblock {\em Annals of the Institute of Statistical Mathematics},
  58:101--114.

\bibitem[Jiang and Yang, 2013]{jiangyang2013}
Jiang, T. and Yang, F. (2013).
\newblock Central limit theorems for classical likelihood ratio tests for
  high-dimensional normal distributions.
\newblock {\em Annals of Statistics}, 41:2029--2074.

\bibitem[Jin et~al., 2014]{jinwangbaietal2014}
Jin, B., Wang, C., Bai, Z.~D., Nair, K.~K., and Harding, M. (2014).
\newblock {Limiting spectral distribution of a symmetrized auto-cross
  covariance matrix}.
\newblock {\em The Annals of Applied Probability}, 24(3):1199 -- 1225.

\bibitem[John, 1971]{john1971}
John, S. (1971).
\newblock Some optimal multivariate tests.
\newblock {\em Biometrika}, 58(1):123--127.

\bibitem[Johnstone, 2006]{Johnstone}
Johnstone, I. (2006).
\newblock High dimensional statistical inference and random matrices.

\bibitem[Jonsson, 1982]{johnson1982}
Jonsson, D. (1982).
\newblock Some limit theorems for the eigenvalues of a sample covariance
  matrix.
\newblock {\em Journal of Multivariate Analysis}, 12(1):1--38.

\bibitem[Ledoit and Wolf, 2002]{ledoit_wolf_2002}
Ledoit, O. and Wolf, M. (2002).
\newblock Some hypothesis tests for the covariance matrix when the dimension is
  large compared to the sample size.
\newblock {\em Annals of statistics}, 30(4):1081--1102.

\bibitem[Li, 2003]{li2003}
Li, Y. (2003).
\newblock A martingale inequality and large deviations.
\newblock {\em Statistics \& Probability Letters}, 62(3):317--321.

\bibitem[Li et~al., 2019]{li2019central}
Li, Z., Wang, Q., and Li, R. (2019).
\newblock Central limit theorem for linear spectral statistics of large
  dimensional kendall's rank correlation matrices and its applications.

\bibitem[Lytova and Pastur, 2009]{lytovapastur2009}
Lytova, A. and Pastur, L. (2009).
\newblock Central limit theorem for linear eigenvalue statistics of the
  {W}igner and sample covariance random matrices.
\newblock {\em Metrika}, 69:153--172.

\bibitem[Mauchly, 1940]{mauchy1940}
Mauchly, J.~W. (1940).
\newblock Significance test for sphericity of a normal {N}-variate
  distribution.
\newblock {\em Annals of Mathematical Statistics}, 11:204--209.

\bibitem[Muirhead, 2009]{muirhead2009}
Muirhead, R.~J. (2009).
\newblock {\em Aspects of Multivariate Statistical Theory}, volume 197.
\newblock John Wiley \& Sons.

\bibitem[Nagel, 2020]{nagel2020}
Nagel, J. (2020).
\newblock A functional {CLT} for partial traces of random matrices.
\newblock {\em Journal of Theoretical Probability}, 98.

\bibitem[Najim and Yao, 2016]{najimyao2016}
Najim, J. and Yao, J. (2016).
\newblock Gaussian fluctuations for linear spectral statistics of large random
  covariance matrices.
\newblock {\em Annals of Applied Probability}, 26(3):1837--1887.

\bibitem[Neuhaus, 1971]{neuhaus1971}
Neuhaus, G. (1971).
\newblock On weak convergence of stochastic processes with multidimensional
  time parameter.
\newblock {\em The Annals of Mathematical Statistics}, 42(4):1285--1295.

\bibitem[Nguyen and Vu, 2014]{nguyen_vu}
Nguyen, H.~H. and Vu, V. (2014).
\newblock Random matrices: Law of the determinant.
\newblock {\em The Annals of Probability}, 42(1):146--167.

\bibitem[Pan, 2014]{pan2014}
Pan, G. (2014).
\newblock Comparison between two types of large sample covariance matrices.
\newblock {\em Annales de l'IHP Probabilit{\'e}s et Statistiques},
  50(2):655--677.

\bibitem[Pan and Zhou, 2008]{panzhou2008}
Pan, G. and Zhou, W. (2008).
\newblock Central limit theorem for signal-to-interference ratio of reduced
  rank linear receiver.
\newblock {\em The Annals of Applied Probability}, 18(3):1232--1270.

\bibitem[Van Der~Vaart and Wellner, 1996]{vandervaart1996}
Van Der~Vaart, A.~W. and Wellner, J.~A. (1996).
\newblock {\em Weak Convergence}.
\newblock Springer.

\bibitem[Wang and Yao, 2013]{wang2013}
Wang, Q. and Yao, J. (2013).
\newblock On the sphericity test with large-dimensional observations.
\newblock {\em Electronic Journal of Statistics}, 7:2164--2192.

\bibitem[Wang et~al., 2018]{wang2018}
Wang, X., Han, X., and Pan, G. (2018).
\newblock The logarithmic law of sample covariance matrices near singularity.
\newblock {\em Bernoulli}, 24(1):80--114.

\bibitem[Yao et~al., 2015]{yao2015}
Yao, J., Zheng, S., and Bai, Z. (2015).
\newblock {\em Large Sample Covariance Matrices and High-Dimensional Data
  Analysis}.
\newblock Cambridge Series in Statistical and Probabilistic Mathematics.
  Cambridge University Press.

\bibitem[Zheng, 2012]{zheng2012}
Zheng, S. (2012).
\newblock Central limit theorems for linear spectral statistics of large
  dimensional $ f $-matrices.
\newblock {\em Annales de l'IHP Probabilit{\'e}s et Statistiques},
  48(2):444--476.

\bibitem[Zheng et~al., 2015]{zheng_et_al_2015}
Zheng, S., Bai, Z., and Yao, J. (2015).
\newblock Substitution principle for clt of linear spectral statistics of
  high-dimensional sample covariance matrices with applications to hypothesis
  testing.
\newblock {\em Annals of Statistics}, 43(2):546--591.

\bibitem[Zheng et~al., 2017]{zhengbaiyao2017}
Zheng, S., Bai, Z., and Yao, J. (2017).
\newblock {CLT for eigenvalue statistics of large-dimensional general Fisher
  matrices with applications}.
\newblock {\em Bernoulli}, 23(2):1130 -- 1178.

\end{thebibliography}
\end{document}